\documentclass[letterpaper,10pt]{amsart}

\newif\iftikz
\tikztrue

\usepackage{comment}

\usepackage{amsmath,amssymb,amsthm,amsfonts,amscd,extpfeil}
\usepackage{hyperref}
\usepackage[dvipdfmx]{graphicx}
\usepackage[small,bf]{caption}
\usepackage[subrefformat=parens]{subcaption}
\usepackage{xcolor,colortbl}
\definecolor{ao}{rgb}{0.275,0.51,0.71}
\definecolor{mi}{rgb}{0,0.51,0.51}
\definecolor{gh}{gray}{1}
\usepackage{upgreek}
\usepackage[T1]{fontenc}   
\usepackage[mathscr]{euscript}
\usepackage{hhline}
\usepackage{marginnote}
\DeclareMathAlphabet{\mathpzc}{OT1}{pzc}{m}{it}
\usepackage{mathbbol,mathtools,calc}
\usepackage{bm}
\usepackage{framed}
\usepackage{musicography}
\usepackage{mathdots} 
\usepackage{blkarray}
\usepackage{longtable,booktabs}

\DeclareMathAlphabet{\mathpzc}{OT1}{pzc}{m}{it}
\DeclareMathAlphabet{\mathdutchcal}{U}{dutchcal}{m}{n}
\SetMathAlphabet{\mathdutchcal}{bold}{U}{dutchcal}{b}{n}

\newcommand{\grc}[1]{{\color{gray} #1}}

\newcommand{\midspan}{~\middle\vert~}

\allowdisplaybreaks
\numberwithin{equation}{section}
\usepackage{ytableau,youngtab}
\usepackage{mathdots}
\usepackage{tikz, tikz-cd}
\usetikzlibrary{
	shapes,
	arrows,
	positioning,
	decorations.markings,
	decorations.pathmorphing,
	circuits.logic.US,
	circuits.logic.IEC,
	fit,
	calc,
	plotmarks,
	matrix
}

\tikzset{
>=stealth',
help lines/.style={dashed, thick},
axis/.style={<->},
important line/.style={thick},
connection/.style={thick, dotted},
punkt/.style={
rectangle,
rounded corners,
draw=black, thick,
text width=4.5em,
minimum height=2em,
text centered,
},
pil/.style={
->,
thick,
gray,
shorten <=2pt,
shorten >=2pt,}
}

\usepackage{array}
\usepackage{adjustbox}
\usepackage{cleveref}
\usepackage{enumitem}
\usepackage{DotArrow}

\usepackage[tmargin=3cm,lmargin=3cm,rmargin=3cm,bmargin=3cm,headheight=1cm,footskip=1.5cm]{geometry}

\usepackage{fancyhdr}

\definecolor{darkred}{rgb}{0.7,0,0} 
\newcommand{\defn}[1]{{\color{darkred}\emph{#1}}} 

\newtheorem{proposition}{Proposition}[section]
\newtheorem{lemma}[proposition]{Lemma}
\newtheorem{corollary}[proposition]{Corollary}
\newtheorem{theorem}[proposition]{Theorem}
\newtheorem*{theorem*}{Theorem}
\newtheorem{conjecture}[proposition]{Conjecture}
\crefname{theorem}{Theorem}{Theorems}
\crefname{proposition}{Proposition}{Propositions}
\crefname{lemma}{Lemma}{Lemmas}
\crefname{corollary}{Corollary}{Corollaries}
\crefname{conjecture}{Conjecture}{Conjectures}

\theoremstyle{definition}
\newtheorem{definition}[proposition]{Definition}

\newtheorem{remark}[proposition]{Remark}
\newtheorem*{remark*}{Remark}
\newtheorem{example}[proposition]{Example}

\crefname{remark}{Remark}{Remarks}
\crefname{example}{Example}{Examples}
\crefname{definition}{Definition}{Definitions}

\crefname{section}{Section}{Sections}
\crefname{subsection}{Section}{Sections}
\crefname{subsubsection}{Section}{Sections}
\crefname{figure}{Figure}{Figures}
\crefname{appendix}{Appendix}{Appendices}

\newcommand{\ao}{\textcolor{ao}}
\newcommand{\re}{\textcolor{red}}

\newcommand{\oddr}{\mathrm{odd}}

\newcommand{\THL}{\mathscr{Q}'} 
\newcommand{\pr}{\mathrm{pr}}

\newcommand{\cRSK}{\mathbf{cRSK}}
\newcommand{\KKR}{\mathrm{KKR}}

\newcommand{\supp}{\operatorname{supp}}

\newcommand{\SST}{\mathrm{SST}}
\newcommand{\HWT}{\mathrm{HWT}}

\newcommand{\BBS}{\mathsf{BBS}}
\newcommand{\RC}{\mathsf{RC}}
\newcommand{\partitions}{\mathbb{Y}}
\newcommand{\signatures}{\mathbb{S}}

\newcommand{\abs}[1]{\lvert #1 \rvert}
\newcommand{\Abs}[1]{\left\lVert #1 \right\rVert}

\newcommand{\fsl}{\mathfrak{sl}}
\newcommand{\asl}{\widehat{\fsl}}

\newcommand{\zero}{\mathbf{0}}
\newcommand{\hw}{\mathrm{hw}}
\newcommand{\lw}{\mathrm{lw}}

\newcommand{\KR}[1]{\mathcal{H}^{#1}}  

\DeclareMathOperator{\wt}{wt}  

\usepackage{xparse}

\makeatletter
\protected\def\specialmergetwolists{%
  \begingroup
  \@ifstar{\def\cnta{1}\@specialmergetwolists}
    {\def\cnta{0}\@specialmergetwolists}%
}
\def\@specialmergetwolists#1#2#3#4{%
  \def\tempa##1##2{%
    \edef##2{%
      \ifnum\cnta=\@ne\else\expandafter\@firstoftwo\fi
      \unexpanded\expandafter{##1}%
    }%
  }%
  \tempa{#2}\tempb\tempa{#3}\tempa
  \def\cnta{0}\def#4{}%
  \foreach \x in \tempb{%
    \xdef\cnta{\the\numexpr\cnta+1}%
    \gdef\cntb{0}%
    \foreach \y in \tempa{%
      \xdef\cntb{\the\numexpr\cntb+1}%
      \ifnum\cntb=\cnta\relax
        \xdef#4{#4\ifx#4\empty\else,\fi\x#1\y}%
        \breakforeach
      \fi
    }%
  }%
  \endgroup
}
\makeatother

\DeclareDocumentCommand\rpp{ m m g }{
	\foreach \x [count=\s from 1] in {#1}{
	        {\ifnum\s=1
	                \draw (0,-\s)--(\x,-\s);
	                \fi}
	   \draw (0,-\s-1) to (\x,-\s-1);
	   \foreach \y in {0, ..., \x} {\draw (\y,-\s)--(\y,-\s-1);}
	}
	\specialmergetwolists{/}{#1}{#2}\ziplist
	\foreach \x/\y [count=\yi from 1] in \ziplist{
	    \node[anchor=west] at (\x,-\yi - .5) {$\y$};  
	}
	\IfValueT {#3}
	{\foreach \z [count=\zi from 1] in {#3} {\node[anchor=east] at (0,-\zi - .5) {$\z$};}}  
	{}
}

\def\C{\mathbb{C}}

\def\i{\infty}

\def\N{\mathbb{N}}

\def\1{\bf{1}}

\def\sfT{\mathsf{T}}

\def\s{\sigma}

\def\Z{\mathbb{Z}}

\def\z{\zeta}

\def\Fomin{\overline{\mathsf{F}}}

 \newcommand{\fulltoday}{\number\day\space \ifcase\month\or
    January\or February\or March\or April\or May\or June\or
    July\or August\or September\or October\or November\or December\fi
    \space\number\year}

\definecolor{UQgold}{RGB}{196, 158, 54} 
\definecolor{UQpurple}{RGB}{73, 7, 94} 
\definecolor{dbluecolor}{rgb}{0.01,0.02,0.7}
\definecolor{dgreencolor}{rgb}{0.2,0.4,0.0}

\usepackage{stmaryrd}
\SetSymbolFont{stmry}{bold}{U}{stmry}{m}{n}

\usepackage[colorinlistoftodos]{todonotes}

\usepackage{silence}
\begin{document}

\title[Skew column RSK dynamics and BBS]{Skew column RSK dynamics and the box-ball system}

\author[T. Imamura]{Takashi Imamura}
\address[T. Imamura]{ 
Department of Mathematics and Informatics, Chiba University, Chiba 263-8522 Japan}
\email{imamura@math.s.chiba-u.ac.jp}

\author[M. Mucciconi]{Matteo Mucciconi}
\address[M. Mucciconi]{ 
 Department of Mathematics, National University of Singapore,
 S17, 10 Lower Kent Ridge Road, 119076, Singapore}
\email{matteomucciconi@gmail.com}

\author[T. Sasamoto]{Tomohiro Sasamoto}
\address[T. Sasamoto]{
Department of Physics, Institute of Science Tokyo, Tokyo 152-8551 Japan}
\email{sasamoto@phys.titech.ac.jp}

\author[T. Scrimshaw]{Travis Scrimshaw}
\address[T. Scrimshaw]{
Department of Mathematics, Hokkaido University, 5 Chome Kita 8 Jonishi, Kita Ward, Sapporo, Hokkaido 060-0808 Japan}
\email{tcscrims@gmail.com}

\begin{abstract}
    The Fomin local rules for Schensted column insertion can be seen as a two-lane box-ball system, in which a carrier moves particles forward or laterally.
    Running such two-lane dynamics in parallel on a periodic lattice gives rise to a two-dimensional generalization of the box-ball system, which we call the \emph{skew column RSK dynamics}.
    Equivalently, this is a deterministic dynamics on pairs of skew semistandard Young tableaux $(P_t,Q_t)_{t \in \Z}$. 
    We prove that this dynamics exhibits solitonic behavior and construct an explicit bijection $(P,Q) \leftrightarrow (H_1,H_2,\kappa,\nu)$ that linearizes the time evolution.
    The resulting coordinates consist of two horizontally weak tableaux $H_1,H_2$ recording the asymptotic soliton data, integer riggings $\kappa$, and a weakly decreasing sequence of integers $\nu$.
    A key feature of the construction is an explicit projection from the skew column RSK dynamics to the classical box-ball system; under this projection, the rigging $\kappa$ is precisely the Kerov--Kirillov--Reshetikhin rigging of the associated box-ball configuration.
    Our proof uses two commuting affine crystal structures on pairs of skew tableaux and a novel connectivity theorem for distinguished subgraphs of tensor products of Kirillov--Reshetikhin crystals.
    We also derive Greene-type formulas for the soliton lengths in terms of last-passage percolation on the associated cylindrical environment.
    Finally, by taking generating functions in the linearizing coordinates, we obtain bijective proofs of Cauchy and Kawanaka--Littlewood-type identities for transformed Hall--Littlewood polynomials.

%
\end{abstract}

\maketitle
\tableofcontents

\section{Introduction}
\label{sec:intro}

\subsection{The box-ball system}
\label{subs:overview}
The box-ball system (BBS) is a discrete deterministic dynamical system on a one dimensional discrete lattice introduced by Takahashi and Satsuma~\cite{Takahashi_BBS,Takahashi_Satsuma} as one of the simplest examples of a cellular automaton exhibiting solitonic behavior.
The states of the BBS are given by
\begin{equation} \label{eq:BBS_state}
    \BBS
    =
    \left\{ B\in\{0,1\}^{\Z} \midspan
    \sum_{x \in \Z} b_x < \infty \right\},
\end{equation}
where $1$'s represent balls and $0$'s represent empty boxes.
The BBS transfer matrix
\begin{equation}
    \sfT\colon \BBS\to\BBS,
\end{equation}
giving the one time step evolution of the dynamics is defined by the following carrier algorithm.
Let $B\in \BBS$ be a configuration.
An initially empty carrier of infinite capacity enters the lattice from far to the left and scans the sites from left to right.
Whenever it encounters an occupied site, it picks up the ball and empties that site; whenever it encounters an empty site while carrying at least one ball, it deposits exactly one ball there.
After the carrier has passed all initially occupied sites and deposited all remaining balls, the resulting configuration is \(\sfT(B)\).
Equivalently, the balls are moved from left to right, each to the nearest available empty site strictly to its right.
Iterating $\sfT$ gives a deterministic and reversible dynamics, where an example is illustrated in \Cref{fig:BBS}.

\begin{figure}
    \centering
    \includegraphics[width=0.7\linewidth]{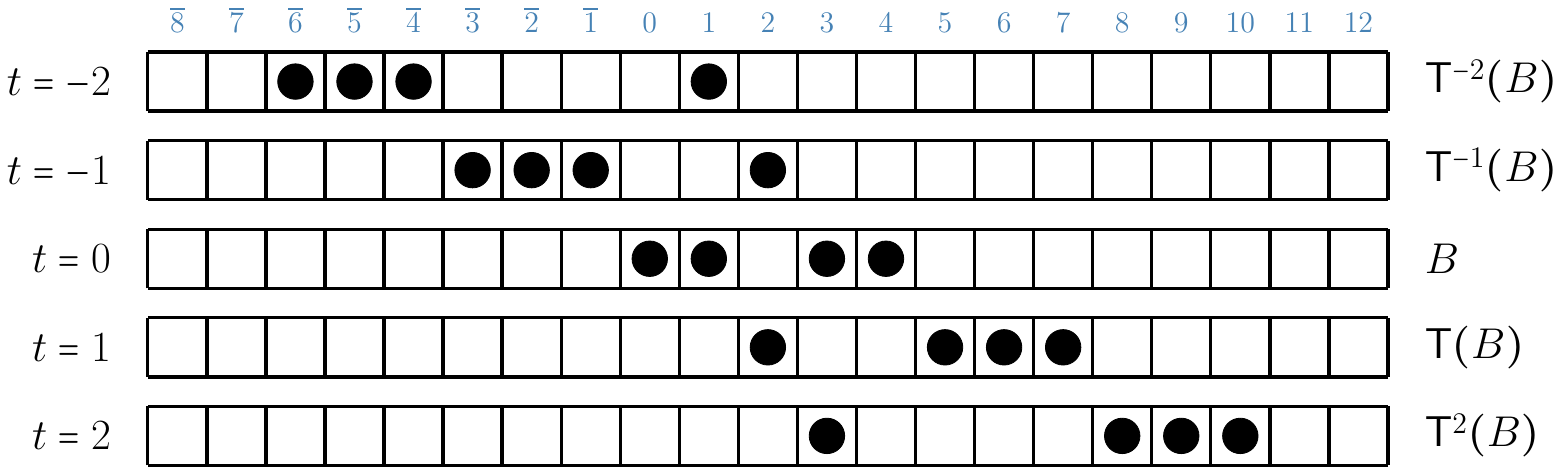}
    \caption{An example of the evolution of the box-ball system. Blue labels on top denote space coordinates; we use the convention $\overline{k} = -k$.}
    \label{fig:BBS}
\end{figure}

The mathematical structures underlying the BBS are remarkably rich and have been studied extensively from the standpoint of combinatorial representation theory, integrable systems, and more recently probability.
Despite its elementary definition, the BBS captures several hallmark features of integrable systems as it is an ultradiscrete limit of the (nonlinear) Korteweg--de Vries (KdV) equation
\[
\frac{\partial\phi}{\partial t} + \frac{\partial^3\phi}{\partial x^3} - 6 \phi \frac{\partial\phi}{\partial x} = 0,
\qquad\qquad
\phi := \phi(x,t).
\]
In more detail, integrability in the KdV equation manifests in the existence of an infinite number of conserved quantities and solitary wave solutions called solitons, which are derived from a Lax formulation~\cite{Lax68} (see also~\cite[S.10.1]{PZ12}), move with speed proportional to their size, and separate after interaction.
To obtain the BBS, we first discretize time, where we obtain the discrete Toda lattice or the discrete modified KdV (see, e.g.,~\cite{HT95}), and then we ultradiscretize (tropicalize) the values $\phi$ (see, e.g.,~\cite{takahashi1997box,tokihiro1996soliton} or~\cite{Inoue-Kuniba-Takagi2012BBS} for a survey).
The BBS retains infinitely many conserved quantities and, as demonstrated in \Cref{fig:BBS}, any configuration under repeated time evolution $\sfT$ eventually decomposes into a collection of (noninteracting) stable clusters of balls which propagate with speeds determined by their sizes (or amplitudes, which play the role of the $\phi$ value) and retain their sizes after interacting (or scattering).
These stable clusters are the \defn{solitons} of the BBS, and under repeated inverse time evolution $\sfT^{-1}$, the state again eventually separates into solitons of the \emph{same} sizes (but in the opposite order).
The nonlinearity of the KdV equation manifests in the phase shift, the change in the positions of the solitons compared to if they did not interact.

Another manifestation of integrability comes from the work of~\cite{HHIKTT01,HKT00} (see also~\cite{Inoue-Kuniba-Takagi2012BBS}), where the dynamics correspond to the crystal limit of the row transfer matrix of the XXZ Heisenberg spin chain~\cite{KS79,KS82}.
More precisely, the states correspond to certain basis elements of a tensor product of the natural 2 dimensional representation of the affine quantum group $U_q(\asl_2)$, known as the spin $\frac{1}{2}$ module, and the carrier corresponds to a fusion module~\cite{KulishReshSkl1981yang} of sufficiently large spin.
By~\cite{KKMMNN92}, these $U_q(\asl_2)$-modules can be encoded combinatorially using Kashiwara's crystal bases~\cite{Kashiwara_crystalizing,Kashiwara_On_Crystal_bases}, and the local interactions of the carrier update rule are precisely the crystal limit of $R$-matrix isomorphism between the tensor products of the large spin and spin $\frac{1}{2}$ modules.

A central result~\cite{KKMMNN92,KKR86,KR86} (see also, \textit{e.g.},~\cite{fukuda_okado_yamada_BBS_energy,HKOTY99,Takagi2005RC} or~\cite{LS19} for a more recent overview) is that the nonlinear BBS dynamics can be linearized through a combinatorial change of coordinates, known as the Kerov--Kirillov--Reshetikhin (KKR) bijection (for $\asl_2$) that can be described by piecewise-linear functions~\cite{KOSTY2006RC}; see also \cite{ferrari2021soliton,mucciconi2024relationships,TNS_BBS} for related linearization procedures.
This transformation maps the nonlinear particle evolution into a simple shift on the angle (or rigging) parameters, providing a discrete analogue of the action-angle transformation for Hamiltonian systems~\cite{KOSTY2006RC,Takagi2005RC}.
This is also considered to be a discrete analog of the quantum inverse scattering transform.
This linearization is well represented by the commutative diagram
\begin{equation} \label{eq:BBS commutative diagram}
    \begin{tikzpicture}[baseline=-2cm]
        \node[] at (2,-1.15) {$B$};
        \draw[->] (2.5,-1.15) -- (5.75,-1.15) node[midway,above] {\scriptsize $\KKR$};
        \draw[->] (2,-1.5) -- (2,-2.5) node[midway,left] {\scriptsize $\sfT$}; 
        \node[] at (6.5,-1.15) {$(\mu, J)$};
        \draw[->] (6.5,-1.5) -- (6.5,-2.5);
        \node[] at (2,-2.85) {$B'$};
        \draw[->] (2.5,-2.85) -- (5.5,-2.85) node[midway,above] {\scriptsize $\KKR$}; 
        \node[] at (6.5,-2.85) {$(\mu, J+\mu)$};
    \end{tikzpicture}
\end{equation}
where the image $(\mu,J)$ of the KKR map consists of a partition $\mu$ encoding the soliton amplitudes (in \Cref{fig:BBS}, $\mu = (3,1)$) and an array of integers $J$ that essentially identifies the location of solitons (in \Cref{fig:BBS}, $J=(-4+3t, -1+t)$).

More recently, the BBS has attracted attention from the viewpoint of non-equilibrium statistical mechanics. The theory of generalized hydrodynamics predicts that integrable systems exhibit atypical hydrodynamic behavior, dominated by persistent quasiparticle modes, whose evolution follow a system of infinitely many coupled partial differential equations; see \cite{Doyon_lecture_notes_GHD} and references therein.
While such predictions remain conjectural for many models, they can be rigorously verified for the BBS thanks to its combinatorial tractability. The stationary and translation invariant measures of the BBS can be explicitly constructed using variants of the KKR bijection, which can be rigorously analyzed in several asymptotic limits \cite{croydon2021generalized,ferrari2021soliton,kuniba2020generalized,olla2024scaling}.


\subsection{Overview}

The goal of this paper is to introduce and study a two-dimensional extension of the BBS. The starting point is a two-lane version of the BBS, defined in \Cref{subs:two_lane_BBS}, in which the carrier transports pairs of particles which are allowed to move laterally between the two lanes. By running these two-lane carrier dynamics in parallel on a periodic planar lattice of width $2n$, we obtain a deterministic particle system that we call the \defn{skew column RSK dynamics} (cRSK). The name reflects the fact that the local particle rule is a formulation of the Fomin growth rule associated with the Robinson--Schensted--Knuth (RSK) correspondence under column Schensted insertion. 

This particle system admits several equivalent geometric and combinatorial representations. It can be represented as an evolving ensemble of non-intersecting lattice paths with $2n$-steps periodicity, as a two-dimensional field of interlacing signatures, or as a sequence of pairs of tableaux with entries at most $n$, illustrated respectively in  \Cref{fig:scattering example}, \Cref{fig:cylinder} and \Cref{fig:RSK dynamics example}.

\begin{figure}
    \centering
    \includegraphics[width=.7\linewidth]{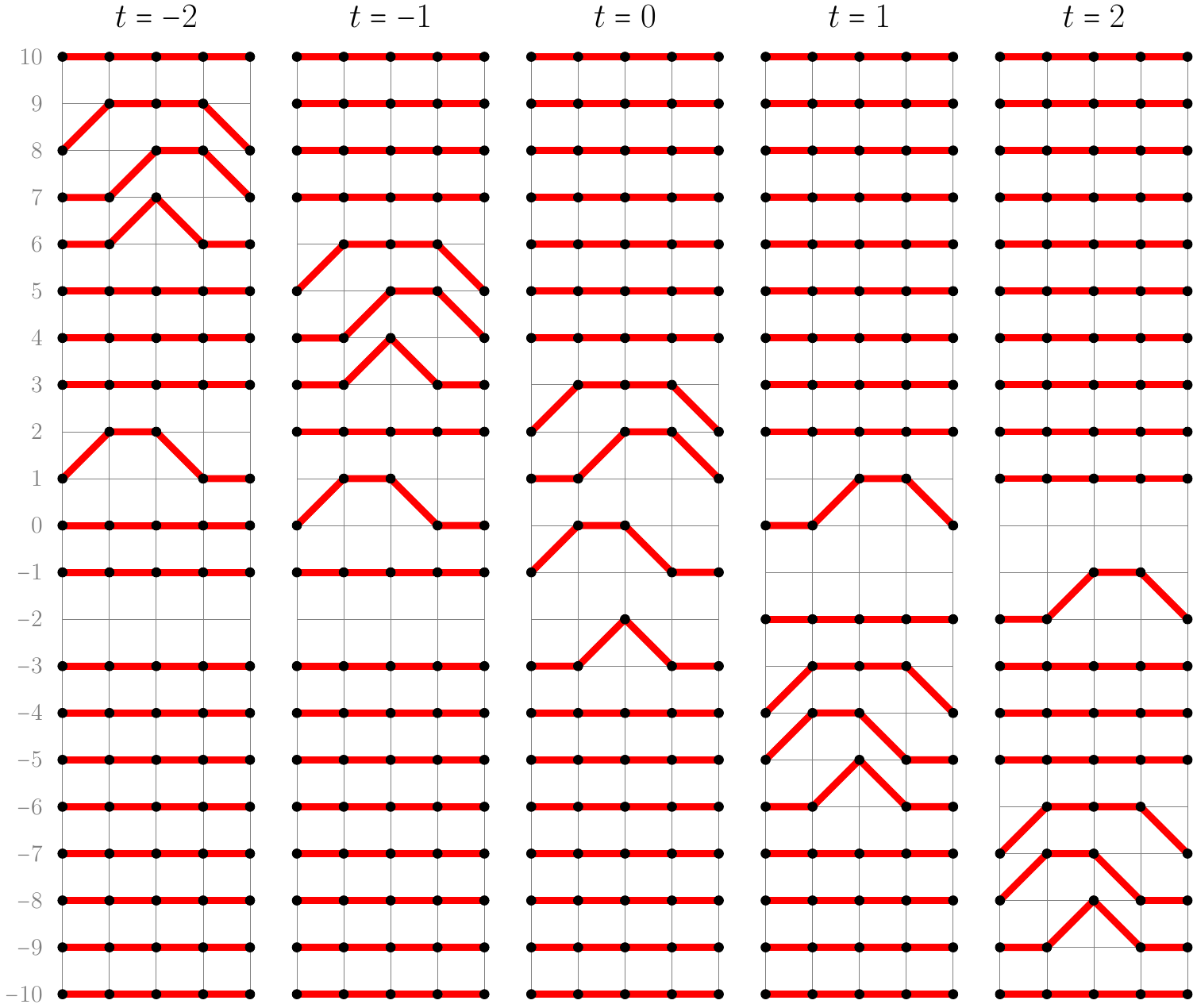}
    \caption{A visualization of the skew column RSK dynamics as an evolving ensemble of non-intersecting lattice paths.} \label{fig:scattering example}
\end{figure}

Two-dimensional analogues of the BBS have appeared in the literature under the name of colored BBS, formulated through higher-rank combinatorial $R$-matrices. The cRSK dynamics is different from these constructions. Its definition does not directly involve an $R$-matrix. Moreover, its state space carries two commuting affine crystal structures, giving it a symmetry structure larger than the one of the ordinary colored BBS; see \Cref{th:crystal_sym}. As the snapshots in \Cref{fig:scattering example} illustrate, the dynamics also admits a natural geometric representation as an ensemble of non-intersecting lattice paths evolving on a cylinder, a viewpoint that could be valuable for probabilistic questions.

A closely related integrable dynamics was defined in \cite{imamura2021skewRSK} by the first three authors under the name of skew RSK dynamics. That dynamics, however, does not generalize the BBS. 
In contrast, the column-insertion version considered here contains the classical BBS as the $n=1$ case, and we believe that the use of column insertion is a more canonical option from a combinatorial standpoint. For instance, column insertion admits generalizations to other root systems and these provide direct description of the combinatorial $R$-matrices defining BBS of other affine Lie types (more generally known as \emph{soliton cellular automata}); e.g.,~\cite{baker2000insertion,HKOT00,HKOT02,lecouvey2002schensted,lecouvey2003schensted}. By contrast, to the best of the authors' knowledge, generalizations of row insertion \cite{BS91,Berele86,Okada93,Sundaram86,Sundaram90II,Sundaram90,Terada93} have not been linked to BBS-type dynamics.

Our main result (\Cref{th:main_bij}) is a linearization algorithm for the cRSK dynamics, generalizing the KKR bijection to this two dimensional setting. This is a combinatorial bijection, which we denote by $\Upsilon_{\mathrm{col}}$, between pairs of semistandard Young tableaux $(P,Q)$, which encode the paths represented in \Cref{fig:scattering example} and quadruples $(H_1,H_2,\kappa,\nu)$. Here $H_1,H_2$ are \emph{horizontally weak tableaux}, which encode the conservation laws of the dynamics, $\nu$ is a signature (a weakly decreasing finite sequence of integers) and $\kappa$ is a list of integers identifying the coordinates of the solitons. The linearization produced by the bijection $\Upsilon_{\mathrm{col}}$ of \Cref{th:main_bij} is well illustrated by the commutative diagram
\begin{equation} \label{eq:cRSK commutative diagram}
    \begin{tikzpicture}[baseline=-2cm]
        \node[] at (2,-1.15) {$(P,Q)$};
        \draw[->] (2.6,-1.15) -- (5.25,-1.15) node[midway,above] {\scriptsize $\Upsilon_{\mathrm{col}}$};
        \draw[->] (2,-1.5) -- (2,-2.5) node[midway,left] {\scriptsize $\cRSK$}; 
        \node[] at (6.5,-1.15) {$(H_1,H_2,\kappa,\nu)$};
        \draw[->] (6.5,-1.5) -- (6.5,-2.5);
        \node[] at (2,-2.85) {$(P',Q')$};
        \draw[->] (2.75,-2.85) -- (5,-2.85) node[midway,above] {\scriptsize $\Upsilon_{\mathrm{col}}$}; 
        \node[] at (6.5,-2.85) {$(H_1,H_2,\kappa+\mu,\nu)$};
    \end{tikzpicture}
\end{equation}
generalizing that of the BBS in \eqref{eq:BBS commutative diagram}. Moreover the cRSK generalizes the BBS dynamics in the sense that there exists a projection $\Pr$ mapping the space of states of the cRSK dynamics onto BBS configurations that lifts the KKR bijection as described by the following commutative diagram
\begin{equation} \label{eq:commutative_diagram_RSK_BBS}
        \begin{tikzpicture}[baseline=-2cm]
        \node[] at (0,0) {$(P, Q)$};
        \draw[->] (.75,0) -- (7.25,0) node[midway,above] {\scriptsize $\Upsilon_{\mathrm{col}}$}; 
        \draw[->] (0,-.35) -- (0,-3.65) node[midway,left] {\scriptsize $\cRSK$};
        \draw[->] (.6,-.35) -- (1.6,-.85) node[midway,xshift=.2cm,yshift=.2cm] {\scriptsize $\mathrm{Pr}$};
        \node[] at (8.5,0) {$(H_1, H_2, \kappa, \nu)$};
        \draw[->] (8.5,-.35) -- (8.5,-3.65);
        \draw[->] (8.5-.6,-.35) -- (8.5-1.6,-.85) node[midway,xshift=-.2cm,yshift=.2cm] {\scriptsize $\mathrm{Pr}$};
        \node[] at (0,-4) {$(P', Q')$};
        \draw[->] (1,-4) -- (7,-4) node[midway,above] {\scriptsize $\Upsilon_{\mathrm{col}}$}; 
        \draw[->] (.6,-4+.35) -- (1.6,-4+.85) node[midway,xshift=-.2cm,yshift=.2cm] {\scriptsize $\mathrm{Pr}$};
        \node[] at (8.5,-4) {$(H_1, H_2, \kappa+\mu, \nu)$};
        \draw[->] (8.5-.6,-4+.35) -- (8.5-1.6,-4+.85) node[midway,xshift=.2cm,yshift=.2cm] {\scriptsize $\mathrm{Pr}$};
        \node[] at (2,-1.15) {$B$};
        \draw[->] (2.5,-1.15) -- (5.75,-1.15) node[midway,above] {\scriptsize $\KKR$};
        \draw[->] (2,-1.5) -- (2,-2.5) node[midway,left] {\scriptsize $\sfT$}; 
        \node[] at (6.5,-1.15) {$(\mu, J)$};
        \draw[->] (6.5,-1.5) -- (6.5,-2.5);
        \node[] at (2,-2.85) {$B'$};
        \draw[->] (2.5,-2.85) -- (5.5,-2.85) node[midway,above] {\scriptsize $\KKR$}; 
        \node[] at (6.5,-2.85) {$(\mu, J+\mu)$};
    \end{tikzpicture}
\end{equation}
where $\kappa = J+\mu$; see \Cref{thm:projection BBS}.

Beyond dynamics, the linearization scheme also yields new combinatorial consequences. In particular it provides bijective proofs of summation identities for the class of symmetric polynomials called \defn{transformed Hall--Littlewood polynomials} $\THL_{\mu}(x; q)$ (\Cref{def:transformed_Hall_Littlewood}); see \Cref{thm:Cauchy and Littlewood id} and \Cref{thm:refined Cauchy and Littlewood id}.
Related special functions have previously appeared in the study of the BBS \cite{KS10,KMOTU00,Nakayashiki_Yamada}, but using such connection as a tool to establish summation identities has been limited to date.
On the other hand, transformed Hall--Littlewood polynomials arise in probabilistic context as they describe the law of certain models of randomly growing interfaces \cite{aggarwal_borodin_wheeler_tPNG,BufetovMucciconiPetrov2018}. As a result these new summation identities provide probabilistic insights about these solvable growth processes; see \cite{He2023boundaryASEP,IMS2022KPZ_free_fermions,imamura_periodic_PNG_FPSAC} for related results.

\subsection{A two-lane box-ball map} \label{subs:two_lane_BBS}

The two-dimensional dynamics we study in this paper are built by composing a simple local update acting on BBS configurations on two adjacent lanes. We call this the \defn{two-lane box-ball map}
\begin{equation} \label{eq:two_lane_transfer_matrix}
    \mathsf{F} \colon (A,B) \in \BBS \times \BBS \longrightarrow (A',B') \in \BBS \times \BBS
\end{equation}
and it is defined as follows.

Let $A=(a_x)_{x\in \Z},B=(b_x)_{x\in \Z} \in \BBS$ be two particle configurations. 
An initially empty carrier enters the lattice $\mathbb{Z}$ from far to the left and scans the sites from left to right. The carrier transports pairs of balls. For $x \in \mathbb{Z}$, let $c_x \in \N$ denote the number of pairs of balls carried immediately after the carrier has passed site $x$. Suppose that the carrier arrives at site $x$ carrying $c_{x-1}$ pairs. If both boxes at site $x$ are occupied, the carrier removes the two balls and adds one pair to its load. If exactly one of the two boxes is occupied, the carrier swaps the occupation between the two lanes. Finally, if both boxes are empty, the carrier deposits one ball in each lane, provided that $c_{x-1}>0$. After the carrier has passed all initially occupied sites and deposited all remaining pairs of balls, the resulting configuration is $(A',B')= \mathsf{F}(A,B)$. For an example, see \Cref{fig:two_lane_BBS_example}, left panel.

\begin{figure}
    \centering
    \includegraphics[width=0.9\linewidth]{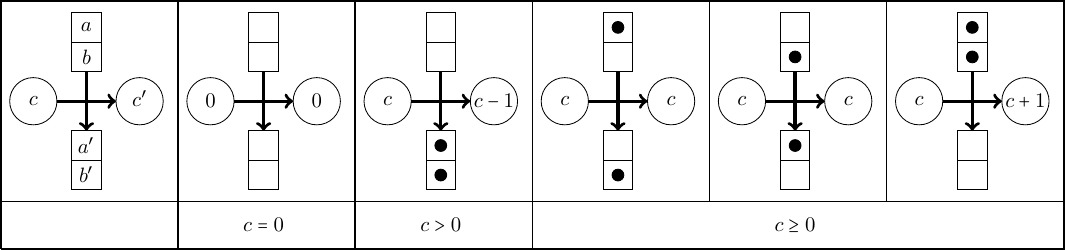}
    \caption{The local update $\mathsf{f}(a,b,c)=(a',b',c')$ in the two-lane BBS.}
    \label{fig:two-lane R-matrix}
\end{figure}

Equivalently, define
\begin{equation} \label{eq:f Fomin}
    \mathsf{f} \colon \{0,1\} \times \{0,1\} \times \N
    \longrightarrow
    \{0,1\} \times \{0,1\} \times \N,
    \qquad
    \mathsf{f}(a,b,c) = (a',b',c'),
\end{equation}
where
\begin{equation} \label{eq:f Fomin_1}
    c'=\max\{0,a+b+c-1\},
    \qquad
    a'=b+c-c',
    \qquad
    b'=a+c-c'.
\end{equation}
See \Cref{fig:two-lane R-matrix} for a representation of the map $\mathsf{f}$. Then, the states $A'=(a'_x)_{x\in \Z}$, $B'=(b'_x)_{x\in \Z}$ are given by
\[
    (a_x',b_x',c_x)
    =
    \mathsf{f}(a_x,b_x,c_{x-1})
    \qquad \text{for all } x\in \mathbb{Z},
    \qquad
    c_{-\infty}=0.
\]
It will be useful to record by \defn{$c_0(A,B)$} the carrier load immediately after the carrier has crossed the origin during
the update of the pair $(A,B)$.

\begin{remark} \label{rem:from_two_lane_to_BBS}
    The classical BBS transfer matrix $\sfT$ is recovered from $\mathsf{F}$ when the configurations in the two lanes are identical. In other words $B'=\sfT(B)$ if and only if $(B',B') = \mathsf{F} (B,B)$.
\end{remark}

\begin{figure}
    \centering
    \begin{minipage}{0.45\textwidth}
        \centering
        \includegraphics[width=0.6\textwidth]{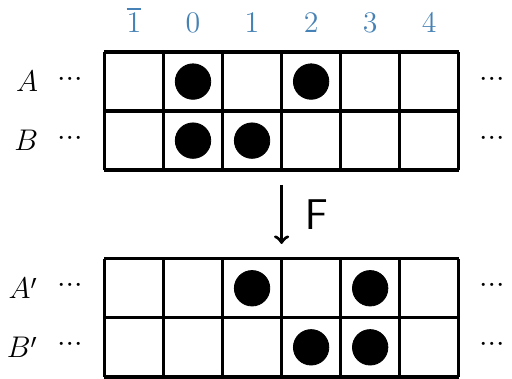}
    \end{minipage}
    \quad
    \begin{minipage}{0.45\textwidth}
        \centering
        \includegraphics[width=.7\linewidth]{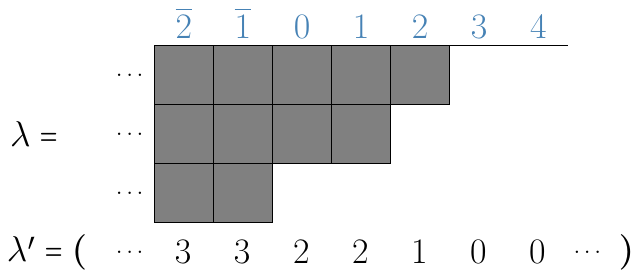}
    \end{minipage}

    \caption{{\bf Left:} A two-lane BBS update $(A,B) \to (A',B')=\mathsf{F}(A,B)$. In this case $c_0(A,B)=1$. {\bf Right:} The diagram of the signature $\lambda = (2,1,-1)$ and its conjugate $\lambda'$.}
    \label{fig:two_lane_BBS_example}
\end{figure}

Compared with the classical BBS transfer matrix, the two-lane BBS map allows particles to move laterally: whenever the carrier encounters a ball-hole pair, the ball is transferred to the other lane. This makes it possible to build nontrivial multi-lane dynamics by applying the two-lane map \(\mathsf F\) in a staggered fashion to adjacent pairs of lanes. For instance, on a collection of lanes arranged cyclically, one may update in parallel all pairs \((1,2),(3,4),\dots\) at odd times and all pairs \((2,3),(4,5),\dots\) at even times, with the last and first lanes also paired when periodic boundary conditions are imposed in the vertical direction. In this way one obtains a genuinely two-dimensional dynamics in which particles may move both longitudinally and laterally; see \Cref{fig:four_lane_BBS_example} for an illustration on four lanes. 

\begin{figure}
    \centering
    \includegraphics[width=\linewidth]{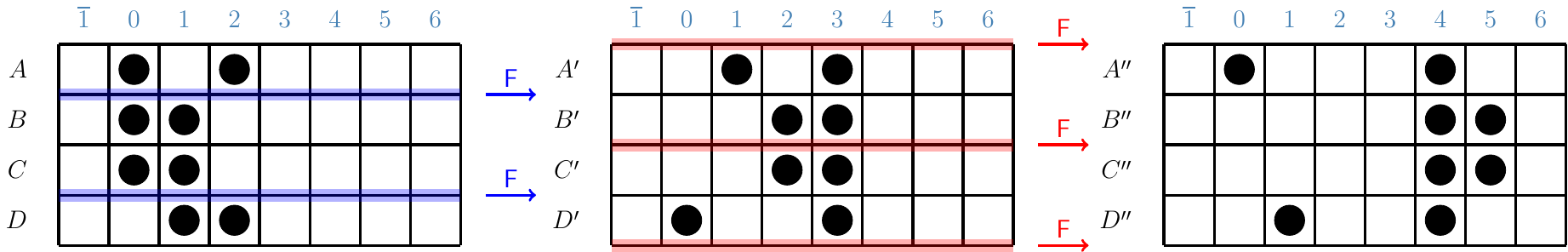}
    \caption{Two time steps of the staggered two-lane box-ball dynamics on four lanes with periodic boundary conditions in the vertical direction, so that the top and bottom lanes are adjacent. At each step the two-lane BBS map $\mathsf{F}$ is applied in parallel to disjoint adjacent pairs of lanes: at odd times to the pairs indicated by the blue shaded lines, and at even times to the pairs indicated by the red shaded lines. Starting from the configuration $(A,B,C,D)$, this gives $(A',B',C',D')$ after one step and $(A'',B'',C'',D'')$ after two steps.}
    \label{fig:four_lane_BBS_example}
\end{figure}

\subsection{Fields of interlacing signatures on a cylindrical lattice} 

It will be convenient to encode BBS particle configurations as differences of interlacing signatures and to translate the two-lane BBS update into an update of pairs of signatures. In this language, the staggered multi-lane dynamics discussed above is naturally encoded by a field of signatures on a cylindrical lattice. We therefore begin by introducing the relevant notions.


\medskip

A \defn{signature} of length $\ell \ge 0$ is a weakly decreasing sequence of integers $\lambda = (\lambda_1 \ge \cdots \ge \lambda_\ell)$. We denote by $\mathbb{S}_\ell$ the set of signatures of length $\ell$, and set $\mathbb{S} = \bigsqcup_{\ell \ge 0} \mathbb{S}_\ell$. The \defn{conjugate}\footnote{With this convention, $\lambda'$ is not itself a signature, since it is a sequence indexed by $\mathbb{Z}$.} of a signature $\lambda$ is the sequence $\lambda'=(\lambda_i')_{i\in\mathbb{Z}}$ defined by
\[
    \lambda_i' = \#\{j \mid \lambda_j \ge i\}.
\]
See \Cref{fig:two_lane_BBS_example}, right panel, for an example. Two signatures $\kappa,\mu \in \mathbb{S}_\ell$ of the same length are said to \defn{interlace}, which we write as $\kappa \preceq \mu$, by the equivalent conditions:
\begin{equation}
    \kappa \preceq \mu
    \qquad \Longleftrightarrow \qquad
    \mu' - \kappa' \in \BBS
    \qquad \Longleftrightarrow \qquad
    \mu_1 \ge \kappa_1 \ge \mu_2 \ge \kappa_2 \ge \cdots.
\end{equation}
Partitions are the special case of signatures with non-negative entries and their set is denoted by $\partitions$.

We now use the two-lane BBS to define a local growth rule for signatures. Let
$\kappa,\lambda,\mu$ be signatures satisfying $\lambda\succeq\kappa\preceq\mu$. The two increments $\mu'-\kappa'$, $\lambda'-\kappa'$ are BBS configurations, and we regard them as input for the two-lane BBS. Recalling the map $\mathsf F$ of \eqref{eq:two_lane_transfer_matrix}, we define the \defn{Fomin growth operator}
\begin{equation}
\label{eq:fomin-operator}
\Fomin(\kappa,\mu,\lambda)=\nu,
\qquad\text{where}\qquad
(\nu'-\mu',\nu'-\lambda')
=
\mathsf F(\mu'-\kappa',\lambda'-\kappa')
.
\end{equation}
For a graphical representation see \Cref{fig:cylinder}, left panel. For example, using the two-lane BBS update of \Cref{fig:two_lane_BBS_example}, left panel we can compute 
    \[
        \nu = \Fomin (\kappa,\mu,\lambda) \qquad \text{with} \qquad \kappa = (2,1,-1), \quad \mu = (2,2,0), \quad \lambda= (2,1,1), \quad \nu = (3,2,1),
    \]
since $A= \mu'-\kappa'$, $B= \lambda'-\kappa'$, $A'=\nu'-\mu'$, $B'=\nu'-\lambda'$. 
The operator $\Fomin$ is invertible in the following sense: Given signatures $\lambda,\mu,\nu$ satisfying $\lambda \preceq \nu \succeq \mu$, there exists a unique signature $\kappa$ such that $\Fomin(\kappa,\mu,\lambda)=\nu$. This is an immediate consequence of the fact that the map $\mathsf{F}$ is a bijection.

We proceed by imposing periodic boundary conditions. For any $n \ge 1$, define the cylindrical lattice
\begin{equation}
\label{eq:cylinder}
\mathscr C_n=\mathbb Z^2/\sim,
\qquad
(x,y)\sim(x',y')
\Longleftrightarrow
(x',y')=(x+kn,y-kn)
\quad\text{for some }k\in\mathbb Z.
\end{equation}
Similarly we define the dual lattice $\mathscr{C}_n' = (\Z')^2 / \sim$, which is the lattice of faces of $\mathscr{C}_n$ and $\Z' = \Z + \frac{1}{2}$.
Let $\mathbf{e}_1=(1,0)$ and $\mathbf{e}_2=(0,1)$. A \defn{simple loop} (or down-right loop) on $\mathscr C_n$ is a path
\begin{equation}
    \mathbf{p}= (p_0 , p_1, \dots,  p_{2n} ) \in \mathscr{C}_n^{2n+1},
    \qquad \text{such that}
    \qquad p_{i+1} - p_i \sim  \mathbf{e}_1, -\mathbf{e}_2
    \quad \text{and} \quad p_{2n} \sim p_0.
\end{equation}

\begin{definition}[Fomin field of interlacing signatures]
\label{def:fomin_fields}
A \defn{Fomin field} of interlacing signatures of length $\ell$ on $\mathscr C_n$ is a map $\boldsymbol{\lambda}\colon \mathscr{C}_n \to \mathbb{S}_\ell$ such that for every $p\in\mathscr C_n$,
\begin{equation}
    \boldsymbol{\lambda}(p) \preceq \boldsymbol{\lambda}(p+\mathbf{e}_i), \quad\text{for } i\in\{ 1,2\}, \qquad \text{and} \qquad \boldsymbol\lambda(p+\mathbf e_1+\mathbf e_2)
    =
    \Fomin \left( \boldsymbol\lambda(p), \boldsymbol\lambda(p+\mathbf e_1), \boldsymbol\lambda(p+\mathbf e_2) \right).
\end{equation}
We associate to a Fomin field $\boldsymbol{\lambda}$ an \defn{environment} $\mathbf{c}\colon \mathscr{C}_n'\to\N$ as follows. If $p'=p+(\mathbf{e}_1+\mathbf{e}_2)/2$, then 
\[
    \mathbf{c}(p') = c_0 (\boldsymbol{\lambda}'(p+\mathbf{e}_1)-\boldsymbol{\lambda}'(p), \boldsymbol{\lambda}'(p+\mathbf{e}_2)-\boldsymbol{\lambda}'(p)),
\]
where $c_0(A,B)$ was defined in \Cref{subs:two_lane_BBS}.
\end{definition}

\begin{figure}
    \centering
    \begin{minipage}{0.2\textwidth}
        \centering
        \includegraphics[width=0.9\textwidth]{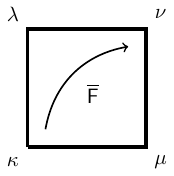}
    \end{minipage}
    \qquad
    \begin{minipage}{0.7\textwidth}
        \centering
        \includegraphics[width=1\linewidth]{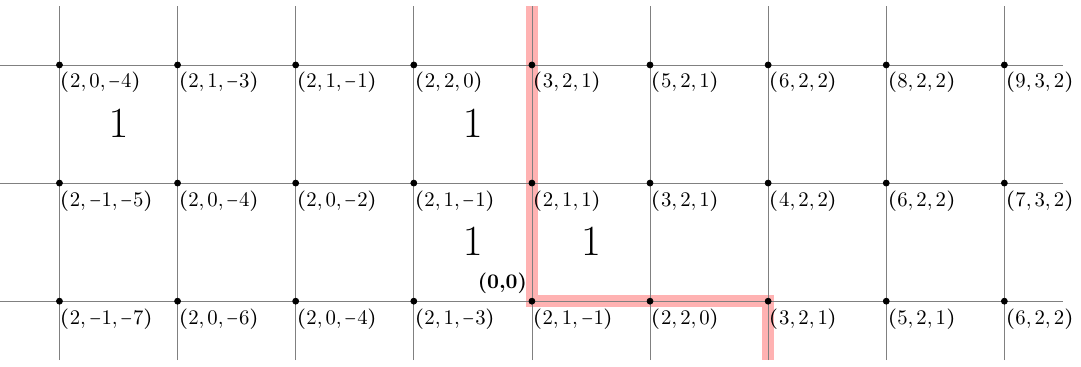}
    \end{minipage}

    \caption{{\bf Left:} A depiction of the Fomin growth operator. {\bf Right:} A Fomin field on the cylindrical lattice $\mathscr{C}_2$. The red shaded path is the simple loop $\mathbf{p} = ((0,2),(0,1),(0,0),(1,0),(2,0)) $. The non-zero values of the environment $\mathbf{c}(p')$ are drawn on faces of $\mathscr{C}_2$.}
    \label{fig:cylinder}
\end{figure}

An example of a Fomin field of interlacing signatures on $\mathscr{C}_n$ is shown in \Cref{fig:cylinder}; the values of the corresponding environment are drawn in the faces of the lattice. 
Since the Fomin operator is invertible, specifying the values of a Fomin field on any simple loop $\mathbf{p}\subset \mathscr{C}_n$ determines all neighboring values recursively, both forward and backward. Thus any compatible interlacing data along such a loop extends uniquely to a Fomin field on the whole cylinder.

Choosing either the rightward or the upward direction on $\mathscr{C}_n$ as a
time direction turns the same construction into a dynamical system on
sequences of signatures. This point of view is illustrated in
\Cref{fig:scattering example}. There, for the cylinder $\mathscr{C}_2$ shown
in \Cref{fig:cylinder}, we consider the simple loops
\[
    \mathbf{p}_t
    =
    \bigl(
        p_0^t,p_1^t,p_2^t,p_3^t,p_4^t
    \bigr),
    \qquad
    t\in\{-2,-1,0,1,2\},
\]
where
\begin{equation}
    p_0^t=(2t,2),\qquad
    p_1^t=(2t,1),\qquad
    p_2^t=(2t,0),\qquad
    p_3^t=(2t+1,0),\qquad
    p_4^t=(2t+2,0).
\end{equation}
The signatures along these loops are represented by the point configurations
\begin{equation*}
    \begin{aligned}
        &\quad \left\{ \bigl(j,\boldsymbol{\lambda}_i'(p_j^t)-i\bigr) \midspan i\in\Z \right\}, \quad j\in \{0,\dots,4\},
        \\[1em]
        &
        \qquad \qquad \qquad t\in\{-2,\dots,2\}
    \end{aligned}
    \qquad \qquad
    \vcenter{\hbox{\includegraphics[width=.25\linewidth]{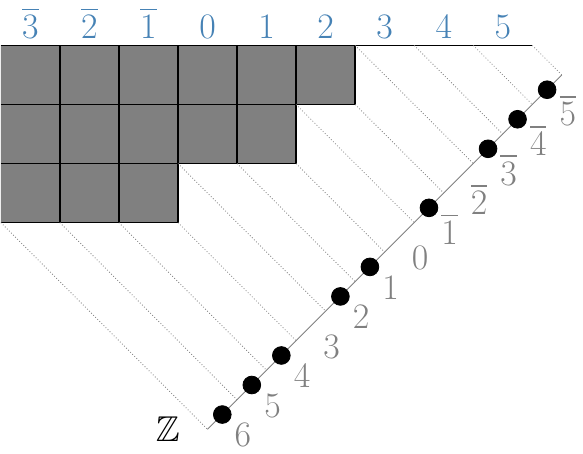}}}.
\end{equation*}
On the side we depicted the point configuration $\{\boldsymbol{\lambda}_i'(p^0_0)-i \mid i\in\mathbb{Z}\}$, where $p^0_0=(0,0)$ and $\boldsymbol{\lambda}(0,0) = (2\ge 1\ge -1)$. 

A continuous limit of the dynamics illustrated in \Cref{fig:scattering example} was introduced in \cite{imamura_periodic_PNG_FPSAC} under the name of ``periodic multi-layer polynuclear growth''. In such special limit the dynamical rules admit a simpler direct description which bypasses the two-lane BBS.

The example in \Cref{fig:scattering example} suggests that the resulting dynamics on sequences of signatures exhibits solitonic behavior analogous to that of the BBS. This solitonic behavior is reflected in the following stabilization phenomenon.

\begin{theorem}[Stabilization phenomenon and conservation laws]
\label{prop:stabilization intro}
    Let $\boldsymbol{\lambda}$ be a Fomin field of interlacing signatures of length $\ell$ on $\mathscr{C}_n$. Then there exists an integer partition $\mu=(\mu_1,\dots,\mu_{\ell'})$, with $\ell'\le\ell$ and a time $t^*>0$ such that, for all
    $t\ge t^*$,
    \begin{equation}
        \boldsymbol{\lambda}(t+n,i)
        =
        \boldsymbol{\lambda}(t,i)+(\mu_1,\dots, \mu_{\ell'},0\dots,0),
        \qquad
        \boldsymbol{\lambda}(-t-n,i)
        =
        \boldsymbol{\lambda}(-t,i)-(0,\dots,0,\mu_{\ell'}, \dots,\mu_1).
    \end{equation}
\end{theorem}

\Cref{prop:stabilization intro} contains two statements.
The first one, which is relatively elementary, is that the forward and backward asymptotic signatures of a Fomin field have constant increments.
The second, and substantially more delicate one, is that the two asymptotic increments coincide up to reversal.
This second statement is \Cref{thm:scattering_rules} in the text.

To describe more explicitly the solitonic features of Fomin fields and their relation with the BBS, we next introduce the cRSK dynamics.

\subsection{The skew column RSK dynamics}

Recall that given signatures $\rho \subset \lambda$ (i.e.~$\rho_i \le \lambda_i$ for all $i$), a semistandard Young tableau of skew shape $\lambda/\rho$ is a sequence of $n+1$ interlacing signatures
\begin{equation} \label{eq:semistandard-tab}
    P = (\rho = \lambda^{(0)} \preceq \lambda^{(1)} \preceq \cdots \preceq \lambda^{(n)} = \lambda).
\end{equation}
Equivalently, $P$ is a filling of the cells of the Young diagram $\lambda/\rho$ in which the cells $(j,k)$ such that $\lambda_j^{(i-1)} < k \le \lambda_j^{(i)}$ are assigned the label $i$, for each $i=1,\dots,n$. This filling is weakly
increasing along rows and strictly increasing along columns. We denote by \defn{$\mathrm{SST}(\lambda/\rho,n)$} the set of semistandard Young tableaux of skew shape $\lambda/\rho$ with labels in $\{1,\dots,n\}$.

    
The following is the main definition of this paper.

    \begin{definition}[Skew column RSK dynamics] \label{def:cRSK}
        Fix a pair $(P,Q)$ of semistandard Young tableaux of skew shape $\lambda/\rho$
        \begin{equation} \label{eq:tableauxPQ}
            P = (\rho = \lambda^{(0)} \preceq \lambda^{(1)} \preceq  \cdots \preceq \lambda^{(n)}= \lambda), \qquad Q = (\rho = \mu^{(0)} \preceq \mu^{(1)} \preceq  \cdots \preceq \mu^{(n)}= \lambda) 
        \end{equation}
        and construct the Fomin field $\boldsymbol{\lambda}=\boldsymbol{\lambda}[P,Q]$ on $\mathscr{C}_n$ such that
        \begin{equation} \label{eq:tableaux_from_loop}
            \boldsymbol{\lambda}(i,0) = \lambda^{(i)}, \qquad \boldsymbol{\lambda}(0,i) = \mu^{(i)},
            \qquad
            \text{for } i=0,\dots,n.
        \end{equation}
        The \defn{skew RSK column dynamics} (or cRSK dynamics) with initial conditions $(P,Q)$ is the sequence of pairs of semistandard Young tableaux $(P_t,Q_t)_{t\in \mathbb{Z}}$ of same skew shape such that 
        \begin{equation*}
            P_t = (\boldsymbol{\lambda}(tn,0) \preceq \boldsymbol{\lambda}(tn+1,0) \preceq  \cdots \preceq \boldsymbol{\lambda}(tn+n,0)), \qquad Q_t = (\boldsymbol{\lambda}(tn,0) \preceq \boldsymbol{\lambda}(tn,1) \preceq  \cdots \preceq \boldsymbol{\lambda}(tn,n)).
        \end{equation*}
        See \Cref{fig:RSK dynamics example} for an example.
    \end{definition}

    \begin{figure}
        \centering
        \includegraphics[width=.9\linewidth]{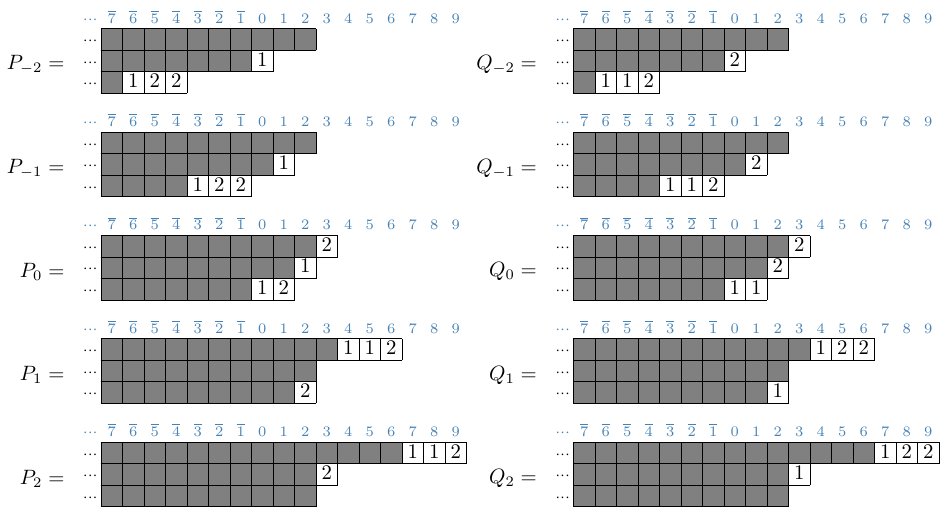}
        \caption{An example of evolution of the cRSK dynamics.}
        \label{fig:RSK dynamics example}
    \end{figure}

    In \Cref{sec:cRSK}, we give an alternative description of the cRSK dynamics in terms of Schensted column insertion, which is sometimes more convenient for proving properties of the system.

    \begin{remark} \label{rem:RSK to BBS}
        For $n=1$, the cRSK dynamics become the classical BBS. Indeed, if $\kappa \preceq \lambda$ and $\nu=\Fomin(\kappa,\lambda,\lambda)$, then the configurations $B=\lambda'-\kappa'$, $B'=\nu'-\lambda'$ satisfy $B'=\mathsf{T}(B)$ as described in \Cref{rem:from_two_lane_to_BBS}.
    \end{remark}

A consequence of the stabilization phenomenon of \Cref{prop:stabilization intro} of Fomin fields is that, asymptotically, pairs of tableaux $(P_t,Q_t)$ evolving according to the cRSK dynamics evolve linearly.
Namely, for large enough $t$, $P_{t+1}$ (resp.~$Q_{t+1}$) will be obtained from $P_t$ (resp.~$Q_t$) by translating its partitions (or labeled cells) by $\mu$.
This is evident in the example of \Cref{fig:RSK dynamics example}, where $P_2$ (resp.~$Q_2$) is obtained from $P_1$ (resp.~$Q_1$) by rigidly shifting the first and second row by $3$ and $1$ boxes to the right.
Similarly $P_{-2}$ (resp.~$Q_{-2}$) is obtained from $P_{-1}$ (resp.~$Q_{-1}$) by rigidly shifting the last and second last row by $3$ and $1$ boxes to the left.

We now extract from the asymptotic regime a pair of tableaux encoding the soliton data. Given a partition $\mu$, a \defn{horizontally weak tableau} of shape $\mu$ is a filling of the Young diagram of $\mu$ whose entries are weakly increasing along rows, with no condition imposed along columns.
We denote by \defn{$\HWT(\mu,n)$} the set of horizontally weak tableaux of shape $\mu$ with labels in $\{1,\dots,n\}$.
    
\begin{definition}[Soliton data] \label{def:Phi}
    Given a pair $(P,Q)$ of semistandard Young tableaux of the same skew shape, define the map 
    \begin{equation}
        \Phi \colon (P,Q) \in \bigsqcup_{\lambda,\rho \in \mathbb{S}} \mathrm{SST} ( \lambda/\rho , n)^2 \longrightarrow (H_1, H_2) \in \bigsqcup_{\mu\in\partitions} \HWT(\mu,n)^2,
    \end{equation}
    where, letting $\boldsymbol{\lambda} = \boldsymbol{\lambda}[P,Q]$, we set
    \begin{subequations}
    \begin{align}
        \label{eq:H_1 limit}
        \#\{ i\text{-cells at row }j \text{ of }H_1 \} & = \lim_{t\to \infty} \boldsymbol{\lambda}_j(t,i)-\boldsymbol{\lambda}_j(t,i-1),
        \\
        \label{eq:H_2 limit}
        \#\{ i\text{-cells at row } j \text{ of }H_2 \} & = \lim_{t\to \infty} \boldsymbol{\lambda}_j(t+i,0)-\boldsymbol{\lambda}_j(t+i-1,0). 
    \end{align}
    \end{subequations}
\end{definition}

For example, setting $(P,Q)=(P_0,Q_0)$ as in \Cref{fig:RSK dynamics example}, we have
\begin{equation}
    \ytableausetup{aligntableaux = center,smalltableaux}
    \Phi(P,Q) = (H_1,H_2) \qquad \text{with} \qquad H_1 = \begin{ytableau} 
        1 & 1 & 2 \\ 2
    \end{ytableau}
    ,
    \quad
    H_2 = \begin{ytableau} 
        1 & 2 & 2 \\ 1
    \end{ytableau}
    ~
    .
\end{equation}

\subsection{Main results}

The main result of this paper is an explicit linearization algorithm for the
cRSK dynamics. The construction relates the dynamics introduced
above to the classical BBS, and identifies the asymptotic soliton
data as part of a complete set of coordinates.

Let $\N = \{0, 1, \ldots\}$ be the set of natural numbers, including 0. To state our main result, we introduce the \defn{dual energy} $D' \colon \mathrm{HWT}(\mu,n)\rightarrow \mathbb{N}$ of a horizontally weak tableau. It is a natural statistic originating in the study of lattice models (see, e.g.,~\cite{Kang_Kashiwara_et_al} for some history) and its definition is given in \Cref{def:energy_function}. The energy has the property that $D'(H) = 0$ if and only if the horizontally weak tableau $H$ satisfies the semistandard property; i.e., if the filling is also strictly increasing column-wise. On the other hand if $H$ does not have the semistandard property, we have $D'(H)>0$ and $D'(H)$ assumes maximal value if the tableau $H$ only has entries equal to 1. We also introduce the sets
\begin{subequations}
\begin{align} \label{eq:tKappa}
    \mathcal{K}(\mu) := \{\kappa=(\kappa_1,\dotsc,\kappa_{\ell(\mu)})\in\Z^{\ell(\mu)} \mid \kappa_i\ge \kappa_{i+1} \text{~if~} \mu_i=\mu_{i+1}\},
\allowdisplaybreaks \\
    \label{eq:tKappa+}
    \mathcal{K}_+(\mu) := \{\kappa=(\kappa_1,\dotsc,\kappa_{\ell(\mu)})\in\N^{\ell(\mu)} \mid \kappa_i\ge \kappa_{i+1} \text{~if~} \mu_i=\mu_{i+1}\}.
\end{align}
\end{subequations}

\begin{theorem} \label{th:main_bij}
There exists an explicit bijection
\begin{equation} \label{eq:Upsilon}
    \Upsilon_{\mathrm{col}} \colon (P,Q) \in \bigsqcup_{\lambda,\rho \in \signatures} \mathrm{SST}(\lambda/\rho, n)^2
    \longrightarrow (H_1,H_2,\kappa,\nu) \in
    \bigsqcup_{\mu\in\partitions} \HWT(\mu,n)^2 \times \mathcal{K}(\mu) \times \signatures
\end{equation}
with the following properties: 
\begin{gather}
    \label{eq:PQHH}
    (H_1, H_2) = \Phi(P,Q),
    \\
    \label{eq:volume_pre}
    \abs{\rho} = D'(H_1) + D'(H_2) + \abs{{\kappa}} + \abs{\nu},
    \\
    \label{eq:length_pre}
    \ell(\lambda) = \ell(\mu)+\ell(\nu),
\end{gather}
where $\ell(\mu)$ is the number of positive entries of the partition $\mu$ and\footnote{notice that if $\rho$ is a signature, then $|\rho|$ is not necessarily positive.} $\abs{\rho} = \sum_i\rho_i$.
Moreover, $\rho, \lambda$ are integer partitions if and only if $\nu$ is also a partition and $\kappa \in \mathcal{K}_+(\mu)$.
\end{theorem}

The first two components $H_1,H_2$ of $\Upsilon_{\mathrm{col}}(P,Q)$ are the soliton data introduced in \Cref{def:Phi} as prescribed by~\eqref{eq:PQHH}.
The component $\nu$ is also obtained directly from the Fomin field.
Let $\boldsymbol{\lambda}=\boldsymbol{\lambda}[P,Q]$, and set
\[
    m
    =
    \max
    \left\{
        i \midspan
        \lim_{t\to -\infty}
        \boldsymbol{\lambda}_i(t,0)
        \text{ exists and is finite}
    \right\}.
\]
Then $\nu\in\mathbb{S}_m$ is defined by
\begin{equation} \label{eq:nu}
    \nu_i
    =
    \lim_{t\to -\infty}
    \boldsymbol{\lambda}_i(t,0),
    \qquad
    i=1,\dots,m.
\end{equation}
If $m=0$, then $\nu=\varnothing$. In the example of \Cref{fig:RSK dynamics example}, one has $m=1$, since the backward cRSK dynamics leaves only the first row constant. Hence, for $(P,Q)=(P_0,Q_0)$ in that example, the corresponding signature is $\nu=(2)$.


It remains to describe the component $\kappa$. This is the most delicate part
of the construction. Its definition relies on a fundamental projection from
the cRSK dynamics to the classical BBS.

\begin{theorem} \label{thm:projection BBS}
    There exists an explicit projection
    \begin{equation}
        \Pr \colon (P,Q) \in \bigsqcup_{\lambda,\rho \in \mathbb{S}} \mathrm{SST} ( \lambda/\rho , n)^2 \longrightarrow B \in \BBS,
    \end{equation}
    such that, assuming $\Upsilon_{\mathrm{col}}(P,Q) = (H_1,H_2,\kappa,\nu)$ with $H_1,H_2$ of shape $\mu$, then
    \begin{equation}
        \KKR(B)=(\mu,J)
        \qquad
        \text{and}
        \qquad
        \kappa = J+\mu.
    \end{equation} 
\end{theorem}

\begin{remark}
    \Cref{thm:projection BBS} explains why the map $\Upsilon_{\mathrm{col}}$ gives linearized coordinates for the cRSK dynamics, as illustrated in \eqref{eq:cRSK commutative diagram} and \eqref{eq:commutative_diagram_RSK_BBS}. Indeed, under $\Upsilon_{\mathrm{col}}$, the data of a pair $(P,Q)$ is decomposed into conserved quantities $(H_1,H_2,\nu)$ and an angle variable $\kappa$. The latter is identified, via the projection $\Pr$ and the KKR correspondence, with the rigging data of the associated BBS configuration. Since the KKR correspondence linearizes the BBS evolution by translating the riggings, the $\kappa$ coordinates evolve linearly under the cRSK dynamics.
\end{remark}

The construction of the projection $\Pr$ is the main novel technical result of this paper and it uses symmetries of the cRSK dynamics, namely the operators
\[
    E_i^{(1)}, F_i^{(1)} ,E_i^{(2)}, F_i^{(2)},
    \qquad
    i\in\{0,\dots,n-1\},
\]
defined in \Cref{sec:crystals_pairs}. These are Kashiwara operators for an $(\widehat{\mathfrak{sl}}_n\oplus\widehat{\mathfrak{sl}}_n)$-crystal structure on pairs of skew semistandard tableaux; equivalently, they define two commuting affine crystal structures. A key input in the construction is a novel connectivity property of certain subgraphs of the Kirillov--Reshetikhin
crystals $\HWT(\mu,n)$, which we call \emph{(dual) regular subgraphs}; see \Cref{def:reg_subgraph} and \Cref{def:dual_subgraph}. These connectivity results, proved in \Cref{thm:connectedness} and \Cref{thm:dual_connectedness} are one of the main technical ingredients of the paper and may be of independent interest. 

In \Cref{fig:projection example}, we illustrate the construction of $\Pr$ for
the pair $(P,Q)=(P_0,Q_0)$ from \Cref{fig:RSK dynamics example}. A repeated application of the
crystal operators reduces the pair $(P,Q)$ to a pair of equal tableaux, which is then
identified with the BBS configuration $B$
shown in \Cref{fig:BBS}.

\begin{figure}[h]
    \centering
    \includegraphics[width=.8\linewidth]{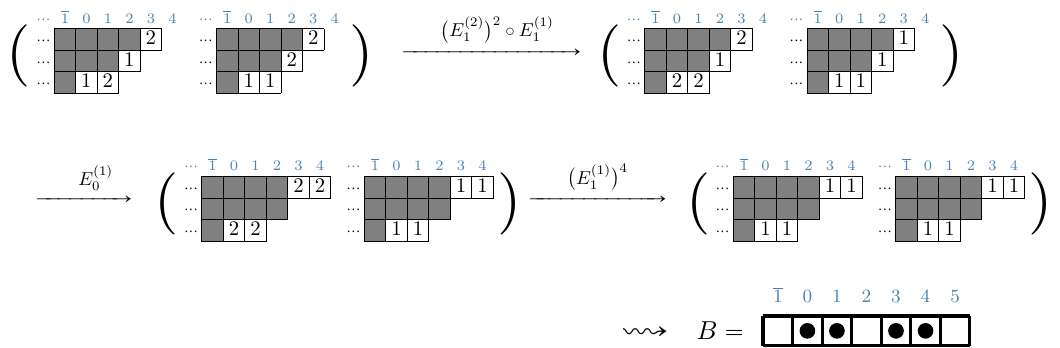}
    \caption{
    The construction of the projection $\Pr$ in \Cref{thm:projection BBS}.
    The final pair of equal skew tableaux, whose entries are all $1$, is identified with a BBS configuration by placing balls at coordinates of the columns containing
    $1$-labelled cells.
    }
    \label{fig:projection example}
\end{figure}

The connectivity results described above ensure that any pair $(P,Q)$ can be moved, by crystal operators, to a pair $(T,T)$ of identical tableaux whose entries are all $1$. This representative determines the projected BBS configuration. Moreover, the change in the shape of the tableaux along this transformation is controlled by the dual energies of the soliton data $(H_1,H_2)$; this bookkeeping is the key ingredient in the proof of the identity \eqref{eq:volume_pre}.


\subsection{Greene's invariants}

It is possible to give an alternative description of the asymptotic partition $\mu$ associated to a Fomin field of signatures $\boldsymbol{\lambda}$. This is a generalization of Greene's theorem which describes the shape of semistandard tableaux which arise as the image of a matrix of natural numbers under the RSK correspondence \cite{Greene1974}. 

Given a Fomin field of signatures $\boldsymbol{\lambda}$, it is straightforward to notice, from the definition of the Fomin growth operator \eqref{eq:fomin-operator}, that $\boldsymbol{\lambda}$ can be reconstructed from the environment $\mathbf{c}$ and from the signature $\nu$ given by \eqref{eq:nu}. The environment records the carrier data at the origin in every elementary square. Together with the backward asymptotic signature $\nu$, this data determines the field $\boldsymbol{\lambda}$ uniquely, by recursively applying the inverse Fomin growth rule. This proves a variant of periodic RSK correspondence of \cite{sagan1990robinson}, which we state next.

For any set $A$ and any set $B \subseteq \mathbb{Z}$ such that $0\in B$, we denote the set of finitely supported functions on $A$ taking values in $B$ as 
\begin{equation} \label{eq:set_environments}
    \mathfrak{F}_{\mathrm{fin}}(A,B) = \{ \mathbf{c}\colon A \to B \mid \mathrm{supp}(\mathbf{c}) \text{ is finite} \}.
\end{equation}

\begin{proposition}[Periodic column RSK correspondence] \label{prop:column RSK}
    We have a bijection
    \[
        \Psi\colon (P,Q) \in \bigsqcup_{\lambda,\rho \in \mathbb{S}} \mathrm{SST} ( \lambda/\rho , n)^2 \longrightarrow (\mathbf{c},\nu) \in \mathfrak{F}_{\mathrm{fin}}(\mathscr{C}_n',\N) \times \mathbb{S},  
    \]
    where letting $\boldsymbol{\lambda} = \boldsymbol{\lambda}[P,Q]$, we define $\mathbf{c}$ as the environment associated to $\boldsymbol{\lambda}$ and $\nu$ is defined as in \eqref{eq:nu}.
\end{proposition}

Composing the periodic column RSK correspondence from \Cref{prop:column RSK} with the bijection of \Cref{th:main_bij} we obtain a refined correspondence, in which the soliton data are read directly from the environment $\mathbf{c}$. In particular, the shape $\mu$ appearing in the linearizing coordinates admits a Greene-type description in terms of last passage percolation times across $\mathbf{c}$.

A \defn{strict up-right path} is a bi-infinite sequence $\boldsymbol{\omega} = (p_i)_{i\in \mathbb{Z}} \subset \mathscr{C}_n'$ such that $p_{i+1} - p_i \sim (a,b)$ with $a,b>0$, for all $i\in \mathbb{Z}$. We define two variants of \defn{last passage percolation times} across an environment $\mathbf{c} \in \mathfrak{F}_{\mathrm{fin}}(\mathscr{C}_n',\N)$ as
\begin{subequations}
\begin{align}
    \mathrm{LPP}_k^{\circlearrowleft} (\mathbf{c}) & = \max_{\substack{\mathbf{p}^{(1)},\dots,\mathbf{p}^{(k)}\\ \text{non-intersecting} \\ \text{simple loops} }} \left\{ \sum_{i=1}^k \sum_{j=0}^{2n-1} \mathbf{c}(\mathbf{p}^{(i)}_j) \right\}
\\
    \mathrm{LPP}_k^{\nearrow} (\mathbf{c}) & = \max_{\substack{\mathbf{\omega}^{(1)},\dots,\mathbf{\omega}^{(k)}\\ \text{strict up-right} \\ \text{paths} }} \left\{ \sum_{p \in \mathscr{C}_n' } \min\left\{ \mathbf{c(p)}, \sum_{i=1}^{k} \mathbf{1}_{p \in \mathbf{\omega}^{(i)}} \right\}  \right\}.
\end{align}
\end{subequations} 
In words, the statistics $\mathrm{LPP}_k^{\circlearrowleft} (\mathbf{c})$ measures the maximal cumulative weight that $k$ walkers can collect from $\mathbf{c}$ after performing a simple loop each (loops cannot intersect each other). To interpret the statistics $\mathrm{LPP}_k^{\nearrow} (\mathbf{c})$, imagine that each cell $p \in \mathscr{C}_n'$ carries $\mathbf{c}(p)$ units of weight and each time an individual visits $p$ it may extract at most one unit from that cell. Then $\mathrm{LPP}_k^{\nearrow} (\mathbf{c})$ is the maximal weight that $k$ individuals can extract from $\mathbf{c}$ along $k$ strict up-right paths.

For the next theorem we define the negative half infinite cylinder
\[
    \mathscr{C}_n'^- = \left\{ p \in \mathscr{C}_n' \midspan p \sim (i',j') \text{ for some } i' \in \Z'_{< 0} \text{ and } j' \in \left\{ \frac{1}{2},\dots,n-\frac{1}{2} \right\}\right\}.
\]

\begin{theorem} \label{thm:LPP}
    There exists a bijection 
    \begin{equation} \label{eq:Upsilon_tilde}
        \widetilde{\Upsilon}_{\mathrm{col}} \colon \mathbf{c} \in \mathfrak{F}_{\mathrm{fin}}(\mathscr{C}_n',\N) \longrightarrow (H_1,H_2,\kappa) \in \bigsqcup_{\mu \in \partitions} \HWT(\mu,n)^2 \times \mathcal{K}(\mu),
    \end{equation}
    such that, assuming $\Upsilon_{\mathrm{col}}(P,Q) = (H_1,H_2,\kappa,\nu)$ and $\Psi(P,Q)=(\mathbf{c},\nu)$, we have $\widetilde{\Upsilon}_{\mathrm{col}}(\mathbf{c}) = (H_1,H_2, \kappa)$. Moreover, the following properties hold.
    \begin{enumerate}[label = {\rm (\Roman*)}]
        \item \label{item:LPP} Assuming that $\mu$ is the common shape of tableaux $H_1,H_2$, we have
        \begin{equation} \label{eq:LPP}
            \mathrm{LPP}_k^{\circlearrowleft} (\mathbf{c}) = \mu_1 + \cdots + \mu_k, \qquad \text{and} \qquad
            \mathrm{LPP}_k^{\nearrow} (\mathbf{c}) = \mu_1'+\cdots+\mu_k'.
        \end{equation}
        \item \label{item:weight_environment} For all $i \in \{1,\dots,n\}$ we have
        \begin{equation} \label{eq:weight_environment}
            \#\{ i\text{-cells of } H_1 \} = \sum_{\ell \in \mathbb{Z}'} \mathbf{c} \left(i-\frac{1}{2},\ell \right) \qquad
            \text{and}
            \qquad
            \#\{ i\text{-cells of } H_2 \} = \sum_{\ell \in \mathbb{Z}'} \mathbf{c}\left (\ell,i-\frac{1}{2} \right).
        \end{equation}
        \item \label{item:volume_environment}
        We have
        \begin{equation} \label{eq:volume_environment}
            \sum_{i,j=1}^n \sum_{k \in \Z } k \, \mathbf{c} \left(i - \frac{1}{2},j-kn-\frac{1}{2} \right) = D'(H_1)+ D'(H_2)+ | \kappa|.
        \end{equation}
        \item \label{item:bij environment HWT restriction} The map $\widetilde{\Upsilon}_{\mathrm{col}}$ restricts to a bijection
        \[
            \widetilde{\Upsilon}_{\mathrm{col}}\colon \mathbf{c} \in \mathfrak{F}_{\mathrm{fin}}(\mathscr{C}_n'^-,\N) \longrightarrow (H_1,H_2,\kappa) \in \bigsqcup_{\mu \in \partitions} \HWT(\mu,n)^2 \times \mathcal{K}_+(\mu). 
        \]
    \end{enumerate}
\end{theorem}

The proof of \Cref{thm:LPP} is given in \Cref{subs:LPP proof}.

\begin{remark}
    \Cref{thm:LPP} appears to be a generalization of the affine RS(K) correspondence by~\cite{CPY18,CLP23,KL25,Shi_affine_RS}.
    It would be interesting to establishing a precise connection between these related constructions and results.
\end{remark}

\begin{remark}
    Consider the $n=1$ case of \Cref{thm:LPP}. In this case simple loops on $\mathscr{C}_1'$ are pairs $(p,p+\mathbf{e}_1)$ with $p \in \mathscr{C}_1'$ while strict up-right paths are sequences $(p_i)_{i\in \mathbb{Z}}$ such that $p_{i+1} = p_i + k \mathbf{e}_1$ with $k\ge 2$. Moreover, by \Cref{rem:RSK to BBS} in the $n=1$ case the numbers $\mathbf{c}(t-1/2,1/2)$ record the number of balls carried by the BBS carrier while crossing the origin at time $t$.
    Then, \Cref{thm:LPP} states that the total size of the $k$ largest solitons is $\mathrm{LPP}_k^{\circlearrowleft} (\mathbf{c})$. Similarly the difference $\mathrm{LPP}_k^{\nearrow} (\mathbf{c})-\mathrm{LPP}_{k-1}^{\nearrow} (\mathbf{c})$ counts the number of solitons of length $k$ or more. These statements appear to be new, but they are reminiscent of Greene-type theorems for the multi-species BBS found in \cite{lewis2024scaling}.
\end{remark}

\subsection{Application to symmetric functions}

The linearization of the cRSK dynamics also has applications to symmetric functions. By taking suitable generating functions on both sides of the bijection \eqref{eq:Upsilon_tilde}, we obtain bijective proofs of Cauchy and Kawanaka--Littlewood identities for transformed Hall--Littlewood
polynomials. 

We will use the $q$-Pochhammer symbols
\[
    (z;q)_k = \prod_{i=0}^{k-1} (1-z q^i) \quad \text{for } k\in \N\cup\{\infty\}.
\]
The functions $\THL_\mu(x;q)$ and the coefficients $c_\mu(q)$, $c_\mu(\alpha,\beta;q)$ appearing in the statements below are defined explicitly in \Cref{sec:summation_id}.

\begin{theorem}[Cauchy and Kawanaka--Littlewood identities] \label{thm:Cauchy and Littlewood id}
    Fix a parameter $|q|<1$ and $2n$ parameters $x_1,\dots, x_n, y_1,\dots, y_n$ with absolute value $<1$. Then, we have
    \begin{equation} \label{eq:Cauchy_id}
        \prod_{i,j=1}^n \frac{1}{(x_i y_j;q)_\infty}  = \sum_{\mu} c_{\mu}(q) \THL_\mu(x;q) \THL_\mu(y;q).
    \end{equation}
    Moreover, for any $\alpha,\beta \in \mathbb{R}$ we have
    \begin{equation} \label{eq:littlewood_id}
            \prod_{i = 1}^n \frac{(-\alpha x_i;q^2)_\infty (-q \beta x_i;q^2)_\infty}{(x_i^2;q^2)_\infty} \prod_{1\le i < j \le n} \frac{1}{(x_i x_j;q^2)_\infty}
            =
            \sum_{\mu} c_\mu(\alpha,\beta;q) \THL_\mu(x;q^2).
    \end{equation}
\end{theorem}

Taking suitable generating function of left and right hand sides of \eqref{eq:Upsilon} we obtain the following refinement of \Cref{thm:Cauchy and Littlewood id}. Analogous identities for $q$-Whittaker polynomials were proven in \cite{imamura2021skewRSK,IMS_matching,He2023boundaryASEP}. These were further generalized through the addition of an extra parameter in \cite{He-Wheeler2023qWhittaker,he_wheeler_FPSAC2025,he2025free} by He and Wheeler, through algebraic techniques.

\begin{theorem}[Refined Cauchy and Kawanaka--Littlewood identities] \label{thm:refined Cauchy and Littlewood id}
    Fix a parameter $|q|<1$ and $2n$ parameters $x_1,\dots, x_n, y_1,\dots, y_n$ with absolute value $<1$. Then, for any $N\in \N$, we have
    \begin{align} \label{eq:refined_Cauchy_id}
        \sum_{\rho \subset \lambda \, : \, \lambda_1' = N} q^{|\rho|} s_{\lambda/\rho}(x) s_{\lambda/\rho}(y) & = \sum_{\ell =0}^N \frac{q^\ell}{(q;q)_\ell} \sum_{\mu \, : \, \mu_1' = N-\ell} c_{\mu}(q) \THL_\mu(x;q) \THL_\mu(y;q),
    \\
    \label{eq:littlewood_id_refined}
        \sum_{\rho \subset \lambda \,:\,\lambda_1'=N}  \alpha^{\oddr(\lambda)} \beta^{\oddr(\rho)} q^{|\rho|} s_{\lambda / \rho}(x) & =  \sum_{k=0}^N \frac{[q^2 + q \alpha \beta]_{q^2}^k}{(q^2;q^2)_k} 
        \sum_{\mu_1'=N-k} c_\mu(\alpha,\beta;q) \THL_\mu(x;q^2),
    \end{align}
    where $\oddr(\lambda)$ counts the number of odd elements of $\lambda$ and the notation $[a+b]^n_p$ is defined in \eqref{eq:rogers_szego}.
\end{theorem}

The proof of \Cref{thm:refined Cauchy and Littlewood id} is given in \Cref{sec:summation_id}.

\subsection{Discussion}

The system with cRSK dynamics is a direct integrable two-dimensional generalization of the BBS.
This model lies at the intersection of several well-developed areas of combinatorics, probability, and integrable systems.
On one side is the theory of the BBS: the classical BBS has been solved by combinatorial methods, related to continuous soliton equations through ultradiscretization, generalized using affine crystal theory and other Lie symmetries, and studied asymptotically in connection with generalized hydrodynamics.
The other side consists of growth processes and line ensembles, which are related through RSK.
They are central objects in integrable probability and in the study of KPZ universality; see, e.g.,~\cite{BorodinGorinSPB12,dauvergne2022rsk,zygouras_review} and references therein.
The cRSK dynamics interpolates between these two structures.
We have defined it as a dynamics on $2n$-step periodic line ensembles:
When $n=1$ it reduces to the classical BBS, and after choosing suitable random environments and sending $n\to\infty$, these periodic line ensembles converge to the line ensembles associated with last passage percolation studied in~\cite{corwin2014brownian,johansson2003discrete,PhahoferSpohn2002}.

One natural direction suggested by this connection is therefore probabilistic.\footnote{This direction arose from a conversation with Amol Aggarwal, whom we thank.}
Given a random environment on the cylindrical lattice, one may run the cRSK dynamics and study the long-time behavior of the associated line ensemble.
Questions of local convergence to equilibrium and mixing for this dynamics are then closely related to universality questions in last passage percolation.
Indeed, the Greene-type theorem proved in this paper expresses the soliton data encoded in the line ensemble in terms of last-passage values of the environment.
Thus, one may hope to combine these Greene-type formulas with structural characterizations of line ensembles, such as those proved by Aggarwal and Huang \cite{aggarwal2026strong}, to obtain new perspectives on universality in last passage percolation.
This is particularly intriguing because convergence to equilibrium for discrete integrable systems appears to be largely unexplored, even for the classical BBS.

It would also be interesting to look for a rational lift of the cRSK dynamics.
The classical BBS arises as an ultradiscrete limit of integrable lattice models such as the Toda lattice or the modified KdV equation, while RSK-type growth rules are closely related to tropical forms of the octahedron recurrence (see, e.g.,~\cite{danilov2005arrays,danilov2007octahedron,kuniba2011TY_Syst,NoumiYamada2004}).
The direct connection with the BBS also points toward possible extensions to other Lie types.
As commented above, the use of column insertion may be instrumental here, since column insertion procedures have analogues for other root systems and are closely related to combinatorial $R$-matrices for soliton cellular automata.

A separate direction is to keep the two-lane BBS local rule fixed but consider it in other geometries.
In this paper, the dynamics is run in parallel on a periodic lattice, leading to the cylindrical cRSK dynamics.
One could instead impose different boundary conditions, for instance on intervals with reflecting or open boundaries; such variants may produce new symmetry structures and models related to other Lie types.

These possible extensions may also have consequences for symmetric functions.
In the present paper, the linearization of the cRSK dynamics gives bijective proofs of Cauchy and Kawanaka--Littlewood type identities for transformed Hall--Littlewood polynomials (of type $A$), while the linearization in~\cite{imamura2021skewRSK} led to identities for (type $A$) $q$-Whittaker polynomials.
This suggests that analogous linearization schemes, arising from other Lie symmetries, boundary geometries, or rational lifts, could provide bijective proofs of Cauchy-type identities for the corresponding families of (specializations of) Macdonald-type polynomials.


\subsection{Acknowledgements}

The authors thank Masato Okado and Atsuo Kuniba for valuable conversations about results of \Cref{sec:affine_leading}. The authors also thank Amol Aggarwal and Mark Shimozono for useful discussions.

The work of T.I.\ was partially supported by JSPS KAKENHI Grant Nos.\ JP20K03626, JP22H01143, and JP24K06773.
T.Sc.\ was partially supported by Grant-in-Aid for JSPS Fellows 21F51028 and JSPS KAKENHI Grant Number JP23K12983.
T.I., M.M.\ and T.Sa.\ thank the International Centre for Theoretical Sciences (ICTS) for the program ``Discrete integrable systems: difference equations, cluster algebras and probabilistic models,'' where part of the work was completed.
This material is based upon work supported by the National Science Foundation under Grant No.\ DMS-1929284 while T.Sc.\ was in residence at the Institute for Computational and Experimental Research in Mathematics in Providence, RI, during the ``Categorification and Computation in Algebraic Combinatorics'' Fall 2025 semester program.

\section{Preliminaries}\label{sec:Pre}

\subsection{Signatures and partitions}
In the introduction we have defined signatures $\lambda$, their conjugates $\lambda'$ and the notion of interlacing. Although the conjugate of a signature is not a partition, but an infinite sequence of non-negative integers, we will use the convention that $(\lambda')'=\lambda$ for any signature $\lambda$. We have also defined semistandard Young tableaux with skew shape given by signatures. A \defn{partition} is a signature whose entries are non-negative, with trailing zeros usually omitted. If $\lambda$ is a partition its Young diagram consists only of cells $(i,j)$ such that $1 \le j \le \lambda_i$ and its size is $|\lambda| = \lambda_1 + \cdots + \lambda_\ell$. The conjugate of a partition $\lambda$ is $\lambda' = (\lambda_i')_{i \ge 1}$ and it is also a partition.

\subsection{Standard tableaux} \label{subs:standard_tab}

In the introduction we have defined the notion of a semistandard tableau. Given a semistandard tableau $P =  (\rho = \lambda^{(0)} \preceq \lambda^{(1)} \preceq \cdots \preceq \lambda^{(n)} = \lambda)$ we call $\lambda/\rho$ its \defn{shape}, $\rho$ its \defn{internal shape} and $\lambda$ its \defn{external shape}. We also define the multiplicities
\begin{equation}\label{eq:multiplicities_m}
    \wt_i(P) = |\lambda^{(i)} / \lambda^{(i-1)}|, \qquad \text{and} \qquad \mathsf{M}_i(P) = \sum_{j =1 }^i \wt_j(P) + \mathbf{1}_{\wt_j(P)=0}, \qquad  \text{for } i\ge 1.
\end{equation}
Representing $P$ as a filling of a Young diagram, $\wt_i(P)$ counts the number of $i$-labeled cells.
We say that a tableau $P$ is \defn{partially standard}\footnote{In literature a tableau of $n$ boxes is \emph{standard} if $\wt_i(P)=1$ for $i=1,\dots,n$.} if we have $\wt_i(P) \in \{0,1\}$ for all $i$. 
From a semistandard tableau $P$ we can produce a partially standard one in a canonical way, called \defn{partial standardization}. For all $i$, whenever $\wt_i(P)>0$ we add between consecutive partitions $\lambda^{(i-1)}$ and $\lambda^{(i)}$ intermediate partitions $\lambda^{(i,0)},\dots,\lambda^{(i,\wt_i(P))}$ so that
\[
    \lambda^{(i-1)} = \lambda^{(i,0)} \preceq \lambda^{(i,1)} \preceq \cdots \preceq \lambda^{(i,\wt_i(P))} = \lambda^{(i)}
    \qquad \text{with} \qquad \left( \lambda^{(i,j)} \right)' - \left( \lambda^{(i,j-1)} \right)' = \mathbf{e}_{m_{i,j}},
\]
where
$$
    m_{i,1} < \cdots < m_{i,\wt_i(P)}.
$$
In terms of filling of Young diagrams the above procedure consists of replacing in the order from left to right the $\wt_i(P)$ $i$-entries in $P$ (whose column coordinates are $m_{i,1} < \cdots < m_{i,\wt_i(P)}$) with $\mathsf{M}_{i-1}(P)+1,2,\dotsc, \mathsf{M}_i(P)$ (where by convention $\mathsf{M}_0(P) = 0$), for all $i=1,\dots,n$.

\begin{example}
    \label{ex:sst}
    An example of a semistandard Young tableau in $\SST(\lambda/\rho,n)$ with $n=3$ and $ \lambda = (2,0,0), \rho=(0,-1,-2)$ is
    \begin{equation} \label{eq:example tableau}
        \ytableausetup{aligntableaux = center,smalltableaux}
        ((0,-1,-2) \preceq (1,0,-2) \preceq (1,0,0) \preceq (2,0,0)) =
            \begin{ytableau}
                \none[\scriptstyle \ao{\cdots}] & \none[\scriptstyle \ao{ \overline{2}}] & \none[\scriptstyle \ao{\overline{1}}] & \none[\scriptstyle \ao{0}] & \none[\scriptstyle \ao{1}] & \none[\scriptstyle \ao{2}] 
                \\
                \none[\scriptstyle \cdots] & *(gray) & *(gray) & *(gray) & 1 & 3
                \\
                \none[ \scriptstyle \cdots] & *(gray) & *(gray) & 1 
                \\
                \none[\scriptstyle \cdots] & *(gray) & 2 & 2
            \end{ytableau}
            \,,
    \end{equation}
    while an example of a semistandard Young tableau in $\SST(\lambda/\rho,n)$ with $n=3$ and shape $\lambda= (5,4,2)$, $\rho = (2,1,1)$ is
    \begin{equation}
            ((2,1,1) \preceq (4,2,1) \preceq (5,4,1) \preceq (5,4,2)) =
            \begin{ytableau} 
                *(gray) & *(gray)  &1&1&2
                \\
                *(gray) & 1 & 2 & 2
                \\
                *(gray) & 3 
            \end{ytableau}
            \, .
    \end{equation}
    Unless the shape of a semistandard Young tableau is a skew partition we will write labels on top of each column (reported in \ao{blue} in the right hand side tableau in \eqref{eq:example tableau}) where barred labels $\overline{i}$, $i=1,2,\ldots$, denote negative numbers $-i$. In this case we also think of the tableaux as having infinitely many empty boxes to the left of its shape. The partial standardization of the tableau in \eqref{eq:example tableau} is
    \begin{equation}
        ((0,-1,-2) \preceq (0,0,-2)  \preceq (1,0,-2) \preceq (1,0,-1) \preceq (1,0,0) \preceq (2,0,0)) =
            \begin{ytableau}
                \none[\scriptstyle \ao{\cdots}] & \none[\scriptstyle \ao{ \overline{2}}] & \none[\scriptstyle \ao{\overline{1}}] & \none[\scriptstyle \ao{0}] & \none[\scriptstyle \ao{1}] & \none[\scriptstyle \ao{2}] 
                \\
                \none[\scriptstyle \cdots] & *(gray) & *(gray) & *(gray) & 2 & 5
                \\
                \none[ \scriptstyle \cdots] & *(gray) & *(gray) & 1 
                \\
                \none[\scriptstyle \cdots] & *(gray) & 3 & 4
            \end{ytableau}
            \,.
    \end{equation}
\end{example}

\subsection{Graphical presentation of Fomin growth operator} \label{subs:fomin operator}

The Fomin growth operator was defined in \Cref{sec:intro}. Here we present a slightly different formulation for it and give its graphical interpretation (which is a variant of Viennot's shadow lines construction \cite{Viennot_une_forme_geometrique}), which will make some of its properties transparent. First recall the piecewise linear transformation $\mathsf{f}$ defined in \eqref{eq:f Fomin}.

To evaluate $\mathsf{f}(a,b,c) = (a',b',c')$, consider a unit cell containing $c$ distinct points arranged along a diagonal, with points on the right placed lower than those on the left. From each of the $c$ points, draw two rays: one upward and one rightward. If $a = 1$, insert a vertical ray entering the cell from below, positioned just to the right of the rightmost of the $c$ points in the cell. If $b = 1$, insert a horizontal ray entering the cell from the left, located just above the topmost point.  Rays travel straight and stop when they intersect another ray or exit through the top or right boundary of the cell. Then $c'$ counts number of ray intersections within the cell, $a'$ counts the number of rays exiting through the top and $b'$ counts number of rays exiting through the right.

The example below shows the evaluations:
\begin{equation}
    \begin{array}{c@{\quad}c@{\quad}c}
    \mathsf{f}(1,0,3) = (1,0,3) & \mathsf{f}(1,1,3) = (0,0,4) & \mathsf{f}(0,0,3) = (1,1,2)
    \\
        \includegraphics[width = .28\linewidth]{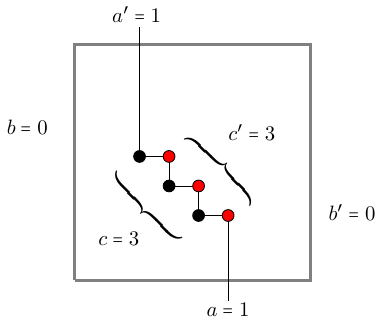}
        &
        \includegraphics[width = .28\linewidth]{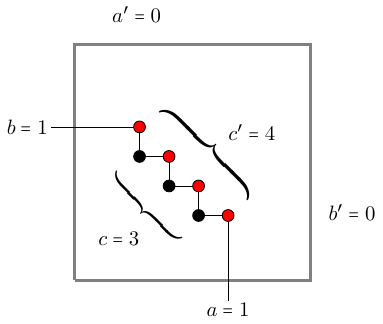}
        &
        \includegraphics[width = .28\linewidth]{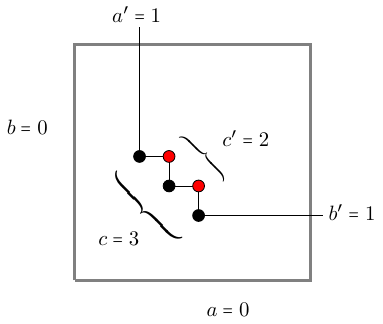}
    \end{array}.
\end{equation}

The map $\mathsf{f}$ is the building block for the Fomin growth operator $\Fomin$ defined in \eqref{eq:fomin-operator}. For signatures $\mu \succeq \kappa \preceq \lambda$ of the same length, we have
\begin{equation}
    \Fomin(\kappa, \mu, \lambda) = \nu,
\end{equation}
where the signature $\nu$ satisfying $\lambda \preceq \nu \succeq \mu$ is defined by the relations
\begin{equation} \label{eq:explicit_fomin}
    \mu'-\kappa' = \sum_{i\in \mathbb{Z}} a_i \mathbf{e}_i,
    \qquad
    \lambda'-\kappa' = \sum_{i\in \mathbb{Z}} b_i \mathbf{e}_i,
    \qquad
    \nu'-\lambda' = \sum_{i\in \mathbb{Z}} a_i' \mathbf{e}_i,
    \qquad
    \nu'-\mu' = \sum_{i\in \mathbb{Z}} b_i' \mathbf{e}_i.
\end{equation}
where $\mathbf{e}_i = \left( \mathbf{1}_{i=k} \right)_{k \in \Z}$ for any $i \in \Z$ and
\begin{equation} \label{eq:f_recursive}
    (a_i',b_i',c_i) = \mathsf{f}(a_i,b_i,c_{i-1}) \qquad \text{for all } i \in \mathbb{Z}, \quad \text{with } \quad c_{-\infty} = 0.
\end{equation}
Since the sequences $(a_i)_{i\in \mathbb{Z}}$, $(b_i)_{i\in \mathbb{Z}}$ are finitely supported, consequence of the fact that $\kappa,\lambda,\mu$ have the same length, the condition $c_{-\infty}=0$ determines uniquely the sequences $(a_i')_{i\in \mathbb{Z}}$, $(b_i')_{i\in \mathbb{Z}}$, $(c_i)_{i\in \mathbb{Z}}$ and hence the operator $\Fomin$ is well defined. See \Cref{fig:Fomin example} for a graphical representation of the Fomin growth operator.

\begin{figure}
    \centering
    \includegraphics[width=0.85\linewidth]{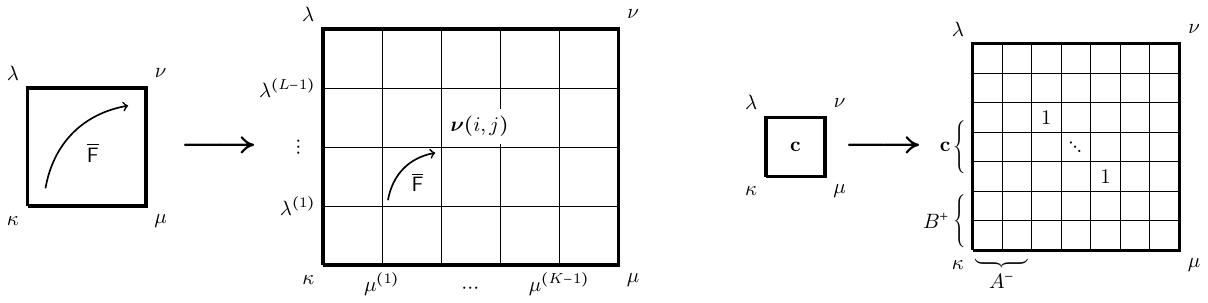}
    \caption{The refinement of a Fomin cell. On the right, the values of the environment following the refinement of a cell}
    \label{fig:refinement_Fomin_cell}
\end{figure}

The graphical construction of the Fomin growth operator, provides a way to refine its action into simpler operations, as depicted in \Cref{fig:refinement_Fomin_cell}. Namely given a triple $\mu \succeq \kappa \preceq \lambda$, we can produce sequences
\begin{equation}
    \mu = \mu^{(K)} \succeq  \cdots \succeq \mu^{(1)} \succeq \mu^{(0)} = \kappa = \lambda^{(0)} \preceq \lambda^{(1)} \preceq \cdots \preceq \lambda^{(L)} = \lambda,
\end{equation}
where $|\mu^{(i)}/ \mu^{(i-1)}|,|\lambda^{(j)}/ \lambda^{(j-1)}| \in \{0,1\}$ for all $i,j$ and for $i_1<\cdots <i_{|\mu / \kappa|}$, $j_1<\cdots < j_{|\lambda / \kappa|}$ we have
\begin{equation}
    \left(\mu^{(i_k)}\right)'- \left(\mu^{(i_k-1)}\right)' = \mathbf{e}_{m_k},
    \qquad 
    \left(\lambda^{(j_r)}\right)'- \left(\mu^{(j_r-1)}\right)' = \mathbf{e}_{\ell_r},
\end{equation}
with
\begin{equation}
    m_1 < \cdots <m_{|\mu/\kappa|}
    \qquad
    \text{and}
    \qquad
    \ell_1 < \cdots < \ell_{|\ell/\kappa|}.
\end{equation}
In other words, partitions $\mu^{(i)}$ (resp. $\lambda^{(j)}$) are constructed from $\kappa$ by adding, step-by-step either no cells or cells of $\mu/\kappa$ (resp $\lambda/\kappa$) from the leftmost to the right. The $\nu = \Fomin (\kappa,\mu,\lambda)$ can be constructed by evaluating the $(M+1) \times (L+1)$ array of interlacing partitions $\{ \boldsymbol{\nu} (i,j) \}_{0 \le i \le M, \, 0 \le j \le L }$ with boundary conditions $\boldsymbol{\nu} (i,0)=\mu^{(i)}$ and $\boldsymbol{\nu} (0,j)=\lambda^{(j)}$ and such that 
$$
    \boldsymbol{\nu} (i+1,j+1) = \Fomin ( \boldsymbol{\nu} (i,j), \boldsymbol{\nu} (i+1,j), \boldsymbol{\nu} (i,j+1) ).
$$
Then $\nu = \boldsymbol{\nu} (M,L)$. 


In the next proposition we examine the values $c_0(p')$ produced by the evaluation of the Fomin growth operator $\Fomin$ at each cell $p'$ the $(M+1) \times (L+1)$ lattice.
\begin{proposition} \label{prop:refinement_environment}
    In the notation established above, let $\Lambda_{M,N} = \{ 0 , \dots, M \} \times \{ 0 , \dots, L \}$ and let $\Lambda'_{M,N} = \{ \frac{1}{2} , \dots, M-\frac{1}{2} \} \times \{ \frac{1}{2} , \dots, L-\frac{1}{2} \}$ be the dual lattice.
    Call $\mathbf{c}$ the value of $c_0$ in the evaluation of $\Fomin (\kappa,\mu,\lambda)$ and for $p' = p+\frac{\mathbf{e}_1 + \mathbf{e}_2}{2}$ call $\widetilde{\mathbf{c}}(p')$ be the value of $c_0$ in the evaluation of $\Fomin (\boldsymbol{\nu}(p), \boldsymbol{\nu}(p+\mathbf{e}_1), \boldsymbol{\nu}(p+\mathbf{e}_2) )$. Then, we have
    \begin{equation} \label{eq:value_c_0}
        \mathbf{c} = \sum_{i \ge 1} \left( \nu_i'-\lambda_i' - \mu_i' +\kappa_i' \right)
    \end{equation}
    and if $\mathbf{c}>0$
    \[
        \widetilde{\mathbf{c}}( p' ) = \begin{cases} 1 \qquad & \text{if } p'=(A^-+k,B^+ + \mathbf{c} + 1-k) \text{ for } k=1,\dots,\mathbf{c},
        \\
        0 \qquad & \text{else},
        \end{cases}
    \]
    where
    \[
        A^- =  \sum_{i\le 0} \nu'_i-\lambda'_i
        \qquad
        B^+ = \sum_{i\ge 1} \lambda'_i-\kappa'_i.
    \]
    If $\mathbf{c}=0$, then $\widetilde{\mathbf{c}}( p' ) = 0$ for all $p'\in \Lambda_{M,N}'$.
    See \Cref{fig:refinement_Fomin_cell}, right panel.
\end{proposition}
\begin{proof}
    First notice that, if $a_i = \mu'_i - \kappa'_i$ and $a_i' = \nu'_i - \lambda'_i$, borrowing the notation of \eqref{eq:explicit_fomin}, \eqref{eq:f_recursive} we have, from \eqref{eq:f Fomin_1}
    \[  
        c_{i-1} = a_i ' - a_i + c_i = \sum_{j \ge i} a_j ' - a_j,
    \]
    which implies \eqref{eq:value_c_0} setting $i=1$.
    
    If $|\lambda/\kappa| = 0$ or $|\mu/ \kappa|=0$, then every coefficient $c_k$ in the evaluation of $\nu = \Fomin (\kappa,\mu,\lambda)$ and in every evaluation $\boldsymbol{\nu} (i+1,j+1) = \Fomin ( \boldsymbol{\nu} (i,j), \boldsymbol{\nu} (i+1,j), \boldsymbol{\nu} (i,j+1) )$ is equal to 0 and the proposition holds.
\\
    Assume that $|\lambda / \kappa|,|\mu/\kappa|>0$ and define sequences $(m_k(
    \ell
    ):k \in \{1,\dots,K \} ,\ell \in \{1,\dots,L \})$ with $m_1(\ell) < \cdots < m_K(\ell)$ by
    \[
        \boldsymbol{\nu}(k,\ell)' -\boldsymbol{\nu}(k,\ell-1)' = \mathbf{e}_{m_k(\ell)}. 
    \]
    Then, by direct inspection of the Fomin growth operator we have $m_{k}(\ell+1) - m_{k}(\ell) \in \{ 0,1\}$ and $\widetilde{\mathbf{c}}(k-\frac{1}{2}, \ell + \frac{1}{2}) = 1$ if and only if $m_k(\ell)=0$ and $m_k(\ell+1)=1$. Call $(k_1,\ell_1),\dots, (k_r,\ell_r)$ all coordinates such that $\widetilde{\mathbf{c}}(k_i-\frac{1}{2}, \ell_i + \frac{1}{2}) = 1$. Then by monotonicity of the sequences $m_k(\ell)$, the columns $k_i$ all columns such that
    \begin{equation} \label{eq:characterization_columns}
        m_{k_i}(0) \le 0 \qquad \text{and} \qquad m_{k_i}(L)>0
    \end{equation}
    and it must necessarily be
    \[
        k_i = k_0 + i \qquad \text{for } i=1,\dots,r.
    \]
    Arguing by symmetry with respect to the transposition of the lattice $\Lambda_{K,L}$ we must also have
    \[
        j_i = j_0 - i \qquad \text{for } i=1,\dots,r.
    \]
    It remains to show that $r=\mathbf{c}$ and that $k_0 = A^-$, $j_0 = B^+ + \mathbf{c} +1$. For the first equality notice that
    \begin{equation*}
        \begin{split}
            r &= \# \{ k \mid m_k(0) \le 0 \text{ and } m_k(L)>0  \} 
            \\
            &= \sum_{i \le 0} \left( \boldsymbol{\nu}_i(K,0)' - \boldsymbol{\nu}_i(0,0)' \right) - \sum_{i \le 0} \left( \boldsymbol{\nu}_i(K,L)' - \boldsymbol{\nu}_i(0,L )' \right) 
            \\
            & = \sum_{i \ge 1} (\nu_i' - \lambda_i') - \sum_{i \ge 0} (\mu_i' - \kappa_i') = \mathbf{c}.
        \end{split}
    \end{equation*}
    By the characterization \eqref{eq:characterization_columns} we have $k_0 = \max \{ k \ge 0 \mid m_k(L) \le 0  \} = \sum_{i \le 0} \nu_i' - \lambda_i' = A^-$. Finally, the equality $j_0 = B^+ + \mathbf{c} +1$ can be obtained arguing again by symmetry with respect to transposition of the lattice $\Lambda_{K,L}$, completing the proof.  
\end{proof}

An advantage of this factorization is that the evaluation of the array $\boldsymbol{\nu} (i,j)$ only involves the action of the Fomin growth operator $\Fomin(\kappa, \mu, \lambda)$ on triples of partitions $\mu \succeq \kappa \preceq \lambda$ such that $|\mu / \kappa|, |\lambda/ \kappa| \in \{ 0,1\}$; see \Cref{fig:Fomin example} right panel.

\begin{figure}
    \centering
    \includegraphics[height=5cm]{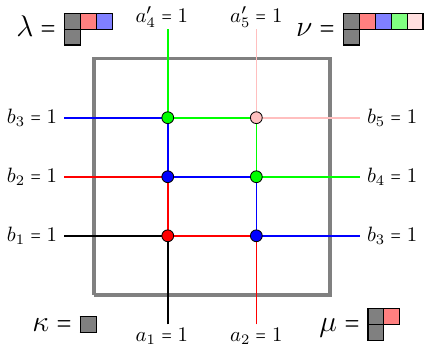}
    \qquad \qquad
    \includegraphics[height=5cm]{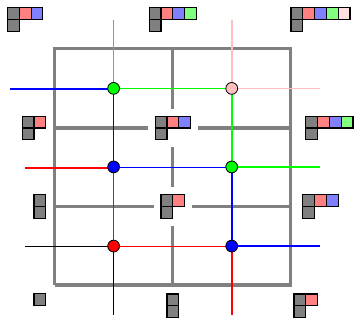}
    \caption{In the left panel a graphical representation of evaluation of the Fomin operator. In the right panel its partial standardization.}
    \label{fig:Fomin example}
\end{figure}

Let us now report two basic properties of the map $\Fomin$. The first is a conservation law.

\begin{proposition} \label{prop:conservation_Fomin}
    If $\nu = \Fomin(\kappa,\mu,\lambda)$, then $|\nu/\lambda| = |\mu/\kappa|$ and $|\nu/\mu| = |\lambda/\kappa|$.    
\end{proposition}
\begin{proof}
    From \eqref{eq:explicit_fomin} we see that $|\nu/\lambda|=\sum_{i\in \mathbb{Z}} a_i' = \sum_{i\in \mathbb{Z}} \left(a_i + c_i -c_{i-1}\right) = \sum_{i\in \mathbb{Z}} a_i = |\mu/\kappa|$, where we used \eqref{eq:f_recursive} and the definition \eqref{eq:f Fomin} of the map $\mathsf{f}$. 
\end{proof}

Another basic property of the Fomin operator $\Fomin$ is the commutation with the global column-wise addition of a signature.

\begin{proposition} \label{prop:symmetry Fomin}
    For signatures $\mu \succeq \kappa \preceq \lambda$, if $\nu = \Fomin(\kappa,\mu, \lambda)$ and we have
    \begin{equation}
        \Fomin\bigl( (\kappa' + \rho')', (\mu' + \rho')' , (\lambda'+\rho')' \bigr) =  (\nu' + \rho')'.
    \end{equation}
    In other words the action of the Fomin growth operator commutes with the column-wise addition of a signature $\rho$. 
\end{proposition}
\begin{proof}
    This is straightforward since the evaluation of the Fomin growth operator $\Fomin(\kappa,\mu,\lambda)$ only depends on the differences $\lambda'-\kappa'$ and $\mu'-\kappa'$.
\end{proof}

\subsection{Horizontally weak tableaux}

\begin{definition}[Horizontally weak tableau]
A \defn{horizontally weak tableau} (HWT) is a filling of Young diagram with positive integers which is nondecreasing in each row while there are no conditions on the columns. We call the associated Young diagram $\mu$ the \defn{shape of
the HWT}. Let \defn{$\HWT(\mu, n)$} denote the set of HWT of shape $\mu$ with entries in the alphabet $\{1, \dotsc, n\}$. 

For later use we also extend the above definition in the case where the shape $\mu$ is not necessarily the Young diagram but $\mu\in \Z_+^{\ell}$, $\ell=1,2,\cdots$ in a similar way.
\end{definition}

\begin{example}\label{ex:HWTs}
The following are examples of the HWTs:
\begin{align}
\begin{ytableau}
    1 & 1 & 2 & 2\\
    2 & 2 & 3\\
    3 & 3
\end{ytableau}
\,
,
\qquad
\begin{ytableau}
    1 & 1 & 2 & 2\\
    2 & 2 & 3\\
    1 & 2
\end{ytableau}
\,
,
\qquad
\begin{ytableau}
    2 & 2 & 3\\
    1 & 1 & 2 & 2\\
    1 & 2
\end{ytableau}
\,
,
\qquad
\begin{ytableau}
    2 & 3 \\
    1 & 1 & 2 & 2\\
    1 & 2 & 2 & 3 & 3
\end{ytableau}
\,
.
\end{align}
Note that the first example is a semistandard Young tableau, i.e., entries are additionally strictly increasing vertically from top to bottom while the second one is not. On the other hand, the rightmost two tableaux have shape respectively $\mu = (3,4,2)$ and $\mu=(2,4,5)$, which are not partitions. 
\end{example}

\subsection{Weight and row reading word}
We define two functions, the weight and the row reading word, which we mainly use to introduce the crystal structures for the set of tableaux in \cref{sec:affine_leading}.

\begin{definition}[Weight]
    \label{def:weight}
    For a combinatorial object $P$ with entries $\{1,2,\dotsc,n\}$ (such as a semistandard tableau, HWT, a word, etc.),
 the \defn{weight} of $P$ is defined by
 \begin{align}
     \label{eq:weight_def}
    \wt(P) := (\wt_1(P), \dotsc, \wt_n(P)) \in \Z^n,
 \end{align}
 where, for $i=1,\dotsc,n$, $\wt_i(P)$ is the number of `$i$' entries in $P$.
\end{definition}

\begin{definition}[Row reading word]
    \label{def:rowreadingword}
    For an HWT or semistandard tableau, the \defn{row reading word} is defined by concatenating the parts with positive elements in all rows starting from the bottom.
\end{definition}

\begin{example}
    \label{ex:rowreadingword}
    For the semistandard and horizontally weak tableaux
\begin{align}
    \label{eq:exrowreadingword}
    \begin{ytableau}
        \none[\scriptstyle \ao{\cdots}] & \none[\scriptstyle \ao{ \overline{2}}] & \none[\scriptstyle \ao{\overline{1}}] & \none[\scriptstyle \ao{0}] & \none[\scriptstyle \ao{1}] & \none[\scriptstyle \ao{2}] & \none[\scriptstyle \ao{3}] & \none[\scriptstyle \ao{4}] & \none[\scriptstyle \ao{5}]
        \\
        \none[\scriptstyle \cdots] & *(gray) & *(gray) & *(gray) & *(gray) & *(gray) & 1 & 1 & 1
        \\
        \none[ \scriptstyle \cdots] & *(gray) & *(gray) & *(gray) & *(gray) & 1 & 2 & 2
        \\
        \none[\scriptstyle \cdots] & *(gray) & *(gray) & *(gray) & *(gray)
        \\
        \none[\scriptstyle \cdots] & *(gray) & *(gray) & 1
    \end{ytableau}
    \qquad
    \text{and}
    \qquad
    \begin{ytableau}
        1 & 2 & 2 & 3
        \\
        1 & 2 & 2
        \\
        1 & 1
        \\
        2
    \end{ytableau}
\end{align}
the row reading words are $1122111$ and $2111221223$, respectively.
\end{example}

\section{Column RSK dynamics}
\label{sec:cRSK}

\subsection{Refinements of fields of interlacing signatures on the cylinder}

Fomin fields of interlacing signatures on the cylindrical lattice $\mathscr{C}_n$ were defined in \Cref{def:fomin_fields}. To each Fomin field $\boldsymbol{\lambda}$ we can associate a pair of semistandard tableaux $(P,Q)$ by reading the signatures along the simple loop 
\[
    (0,0) , (1,0), \dots, (n,0) \sim (0,n), (0,n-1) , \dots , (0,1), (0,0) 
\]
as done in \eqref{eq:tableauxPQ} and \eqref{eq:tableaux_from_loop}. In this case we write $\boldsymbol{\lambda} = \boldsymbol{\lambda}[P,Q]$, since the field is fully determined by the tableaux $(P,Q)$. More in general the Fomin field $\boldsymbol{\lambda}$ would be fully determined once one assigns its values along any simple loop $\mathbf{p} \subset \mathscr{C}_n$.


We say that a field of interlacing signatures $\boldsymbol{\lambda}$ is \defn{fully refined} if
\[
    |\boldsymbol{\lambda}(p+\mathbf{e}_\varepsilon) / \boldsymbol{\lambda}(p)| \in \{ 0 ,1 \} \qquad \text{for } \varepsilon \in \{0,1\}, \qquad \text{for all } p \in \mathscr{C}_{m,n}.
\]
It is clear that if a Fomin field $\boldsymbol{\lambda}$ is fully refined the associated tableaux $(P,Q)$ (i.e. $\boldsymbol{\lambda}=\boldsymbol{\lambda}[P,Q]$) are partially standard. The converse is also true. This is because if $\lambda \succeq \kappa \preceq \mu$ are interlacing signatures with $|\lambda / \kappa|, |\mu/ \kappa| \in \{0,1\}$, then $\nu = \Fomin (\kappa,\mu,\lambda)$ also has the property that $|\nu / \lambda |, |\nu/\mu| \in \{0,1\}$. Such property follows directly from the definition of the operator $\Fomin$. It is also evident from the graphical construction presented in \Cref{subs:fomin operator}, where the sizes $|\lambda / \kappa|, |\mu/ \kappa|,|\nu / \lambda |,|\nu/\mu|$ are respectively the number of colored rays crossing the left, bottom, top and right side of a cell used to evaluate $\Fomin$ as in \Cref{fig:Fomin example}.


The partial standardization procedure of semistandard tableaux described in \Cref{subs:standard_tab} provides a natural way to refine Fomin fields of signatures.

\begin{proposition} \label{prop:standardization_field}
    Let $\boldsymbol{\lambda}$ be a Fomin field of interlacing signatures on $\mathscr{C}_n$, and let $N = \abs{\boldsymbol{\lambda}(n,0)/\boldsymbol{\lambda}(0,0)}$. There exists a fully refined Fomin field $\widetilde{\boldsymbol{\lambda}}$ on $\mathscr{C}_N$ and an explicit embedding $\phi \colon \mathscr{C}_n \hookrightarrow \mathscr{C}_N$ such that $\boldsymbol{\lambda}(p) = \widetilde{\boldsymbol{\lambda}} (\phi(p))$. Moreover, let $T$ be any semistandard tableau of the form
    \begin{equation} \label{eq:generic_tableau}
        T = \bigl( \boldsymbol{\lambda}(p) \preceq \boldsymbol{\lambda}(p+\mathbf{e}_\varepsilon) \preceq \cdots \preceq \boldsymbol{\lambda}(p+ r \mathbf{e}_\varepsilon) \big)
    \end{equation}
    for some $p \in \mathscr{C}_n, r \in \{0,\dots,n\}, \varepsilon \in \{1,2\}$, then the tableau
    \begin{equation} \label{eq:generic_tableau_standard}
        \widetilde{T} = \bigl( \widetilde{\boldsymbol{\lambda}}(\phi(p)) \preceq \widetilde{\boldsymbol{\lambda}}(\phi(p)+\mathbf{e}_\varepsilon) \preceq \cdots \preceq \widetilde{\boldsymbol{\lambda}}(\phi(p+ r \mathbf{e}_\varepsilon) \big)
    \end{equation}
    is the partial standardization of $T$.
\end{proposition}

\begin{proof}
    We can refine each evaluation
    \[
        \boldsymbol{\lambda}(p+\mathbf{e}_1+\mathbf{e}_2) = \Fomin \bigl( \boldsymbol{\lambda}(p), \boldsymbol{\lambda}(p + \mathbf{e}_1) , \boldsymbol{\lambda}(p + \mathbf{e}_2)  \bigr),
    \]
    as described in \Cref{subs:fomin operator}. Such refinement produces a field $\widetilde{\boldsymbol{\lambda}}$ on $\mathscr{C}_N$, obtained replacing each cell $p'=(i+kn + \frac{1}{2} ,j+kn+ \frac{1}{2}) \in \mathscr{C}_n'$ for $i,j\in\{0,\dots,n-1\}, k,l\in \mathbb{Z}$ with a matrix of cells
    \[
        \left\{ \bigl( \alpha +k N + \frac{1}{2}, \beta +k N + \frac{1}{2} \bigr) \midspan M_i \le \alpha \le M_{i+1}-1, ~ M_j' \le \beta \le M_{j+1}'-1 \right\} \subset \mathscr{C}_N',
    \]
    with $M_i = | \boldsymbol{\lambda}(i,j) / \boldsymbol{\lambda}(0,j) | = | \boldsymbol{\lambda}(i,0) / \boldsymbol{\lambda}(0,0) |$ and $M_j' = | \boldsymbol{\lambda}(i,j) / \boldsymbol{\lambda}(i,0) | = | \boldsymbol{\lambda}(0,j) / \boldsymbol{\lambda}(0,0) |$, by \Cref{prop:conservation_Fomin}. Then, by construction, we have $\boldsymbol{\lambda}(p) = \widetilde{\boldsymbol{\lambda}} (\phi(p))$, where the embedding  $\phi$ is defined as
    \[
        \phi( i+ k n, j+\ell n) = ( M_i+k N, M_j'+ \ell N),
    \]
    for $i,j\in\{0,\dots,n-1\}, k,l\in \mathbb{Z}$. The fact that any tableau $\widetilde{T}$ of the form \eqref{eq:generic_tableau_standard} is the standardization of its pre-image $T$ under $\phi$, as in \eqref{eq:generic_tableau}, is evident by construction.
\end{proof}

\begin{figure}
    \centering
    \includegraphics[width=\linewidth]{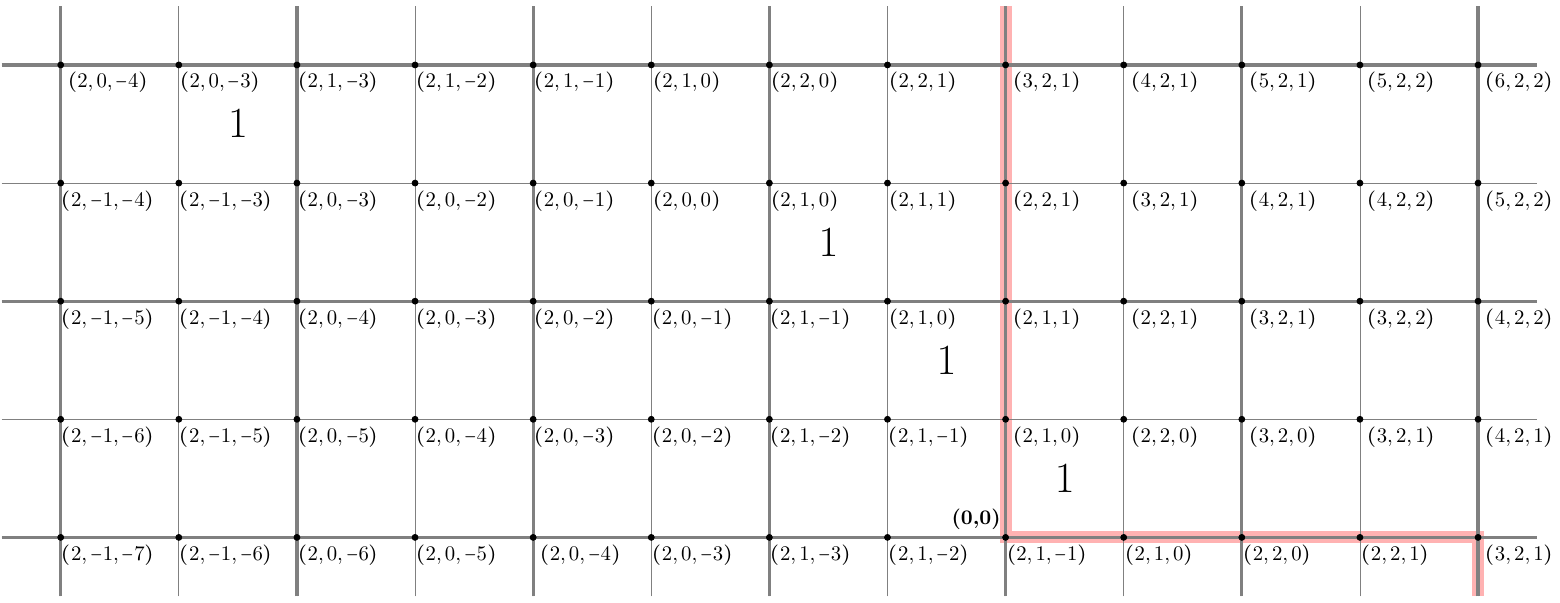}
    \caption{The refinement of the field $\boldsymbol{\lambda}$ of \Cref{fig:cylinder} as described in \Cref{prop:standardization_field}.}
    \label{fig:refinement_field}
\end{figure}

\Cref{prop:standardization_field} will allow us to prove statements for Fomin fields of signatures by analyzing only fully refined fields, which are simpler to treat.

\begin{remark}
    The standardization is a common technique in algebraic combinatorics; see e.g. \cite{Stanley1999}.
\end{remark}

\subsection{Basic properties of the skew column RSK dynamics} \label{subs:definition dynamics Fomin}

Let us give some basic definitions.

\begin{definition}[Skew column RSK map and shifts $\iota_1,\iota_2$] \label{def:cRSK_i1_i2}
    Fix  $(P,Q)\in\bigsqcup_{\lambda,\rho}\SST(\lambda/\rho,n)^2$ and construct the field of interlacing partitions $\boldsymbol{\lambda} = \boldsymbol{\lambda}[P,Q]$.
    The \defn{skew column RSK map} $\cRSK$ is the map $\cRSK(P,Q) = (P',Q')$, where
    \begin{equation}
        \cRSK(P,Q) = (P',Q') ~: ~P'= \bigl( \boldsymbol{\lambda} (n,0) \preceq \cdots \preceq \boldsymbol{\lambda} (2n,0) \bigr) ,\quad Q' = \bigl( \boldsymbol{\lambda} (n,0) \preceq \cdots \preceq \boldsymbol{\lambda} (n,n) \bigr).
    \end{equation}
    We also define the operations $\iota_1 ,\iota_2$ as
    \begin{align*}
       \iota_1(P,Q) & = (\widetilde{P},\widetilde{Q}) ~ : ~  \widetilde{P} = \bigl( \boldsymbol{\lambda} (1,0) \preceq \cdots \preceq \boldsymbol{\lambda} (n+1,0) \bigr) ,~~ \widetilde{Q} = \bigl( \boldsymbol{\lambda} (1,0) \preceq \cdots \preceq \boldsymbol{\lambda} (1,n) \bigr),
    \\
       \iota_2(P,Q) & = (\overline{P},\overline{Q}) ~ : ~  \overline{P} = \bigl( \boldsymbol{\lambda} (0,1) \preceq \cdots \preceq \boldsymbol{\lambda} (n,1) \bigr) ,~~ \overline{Q} = \bigr( \boldsymbol{\lambda} (0,1) \preceq \cdots \preceq \boldsymbol{\lambda} (0,n+1) \bigr).
    \end{align*}
\end{definition}

It is clear that $\cRSK, \iota_1,\iota_2$ are invertible maps.
For examples, see \Cref{fig:iota_1_iota_2}.

\begin{figure}
    \centering
    \includegraphics[width=1\linewidth]{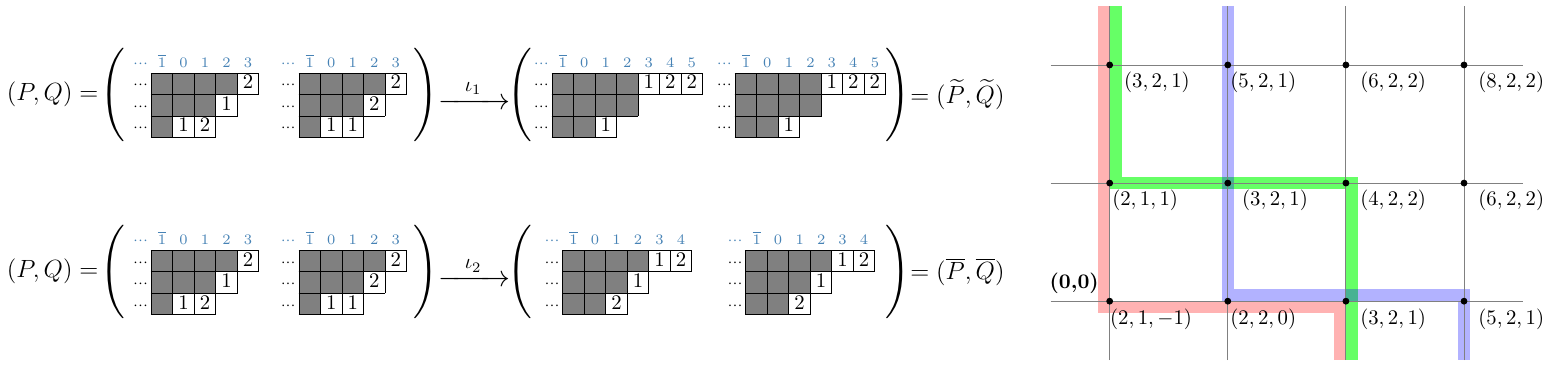}
    \caption{In the left panel, examples of evaluation of $\iota_1, \iota_2$. They correspond to sequences of interlacing signatures displayed on the lattice in the right panel. Signatures along the red down-right loop correspond to tableaux $(P,Q)$, those along the blue loop to $\iota_2(P,Q)=(\widetilde{P},\widetilde{Q})$, while those on along the green loop to $\iota_1(P,Q)=(\overline{P},\overline{Q})$}
    \label{fig:iota_1_iota_2}
\end{figure}

The cRSK dynamics of \Cref{def:cRSK} with initial conditions $(P,Q)$ could be equivalently defined as the sequence $(P_t,Q_t)$ with $(P_t,Q_t) = \cRSK^t(P,Q)$, for all $t\in \mathbb{Z}$; this is evident from the definition of the map $\cRSK$.

We now enumerate several basic properties of Fomin fields of signatures.

\begin{proposition} \label{prop:properties of iota and cRSK}
    Let $(P,Q)\in\bigsqcup_{\lambda,\rho}\SST(\lambda/\rho,n)^2 $ and let $\boldsymbol{\lambda} = \boldsymbol{\lambda}[P,Q]$ be the corresponding field of interlacing partitions on $\mathscr{C}_{n}$. The following properties hold.
    \begin{enumerate}[label = {\rm (\Roman*)}]
        \item \label{item:iota2 field} if $(\widetilde{P},\widetilde{Q})=\iota_2(P,Q)$ and $\widetilde{\boldsymbol{\lambda}} = \boldsymbol{\lambda}[\widetilde{P},\widetilde{Q}]$, then $\widetilde{\boldsymbol{\lambda}}(x,y) = \boldsymbol{\lambda}(x,y+1)$ for all $(x,y) \in \mathscr{C}_{n}$;
        \item \label{item:iota1 field} if $(\widetilde{P},\widetilde{Q})=\iota_1(P,Q)$ and $\widetilde{\boldsymbol{\lambda}} = \boldsymbol{\lambda}[\widetilde{P},\widetilde{Q}]$, then, $\widetilde{\boldsymbol{\lambda}}(x,y) = \boldsymbol{\lambda}(x+1,y)$ for all $(x,y) \in \mathscr{C}_{n}$;
        \item \label{item:cRSK field} if $(\widetilde{P},\widetilde{Q})=\cRSK(P,Q)$ and $\widetilde{\boldsymbol{\lambda}} = \boldsymbol{\lambda}[\widetilde{P},\widetilde{Q}]$, then $\widetilde{\boldsymbol{\lambda}}(x,y) = \boldsymbol{\lambda}(x+n,y) = \boldsymbol{\lambda}(x,y+n)$ for all $(x,y) \in \mathscr{C}_{n}$;
        \item \label{item:iota_swap} we have $\iota_2(P,Q) = \iota_1(Q,P)$.
        \item \label{item:cRSK iota} we have $\cRSK = \iota_2^n = \iota_1^n$;
        \item \label{item:commutation iota} we have $\iota_1 \circ \iota_2 = \iota_2 \circ \iota_1$;
        \item \label{item:content_i2} Let $(\overline{P},\overline{Q}) = \iota_2(P,Q)$. Then $\wt_i(\overline{P}) = \wt_i(P)$ and $\wt_1(\overline{Q}) = \wt_{n}(Q)$, while $\wt_i(\overline{Q}) = \wt_{i-1}(Q)$ for $i=2,\dots,n$
        \item \label{item:iota_standardization} Call $(R,T)$ and $(\overline{R},\overline{T})$ the pairs of tableaux obtained by standardization of the pairs $(P,Q)$ and $(\overline{P},\overline{Q}) = \iota_2(P,Q)$. Then $(\overline{R},\overline{T}) = \iota_2^{\wt_1(Q)}(R,T)$.
        \item \label{item:sum environment} Let $\mathbf{c}$ be the environment associated to the field $\boldsymbol{\lambda}$, as in \Cref{def:fomin_fields}. Then, for all $(x,y)\in\mathscr{C}_n$, we have 
        \begin{equation} \label{eq:sum environment}
            |\boldsymbol{\lambda}(x+1,y)/\boldsymbol{\lambda}(x,y)| = \sum_{j \in \mathbb{Z}'} \mathbf{c} \left(x+\frac{1}{2},j \right)
            \qquad
            \text{and}
            \qquad
            |\boldsymbol{\lambda}(x,y+1)/\boldsymbol{\lambda}(x,y)| = \sum_{j \in \mathbb{Z}'} \mathbf{c} \left(j,y+\frac{1}{2} \right).
        \end{equation}
    \end{enumerate}
\end{proposition}
\begin{proof}
    Statements \ref{item:iota2 field}, \ref{item:iota1 field}, \ref{item:cRSK field} follow straightforwardly from \Cref{def:cRSK_i1_i2}. They further imply \ref{item:iota_swap}, \ref{item:cRSK iota}, \ref{item:commutation iota}. The property \ref{item:content_i2} is straightforward from \Cref{def:cRSK_i1_i2} and \Cref{prop:conservation_Fomin}. The property  \ref{item:iota_standardization} is an immediate consequence of \Cref{prop:standardization_field} and \ref{item:iota2 field}. 
    
    Finally, we prove \ref{item:sum environment}. Since the environment $\mathbf{c}$ is finitely supported, let $m\in \mathbb{Z}'$ be large enough so that $\mathbf{c}(x+1/2,-m-i) = \mathbf{c}(x+1/2,m+i)=0$ for all $i\ge 0$ and 
    \[
        \boldsymbol{\lambda}_{-j}'(x+1,m+i) - \boldsymbol{\lambda}_{-j}'(x,m+i) = 0 = \boldsymbol{\lambda}_j'(x+1,-m-i) - \boldsymbol{\lambda}_j'(x,-m-i) \qquad \text{for all } i,j\ge 0.
    \]
    In other words the signatures $\boldsymbol{\lambda}(x+1,m+i)$ and $\boldsymbol{\lambda}(x,m+i)$ (resp. $\boldsymbol{\lambda}(x+1,-m-i)$ and $\boldsymbol{\lambda}(x,-m-i)$) only differ by cells at some positive (resp. negative) column. Then, by \Cref{prop:conservation_Fomin} and \eqref{eq:value_c_0}, we have
    \begin{equation*}
        \begin{aligned}
            | \boldsymbol{\lambda}(x+1,y)  / \boldsymbol{\lambda}(x,y) | & = | \boldsymbol{\lambda}(x+1,m)  / \boldsymbol{\lambda}(x,m) | = \sum_{j \ge 1}  \boldsymbol{\lambda}_j'(x+1,m)  - \boldsymbol{\lambda}_j'(x,m)
            \\
            &
            = \sum_{j \ge 1}  \boldsymbol{\lambda}_j'(x+1,-m)  - \boldsymbol{\lambda}_j'(x,-m) + \sum_{|j| < m} \mathbf{c}\left(x+\frac{1}{2},j\right) = \sum_{j \in \mathbb{Z}'} \mathbf{c}\left(x+\frac{1}{2},j\right).
        \end{aligned}
    \end{equation*}
    This proves the first relation in \eqref{eq:sum environment}. The second relation can be shown analogously.
\end{proof}

\subsection{Alternative description through Schensted column insertion} \label{subs:alternative_description_cRSK}

In \Cref{subs:definition dynamics Fomin}, we defined the maps $\cRSK, \iota_1, \iota_2$ in terms of the Fomin growth operator $\Fomin$. Here we give an alternative, more combinatorial, description in terms of the Schensted column insertion.

\begin{definition}[Column insertion] \label{def:c_insertion}
    Let $P\in \bigsqcup_{\lambda,\rho\in\signatures}\SST(\lambda/\rho,n)$ be a tableau, let $x \in \{1,\dots,n\}$ and let $c\in \mathbb{Z}$ be a column label. The column insertion of $x$ in $P$ at column $c$ is the following procedure
    \begin{enumerate}
        \item Scan the column $c$ of the tableau $P$ and call $x'$ the smallest entry larger or equal to $x$. If such an entry exists, replace it with $x$ and insert $x'$ into column $c+1$ of $P$.
        \item If all entries of the $c$-th column of $P$ are strictly smaller than $x$ (which includes the case where the column is empty) create a new cell at the bottom of column $c$ and assign to it the value $x$.
    \end{enumerate}
\end{definition}

\begin{definition}[Column internal insertion $\mathcal{C}_{[c]}$]
    Let $P\in \bigsqcup_{\lambda,\rho\in\signatures}\SST(\lambda/\rho,n)$ be a tableau and let $c$ be a column label such that the $c$-th column of $P$ possesses a \emph{corner cell}, namely a labeled cell $(r,c)$ such that both cells $(r-1,c),(r,c-1)$ are not labeled cells of $P$. Let $x$ be the entry of the corner cell of $c$-th column of $P$. We define the \defn{column internal insertion}  $\mathcal{C}_{[c]}(P) \in \bigsqcup_{\lambda,\rho\in\signatures}\SST(\lambda/\rho,n)$ as the tableau obtained by removing the corner at column $c$ of $P$ and performing column insertion of $x$ into $P$ at column $c+1$.
\end{definition}

The operation of column (internal) insertion can be realized using the Fomin operator, as described below.

\begin{proposition} \label{prop:Fomin and column insertion}
    Let $P \in \SST(\lambda/\rho,n)$ be a tableau and let $\widetilde{P} = \mathcal{C}_{[k]} (P)$ be the tableau obtained performing a column internal insertion at column $k$. If 
    \begin{equation}
        P = \left(\rho = \lambda^{(0)} \preceq \lambda^{(1)} \preceq \cdots \preceq \lambda^{(n)} = \lambda \right)
        \quad
        \text{and}
        \quad
        \widetilde{P} = \left( \widetilde{\rho} = \nu^{(0)} \preceq \nu^{(1)} \preceq \cdots \preceq \nu^{(n)} \right),
    \end{equation}
    then,
    \begin{equation}
        \widetilde{\rho} = \left( \rho' + \mathbf{e}_k \right)' \quad
        \text{and}
        \quad
        \nu^{(i+1)} = \Fomin (\lambda^{(i)}, \lambda^{(i+1)}, \nu^{(i)}) \quad \text{for all } i=1,\dots,n.
    \end{equation}
    In other words, $\widetilde{\rho}$ is obtained from $\rho$ by adding an empty cell at the $k$-th column, while all other partitions $\nu^{(i)}$ are determined iteratively using the Fomin growth operator $\Fomin$.
\end{proposition}
\begin{proof}
    The relation between the Fomin growth rules and the Schensted insertion were discovered in \cite{Fomin1986,fomin1995schensted} (see also \cite{krattenthaler2006growth}) and the argument goes as follows. By the factorization properties of the Fomin growth operator and the iterative definiton of the column insertion, it suffices to show that the proposition holds for tableaux $P$ such that $|\lambda^{(i)}/\lambda^{(i-1)}|\in \{0,1\}$. Under these assumptions one can proceed by induction on the index $i=0,\dots,n-1$, matching the partitions $\nu^{(i+1)}$ produced by the column inserition procedure with the output of $ \Fomin(\lambda^{(i)}, \lambda^{(i+1)}, \nu^{(i)})$. We leave these checks to the reader; see \Cref{ex:c_insertion}.
\end{proof}

\begin{example}\label{ex:c_insertion}
Below we compute the internal insertion $\mathcal{C}_{[1]}(P)$ step-by-step as
\[
    \raisebox{-.5\height}{
        \includegraphics[width=.95\linewidth]{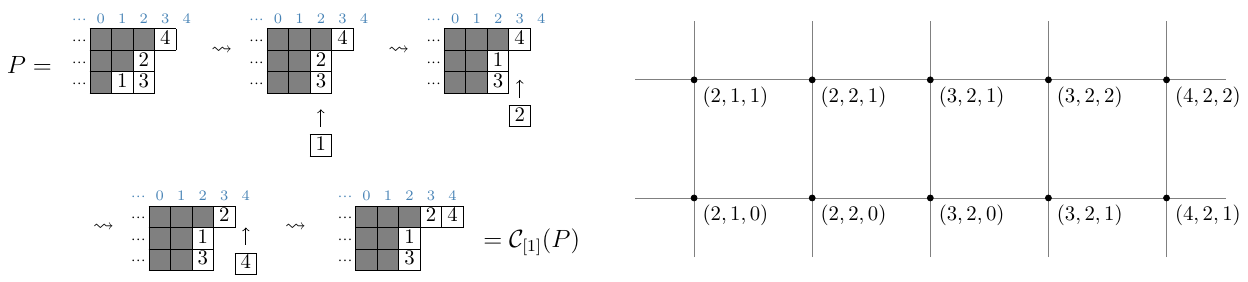}
        }
\]
On the right panel we observe the representation of the column internal insertion as described in \Cref{prop:Fomin and column insertion}: the bottom row represents the tableau $P$, while the row above is $\mathcal{C}_{[1]}(P)$.
\end{example}

Using \Cref{prop:Fomin and column insertion} we can translate the action of the operations $\iota_1$, $\iota_2$ and $\cRSK$ on pairs of semistandard tableaux using the column internal insertion as a building block.

\begin{proposition}[$\iota_2$ through column insertion] 
\label{prop:iota2 by column insertion}

    Let $(P,Q)\in \SST(\lambda/\rho,n)^2$. Then $(\overline{P},\overline{Q}) = \iota_2(P,Q) \in \SST(\overline{\lambda}/ \overline{\rho},n)^2$ can be computed in the following way:
    \begin{enumerate}[label = {\rm (\Roman*)}]
        \item  Let $k = \wt_1(Q)$ and $i_1< i_2<\cdots <i_k$ be the column labels of the 1-cells (the cells with label 1) of $Q$. Then
        \begin{align}\label{eq:iota2P'}
            \overline{P}=\mathcal{C}_{[i_k]}\cdots \mathcal{C}_{[i_1]}(P).
        \end{align}
        \item \label{item:construction_Q} $\overline{Q}$ is obtained by vacating all $1$-cells, decreasing by 1 the labels of all remaining cells and assigning label $n$ to all cells of the skew Young diagram $\widetilde{\lambda}/\lambda$ (i.e. assigning label $n$ to all new cells added to $P$ after the action of $\mathcal{C}_{[i_j]},~j=1,\dotsc,k$).
    \end{enumerate}

    By the symmetry of \Cref{prop:properties of iota and cRSK}, \Cref{item:iota_swap}, the operation $\iota_1$ can be realized as $\iota_1 (P,Q) = \iota_2(Q,P)$. Finally, by \Cref{prop:properties of iota and cRSK}, \Cref{item:cRSK iota}, we can realize the $\cRSK$ map iterating either $\iota_2$ or $\iota_1$ as $\cRSK = \iota_2^n = \iota_1^n$. 
\end{proposition}

\begin{proof}
    Let $R,T,\overline{R}, \overline{T}$ be respectively the partial standardizations of tableaux $P,Q,\overline{P}, \overline{Q}$. Since we can recover $\overline{P}, \overline{Q}$ by $\overline{R}, \overline{T}$ and the weights $\wt_i(\overline{P}),\wt_i(\overline{Q})$ for $i=1,\dots,n$ and the latter are prescribed by \Cref{prop:properties of iota and cRSK} \ref{item:content_i2}, we focus on the description of $\overline{R}, \overline{T}$. By \Cref{prop:properties of iota and cRSK} \ref{item:iota_standardization} we have $(\overline{R}, \overline{T})=\iota_2^{\wt_1(Q)}(R,T)$ and by \Cref{prop:iota2 by column insertion} each of the $\wt_1(Q)$ iterations of the operation $\iota_2$ can be realized through a column insertion. More precisely letting $(\overline{R}_\ell,\overline{T}_\ell) = \iota_2^\ell(R,T)$ we have that $\overline{R}_\ell = \mathcal{C}_{[i_\ell]}(\overline{R}_{\ell-1})$ where $i_\ell$ is the column label of the 1-cell in $\overline{T}_{\ell-1}$, which is also equal to the column label of the $\ell$-th 1-cell in $Q$ from the left. Translating this procedure in terms of the tableau $P$, proves \eqref{eq:iota2P'}. On the other hand statement \ref{item:construction_Q} is immediate from \Cref{def:cRSK_i1_i2}, completing the proof.
\end{proof}

\begin{example}\label{ex:iota2}
In \Cref{fig:iota_1_iota_2} we showed an example of the evaluation of $\iota_2$. Below we evaluate $\iota_2(P,Q)$ using the alternative algorithm described in \Cref{prop:iota2 by column insertion}: 
\[
    \includegraphics[width=.7\linewidth]{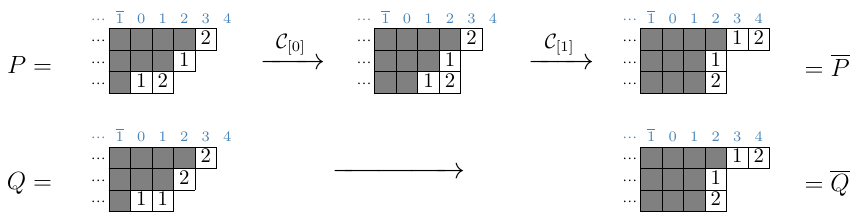} . 
\]
\end{example}

\subsection{Asymptotic stabilization} Here we discuss the stabilization phenomenon, produced after iterating the cRSK map $\cRSK^n(P,Q)$, or its inverse for a sufficiently large number of times.

\begin{proposition} \label{prop:stable field}
    Fix $(P,Q)\in\bigsqcup_{\lambda,\rho}\SST(\lambda/\rho,n)^2$ and let $\boldsymbol{\lambda} = \boldsymbol{\lambda}[P,Q]$ be the corresponding field of interlacing signatures on $\mathscr{C}_{n}$. Then, there exist $\mu\in\mathbb{Y}$, $t_+\in\N$ and $\mu^-\in\mathbb{Y}^-$ and $t_-\in\Z_{\le 0}$ such that
    \begin{subequations}
    \begin{align} \label{eq:forward stabilization field}
        \boldsymbol{\lambda} ((t+1)n+x,y) & = \boldsymbol{\lambda} (tn+x,y) + \mu, && \text{ for all } t > t_+, ~~ x,y \in \{0,\dots,n\},
    \\ \label{eq:backward stabilization field}
        \boldsymbol{\lambda} ((t-1)n+x,y) & = \boldsymbol{\lambda} (t n+x,y) - \mu^-, && \text{ for all } t < t_-, ~~ x,y \in \{0,\dots,m\}.
    \end{align}
    \end{subequations}
    Here $\mathbb{Y}^-:=\bigsqcup_{n\ge 1}\mathbb{Y}^-_n$ with $\mathbb{Y}^-_n:=\{(\lambda^{-}_1,\lambda^{-}_2,\dotsc,\lambda^-_n)\in \Z_{\le 0}^n \mid \lambda^-_1\le\cdots\le\lambda^-_n\}$.
\end{proposition}

The proof of \Cref{prop:stable field} is based on a simple monotonicity property of the column internal insertion, explained next.

\begin{lemma} \label{prop:cRSK_upright}
    Let $P \in \mathrm{SST}(\lambda/\rho,n)$ and $\widetilde{P}=\mathcal{C}_{[k]}(P)\in \mathrm{SST}(\widetilde{\lambda}/\widetilde{\rho},n)$ for some admissible column $k$. Fix $i\in \{1,\dots,n\}$ and let $(r,c) \in \lambda/\rho$, $(\widetilde{r},\widetilde{c}) \in \widetilde{\lambda}/ \widetilde{\rho}$ be the cell containing the $j$-th leftmost $i$ label respectively in $P$, $\widetilde{P}$ (assuming it exists). Then
    \begin{equation}
        r \ge \widetilde{r} \qquad \text{and} \qquad c \le \widetilde{c}.
    \end{equation}
    In words, any tagged cell of $P$ moves weakly in up-right direction under column internal insertion.
\end{lemma}

\begin{proof}
    This is a straightforward property of the column insertion.
\end{proof}

\begin{proof}[Proof of \Cref{prop:stable field}]
    Without loss of generality we can assume that the field $\boldsymbol{\lambda}$ is fully refined and as a result the tableaux $P,Q$ are partially standard. By \Cref{prop:properties of iota and cRSK}, we have, for all $\tau \in \mathbb{Z}$,
    \begin{equation}\label{eq:shift field}
        \boldsymbol{\lambda}(x,y+\tau) = \boldsymbol{\lambda}[\iota_2^{\tau}(\widetilde{P},\widetilde{Q})](n,0),
    \end{equation}
    where $(\widetilde{P},\widetilde{Q})= \iota_1^{x-n} \circ \iota_2^y (P,Q)$. By \Cref{prop:iota2 by column insertion} the operator $\iota_2(\widetilde{P},\widetilde{Q})$ can be realized as a column internal insertion in $\widetilde{P}$ at a column determined from $\widetilde{Q}$. Fix $i \in \{1,\dots,n\}$, so that $\widetilde{P}$ has a unique $i$-cell and let $(r_0,c_0)$ be its coordinate. Similarly define $(r_\tau,c_\tau)$ be the coordinate of the $i$-cell of the $P$-tableau of $\iota_2^\tau(\widetilde{P},\widetilde{Q})$. By \Cref{prop:cRSK_upright}, the sequence $r_0,r_1,\dots$ is weakly decreasing and since it is also bounded from below by 1, there exists $\tau'$ such that $r_{\tau'} = r_{\tau'+1}=\cdots$. Since the tagged cell was arbitrary, there must exists $\tau_*$ such that the content of each row of the $P$-tableau of $\iota_2^\tau(\widetilde{P},\widetilde{Q})$ remains constant for $\tau \ge \tau_*$. Setting $\tau = t n$ and using the identity $\iota_2^n = \iota_1^n$ we have
    \[
        \iota_2^{t n}(\widetilde{P},\widetilde{Q}) =  \iota_1^{ t n} (\widetilde{P},\widetilde{Q}),
    \]
    and the above argument applied to the $Q$-tableau of $\iota_1^{tn}(\widetilde{P},\widetilde{Q})$ proves that there exists $t_+$ large enough such that, for all $t \ge t_+$, the row content of each of the tableau of the pair $\iota_1^{tn}(\widetilde{P},\widetilde{Q})$ does not change. Such stabilization of row contents implies, following the internal insertions in the computation of the transition
    \begin{equation}
        \iota_2^{t n}(\widetilde{P},\widetilde{Q}) \to \iota_2^{t n + n}(\widetilde{P},\widetilde{Q}) = \cRSK ~ \iota_2^{t n}(\widetilde{P},\widetilde{Q}), 
    \end{equation}
    that there exists an array $\mu$ such that
    \begin{equation} \label{eq:stabilization external shape}
        \boldsymbol{\lambda}[\iota_2^{t n}(\widetilde{P},\widetilde{Q})](2n,0) = \boldsymbol{\lambda}[\iota_2^{t n}(\widetilde{P},\widetilde{Q})](n,0) + \mu.
    \end{equation}
    Here we used the fact that $\boldsymbol{\lambda}[\iota_2^{t n}(\widetilde{P},\widetilde{Q})](n,0)$ is the external shape of the tableaux $\iota_2^{t n}(\widetilde{P},\widetilde{Q})$. Moreover the array $\mu$ must necessarily be a partition for the relation \eqref{eq:stabilization external shape} to hold for arbitrarily large $t$. Combining relations \eqref{eq:shift field} and \eqref{eq:stabilization external shape} completes the proof of \eqref{eq:forward stabilization field}. The proof of \eqref{eq:backward stabilization field}  follows an analogous argument and therefore we omit it.
\end{proof}

\begin{corollary} \label{prop:stable}
    Consider $(P_t,Q_t)_{t\in \mathbb{Z}}$ the cRSK dynamics with initial data $(P,Q)$. Let $\lambda/\rho$ be the common shape of $P,Q$ with $\lambda,\rho \in \mathbb{S}_\ell$. Then, there exist $\mu = (\mu_1 \ge \cdots \ge \mu_{\ell'}>0) \in\mathbb{Y}$, $t_+\in\N$ and $\mu^-= (0<\mu_{k}^- \le \cdots \le \mu_1^-) \in\mathbb{Y}^-$ and $t_-\in\Z_{\le 0}$ such that, for all $t \ge t_+$
    \[
        P_t  = P_{t_+} + (t-t_+) \times (\mu_1,\dots,\mu_{\ell'},0,\dots,0),
        \qquad Q_t = Q_{t_+} + (t-t_+) \times (\mu_1,\dots,\mu_{\ell'},0,\dots,0), 
    \]
    and, for all $t \le t_-$
    \[
         P_t = P_{t_-} + (t-t_-) \times (0, \dots, 0 , \mu^-_{k} ,\dots , \mu^-_{1}),
         \qquad Q_t  = Q_{t_-} + (t-t_-) \times (0, \dots, 0 , \mu^-_{k} ,\dots , \mu^-_{1}).
    \]
    Here, if $T= (\lambda^{(0)} \preceq \cdots \preceq \lambda^{(n)})$ is a semistandard Young tableau, we denoted by $T + \nu= (\lambda^{(0)} + \nu \preceq \cdots \preceq \lambda^{(n)}+\nu)$.
\end{corollary}

Based on \cref{prop:stable}, we define give the following two definitions.

\begin{definition}[Stability] \label{def:stability}
    Let $(P,Q) \in \SST(\lambda,\rho,n)$ and for all $t \in \mathbb{Z}$ let $(P_t,Q_t)=\cRSK^t(P,Q)$. We call the pair $(P,Q)$ \defn{stable} with \defn{asymptotic shape} $\mu = (\mu_1 \ge \cdots \ge \mu_{\ell'}>0)$ if 
    \[
        P_t = P + t \times (\mu_1 ,\dots, \mu_{\ell'},0,\dots,0), \qquad Q_t = Q + t \times (\mu_1 ,\dots, \mu_{\ell'},0,\dots,0)
        \qquad \text{for all } t\ge 0.
    \]
    Similarly, we call the pair $(P,Q)$ \defn{negatively stable} with \defn{negative asymptotic shape} $\mu^- = (0< \mu_{k}^-\le \cdots \le \mu_1^-)$ if
    \[
        P_t = P + t \times (0,\dots,0,\mu_k^-, \dots,\mu_1^-), \qquad Q_t = Q + t \times (0,\dots,0,\mu_k^-, \dots,\mu_1^-)
        \qquad \text{for all } t \le 0.
    \]
\end{definition}


The next definition should be paired with \Cref{def:Phi} given in the introduction.

\begin{definition} \label{def:Phi_minus}
    Given a pair of semistandard Young tableaux $(P,Q)$ of skew shape we define the map 
    \begin{equation}
        \Phi^- \colon (P,Q) \in \bigsqcup_{\lambda,\rho \in \mathbb{S}} \mathrm{SST} ( \lambda/\rho , n)^2 \longrightarrow (H_1^-, H_2^-) \in \bigsqcup_{\mu^- \in \partitions^-} \HWT(\mu^-,n)^2,
    \end{equation}
    where $\mathbb{Y}^-:=\bigsqcup_{k\ge 1}\{(\mu_k,\dotsc,\mu_1)\in\Z_+^k \mid \mu_k\le\cdots\le\mu_1\}$, and letting $\boldsymbol{\lambda} = \boldsymbol{\lambda}[P,Q]$, we set
    \begin{subequations}
    \begin{align}
        \#\{ i\text{-cells at row }j \text{ of }H_1^- \} & = \lim_{t\to -\infty} \boldsymbol{\lambda}_{L+j}(t,i)-\boldsymbol{\lambda}_{L+j}(t,i-1),
        \\
        \#\{ i\text{-cells at row } j \text{ of }H_2^- \} & = \lim_{t\to -\infty} \boldsymbol{\lambda}_{L+j}(t+i,0)-\boldsymbol{\lambda}_{L+j}(t+i-1,0). 
    \end{align}
    \end{subequations}
    Above $L = \max\{ i \ge 0 \mid \lim_{t \to -\infty} \boldsymbol{\lambda}_{i}(t+i,0) \neq - \infty \}$, where $L=0$ if all the limits are $-\infty$.
\end{definition}

In other words, the map $\Phi^-(P,Q)$ collects the rows of negatively stable pairs $(P_t,Q_t)$ into horizontally weak tableaux with shape given by a reverse partition $\mu^-$. We will prove in \Cref{thm:scattering_rules} that the negatively asymptotic shape $\mu^-$ is the sorting of the positively asymptotic shape $\mu$. We will also prove that $H_i^-$ can be computed from $H_i$ through the use of the combinatorial $R$-matrix. 

\begin{example} \label{ex:projection}
    Let $(P,Q) = (P_0,Q_0)$ from \Cref{fig:RSK dynamics example}. Then, we have
    \begin{align*}
        \ytableausetup{aligntableaux = center,smalltableaux}
        \Phi\left(            \begin{ytableau}
            \none[\scriptstyle \ao{\cdots}] & \none[\scriptstyle \ao{\overline{1}}] & \none[\scriptstyle \ao{0}] & \none[\scriptstyle \ao{1}] & \none[\scriptstyle \ao{2}]& \none[\scriptstyle \ao{3}]
            \\
            \none[\scriptstyle \cdots] & *(gray) & *(gray) & *(gray) & *(gray) & 2
            \\
            \none[ \scriptstyle \cdots] & *(gray) & *(gray) & *(gray) & 1 
                \\
                \none[\scriptstyle \cdots]  & *(gray) & 1 & 2
            \end{ytableau}
            \,
            ,
            \, \begin{ytableau}
            \none[\scriptstyle \ao{\cdots}] & \none[\scriptstyle \ao{\overline{1}}] & \none[\scriptstyle \ao{0}] & \none[\scriptstyle \ao{1}] & \none[\scriptstyle \ao{2}]& \none[\scriptstyle \ao{3}]
            \\
            \none[\scriptstyle \cdots] & *(gray) & *(gray) & *(gray) & *(gray) & 2
            \\
            \none[ \scriptstyle \cdots] & *(gray) & *(gray) & *(gray) & 2 
                \\
                \none[\scriptstyle \cdots]  & *(gray) & 1 & 1
            \end{ytableau}
        \right) & = 
        \left( 
            \begin{ytableau}
                1 & 1 & 2
                \\
                2
            \end{ytableau}
            \,
            ,
            \, 
            \begin{ytableau}
                1 & 2 & 2
                \\
                1
            \end{ytableau}
        \right),
    \allowdisplaybreaks \\
        \ytableausetup{aligntableaux = center,smalltableaux}
        \Phi^-\left(            \begin{ytableau}
            \none[\scriptstyle \ao{\cdots}] & \none[\scriptstyle \ao{\overline{1}}] & \none[\scriptstyle \ao{0}] & \none[\scriptstyle \ao{1}] & \none[\scriptstyle \ao{2}]& \none[\scriptstyle \ao{3}]
            \\
            \none[\scriptstyle \cdots] & *(gray) & *(gray) & *(gray) & *(gray) & 2
            \\
            \none[ \scriptstyle \cdots] & *(gray) & *(gray) & *(gray) & 1 
                \\
                \none[\scriptstyle \cdots]  & *(gray) & 1 & 2
            \end{ytableau}
            \,
            ,
            \, \begin{ytableau}
            \none[\scriptstyle \ao{\cdots}] & \none[\scriptstyle \ao{\overline{1}}] & \none[\scriptstyle \ao{0}] & \none[\scriptstyle \ao{1}] & \none[\scriptstyle \ao{2}]& \none[\scriptstyle \ao{3}]
            \\
            \none[\scriptstyle \cdots] & *(gray) & *(gray) & *(gray) & *(gray) & 2
            \\
            \none[ \scriptstyle \cdots] & *(gray) & *(gray) & *(gray) & 2 
                \\
                \none[\scriptstyle \cdots]  & *(gray) & 1 & 1
            \end{ytableau}
        \right) & = 
        \left( 
            \begin{ytableau}
                1
                \\
                1 & 2 & 2
            \end{ytableau}
            \,
            ,
            \, 
            \begin{ytableau}
                2
                \\
                1 & 1 &2
            \end{ytableau}
        \right).
    \end{align*}
\end{example}

\subsection{Kernels of tableaux}

By \Cref{prop:symmetry Fomin}, given a Fomin field of interlacing signatures $\boldsymbol{\lambda}$ on $\mathscr{C}_n$ and a signature $\nu \in \mathbb{S}$ we can define another Fomin field of signature $\overline{\boldsymbol{\lambda}}(p) = (\boldsymbol{\lambda}(p)'+\nu')'$ for all $p\in\mathscr{C}_n$ by globally adding the signature $\nu$ to $\boldsymbol{\lambda}$ column-by-column. In a similar fashion we can globally substract from the field $\boldsymbol{\lambda}$ a signature, provided such operation returns a signature at each point $p \in \mathscr{C}_n$.

\begin{definition}\label{def:kernel}
    Let $\boldsymbol{\lambda}$ be a Fomin field of signatures. We define $\nu = \mathrm{ker}(\boldsymbol{\lambda})$ to be the maximal signature $\nu$ such that $\boldsymbol{\lambda}(p)' -\nu'$ is a transposed signature for all $p \in \mathscr{C}_n$. Similarly, given $(P,Q)\in\bigsqcup_{\lambda,\rho}\SST(\lambda/\rho,n)^2$, we define $\ker(P,Q) = \ker(\boldsymbol{\lambda}[P,Q])$.
\end{definition}

Notice that if $\nu, \eta \in \mathbb{S}$ are signatures such that both $\boldsymbol{\lambda}(p)' - \nu'$ and $\boldsymbol{\lambda}(p)' - \eta'$ are transposed signatures, then also the signature $\max(\nu,\eta)$ has the same property. This implies that there exists a unique maximal signature $\nu$ such that $\boldsymbol{\lambda}(p)' - \nu'$ is a transposed signature for all $p\in \mathscr{C}_n$ justifying \Cref{def:kernel}. 

We now state some basic proposition of the kernel.

\begin{proposition} \label{prop:properties_kernel}
    Let $(P,Q)\in\bigsqcup_{\lambda,\rho}\SST(\lambda/\rho,n)^2 $ and let $\boldsymbol{\lambda} = \boldsymbol{\lambda}[P,Q]$ be the corresponding field of interlacing partitions on $\mathscr{C}_{n}$. Let also $\mu$ be the asymptotic shape of the pair $(P,Q)$ and $\nu = \ker(P,Q)$. The following properties hold.
    \begin{enumerate}[label = {\rm (\Roman*)}]
        \item \label{item:kernel_iota} We have $\nu = \ker \iota_1(P,Q)= \ker \iota_2(P,Q) = \ker \cRSK (P,Q)$.
        \item \label{item:kernel_loop} Fix a down right loop $\mathbf{p}$ of $\mathscr{C}_n$ and let $\widetilde{\nu}$ be the maximal signature such that $\boldsymbol{\lambda}(p)'-\widetilde{\nu}'$ is a transposed signature for all $p\in \mathbf{p}$. Then $\widetilde{\nu} = \nu$.
        \item \label{item:ker_length} If $\lambda$ is the shape of $P$, then $\ell(\lambda) = \ell(\nu) + \ell(\mu)$.
        \item \label{item:ker_asymptotic_pos} We have $\nu_i = \lim_{t \to + \infty} \boldsymbol{\lambda}_{\ell(\mu)+i}(p+t\,\mathbf{e}_1)$ for all $i=1,\dots \ell(\nu)$ and all $p \in \mathscr{C}_n$.
        \item \label{item:ker_asymptotic_neg} We have $\nu_i = \lim_{t \to - \infty} \boldsymbol{\lambda}_i(x-t,y)$ for all $i=1,\dots \ell(\nu)$ and all $x,y \in \{0,\dots,n\}$.
    \end{enumerate}
\end{proposition}

\begin{proof}
    Properties \ref{item:kernel_iota} are immediate consequences of \Cref{prop:properties of iota and cRSK} \ref{item:iota2 field},\ref{item:iota1 field},\ref{item:cRSK field}, which express the Fomin fields of $\iota_2(P,Q), \iota_1(P,Q), \cRSK(P,Q)$ as coordinate shifts of $\boldsymbol{\lambda}$. As a result the corresponding kernels are all identical. 
    To show \ref{item:kernel_loop}, notice that $\nu' \le \widetilde{\nu}'$, since the constraints for signature $\nu$ contain those for $\widetilde{\nu}$. On the other hand, define the field $\widetilde{\boldsymbol{\lambda}}$ to be the (unique) Fomin field such that $\widetilde{\boldsymbol{\lambda}}(p) ' = \boldsymbol{\lambda}(p)'-\widetilde{\nu}'$ for all $p\in\mathbf{p}$. Then, by \Cref{prop:symmetry Fomin} the relation $\widetilde{\boldsymbol{\lambda}}(p) ' = \boldsymbol{\lambda}(p)'-\widetilde{\nu}'$ holds for all $p \in \mathscr{C}_n$, which implies that $\widetilde{\nu}' \le \nu'$ and as a result that $\widetilde{\nu} = \nu$.
    
    To show \ref{item:ker_length} consider, for a fixed $p\in\mathscr{C}_n$ and a fixed $i$ the sequence $L_t(i)=\boldsymbol{\lambda}_i(p+t\,\mathbf{e}_1)$. Then, by monotonicity properties of the Fomin growth operator the sequence $L_t(i)$ is monotonic increasing and by \Cref{prop:stable field}, $L_t(i)$ is diverging in $t$ for $i \in \{1,\dots,\ell(\mu)\}$ and converging to a finite value $L_\infty(i)$ for $i\in \{ \ell(\mu) +1 ,\dots \ell(\lambda)\}$. Define the signature $\widetilde{\nu}_i = L_\infty(\ell(\mu)+i)$, for $i=1\dots,\ell(\lambda)-\ell(\mu)$ and let $t_*$ such that $\boldsymbol{\lambda}_{\ell(\mu)+i}(p+t\,\mathbf{e}_1) = \widetilde{\nu}_i$ for all $t \ge t_*$. Fix now $\mathbf{p}$ a down right loop such that $(0,0)\in \mathbf{p}$. Then it is clear that $\widetilde{\nu}$ is the maximal signature such that $\boldsymbol{\lambda}(p+(t_*+n)\,\mathbf{e}_1)'-\widetilde{\nu}'$ is a transposed signature for all $p \in \mathbf{p}$, which implies, by \ref{item:kernel_iota}, that $\widetilde{\nu}=\nu$. Notice that we also showed \ref{item:ker_length}. One can prove \ref{item:ker_asymptotic_neg} through analogous argument as for \ref{item:ker_asymptotic_pos}.
\end{proof}

\begin{example}
    Let $\boldsymbol{\lambda}$ be the Fomin of \Cref{fig:cylinder} and let $(P,Q) = (P_0,Q_0)$ from \Cref{fig:RSK dynamics example}, so that $\boldsymbol{\lambda} = \boldsymbol{\lambda}[P,Q]$. Then, we have
    \[
        \ytableausetup{aligntableaux = center,smalltableaux}
        \ker(\boldsymbol{\lambda}) =
        \ker \left(          
        \begin{ytableau}
            \none[\scriptstyle \ao{\cdots}] & \none[\scriptstyle \ao{\overline{1}}] & \none[\scriptstyle \ao{0}] & \none[\scriptstyle \ao{1}] & \none[\scriptstyle \ao{2}]& \none[\scriptstyle \ao{3}]
            \\
            \none[\scriptstyle \cdots] & *(gray) & *(gray) & *(gray) & *(gray) & 2
            \\
            \none[ \scriptstyle \cdots] & *(gray) & *(gray) & *(gray) & 1 
                \\
                \none[\scriptstyle \cdots]  & *(gray) & 1 & 2
            \end{ytableau}
            \,
            ,
            \, \begin{ytableau}
            \none[\scriptstyle \ao{\cdots}] & \none[\scriptstyle \ao{\overline{1}}] & \none[\scriptstyle \ao{0}] & \none[\scriptstyle \ao{1}] & \none[\scriptstyle \ao{2}]& \none[\scriptstyle \ao{3}]
            \\
            \none[\scriptstyle \cdots] & *(gray) & *(gray) & *(gray) & *(gray) & 2
            \\
            \none[ \scriptstyle \cdots] & *(gray) & *(gray) & *(gray) & 2 
                \\
                \none[\scriptstyle \cdots]  & *(gray) & 1 & 1
            \end{ytableau}
        \right) 
        = 
        \begin{ytableau}
            \none[\scriptstyle \ao{\cdots}] & \none[\scriptstyle \ao{\overline{1}}] & \none[\scriptstyle \ao{0}] & \none[\scriptstyle \ao{1}] & \none[\scriptstyle \ao{2}]
            \\
            \none[\scriptstyle \cdots] & *(gray) & *(gray) & *(gray) & *(gray) 
        \end{ytableau}
        =
        (2) \in \mathbb{S}_1.
    \]
    In line with \Cref{prop:properties_kernel} \ref{item:kernel_loop}, we see that the signature $\nu=(2)$, shaded in yellow below, is the maximal signature one can subtract column-wise simultaneously from $P,Q$ as
    \[
        \ytableausetup{aligntableaux = center,smalltableaux}
        \left(          
        \begin{ytableau}
            \none[\scriptstyle \ao{\cdots}] & \none[\scriptstyle \ao{\overline{1}}] & \none[\scriptstyle \ao{0}] & \none[\scriptstyle \ao{1}] & \none[\scriptstyle \ao{2}]& \none[\scriptstyle \ao{3}]
            \\
            \none[\scriptstyle \cdots] & *(yellow) & *(yellow) & *(yellow) & *(yellow) & 2
            \\
            \none[ \scriptstyle \cdots] & *(gray) & *(gray) & *(gray) & 1 
                \\
                \none[\scriptstyle \cdots]  & *(gray) & 1 & 2
            \end{ytableau}
            \,
            ,
            \, \begin{ytableau}
            \none[\scriptstyle \ao{\cdots}] & \none[\scriptstyle \ao{\overline{1}}] & \none[\scriptstyle \ao{0}] & \none[\scriptstyle \ao{1}] & \none[\scriptstyle \ao{2}]& \none[\scriptstyle \ao{3}]
            \\
            \none[\scriptstyle \cdots] & *(yellow) & *(yellow) & *(yellow) & *(yellow) & 2
            \\
            \none[ \scriptstyle \cdots] & *(gray) & *(gray) & *(gray) & 2 
                \\
                \none[\scriptstyle \cdots]  & *(gray) & 1 & 1
            \end{ytableau}
        \right) 
        \xrightarrow[]{\qquad}
        \left(          
        \begin{ytableau}
            \none[\scriptstyle \ao{\cdots}] & \none[\scriptstyle \ao{\overline{1}}] & \none[\scriptstyle \ao{0}] & \none[\scriptstyle \ao{1}] & \none[\scriptstyle \ao{2}]& \none[\scriptstyle \ao{3}]
            \\
            \none[ \scriptstyle \cdots] & *(gray) & *(gray) & *(gray) & 1 &2
                \\
                \none[\scriptstyle \cdots]  & *(gray) & 1 & 2
            \end{ytableau}
            \,
            ,
            \, \begin{ytableau}
            \none[\scriptstyle \ao{\cdots}] & \none[\scriptstyle \ao{\overline{1}}] & \none[\scriptstyle \ao{0}] & \none[\scriptstyle \ao{1}] & \none[\scriptstyle \ao{2}]& \none[\scriptstyle \ao{3}]
            \\
            \none[ \scriptstyle \cdots] & *(gray) & *(gray) & *(gray) & 2 & 2
                \\
                \none[\scriptstyle \cdots]  & *(gray) & 1 & 1
            \end{ytableau}
        \right) .
    \]
\end{example}

\section{Box-ball system and its linearization}

\subsection{Box-ball system dynamics}

In \Cref{subs:overview} we introduced the BBS dynamics with space of states $\BBS$ and transfer matrix $\sfT$. The support of a state $B \in \BBS$ is $\supp(B) := \{x \in \Z \mid b_x = 1\}$. In addition, we define 
\[
    \BBS_{>0}:= \left\{ B = (b_x)_{x\in \Z}\in\BBS \midspan b_x=0 \text{ for all } x\le0 \right\},
\] 
the set of binary sequences supported in $\Z_{>0}$. For any fixed $B_0\in\BBS$, the BBS dynamics is the sequence $(B_t)_{t\in\Z}$ defined by
\begin{align} \label{eq:Bt}
    B_t = \sfT^t(B_0),  
\end{align}
where for $t<0$ we interpret $\sfT^{t}=(\sfT^{-1})^{|t|}$.
    
\begin{definition} \label{def:BBS_asymptotic_shapes}
    Let $B\in\BBS$. A partition $\mu=(\mu_1\ge \cdots \ge \mu_\ell>0)$ is called the \defn{asymptotic shape} of $B$ if, for all sufficiently large $t$, the configuration $\sfT^t(B)$ contains exactly $\ell$ clusters of consecutive balls, of lengths $\mu_1,\dots,\mu_\ell,$ counted with multiplicity.
\end{definition}

Notice that it is in fact non-trivial that the BBS dynamics stabilizes and that eventually states consist of constant (in time) clusters of consecutive particles evolving freely. This a well established fact; see \cite{Takahashi_Satsuma}.

One could also define the \defn{backward asymptotic shape} of a BBS state $B$ to be a partition $\widetilde{\mu} = (\widetilde{\mu}_1\ge \cdots \ge \widetilde{\mu}_m>0)$ such that, for all sufficiently large $t$, the configuration $\sfT^{-t}(B)$ contains exactly $m$ clusters of consecutive balls, of lengths $\widetilde{\mu}_1,\dots,\widetilde{\mu}_m$ counted with multiplicity. Nevertheless it is well known, and illustrated by \Cref{fig:BBS}, that the backward asymptotic shape coincides with the asymptotic shape $\mu$. This result was proven first in \cite{Takahashi_Satsuma} and recalled in \Cref{thm:KKR linearization} below.

\subsection{Linearization of the BBS dynamics} \label{subs:KKR}

We now recall the linearization of the single-species BBS by the  \defn{Kerov--Kirillov--Reshetikhin (KKR) bijection} \cite{KKR86,KR86}; see also \cite{HKT00,HHIKTT01,HKOTY02,KOSTY2006RC}. We begin by introducing the target space of this bijection. 

\begin{definition}
    A \defn{rigged configuration} is a pair $(\mu,J)$, where $\mu$ is a partition and $J\in \mathcal{K}(\mu)$ is a collection of integers, called \defn{riggings}. We recall the set $\mathcal{K}(\mu)$ was introduced in \eqref{eq:tKappa} We write $\RC$ for the set of all rigged configurations. 
    Finally we define the set 
    \[
        \RC' = \{(\mu,J) \in \RC \mid J_i \ge - \mu_i \}.
    \]
\end{definition}

\begin{theorem}[KKR bijection and linearization, \cite{KOSTY2006RC}] \label{thm:KKR linearization}
    There exists a bijection $\KKR \colon \BBS \leftrightarrow \RC$ such that, if $\KKR(B)=(\mu,J)$, then $\mu$ is the asymptotic shape of $B$ (and also its backward asymptotic shape). The restriction to states $B\in \BBS_{>0}$ gives a bijection $\KKR \colon \BBS_{>0} \leftrightarrow \RC'$.
    Moreover, defining the map $\hat{T}(\mu,J) = (\mu, J + \mu)$, then $\KKR \circ \sfT = \hat{T} \circ \KKR$, as illustrated by the commutative diagram~\eqref{eq:BBS commutative diagram}.
\end{theorem}

Let us describe explicitly the bijection $\KKR$; see also~\cite{Kirillov-Schilling-Shimozono2002_RCbijection} for the most general version for $\asl_n$. An example of \Cref{thm:KKR linearization} is given in \Cref{ex:KKR}

\begin{definition}[The KKR bijection]\label{def:KKR}
    Let $B=(b_i)_{i\in\Z}\in\BBS$ and let $\ell= \min \{ \supp(B)\}$ and $r=\max \{ \supp(B) \}$ denote the positions of the leftmost and rightmost particles. The rigged configuration $\KKR(B)$ is constructed inductively by scanning the interval $[\ell,r]$ from left to right and producing a sequence or rigged configurations $(\mu^{[k]},J^{[k]})_{k \in \{\ell,\dots,r \} }$. To each rigged configuration $(\mu^{[k]},J^{[k]})$ we associate the \defn{vacancy numbers}
    \begin{equation} \label{eq:vacancy}
        p_i^{[k]} := k - 2 \sum_{j=1}^{\ell(\mu^{[k]})} \min(\mu_j^{[k]}, i).
    \end{equation}
    Given a rigged configuration $(\mu^{[k]},J^{[k]})$, we say that $i$ is a \defn{singular string} of length $\mu_i^{[k]}$ if
    \[
        J_i^{[k]} = p_{\mu_i^{[k]}}^{[k]}.
    \]
    Set $(\mu^{[\ell-1]},J^{[\ell-1]})=(\emptyset,\emptyset)$. For $k=\ell,\ell+1,\dots,r$, define $(\mu^{[k]},J^{[k]})$ from $(\mu^{[k-1]},J^{[k-1]})$ as follows:
    \begin{enumerate}
        \item If $b_k=0$, then $(\mu^{[k]},J^{[k]})=(\mu^{[k-1]},J^{[k-1]})$.

        \item If $b_k=1$, choose one of the longest singular strings in $(\mu^{[k-1]},J^{[k-1]})$ and replace it with a singular string whose length is larger by one (with vacancy numbers computed with respect to $k$). We assume, by convention, that there always exists a singular string of length $0$.
    \end{enumerate}
    After the step $k=r$, we yield $\KKR(B) = (\mu^{[r]},J^{[r]})$.
\end{definition}

The inverse of the map $\KKR$ is constructed next.

\begin{definition}[The inverse KKR bijection]\label{def:Phi-1}
    Let $(\mu,J)\in\RC$ be a rigged configuration. We reconstruct the BBS state $B=(b_i)_{i\in\Z}= \KKR^{-1}(\mu,J)$ recursively from right to left. First determine the position of the rightmost particle by
    \begin{align} \label{eq:rightmost}
        r=\max_{1\le s \le \ell(\mu)} p_{\mu_s}^{[J_{s}]}.
    \end{align}
    (Here $p_{\mu_s}^{[J_{s}]}$ is computed as in \eqref{eq:vacancy} replacing $\mu^{[k]}$ by $\mu$ and $k$ by $J_s$.)  
    Set $b_k=0$ for all $k>r$ and initialize $(\mu^{[r]},J^{[r]})=(\mu,J)$.

    For $k=r,r-1,r-2,r-3,\ldots$, if $(\mu^{[k]}, J^{[k]}) = (\emptyset, \emptyset)$, then terminate, and otherwise construct $(\mu^{[k-1]},J^{[k-1]})$ and $b_k$ as follows:
    \begin{enumerate}
        \item If $(\mu^{[k]},J^{[k]})$ contains no singular strings, then set
        \[
            b_k=0,
            \qquad\quad
            (\mu^{[k-1]},J^{[k-1]})=(\mu^{[k]},J^{[k]}).
        \]
        \item Otherwise set $b_k = 1$ and choose one of the shortest singular strings of $(\mu^{[k]},J^{[k]})$ and replace it with a singular string whose length is smaller by one (with vacancy numbers computed with respect to $\mu^{[k-1]}$).
    \end{enumerate}
    Suppose this terminates at $\ell$, then we set $b_i = 0$ for all $i < \ell$.
\end{definition}

\begin{example} \label{ex:KKR}
    Consider the BBS configuration $B$ depicted in \Cref{fig:BBS}
    \begin{equation} \label{eq:BBS config example}
        B = \,\, 
        \begin{matrix}
             \ao{\cdots} & \ao{\overline{1}} & \ao{0} & \ao{1} & \ao{2} & \ao{3} & \ao{4} & \ao{5} & \ao{\cdots} \\
             \grc{\cdots} & \grc{0} & 1 & 1 & \grc{0} & 1 & 1 & \grc{0}  &\grc{\cdots}
        \end{matrix}
        \,\,
        .
    \end{equation}
    Then, we can compute $\KKR(B) = (\mu,J)$, with $\mu=(3,1)$ and $J=(-4,0)$, as show in the right panel of \Cref{fig:KKR}. 
    Similar computations show that, letting $(B_t)_{t \in \mathbb{Z}}$ be the BBS dynamics with initial data $B_0=B$, we have
    \[
        \KKR (B_t) = (\mu, J + t \mu) \qquad \text{for all } t\in \mathbb{Z},
    \]
    in agreement with \Cref{thm:KKR linearization}.
\end{example}

\begin{figure}
    \centering
    \includegraphics[width=0.3\linewidth]{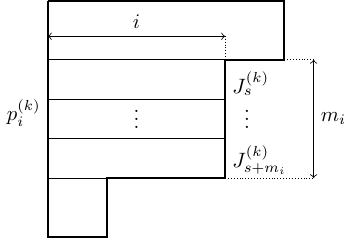}
    \qquad
    \includegraphics[width=0.5\linewidth, trim=2cm 0 2cm 0, clip]{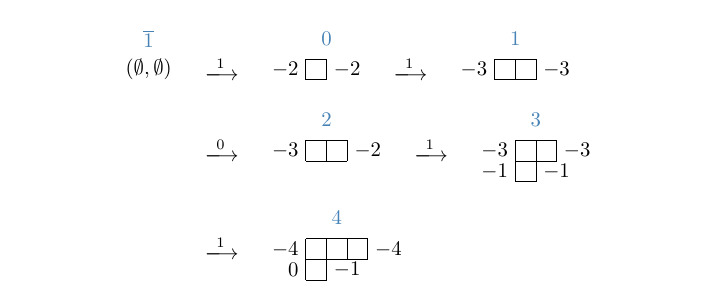}
    \caption{In the left panel, a representation of a rigged configurations $(\mu^{[k]},J^{[k]})$ as the Young diagrams $\mu^{[k]}$ with values of $J_i$ placed at the right of row $i$ and vacancy numbers $p^{[k]}_i$ placed at the left of rows of $\mu^{[k]}$ of length $i$ (see, e.g., \cite{Inoue-Kuniba-Takagi2012BBS}). In the right panel, the evaluation of the map $\KKR(B)$ with $B$ given in~\eqref{eq:BBS config example}.}
    \label{fig:KKR}
\end{figure}

\subsection{\textbf{cRSK}, BBS and KKR}
\label{subsec:cRSK_KKR}

Define the set
\begin{equation} \label{eq:SST1}
    \SST(1) = \bigsqcup_{\rho,\lambda \in \mathbb{S}} \SST(\lambda/\rho,1)
\end{equation}
of semi standard tableaux of skew shape with all labels equal to 1. (Notice that $\SST(\lambda/\rho,1)$ is either empty or it consists of a single element $T=(\rho \preceq \lambda)$.)
In this subsection, we build on the equivalence between BBS and cRSK dynamics on pairs of identical tableaux $(T,T)$ with $T \in \SST(1)$ stated in \Cref{rem:RSK to BBS} and formalized in \Cref{prop:cRSK_BBS}. The main result is \Cref{prop:volume_BBS_tableau}, which relates the shape of $T$ with the rigging associated to $B$.

\begin{definition}[The projection $\Pr$] \label{def:map_to_BBS}
Let $T = (\rho \preceq \lambda)$ be a tableau in $\SST(1)$. Then we can define the projection $\Pr\colon \SST(1) \rightarrow \BBS$ as
\[
    \Pr(T) = \lambda'-\rho'.
\]
\end{definition}

\begin{example}
    Consider the tableau $T \in \SST(1)$ as
    \[
        T =
        \begin{ytableau}
                \none[\scriptstyle \ao{\cdots}] & \none[\scriptstyle \ao{ \overline{1}}] & \none[\scriptstyle \ao{0}] & \none[\scriptstyle \ao{1}] & \none[\scriptstyle \ao{2}] & \none[\scriptstyle \ao{3}] & \none[\scriptstyle \ao{4}]
                \\
                \none[\scriptstyle \cdots] & *(gray) & *(gray) & *(gray) & *(gray) & 1 & 1
                \\
                \none[ \scriptstyle \cdots] & *(gray) & *(gray) & *(gray) & *(gray)
                \\
                \none[\scriptstyle \cdots] & *(gray) & 1 & 1
            \end{ytableau}
            \,.
    \]
    Then, we have $\Pr(T) = B$, where $B$ is given in \eqref{eq:BBS config example}.
\end{example}

\begin{lemma} 
\label{lem:pr_bij}
    The map
    \begin{equation} \label{eq:Xi}
        \Xi \colon T \in \SST(1) \longrightarrow (B,\nu) \in \BBS \times \mathbb{S}, \qquad \text{with} \qquad B=\Pr(T) \quad \text{and} \quad \nu = \ker(T,T)
    \end{equation}
    is a bijection.
    Given $B = (b_i)_{i \in \Z} \in \BBS$ and $\nu \in \mathbb{S}$, the inverse $T=\Xi^{-1}(B,\nu)$ is constructed as follows. Setting $\ell = \sum_{x \in \Z } \mathbf{1}_{b_x=0 } \mathbf{1}_{ b_{x+1} =1}$, we define the signature $\overline{\rho}$ as
    \begin{equation} \label{eq:pr_bij}
        \overline{\rho}_{i} = \sup \left\{ j < \overline{\rho}_{i-1} \midspan (b_j,b_{j+1}) = (0,1)  \right\}, \qquad \text{for } i=1,\dots, \ell,
    \end{equation}
    where we assume the convention $\overline{\rho}_{-1} = + \infty$. Then, we set
    \[
        \rho' = \overline{\rho}' + \nu'
        \qquad
        \text{and}
        \qquad
        \lambda' = \rho'+B.
    \]
\end{lemma}

\begin{proof}
    The map constructed in \eqref{eq:pr_bij} can be described in words as follows. Given a BBS configuration $B = (b_i)_{i\in \Z} \in \BBS$, we scan it from the right to the left and set $\overline{\rho}_i = r_i$ where $r_i$ is the location of the $i$-th hole immediately to the left of a particle. Then, we set $\rho'= \overline{\rho}'+\nu'$,  $\lambda'=\rho'+B$ and the tableau $T=(\rho \preceq \lambda)$ clearly has the property that $\Pr(T) = B$ and $\ker(T,T) = \nu$, by \Cref{prop:properties_kernel}, \ref{item:kernel_loop} since $\nu$ is the maximal signature one can remove from $T$ without breaking the semistandard property.
\end{proof}

The equivalence between the cRSK dynamics of the BBS dynamics is summarized as follows.
Recall that $\sfT$ denotes the one time step evolution of the BBS described in \Cref{subs:overview}.

\begin{proposition}\label{prop:cRSK_BBS}
For any $T \in \SST(1)$, denote $T'=\cRSK (T)$ if $(T',T')=\cRSK(T,T)$. We have $\Pr \circ \, \mathbf{cRSK}(T) = \sfT \circ \Pr (T)$.
\end{proposition}

\begin{proof}
    This is essentially \Cref{rem:RSK to BBS}. If $T=(\rho \preceq \lambda)$ and $T'=(\lambda \preceq \eta) = \cRSK(T)$, then the increments $B=\lambda'-\rho'$ and $B'= \eta' - \lambda'$ are related as $B'=\sfT(B)$, implying $\Pr \circ \, \mathbf{cRSK}(T) = \sfT \circ \Pr (T)$.
\end{proof}

Let $ T = (\rho \preceq \lambda) \in \SST(1)$. 
We will have an explicit relation between the inner shape $\rho$ and the rigged configuration $(\mu,J)= \KKR \circ\Pr (T)$; see \Cref{prop:volume_BBS_tableau} below. For this, a result from \cite{Kuniba-Sakamoto-Yamada2007_tau} will play a key role. For its statement we introduce two functions $\Gamma_B,\tau_{\mu,J} \colon \Z\rightarrow \Z$. 

\begin{definition}
    Let $B \in \BBS$ and let $B_t=(b^t_i)_{i\in \mathbb{Z}}$ be its evolution under the BBS dynamics with initial data $B_0=B$. We define the function $\Gamma_B \colon \Z \to \Z$
    \[
        \Gamma_B(i) = \sum_{t \ge 0} \sum_{j \le i} b_j^t.  
    \]
\end{definition}

The function $\Gamma_B(i)$ counts the number of balls weakly to the left of location $i$ in all configurations $B_0(=B),B_1,B_2,\ldots$.
Since the original configuration $B$ is finitely supported and in the BBS dynamics every particle has positive drift, then $\Gamma_B(i)$ is always a finite quantity. 

The following lemma provides a way to reconstruct the configuration $B$ from the function $\Gamma_B$. For any function $f\colon \Z \to \C$ we define its discrete first and second derivatives 
\begin{equation*}
    \nabla f(i) = f(i) - f(i-1),
    \qquad\qquad
    \Delta f(i) = f(i) - 2f(i-1) + f(i-2).
\end{equation*}

\begin{lemma} \label{lem:gamma_B}
    For any $B \in \BBS$, we have 
    \[
        \Delta \Gamma_B(i)= \mathbf{1}_{b_{i-1} = 0} \mathbf{1}_{b_i = 1}.
    \]
    In particular, $\Gamma_B$ is piecewise linear, and its slopes $\nabla \Gamma_B(i)$ have increments being either 0 or 1.
\end{lemma}
\begin{proof}
Note that $\Delta\Gamma_B(i)=\sum_{t\ge0}(b^t_i-b^t_{i-1})$. Also observe that in the BBS dynamics, the possible configurations of four elements $b^t_{j-1},b^t_j,b^{t+1}_{j-1},b^{t+1}_j$ are, for any $t\in\N$ and $j\in\Z$
    \begin{align}\label{eq:4pts}
        \begin{pmatrix}
            b^t_{j-1} & b^t_j\\[5pt]
            b^{t+1}_{j-1}& b^{t+1}_j
        \end{pmatrix}
        =\begin{pmatrix}
            0&0\\0&0
        \end{pmatrix},
        \begin{pmatrix}
            0&0\\1&1
        \end{pmatrix},
        \begin{pmatrix}
            0&0\\1&0
        \end{pmatrix},
        \begin{pmatrix}
            0&1\\0&0
        \end{pmatrix},
        \begin{pmatrix}
            0&1\\1&0
        \end{pmatrix},
        \begin{pmatrix}
            1&0\\0&1
        \end{pmatrix},
        \begin{pmatrix}
            1&1\\0&0
        \end{pmatrix}.   
    \end{align} 

    Let us focus on the behavior of the partial sum $\sum_{t=s}^{k}(b^t_i-b^t_{i-1})$ for some $s \le k$. By checking all cases in \eqref{eq:4pts}, we see that if $(b^s_{i-1},b^s_{i}) \in \{ (0,0),(1,0), (1,1) \}$, then we can always find $k\ge s$ such that 
    \[
        \sum_{t=s}^{k}(b^t_i-b^t_{i-1})=0
        \qquad
        \text{and}
        \qquad
        (b^{k+1}_{i-1},b^{k+1}_{i}) \in \{ (0,0),(1,0), (1,1) \}.
    \]
    If $(b^s_{i-1},b^s_{i}) \in \{ (0,0), (1,1)\}$ one can take $k=s$. On the other hand if $(b^s_{i-1},b^s_{i}) = (1,0)$, then one can take $k=s+1$. This shows that $\Delta \Gamma_B(i) = 0$ whenever $(b^0_{i-1},b^0_{i}) \in \{ (0,0),(1,0), (1,1) \}$. Similarly, if $(b^0_{i-1},b^0_{i}) =  (0,1)$, we have $\Delta \Gamma_B(i) =  1 + \sum_{t\ge1}(b^t_i-b^t_{i-1}) = 1$, since $\sum_{t\ge1}(b^t_i-b^t_{i-1}) = 0$ by the previous argument.
\end{proof}

Next we introduce another function $\tau_{\mu,J} \colon \Z \rightarrow \Z$.
\begin{definition}
    For any $(\mu,J) \in \RC$ define the \defn{ultradiscrete tau function} $\tau_{\mu,J} \colon \Z \to \Z$ as
    \[
        \tau_{\mu,J}(i) = - \min_{n \in \{ 0,1 \}^\ell} \left\{ \sum_{ \alpha = 1 }^\ell \left( \mu_\alpha - i \right)n_\alpha + \sum_{\alpha,\beta = 1}^\ell \min(\mu_\alpha ,\mu_\beta) n_\alpha n_\beta \right\}.
    \]
\end{definition}

In \cite{Kuniba-Sakamoto-Yamada2007_tau}, the authors showed the two functions $\Gamma_B$ and $\tau_{\mu,J}$ coincide.

\begin{proposition}[{\cite[Theorem 7.4 with $n=1$]{Kuniba-Sakamoto-Yamada2007_tau}}] \label{prop:tau_function}
    Let $B\in \BBS$ and let $(\mu,J)=\KKR(B)$. Then, we have
    \[
        \Gamma_B(i) = \tau_{\mu,J}(i) \qquad \text{for all } i\in \Z.
    \]
\end{proposition}

Combining \cref{lem:gamma_B} with \cref{prop:tau_function}, we obtain the following result. 

\begin{proposition} \label{prop:volume_BBS_tableau}
    Let $ T = (\rho \preceq \lambda) \in \SST(1)$ and let $(B,\nu) = \Xi(T)$, where the map $\Xi$ is defined in \eqref{eq:Xi}. Let $(\mu,J) = \KKR (B)$. Then, we have
    \[
        \abs{\rho} = |\nu| + \abs{J} + \abs{\mu} + \sum_{j=1}^{\ell(\mu)} 2(j - 1) \mu_j.
    \]
\end{proposition}

\begin{proof}
    By \Cref{lem:pr_bij}, it suffices to prove the case $\nu = \varnothing$. Otherwise, one can apply the proposition to the tableau $\overline{T} = (\overline{\rho} \preceq \overline{\lambda})$ with $\overline{\rho}' = \rho' - \nu'$ and $\overline{\lambda}' = \lambda' - \nu'$, for which one has $|\overline{\rho}| = |\rho| - |\nu|$.
    
    Assume $\nu = \varnothing$ and denote by $\ell= \ell(\mu) = \ell(\rho) = \ell(\lambda)$. Basic algebraic manipulations show that the tau function $\tau_{\mu,J}$ can be written as
    \begin{equation} \label{eq:tau_Legendre}
        \tau_{\mu,J}(i) = -\min_{\gamma \in \{0,\dots ,\ell \} } \left\{ \sigma_{\mu,J}(\gamma) - i \gamma \right\},
    \end{equation}
    where the function $\sigma_{\mu,J}$ is defined as
    \begin{equation} \label{eq:sigma}
        \sigma_{\mu,J}(\gamma) = \min_{ \substack{n \in \{0,1\}^\ell \\ |n|=\gamma} } \left\{ \sum_{\alpha = 1}^\ell (J_\alpha + \mu_\alpha)n_\alpha + 2 \sum_{1 \le \alpha < \beta \le \ell} \mu_{\alpha} n_\alpha n_\beta \right\}.
    \end{equation}
    Moreover, an immediate consequence of \Cref{prop:tau_function} and \Cref{lem:gamma_B} is that $\tau_{\mu,J}$ is piecewise linear with slopes $\nabla \tau_{\mu,J}(i)$ having increments being $0$ or $1$. Then, by expression \eqref{eq:tau_Legendre}, we have
    \[
        \nabla \tau_{\mu,J} \colon \Z  \to  \{ 0,\dots , \ell \} \qquad \text{surjectively}
    \]
    and there exists a strictly increasing function $\theta\colon \{ 0, \dots,\ell +1 \} \to \Z \cup \{\pm \infty\}$ such that
    \[
        \tau_{\mu,J}(i) = \gamma i - \sigma_{\mu,J}(\gamma), \qquad \text{for all } i\in \{ \theta(\gamma), \dots, \theta(\gamma+1) \},
    \]
    for $\gamma \in \{0,\dots,\ell \}$, with $\theta(0) = -\infty$ and $\theta (\ell +1)=+\infty$.
    \Cref{lem:gamma_B} also indicates that the locations $\theta(\gamma)$ where $\tau_{\mu,J}$ changes slope are 
    \[
        \theta(\gamma) = \inf\{ j > \theta(\gamma-1) \mid (b_j,b_{j+1})=(0,1) \}.
    \]
    In particular, comparing the above expression with \eqref{eq:pr_bij}, we have
    \begin{equation} \label{eq:theta_rho}
        \theta (\ell-i+1) = \rho_i, \qquad \text{for } i\in \{ 1,\dots, \ell \}.
    \end{equation}
    On the other hand, the equality
    \[
        \tau_{\mu,J}\bigl(\theta(\gamma)\bigr) = \tau_{\mu,J}\bigl(\theta(\gamma-1)\bigr) + (\gamma -1) \Bigl( \theta(\gamma) - \theta(\gamma-1) \Bigr),
    \]
    implies that the coefficients $\theta(\gamma)$ can be written as
    \[
        \theta(\gamma) = \nabla \sigma_{\mu,J}( \gamma ), \qquad \text{for } \gamma \in \{ 1,\dots, \ell\}.  
    \]
    The above analysis produces the following chain of identities
    \begin{equation}
        \abs{J}+\abs{\mu} + \sum_{j=1}^\ell 2(j-1) \mu_j = \sigma_{\mu,J}(\ell) = \sum_{\gamma=1}^{\ell} \nabla \sigma_{\mu,J}( \gamma ) = \sum_{\gamma=1}^\ell \theta(\gamma)
        = \abs{\rho},
    \end{equation}
    where the first equality is a simple evaluation of \eqref{eq:sigma} and the fourth equality comes from \eqref{eq:theta_rho}. This completes the proof.
\end{proof}

\section{Affine crystals and leading paths}\label{sec:affine_leading}

\subsection{Affine crystal structure on horizontally weak tableaux} \label{subsec:crystal_HWT}

In this subsection, we equip the set of horizontally weak tableaux $\HWT(\mu,n)$ with an affine crystal structure.
This is achieved by identifying $\HWT(\mu,n)$ with a tensor product of certain Kirillov--Reshetikhin (KR) crystals that are modeled by single row semistandard Young tableauax; see, e.g.,~\cite{HKOTY02,KKM94,KKMMNN92} and~\cite{shimozono_affine}.
Below we will only recall the strictly necessary aspects of Kashiwara's crystal basis theory and refer the reader to \cite{BumpSchilling_crystal_book} for a complete treatment.

An \defn{affine crystal} or \defn{$\asl_n$-crystal} (associated to the Drinfel'd--Jimbo affine quantum group $U_{\upsilon}'(\asl_n)$ without derivation; see, \textit{e.g.},~\cite{CP95}) is a set with Kashiwara operators and weight function that give it the structure of a special edge-colored weighted directed graph.
Our starting point is the following crystal.

\begin{definition}
    Fix $n\in \Z_+$. For $s\ge 0$, the \defn{Kirillov--Reshetikhin crystal} $\KR{1,s}$ is the set of semistandard tableaux with labels in $\{1,\dots,n\}$ of a single row of length $s$, or equivalently nondecreasing words over $\{1,\dotsc,n\}$ of length $s$.
    Additionally, we will sometimes consider the elements $h \in \KR{1,s}$ as multisets, where we are sorting the elements as necessary.
    The Kashiwara operators will be described as a special case of the following definition.
\end{definition}

\begin{definition} \label{def:kashiwara_HWT}
Fix some $\ell > 0$ and  $\mu\in\N^\ell$.
Define the $\asl_n$-crystal
\begin{equation}
    \label{eq:B_mu}
    \KR{\mu}
    \coloneqq
    \KR{1,\mu_\ell} \otimes \KR{1,\mu_{\ell-1}} \otimes \cdots \otimes \KR{1,\mu_1},
\end{equation}
as the Cartesian product $\KR{1,\mu_\ell} \times \KR{1,\mu_{\ell-1}} \times \cdots \times \KR{1,\mu_1}$ with \defn{Kashiwara operators} $e_i, f_i \colon \KR{\mu} \to \KR{\mu} \sqcup \{\zero\}$ for $i=0,1,2,\dotsc,n-1$ defined by the \defn{signature rule} as follows.
Consider some
\[
H := h_{\ell} \otimes h_{\ell-1} \otimes \cdots \otimes h_1 = (h_{\ell}, h_{\ell-1}, \dotsc, h_1) \in \KR{\mu}.
\]
Fix some $i=0,1,2,\dotsc,n-1$ and let $j := i + 1 \pmod{n}$.
\begin{itemize}
\item[(1)] Make a single word $w$ from $H$ in the following way. For $e_i$ or $f_i$ with $i=1,\ldots,n-1$, $w=h_{\ell} \cdots h_2 h_1$ obtained by concatenating all elements in $H$, while for $i=0$, 
we form $w=\tilde{h}_\ell \tilde{h}_{\ell-1}\cdots \tilde{h}_1$ with $\tilde{h}_k$ being the reverse of $h_k$.
\item[(2)] In $w$, replace all $i$'s\footnote{When $i = 0$, we are replacing all $n$'s with ``$)$''s.} with ``$)$''s and all $j$'s with ``$($''s and ignore all other letters.
\item[(3)] Erase each consecutive symbols ``$($'' ``$)$'' until no such consecutive symbols appear; i.e., the remaining pattern of the symbols becomes $)\cdots))((\cdots($.
\item[(4e)] For the remaining symbols, if there is no ``$($'' then $e_i (H) = \zero$, and otherwise $e_i (H)$ is given by replacing the $j$ corresponding to the leftmost ``$($'' with $i$ in $H$.
\item[(4f)] For the remaining symbols, if there is no ``$)$'' then $f_i (H) = \zero$ and otherwise $f_i (H)$ is given by replacing the $i$ corresponding to the rightmost ``$)$'' with $j$ in $H$.
\end{itemize} 
We use the weight function $\wt$ given in \cref{def:weight}; note $\wt(H) = \sum_{j=1}^{\ell} \wt(h_j)$.
The \defn{crystal graph} of $\KR{\mu}$ is the (weighted) directed graph with vertices $\KR{\mu}$ and colored directed edges $H \xrightarrow{i} H'$, for $H,H'\in \KR{\mu}$ and $i=0,1,\dotsc,n-1$, if and only if $H'=f_i(H)$, or equivalently $H=e_i(H')$. See \Cref{fig:crystal graphs} for an example.
\end{definition}

We identify elements of $\HWT(\mu, n)$ with $\KR{\mu}$ by having an element $H = h_{\ell} \otimes \cdots \otimes h_1 \in \KR{\mu}$ correspond to the tableau in $\HWT(\mu,n)$ whose $j$th row is $h_j$.

\begin{example}\label{ex:kashiwara_word}
Let $H=24\otimes233\otimes2\otimes12\otimes12\otimes133\otimes23\in \HWT(\mu,n)$ with $\mu=(2,3,2,2,1,3,2)$ and $n=4$. Applying the signature rule for $i=2$ yields:
\begin{equation}
    \begin{matrix}
        (1) & & w =  & 2 & 4 & 2 & 3 & 3 & 2 & 1 & 2 & 1 & 2 & 1 & 3 & 3 & 2 & 3
        \\
        (2) & &   & ) &  & ) & ( & ( & ) &  & ) &  & ) &  & ( & ( & ) & (
        \\
        (3) & &   & ) &  & ) &  &  &  &  &  &  & ) &  & ( &  &  & (
        \\
        (4\mathrm{e}) & &   & ) &  & ) &  &  &  &  &  &  & ) &  & \re{)} &  &  & (
        \\
        (4\mathrm{f}) & &   & ) &  & ) &  &  &  &  &  &  & \re{(} &  & ( &  &  & (
    \end{matrix}
\end{equation}
so that
\[
    e_2(H)=24\otimes233\otimes2\otimes12\otimes12\otimes1\re{2}3\otimes23
    \qquad 
    \text{and}
    \qquad
    f_2(H)=24\otimes233\otimes2\otimes12\otimes1\re{3}\otimes133\otimes23.
\]
\end{example}

\begin{figure}
    \centering
    \includegraphics[width=\linewidth]{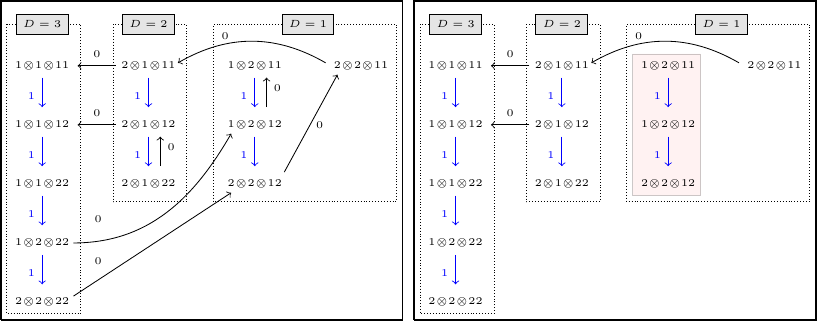}
    \caption{In the left panel the KR crystal graph $\KR{\mu}$ for $\mu=(2,1,1)$ and $n=2$. We have arranged vertically its classically connected components and framed in dotted contours elements with same energy function $D$. In the right panel we removed all $0$ arrows such that $D(f_0(b)) \neq D(b)+1$. The red shaded component is no longer connected to the rest of the graph.}
    \label{fig:crystal graphs}
\end{figure}

\begin{remark}\label{rem:crystal}
    In the crystal graph of $\KR{\mu}$, one can check the following three properties, which are an important aspects of the (affine) crystal structure.
    \begin{enumerate}
        \item The in and out degree of any color and any vertex is at most one.
        \item Let $\epsilon_i$, $i=1,\dotsc,n$, be the standard bases on $\Z^n$; for convenience, we identify $\epsilon_0 = \epsilon_n$.
        Let $\langle\cdot,\cdot\rangle$ be the standard dot product on $\Z^n$ defined by $\langle a,b\rangle=\sum_{i=1}^na_ib_i$ for $a,b\in\Z^n$.
        Then for $i=0,\dotsc,n-1$ and $b\in \KR{\mu}$,
        \begin{align}
                    \label{eq:crystal_2}
            \langle \epsilon_i-\epsilon_{i+1},\wt(H)\rangle=\varphi_i(H)-\varepsilon_i(H),
        \end{align}
        where
        \[
        \varepsilon_i(H)=\max\{m \mid e_i^m(H)\neq\zero\}, \qquad\qquad \varphi_i(H)=\max\{m \mid f_i^m(H)\neq\zero\}.
        \]
    \item For $H,H'\in \KR{\mu}$ satisfying $H'=f_i(H)$ for some $i\in\{0,1,\dotsc,n-1\}$, we have
    \begin{align}
        \label{eq:crystal_3}
        \wt(H')=\wt(H)-\epsilon_i+\epsilon_{i+1}.
    \end{align}
    \end{enumerate}
\end{remark}

\begin{remark}
    \label{rem:promotion}
    Let $\pr\colon \KR{\mu} \rightarrow \KR{\mu}$ be the \defn{promotion operator}~\cite{schutzenberger1963quelques,schutzenberger72} which acts by cyclically relabeling entries $i\to i+1$ for $i=1,\dots n-1$ and $n\to 1$ followed by a reordering within each tensor factor so that elements remain non-decreasing. The Kashiwara operators satisfy the relations
    \begin{equation} \label{eq:0 operators by promotion}
        e_{i-1} = \pr^{-1} \circ e_i \circ \pr
        \qquad \text{ and } \qquad
        f_{i-1} = \pr^{-1} \circ f_i \circ \pr
    \end{equation}
    for $i=1\dots ,n-1$ (see \cite{shimozono_affine}). In particular the operators $e_0,f_0$ can be defined through \eqref{eq:0 operators by promotion}; see also~\cite[Thm.~5.5.1]{lothaire_2002}.
\end{remark}

\begin{remark} \label{rem:KR_connected}
    It was shown in~\cite{Akasaka_Kashiwara} that the crystal graph $\KR{\mu}$ is a strongly connected digraph for arbitrary $\mu$ and $n$; that is, for any two elements $H, H' \in \KR{\mu}$, there exists a \emph{directed} path from $H$ to~$H'$.
\end{remark}

\subsection{Classically connected components} \label{subs:classic_connected_component}

The \defn{classically connected components} of $\KR{\mu}$, are the connected components of the crystal graphs obtained by removing all $0$-arrows; that is, not considering the Kashiwara operators $e_0, f_0$.
It is known~\cite{Kashiwara_crystalizing,Kashiwara_On_Crystal_bases,KKMMNN92} that each such component contains a unique element $H_\hw$ and $H_\lw$ satisfying, respectively,
\[
    e_i(H_\hw) = \zero \qquad \text{and} \qquad  f_i(H_\lw) = \zero \qquad \text{for all} \quad i=1,\dots,n-1.
\] 
We call them the \defn{classically highest} and \defn{classically lowest weight elements}, respectively, and denote their corresponding sets as \defn{$\KR{\mu}_\hw$} and \defn{$\KR{\mu}_\lw$}. 

\begin{example}
    In \Cref{fig:crystal graphs}, the graph in the left panel has four classically connected components. The set of classically highest weight elements is
    \[
        \KR{\mu}_\hw =  \left\{ 1 \otimes 1 \otimes 1 1, \, 2 \otimes 1 \otimes 1 1 , \, 1 \otimes 2 \otimes 1 1, \, 2 \otimes 2 \otimes 1 1 \right\},
    \]
    while that of classically lowest weight elements is
    \[
        \KR{\mu}_\lw =  \left\{ 2 \otimes 2 \otimes 2 2, \, 2 \otimes 1 \otimes 2 2 , \, 2 \otimes 2 \otimes 1 2, \, 2 \otimes 2 \otimes 1 1 \right\}.
    \]
\end{example}

Classically highest and lowest weight elements can be characterized via the \emph{Yamanouchi conditions} on their row reading words (see, \textit{e.g.},~\cite{Stanley1999}).

\begin{remark}
    \label{rem:b1_bl}
    Let $H = h_{\ell}\otimes\cdots\otimes h_2\otimes h_1\in \KR{\mu}$.
    Since each $h_j$ is a weakly increasing word, the signature rule (the Yamanouchi condition; see~\cite[Prop.~7.10.3(d)]{Stanley1999}) implies the following constrains.
    If $H$ is classically highest weight, then $h_1$ consists of only $1$'s.
    Similarly, if $H$ is classically lowest weight, then $h_\ell$ consists of only $n$'s.
    More generally, in classically highest weight elements $H$, the letter $i$ can only appear in $h_j$ for $j \ge i$.
    In classically lowest weight elements, the letter $n-i+1$ can only appear in $h_j$ for $1 \le j \le \ell - i + 1$.
\end{remark}

\subsection{Combinatorial \texorpdfstring{$R$}{R} matrices}

Let $\mu \in \N^\ell$, and let $\sigma \in S_\ell$ be a permutation. It turns out that, for every $n$ the crystal graphs $\KR{\mu}$ and $\KR{\sigma(\mu)}$ are isomorphic~\cite{KKMMNN92} (see also~\cite{Akasaka_Kashiwara,Okado13}). Such isomorphism is unique and it is constructed through the \defn{combinatorial $R$-matrix}, the unique crystal isomorphism $\mathcal{R}\colon \KR{(\mu_1,\mu_2)} \rightarrow \KR{(\mu_2,\mu_1)}$.

We depict the combinatorial $R$-matrix $\mathcal{R}(a \otimes b) = b' \otimes a'$ as the scattering diagram
\begin{equation}
\label{eq:scattering_diagram}
    \begin{tikzpicture}[baseline=0, scale=.5]
        \draw[-] (-1,0) node[anchor=east] {$b$} -- (1,0) node[anchor=west] {$b'$};
        \draw[-] (0,1) node[anchor=south] {$a$} -- (0,-1) node[anchor=north] {$a'$};
    \end{tikzpicture}.
\end{equation}
In \cite[Rule 3.11]{Nakayashiki_Yamada}, the authors obtained an explicit algorithm for $\mathcal{R}$.
As such, we refer to this as the \defn{Nakayashiki--Yamada rule} (for tensor products of single row Kirillov--Reshetikhin crystal), see also~\cite[Sec.~2.2.3]{Inoue-Kuniba-Takagi2012BBS}.
The following definition also includes the local energy that was introduced in \cite{Kang_Kashiwara_et_al}.

\begin{definition}[The Nakayashiki--Yamada rule~{\cite[Rule 3.11]{Nakayashiki_Yamada}}]
    \label{def:Nakayashiki_yamada}
    Let $K \le L$, and let $a\in \KR{K}$, $b\in \KR{L}$. The element $b' \times a' = \mathcal{R}(a \otimes b)$ is constructed as follows.
    \begin{enumerate}
        \item Prepare two columns of $n$ boxes. For each $i=1,\dots,n$, draw $\wt_i(a)$ dots in the $i$-th box (from the top) of the right column and $\wt_i(b)$ dots in the $i$-th box of the left column. 
        \item Pair each dot $x$ in the right column (processed in arbitrary order) to the lowest unpaired dot $y$ in the left column lying in a box strictly higher than that of $x$. If no such dot exists choose the lowest unpaired dot $y$. We refer to pairs $(x,y)$ with $y$ below $x$ as \defn{winding pairs}.
        \item Move all unpaired dots from their box in the left column to the same box in the right column. The resulting configuration of dots determines $a'$ and $b'$.
    \end{enumerate}
    We also define the \defn{local energy function}\footnote{This is typically denoted by $H$ in the literature, but we are using a different name to avoid confusion with HWT.}
    \begin{equation} \label{eq:winding number}
        W(a \otimes b) = \# \text{ of winding pairs in } a \otimes b.
    \end{equation}
    When $L<K$, $\mathcal{R}(a \otimes b)$ is computed exchanging $a$ (resp.\ $a'$) with $b$ (resp.\ $b'$) in (1) (resp.\ (3)) in the above algorithm.
\end{definition}

\begin{remark} \label{rem:pairing_order}
    The output $b' \otimes a'$ of the Nakayashiki--Yamada rule, as well as the winding number is independent of the processing order of dot-pairing in step (2); this was proven in \cite[Prop.~3.20]{Nakayashiki_Yamada}.
\end{remark}

\begin{example}
    \label{ex:NYrule}
    Applying \Cref{def:Nakayashiki_yamada} to $a=112$, $b=1223$, we get $a'=123$, $b'=1122$ as follows:
    \[
        \includegraphics[width=\linewidth]{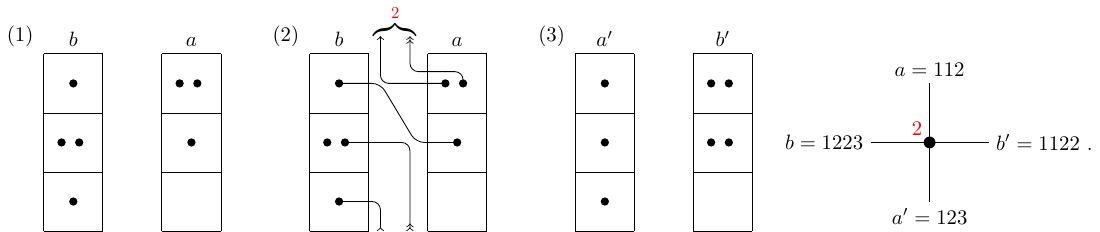}
    \]
    In this example, the local energy function is $2$.
    The corresponding scattering diagram is shown on the right, where the winding number has been written near the crossing.
\end{example}

The combinatorial $R$-matrix is an involution and satisfies the
\defn{Yang--Baxter equation}:
\begin{align}
\label{eq:YBE}
(\mathcal{R}\otimes \mathbf{1})(\mathbf{1}\otimes \mathcal{R})(\mathcal{R}\otimes \mathbf{1})
=
(\mathbf{1}\otimes \mathcal{R})(\mathcal{R}\otimes \mathbf{1})(\mathbf{1}\otimes \mathcal{R}).
\end{align}
Fix $\ell \in \N$. For $i=1,\dots,\ell-1$, define the operator
\begin{equation} \label{eq:sigma_i}
    \sigma_i
    =
    \mathbf{1}^{\otimes (i-1)} \otimes \mathcal{R} \otimes \mathbf{1}^{\otimes (\ell-i-1)},
\end{equation}
which acts on a tensor product $h_1\otimes \cdots \otimes h_\ell$ as $\mathcal{R}$ on the $i$-th and $(i+1)$-th tensor factors and as the identity on the remaining factors.
This defines an action of the symmetric group $S_{\ell}$ on $\bigsqcup_{\mu:\ell(\mu)=\ell} \HWT(\mu,n)$.
Note that the braid relation for $\sigma_i, \sigma_{i+1}$ is the Yang--Baxter equation~\eqref{eq:YBE}.
With a slight abuse of notation, we denote by $\sigma \in S_\ell$ both a permutation and the corresponding operator obtained by composing the $\sigma_i$.
Hence, we have
\begin{equation}\label{eq:sigma_HWT}
    \sigma \KR{\mu} = \sigma_{i_1} \cdots \sigma_{i_k} \KR{\mu} \longrightarrow \KR{s_{i_1} \cdots s_{i_k} \mu} = \KR{\sigma \mu}
\end{equation}
for any choice of reduced expression $\sigma = s_{i_1}\cdots s_{i_k}$ (where $s_i = (i \; i+1)$ denotes the $i$-th elementary transposition).
Since the combinatorial $R$-matrix is a crystal isomorphism, this action commutes with the Kashiwara operators.
In particular, for any $\sigma\in S_\ell$, the map \eqref{eq:sigma_HWT} is the unique crystal isomorphism between $\KR{\mu}$ and $\KR{\sigma(\mu)}$.

\begin{remark}
    \label{rem:reverse_iso}
    Denote by 
    \[
    \sigma^- := (\ell, \ell-1, \dots , 1) = (s_1 \cdots s_{\ell-1}) \cdots  (s_1 s_{2})(s_{1}),
    \]
    which is referred to as the long element as it requires the most number of simple transpositions (see, \textit{e,g.},~\cite{sagan2001symmetric}).
    For $\mu=(\mu_1,\mu_2,\dotsc,\mu_\ell)$ and its reverse $\mu^{-}:= \sigma^-(\mu) = (\mu_\ell,\dotsc,\mu_2,\mu_1)$, the crystal isomorphism $\sigma^- \colon \KR{\mu} \rightarrow \KR{\mu^-}$ can be represented by the scattering diagram of \Cref{fig:scattering diagrams} left panel, where the input is $H=h_{\ell} \otimes \cdots \otimes h_1$ and the output is $\sigma^-(H) = h'_1\otimes\cdots\otimes h'_{\ell}$.
\end{remark}

\subsection{The Lusztig involution} \label{subs:Lusztig}

In the next definition we introduce two involutions $\eta'$ and $\eta$. In particular $\eta'$ is a combinatorial realization of (the crystal limit of) the \defn{Lusztig involution}~\cite{BZ96,Lusztig94} (see also~\cite[Prop.~2.3]{Lenart07}).

\begin{definition}
    \label{def:Lusztig_involution}
    Let $\sigma^-\colon \KR{\mu} \rightarrow \KR{\mu}$ be the map defined in \cref{rem:reverse_iso}. Let $\mathrm{flip}\colon \KR{\mu}\rightarrow \KR{\mu^-}$
    be the map defined by 
    \[
        \mathrm{flip} \colon h_{\ell}\otimes\cdots\otimes h_2\otimes h_1\mapsto h_1\otimes h_2\otimes\cdots\otimes h_\ell. 
    \]
    Let $\mathrm{comp}\colon \KR{\mu} \rightarrow \KR{\mu}$ be the map defined by inverting all words $h_i$, $i=1,\dotsc,\ell$ then changing all $j$-entries to $n+1-j$-entries. Then we define the maps $\eta' \colon \KR{\mu} \rightarrow \KR{\mu^-}$ and $\eta \colon \KR{\mu} \rightarrow \KR{\mu}$ as 
    \begin{align}
        \label{eq:lusztig_involution}
        \eta'=\mathrm{comp}\circ\mathrm{flip} 
        \qquad \text{and}
        \qquad
        \eta=\eta'\circ \sigma^{-}.
    \end{align}
    The map $\eta'$ is called tne \defn{Lusztig involution}.
\end{definition}

\begin{lemma}[{\cite{Okado_Schilling_Schimozono_virtual,Lenart07}}]
    \label{lem:cf_lusztig}
    The map $\eta'$ is an involution such that for each $i \in \{0,\dots,n-1\}$
    \begin{align}
        \label{eq:etap_ef}
        e_{n-i}\circ \eta' = \eta'\circ f_{i},
    \end{align}
    where we set $e_n = e_0$.
\end{lemma}

\begin{proof}
We provide a short proof for completeness.
Equation~\eqref{eq:etap_ef} follows from the signature rule in \cref{def:kashiwara_HWT} and definitions of $\mathrm{flip}$ and $\mathrm{comp}$ in \cref{def:Lusztig_involution}.
\end{proof}

\begin{lemma}
    \label{lem:eta_ef}
    The map $\eta$ is an involution such that for each $i \in \{0,\dots,n-1\}$ we have
    \begin{align}
        \label{eq:eta_ef}
        e_{n-i}\circ \eta=\eta\circ f_{i},
    \end{align}
    where we set $e_n = e_0$.
\end{lemma}

\begin{proof}
    Equation~\eqref{eq:eta_ef} follows from the relation $f_i \circ \sigma^- =\sigma^- \circ f_i$ (since $\sigma^-$ a crystal isomorphism) and~\cref{lem:cf_lusztig}. Relation \eqref{eq:eta_ef} also implies that $\eta^2 \colon \KR{\mu} \to \KR{\mu}$ is a crystal isomorphism and by uniqueness it must be the identity. 
\end{proof}

\begin{remark} \label{rem:eta_hw_lw}
    For any $\mu \in \N^\ell$ the involution $\eta$ restricts to a bijection $\eta \colon \KR{\mu}_\hw  \longleftrightarrow \KR{\mu}_\lw$.
\end{remark}

\subsection{Energy functions}

In this subsection, we describe two notions of energy functions $D,D'$ for $\KR{\mu}$. 

\begin{definition}[Energy functions]\label{def:energy_function}
Recall the winding number $W$ from \eqref{eq:winding number}. For any
$\mu \in \N^\ell$, define the \defn{global energy}
\[
    D(H) =
    \sum_{1 \le i < j \le \ell(\mu)}
    W_i\bigl(\sigma_{i+1}\sigma_{i+2}\cdots \sigma_{j-1}(H)\bigr),
\]
where $W_i(h_1 \otimes \cdots \otimes h_\ell) = W(h_i \otimes h_{i+1})$. The \defn{dual global energy} is
\[
    D'(H) = D(\eta(H)),
\]
where $\eta$ is the map in \eqref{eq:lusztig_involution}.
\end{definition}

The energy functions were introduced in~\cite{HKOTY99,HKOTT02}\footnote{More precisely, only $D$ was defined in~\cite{HKOTY99,HKOTT02}, while both $D$ and $D'$ have subsequently appeared in the literature. See \cite[Rem.~4.10]{LNSSSII} and \cite[Def.~2.4]{Lenart_Schilling_charge} for precise statements.}.  In what follows, rather than the algorithmic definition, we mainly use the structural properties of $D$ recorded below. Alternative explicit descriptions of the energy $D$ and $D'$ are given respectively in \Cref{rem:alternative_def_energy} and \Cref{rem:alternative_def_dual_energy}.

\begin{figure}
    \centering
    \includegraphics[width=.8\linewidth]{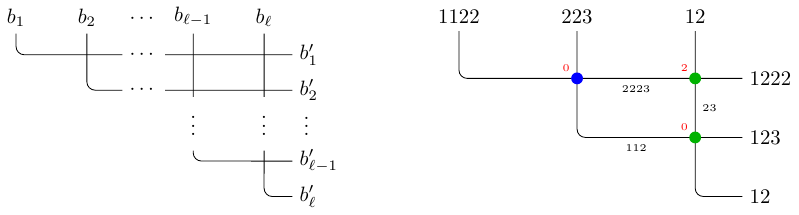}
    \caption{In the left panel the scattering diagrams representing the isomorphism $\overline{\sigma}^-\colon \KR{\mu} \to \KR{\mu^-}$. In the right panel an explicit example: values in red near the crossings are the winding numbers of the corresponding application of the $R$-matrix.}
    \label{fig:scattering diagrams}
\end{figure}

First, the energy functions are constant on classically connected components (see ~\cite{HKOTY99,HKOTT02,HKOTY02,shimozono_affine}).

\begin{lemma}
    \label{lem:energy invariance}
    Let $g_i$ be $e_i$, or $f_i$, for $i \in \{1,2,\dotsc,n-1\}$. If $g_i(H) \neq \zero$, then we have
    \begin{align} \label{eq:energy invariance}
        D\bigl( g_i(H) \bigr) = D(H) \qquad \text{and} \qquad D'\bigl( g_i(H) \bigr) = D'(H),
    \end{align}
    i.e., all elements in a classical connected component have the same energy and dual energy.
\end{lemma}

The following key property of the energy function is well-known to experts in KR crystals, but has not explicitly appeared in the literature as far as the authors are aware.
\footnote{We thank Masato Okado for the proof. This proof is type-independent and holds for arbitrary tensor products of KR crystals.}

\begin{lemma}
    \label{lem:unique_highest}
    Let $\mu\in \N^{\ell}$.
    The classically connected component of $\KR{\mu}$ having the highest energy $D$ is unique, whose highest weight element and energy are
    \begin{align}\label{eq:b_1}
        H_{\max} := \underbrace{1\cdots1}_{\mu_\ell} \otimes \cdots \otimes \underbrace{1\cdots1}_{\mu_{1}}
        \qquad
        \text{and}
        \qquad
        D(H_{\max}) = \Abs{\mu^+} := \sum_{i=1}^{\ell(\mu)} (i-1) \mu_i^+,
    \end{align}
    respectively.
    Here $\mu^+=(\mu^+_1,\dotsc,\mu^+_\ell)\in\partitions_\ell$ is obtained by sorting the elements in $\mu$ in nonincreasing order. 
    Moreover, $D(H_{\max}) = D'(H_{\max})$.
\end{lemma}

\begin{proof}
    By \cref{lem:energy invariance}, $D$ is constant on each classical connected component. Hence it suffices to consider classically highest weight elements.
    By \cite{KKR86,KR86}, the $\asl_n$-KKR bijection (a generalization of the one presented in \Cref{subs:KKR}) identifies highest weight elements $H \in \KR{\mu}$ with $\asl_n$ rigged configurations $(\boldsymbol{\nu},\boldsymbol{J}) = (\nu^{(a)},J^{(a)})_{a=1}^{n-1}$, where $\nu^{(a)}$ are partitions such that $|\nu^{(a)}| = \wt_{a+1}(H)$ and the riggings $J^{(a)}$ are lists of non-negative integers.
    Under this bijection, the quantity $c(H) = \Abs{\mu^+} - D(H)$,  known as the \emph{cocharge} of $b$, is given by the expression (see, \textit{e.g.},~\cite[Eq.~(4.43)]{Inoue-Kuniba-Takagi2012BBS})
    \[
        c(H) = \sum_{a,b=1}^{n-1} \sum_{i,j \ge 1} C_{a,b}\,\min(i,j)\, m_i(\nu^{(a)}) m_j(\nu^{(b)}) \;+\; \sum_{a=1}^{n-1} \sum_{k\ge 1} J^{(a)}_k,
    \]
    where $C_{a,b} = 2\delta_{a=b} - \delta_{|a-b|=1}$. The right hand side is a non-negative, non-degenerate quadratic form in the multiplicities $m_i(\nu^{(a)})$, plus a non-negative contribution from the riggings.
    It is therefore uniquely minimized when $\boldsymbol{\nu}=\boldsymbol{J}=\varnothing$, corresponding to $H=H_{\max}$, where $c(H_{\max})=0$, yielding the claims.

    To show that $D'(H_{\max}) = D(H_{\max})$, we compute
    \[
        e_1^{\abs{\mu}} \cdots e_n^{\abs{\mu}} \eta(H_{\max}) = e_1^{\abs{\mu}} \cdots e_n^{\abs{\mu}}( \underbrace{n \cdots n}_{\mu_{\ell}} \otimes \cdots \otimes \underbrace{n \cdots n}_{\mu_{1}}) = H_{\max},
    \]
    and hence the claim follows.
\end{proof}

\begin{remark}
    We have that $D' \neq D$ in general.
    Consider the case and $\mu = (1,1,1)$ and $n=3$. Below we report the classically highest weight elements of $\KR{\mu}$, their energy $D$, the image of $\eta$, and $D'$.
    \[
    \begin{array}{c|cccc}
    H_{\hw} & 1 \otimes 1 \otimes 1 & 2 \otimes 1 \otimes 1 & 1 \otimes 2 \otimes 1 & 3 \otimes 2 \otimes 1
    \\[.2em]
    D(H_\hw) & 3 & 2 & 1 & 0
    \\[.2em]
    \eta(H_{\hw}) & 3 \otimes 3 \otimes 3 & 3 \otimes 3 \otimes 2 & 3 \otimes 2 \otimes 3  & 3 \otimes 2 \otimes 1
    \\[.2em]
    D'(H_\hw) & 3 & 1 & 2 & 0
    \end{array}
    \]
\end{remark}

The following result was proven in \cite{HKKOTY99} (where the opposite tensor product convention is used).

\begin{proposition}[{\cite[Prop.~1.1]{HKKOTY99}}] \label{lem:HKKOTY99}
    Let $ H = h_{\ell} \otimes \cdots \otimes h_1$ be an element such that, for all $j \in \{ 2, 3, \dotsc, \ell \}$, if $e_0(h_j \otimes h_1) = e_0 (h_j) \otimes h_1$, then $e_0(\widetilde{h}_1 \otimes \widetilde{h}_j) = e_0 (\widetilde{h}_1) \otimes \widetilde{h}_j$, where $\mathcal{R}(h_j \otimes h_1) = \widetilde{h}_1 \otimes \widetilde{h}_j$.
    Then $D(e_0 (H)) = D(H) - 1$.
\end{proposition}

\subsection{Regular subgraphs and connectivity}
\label{subsec:leadingHWT}

In this subsection, we introduce a distinguished subgraph of $\KR{\mu}$ obtained by selecting the $f_0$ arrows that increase the energy by one unit, and study its structural properties.

\begin{definition}[Regular subgraph] \label{def:reg_subgraph}
    Let $H \in \KR{\mu}$ and assume that $f_0(H) \neq \zero$. We say that $H \to f_0(H)$ is a \defn{regular 0-arrow} if 
    \[
        D\bigl(f_0(H)\bigr) = D(H) + 1.
    \]
    The \defn{regular subgraph} of $\HWT(\mu,n)$ is the subgraph obtained by removing all 0-arrows that are not regular. More generally, $H \to g(H)$ is a \defn{regular arrow} if it is either a regular 0-arrow or if $g \in \{ e_1,f_1,\dots,e_{n-1},f_{n-1} \}$.
\end{definition}

\begin{remark}
    The regular subgraph of $\KR{\mu}$ is \emph{graded} by the energy function $D$. In particular along any path in the regular subgraph of $\KR{\mu}$
    \[
        H' = g_{1} \cdots g_{k}(H),
    \]
    we have
    \[
        D(H') = D(H) + \#\{ j \mid g_{j} = f_0 \} - \#\{ j \mid g_{j} = e_0 \}.
    \]
\end{remark}

We now state the main result of this section. 
For $\mu \in \N^{\ell}$, we define $\ell(\mu) := \abs{\{i = 1, \dotsc, \ell \mid \mu_i > 0\}}$ as the number of nonzero entries in $\mu$.

\begin{theorem} \label{thm:connectedness}
    Let $\mu \in \N^{\ell}$.
    If $n\ge \ell(\mu)$, the regular subgraph of the $\asl_n$-crystal $\KR{\mu}$ is connected.
\end{theorem}

\begin{remark}
    While the crystal graph $\KR{\mu}$ is strongly connected, as pointed out in \Cref{rem:KR_connected}, the same is not true in general for its regular subgraph, making the result of \Cref{thm:connectedness} nontrivial. For instance, in \Cref{fig:crystal graphs} right panel, we see that the regular subgraph of $\KR{\mu}$ with $\mu=(2,1,1)$ and $n=2$ is not connected.
\end{remark}

\begin{remark}
    A similar result to \Cref{thm:connectedness} was proved in~\cite[Thm.~6.1]{schilling_tingley} (see also~\cite{Fourier_Schilling_Shimozono_demazure}), where the connectedness was shown for the Demazure subgraph of tensor products of ``single column'' KR crystals (as opposed to single rows that we are considering in the present paper) using properties not present in our case. 
\end{remark}

\begin{remark}
    We have observed in \Cref{fig:crystal graphs} right panel, that the regular subgraph of $\KR{\mu}$ could be disconnected if $n< \ell(\mu)$. We do not expect this to be a general pathology and in fact conjecture that regular subgraph of $\KR{\mu}$ can only be disconnected for $n=2$ and $\ell(\mu)>2$. This should be related to $\asl_2$ not being simply-laced (see \Cref{conj:general_energy_change_one} below).
\end{remark}

To prove \Cref{thm:connectedness} we construct an explicit path in the regular subgraph from the arbitrary element $H\in \KR{\mu}$ to $H_{\max}$.
The main tools for its definition are given next.

\begin{definition}\label{def:hwshift}
    Define the map $\overline{\mathcal{P}}\colon \KR{\mu}_{\mathrm{hw}} \to \KR{\mu}$ such that for $H_{\hw}\in \KR{\mu}_{\mathrm{hw}}$
    \[
        \overline{\mathcal{P}} (H_{\hw}) = \pr^{M}(H_{\hw}) \qquad \text{ with } \qquad M = n-\max(H_\hw),
    \]
    where $\max(H)$ is the largest entry in $H$. For brevity we will denote
    \[
        \overline{H_{\hw}} =\overline{\mathcal{P}} (H_{\hw}).
    \]
    In words $\overline{H_{\hw}}$ is the element obtained by adding $M$ to each entry of $H_\hw$.
\end{definition}

Next, we note that the elements $H_\hw, \overline{H_{\hw}}$ belong to the same classically connected component.

\begin{lemma}\label{lem:classical_bar_component}
    Let $\mu \in \N^\ell$. For any $H_\hw \in \KR{\mu}_{\hw}$, let $M=\max(H_\hw)$. Then, we have 
    \begin{align}
        \label{eq:b_barb}
        \overline{H_{\hw}} = g_1g_2 \cdots g_{M}(H_{\hw}),
    \end{align}
    where for each $j = 1,\dotsc,M$, we set $g_{j} = f^{\wt_j(H_{\hw})}_{j+n-M-1}\cdots f_{j+1}^{\wt_j(H_{\hw})}f_j^{\wt_j(H_{\hw})}$.
\end{lemma}

\begin{proof}
A direct inspection of the map $g_M = f^{\wt_M(H_{\hw})}_{n-1} \cdots f_M^{\wt_M(H_{\hw})}$ shows that its action on $H_\hw$ raises every $M$ entry to an $n$ entry.
Each application of the Kashiwara operators $f^{\wt_M(H_{\hw})}_k$, for $k=M, \dotsc,n-1$ is non-vanishing since the element $f^{\wt_M(H_{\hw})}_{k-1} \cdots f^{\wt_M(H_{\hw})}_M(H_\hw)$ has exactly $\wt_M(H_{\hw})$ $k$ entries and no $k+1$ entry.
By downwards induction on $j = M, \ldots, 1$, the map $g_j$ turns every $j$ entry in the element $g_{j+1} \dots g_M(H_\hw)$ into a $j+n-M$ entry, proving \eqref{eq:b_barb}.
\end{proof}

\begin{example}\label{ex:barb}
For $\mu=(4,3,3,2,1)$ and $n=7$, we consider $H_{\hw}=2\otimes11\otimes123\otimes122\otimes 1111$, and so $\max(H_\hw) = 3$. 
Adding $4 = 7 - 3$ to each entry we get
\[
    \overline{b_{\hw}}=6\otimes 55\otimes 567\otimes 566 \otimes 5555 = (f_4^8 f_3^8 f_2^8 f_1^8) (f_5^4 f_4^4 f_3^4 f_2^4) (f_6 f_5 f_4 f_3) b_{\hw}.
\]
\end{example}

In addition, we can show that $f_0(\overline{H_{\hw}}) \neq \zero$, when $n\ge \ell(\mu)$.

\begin{lemma}
  \label{lem:f0barb}
  Fix $\mu \in \N^\ell$ and suppose $2 \le \ell \leq n$. 
  For any $H_{\hw}\in \KR{\mu}_{\mathrm{hw}}$, we write $\overline{H_{\hw}}=\bar{h}_\ell\otimes \bar{h}_{\ell-1}\otimes\cdots\otimes \bar{h}_1$.
  Then, calling $k$ is the smallest index such that $\overline{h}_k$ contains an $n$, we have
  \begin{align}
    \label{eq:f0bb_rep}
        f_0(\overline{H_{\hw}})=\bar{h}_{\ell}\otimes\cdots\otimes \bar{h}_{k+1}\otimes f_0(\bar{h}_k) \otimes\bar{h}_{k-1}\otimes\cdots\otimes\bar{h}_1.
  \end{align}
\end{lemma}

\begin{proof}
    From \Cref{def:hwshift}, it is clear that there exists at least one $n$ entry in $\overline{H_{\hw}}$.
    Then, by the signature rule defining $f_0$ (\Cref{def:kashiwara_HWT}), the claim reduces to showing that no $1$ entry appears in any $\bar{h}_j$ for $j>k$ (see also~\Cref{rem:b1_bl}).
    Let $H_\hw =h_\ell \otimes \cdots \otimes h_1$.
    Then, the highest weight condition and a straightforward induction over $j$ imply that $\max(h_j) \le j$. Therefore, if $\ell < n$, all entries in $\overline{H_{\hw}}$ are strictly greater than $1$ and the claim follows.
    For the case $\ell = n$, if there is an $n$ entry it must occur in $h_n$ and in such case we have $\overline{H_{\hw}} = H_{\hw}$, $k=n$ and the claim follows.
\end{proof}

\begin{example}
    \label{ex:f0barb}
    We act $f_0$ on $\overline{H_{\hw}}=6\otimes 55\otimes 567\otimes 566 \otimes 5555$ in \cref{ex:barb}, which applying the algorithm, we easily find $f_0(\overline{H_{\hw}})=6\otimes 55\otimes 156\otimes 566\otimes 5555$.
\end{example}

\begin{example}
Note that the conclusion of Lemma~\ref{lem:f0barb} does not hold for all elements $\overline{H_{\hw}}$ in Figure~\ref{fig:crystal graphs}, where the condition $\ell(\mu)\le n$ fails.
Indeed, consider the element $H_{\hw} = \overline{H_{\hw}} = 1 \otimes 2 \otimes 11$ as $n = \max(H_\hw) = 2$, but we have $f_0 (\overline{H_{\hw}}) = \zero$.
\end{example}

The next proposition states that $\overline{H_{\hw}} \to f_0(\overline{H_{\hw}})$ is a regular 0-arrow whenever $\mu$ is a partition, $\ell(\mu) \le n$ and $H_{\hw} \neq H_{\max}$.

\begin{proposition}\label{th:leading_f0}
    Let $\mu \in \N^{\ell}$ such that $2\le \ell(\mu)\le n$.
    Let $H_{\hw}\in \KR{\mu}_{\mathrm{hw}}\setminus\{H_{\max}\}$, and recall the element $\overline{H_{\hw}}$ in $\KR{\mu}$ of \cref{def:hwshift}.
    Then, we have
    \begin{align}\label{eq:leading_f0}
        D\bigl( f_0(\overline{H_{\hw}}) \bigr) = D(\overline{H_{\hw}}) + 1.
    \end{align}
\end{proposition}

\begin{proof}
    We first consider the case $\max(H_\hw) < n$, where we have $\ell=\ell(\mu)<n$. By \Cref{lem:f0barb}, $H = f_0 (\overline{H_{\hw}}) \neq \zero$ is an element with a single $1$ entry and all other entries strictly greater than $1$. Suppose such $1$ entry occurs in the $k$-th tensor factor of $H = h_\ell \otimes \cdots \otimes h_1$ and thus there are no $n$'s appearing in any tensor factors $h_j$ for all $j < k$. Since $H_\hw \neq H_{\max}$, we have $k>1$. We now check that $H$ satisfies the hypotheses of \Cref{lem:HKKOTY99}.
    
    For $j\neq 1$, let $\widetilde{h}_1 \otimes \widetilde{h}_j = \mathcal{R}(h_j \otimes h_1)$. For all $j \neq k$, we have $e_0(h_j \otimes h_1) = e_0(\widetilde{h}_1 \otimes \widetilde{h}_j) = \zero$, since there are no $1$ entries in either factors $h_1,h_j$ (recall the combinatorial $R$-matrix is weight preserving). 
    We also have $e_0(h_k \otimes h_1) = e_0 (h_k) \otimes h_1$, since $h_1$ does not contain any $n$ entry. The Nakayashiki--Yamada rule (\Cref{def:Nakayashiki_yamada}) implies that in the tensor product $\widetilde{h}_1 \otimes \widetilde{h}_k$, the only 1 entry belongs to $\widetilde{h}_1$, while $\widetilde{h}_k$ posseses no $n$ entries. As a result $e_0(\widetilde{h}_1 \otimes \widetilde{h}_k) = e_0(\widetilde{h}_1) \otimes \widetilde{h}_k$ and \Cref{lem:HKKOTY99} implies that
    \[
        D\bigl( e_0(H) \bigr) = D(H) - 1,
    \]
    which is equivalent to \eqref{eq:leading_f0}. 

    
    Now we assume $n = \max(H_\hw)$, which also forces $\ell(\mu) = n$ and $\overline{H_{\hw}} = H_{\hw}$. Denote $H_\hw= h_n' \otimes \cdots \otimes h_1'$ and $H=f_0(H_{\hw}) = h_n \otimes \cdots \otimes h_1$. Then $h_n = f_0(h_n')$, since $h_n'$ is the only tensor factor over which $f_0$ has non-vanishing action, and $h_j=h_j'$ for all $j\in \{1,\dots,n-1\}$. Notice also that $h_1$ consists of a row of length $\mu_1$ with only 1 entries. 
    
    For any $j\neq 1$ denote again $\widetilde{h}_1 \otimes \widetilde{h}_j = \mathcal{R} (h_j \otimes h_1)$. 
    A simple application of the signature rule (\Cref{def:kashiwara_HWT}) gives
    \[
        e_0(h_j \otimes h_1) = 
        \begin{cases}
            e_0(h_j) \otimes h_1 \qquad &\text{if $h_j$ contains at least a 1 entry},
            \\
            h_j \otimes e_0(h_1) \qquad &\text{otherwise}.
        \end{cases}
    \]
    To apply \Cref{lem:HKKOTY99} we are only interested in the first case. If $h_j$ contains at least a 1 entry, then the Nakayashiki--Yamada rule implies that $\widetilde{h}_1$ consists of a single row having all the entries of $h_j$ and additional $\mu_1-\mu_j$ 1 entries, while $\widetilde{h}_j$ is a single row of length $\mu_j$ with only 1's. Then, the signature rule implies that 
    \[
        e_0(\widetilde{h}_1 \otimes \widetilde{h}_j) = e_0(\widetilde{h}_1) \otimes \widetilde{h}_j,
    \]
    if $h_j$ contains at least a 1-entry. By \Cref{lem:HKKOTY99}, also in this case we have $D(e_0(H)) = D(H)-1$, which implies \eqref{eq:leading_f0} and completes the proof.
\end{proof}

The proof of \Cref{th:leading_f0} relies on \Cref{lem:HKKOTY99}, a result from \cite{HKKOTY99}. A direct, alternative proof is given in \cref{sec:proof+1}.

\begin{conjecture}
\label{conj:general_energy_change_one}
The assumption $\ell(\mu) \leq n$ in \cref{th:leading_f0} can be reduced to simply $n > 2$ or $\ell(\mu) \leq 2$ for some other element $\overline{H_{\hw}}$ when $n = \max(H_\hw)$ in the same classical component as $H_{\hw}$.
\end{conjecture}

We also note the necessity of changing the element $\overline{H_{\hw}}$ in~\Cref{conj:general_energy_change_one}.
Consider the element $H_{\hw} = 1 \otimes n \otimes \cdots \otimes 2 \otimes 1$, which has $\overline{H}_{\hw} = H_{\hw}$ (as defined in~\Cref{lem:classical_bar_component}).
Then by the bracketing rule, we have $f_0 (\overline{H_{\hw}}) = \zero$.
Analogous elements exist for any such $\KR{\mu}$ for $\ell(\mu) > n$.


The result of \Cref{th:leading_f0} yields an explicit regular 0-arrow in every classically connected component of $\KR{\mu}$ whose energy is not maximal.

We are now ready to prove \Cref{thm:connectedness}.

\begin{proof}[Proof of \Cref{thm:connectedness}]
    We show that any element $H \in \KR{\mu}$ is connected in the regular subgraph of $\KR{\mu}$ to $H_{\max}$. Let $H \in \KR{\mu}$ and let $H_{\hw}$ be the classically highest weight element in its classically connected component. If $H_{\hw} = H_{\max}$, we simply join $H$ to $H_{\max}$ applying to $H$ sufficiently many $e_i$ operators with $i \in \{1,\dots,n-1\}$. Otherwise, by \Cref{lem:classical_bar_component}, $H$ is connected to $\overline{H_{\hw}}$ via classical arrows and by \Cref{th:leading_f0}, the arrow $\overline{H_{\hw}} \to f_0(\overline{H_{\hw}})$ is a regular $0$-arrow and increases the energy $D$ by $1$. Iterating this procedure, we construct a path in the regular subgraph along which the energy strictly increases at each step where we move from one classically connected component to another. Since the energy is bounded above and the maximal element is unique, this process terminates at the unique classically connected component with $H_{\max}$.
\end{proof}

\begin{remark} \label{rem:alternative_def_energy}
    The explicit path constructed in the proof of \Cref{thm:connectedness} can be used as a definition of the energy function $D$ of a generic element $H \in \KR{\mu}$ as long as $n \ge \ell ( \mu )$. Let $g_1 \cdots g_K$ be the sequence of Kashiwara operators required to connect $H$ to $H_{\max}$ following the construction in the proof of \Cref{thm:connectedness}. Then, we have $D(H) = \Abs{\mu} - \# \{j \mid g_j = f_0\} $.
\end{remark}

\subsection{Dual regular subgraphs and leading paths}

Adapting the strategy of \Cref{subsec:leadingHWT}, we introduce a distinguished subgraph of $\KR{\mu}$ obtained by selecting the $e_0$ arrows that increase the dual energy by one unit, and study its structural properties.

\begin{definition}[Dual regular subgraph] \label{def:dual_subgraph}
    Let $H \in \KR{\mu}$ and assume that $e_0(H) \neq \zero$. We say that $H \to e_0(H)$ is a \defn{dual regular 0-arrow} if 
    \[
        D'\bigl(e_0(H)\bigr) = D'(H) + 1.
    \]
    The \defn{dual regular subgraph} of $\KR{\mu}$ is the subgraph obtained by removing all 0-arrows that are not dual regular. More generally, $H \to g(H)$ is a \defn{dual regular arrow} if it is either a regular 0-arrow or if $g \in \{ e_1,f_1,\dots,e_{n-1},f_{n-1} \}$.
\end{definition}

\begin{remark}
    The dual regular subgraph of $\KR{\mu}$ is \emph{graded} by the dual energy function $D'$. In particular along any path in the dual regular subgraph of $\KR{\mu}$
    \[
        H' = g_{1} \cdots g_{k}(H),
    \]
    we have
    \[
        D'(H') = D'(H) + \#\{ j \mid g_{j} = e_0 \} - \#\{ j \mid g_{j} = f_0 \}.
    \] 
\end{remark}

The following lemma relates regular arrows and dual regular arrows.

\begin{lemma} \label{lem:from regular to dual regular}
    Let $H \to H'$ be a regular arrow. Then $\eta(H) \to \eta(H')$ is a dual regular arrow.
\end{lemma}
\begin{proof}
    Let $H' = g(H)$ for a Kashiwara operator $g \in \{ e_0,f_0, \dotsc, e_{n-1},f_{n-1} \}$. 
    Then, by \Cref{lem:eta_ef} we have $\eta(H') =g^*(\eta(H))$ for some Kashiwara operator $g^*$.
    If $g \notin \{e_0,f_0\}$, then also $g^*\notin \{e_0,f_0\}$ and  $\eta(H) \to \eta (H')$ is a dual regular arrow. If $g=f_0$, then $g^*=e_0$. By regular 0-arrow condition we have $D(H') = D(H)+1$, which implies
    \[
        D'\bigl(e_0\eta(H)\bigr) = D'\bigl(\eta(H')\bigr) = D(H') = D(H)+1 = D'\bigl(\eta(H)\bigr) + 1.
    \]
    In the second equality we used the definition of dual energy $D'$ and the fact that $\eta$ is an involution. Hence $\eta(H) \to \eta (H')$ is a dual regular 0-arrow. The case $g=e_0$ and $g^*=f_0$ is treated analogously.
\end{proof}

As a consequence of \Cref{thm:connectedness} and \Cref{lem:from regular to dual regular}, we give the following result.

\begin{theorem}\label{thm:dual_connectedness}
    Let $\mu \in \N^{\ell}$ such that $\ell(\mu) \le n$. Then, the dual regular subgraph of $\KR{\mu}$ is connected.
\end{theorem}
\begin{proof}
    Fix $H,H' \in \KR{\mu}$ and call $\widetilde{H}=\eta(H)$, $\widetilde{H}'=\eta(H')$. By \Cref{thm:connectedness}, there exists a path in the regular subgraph of $\KR{\mu}$ connecting  $\widetilde{H}$ to $\widetilde{H}'$. By \Cref{lem:from regular to dual regular} the image of such path under the involution $\eta$ consists of dual regular arrows and it joins $\eta(\widetilde{H})=H$ and $\eta(\widetilde{H}')=H'$. Since $H,H'$ are generic, the dual regular subgraph of $\KR{\mu}$ is connected.
\end{proof}

Based on \Cref{thm:dual_connectedness}, we now construct an explicit path on the dual regular graph joining any element $H \in \KR{\mu}$ to $H_{\max}$. For this we define, similarly to \Cref{def:hwshift} the map $\underline{\mathcal{P}}: \KR{\mu}_\lw \to \KR{\mu}$ such that for any $H_\lw \in \KR{\mu}_\lw$
\begin{equation} \label{eq:underline_P}
    \underline{\mathcal{P}} (H_\lw) = \pr^{-(m-1)}(H_\lw)
    \qquad
    \text{with}
    \qquad
    m = \min(H_\lw),
\end{equation}
where $\min(H)$ is the value of the minimal entry in $H$.

\begin{lemma} \label{lem:explicit leading path}
    Let $\mu \in \N^{\ell}$.
    Given $H_\lw \in \KR{\mu}_\lw$, denote
    \[
        \underline{H_\lw} = \underline{\mathcal{P}}(H_\lw),
    \]
    and let $H_\hw = \eta (H_\lw) \in \KR{\mu}_\hw$.
    Then we have
    \begin{equation} \label{eq:lwshift}
        \underline{H_\lw} = \eta (\overline{H_\hw}).
    \end{equation}
    Moreover, if $H_\hw \neq H_{\max}$, then $\underline{H_\lw} \to e_0 (\underline{H_\lw})$ is a dual regular arrow.
\end{lemma}
\begin{proof}
    First recall that $H_\hw$ is a highest weight element by \Cref{rem:eta_hw_lw}. The identity \eqref{eq:lwshift} is a consequence of \Cref{lem:classical_bar_component} and \Cref{lem:eta_ef}, from which we have, setting $M= \max(H_\hw)$,
    \[
        \eta(\overline{H_{\hw}}) = \eta\bigl(g_1 \cdots g_M(H_\hw)\bigr) = g_1^* \cdots g_M^*(H_\lw) = \underline{H_\lw}.
    \]
    Here operators $g_i$ are combinations of Kashiwara operators which raise the value of every $i$-entry in $H_{\hw}$ by $n-M$. Similarly, by \Cref{lem:eta_ef}, operators $g_i^*$ lower the value of the $n-i+1$ entry in $H_{\lw}$ by $m-1$, where $m = n-M+1$. 

    The fact that $\underline{H_\lw} \to e_0 (\underline{H_\lw})$ is a dual regular arrow follows from \Cref{th:leading_f0} and \Cref{lem:from regular to dual regular}.
\end{proof}

In the definition below we denote by $\langle g_1,\dots, g_k \rangle$ the free group generated by operators $g_1,\dots, g_k$.

\begin{definition} \label{def:leading_path}
    Consider the $\asl_n$-crystal $\KR{\mu}$ with $\mu \in \N^{\ell}$.
    Fix $H \in \KR{\mu}$, and assume that $n\ge \ell(\mu)$.
    The \defn{leading path} $\mathcal{L}_H$ from $H$ to $H_{\max}$ is a composition of Kashiwara operators constructed recursively as follows.
    \begin{enumerate}
        \item If $D'(H)=D'(H_{\max})$, then $H$ and $H_{\max}$ lie in the same classical connected component by \Cref{lem:unique_highest}. In this case we choose a composition of classical raising operators
        \[
            \mathcal{L}_H \in \langle e_1,\dots,e_{n-1}\rangle,
            \qquad \text{such that} \qquad \mathcal{L}_H(H)=H_{\max}.
        \]
    \item If $D'(H)\ne D'(H_{\max})$, let $\mathcal{G}_H \in \langle e_1 ,f_1,\dots, e_{n-1}, f_{n-1} \rangle$ be a composition of classical Kashiwara operators such that $\mathcal{G}_H(H)=\underline{H_{\mathrm{lw}}},$ where $H_{\mathrm{lw}}$ is the lowest weight element in the classical connected component of $H$. Set $H' = e_0(\underline{H_{\mathrm{lw}}})$. By \Cref{lem:explicit leading path}, the arrow $\underline{H_{\mathrm{lw}}}\longrightarrow H'$ is dual regular and $D'(H')=D'(H)+1$.
    We then define
    \[
        \mathcal{L}_H
        =
        \mathcal{L}_{H'}\circ e_0\circ \mathcal{G}_H.
    \]
    \end{enumerate}
\end{definition}

\begin{remark} \label{rem:alternative_def_dual_energy}
    The leading path gives an alternative definition of the dual energy $D'$. Assuming the leading path of the element $H$ is written as $\mathcal{L}_H = g_1 \cdots g_K$ for $g_i$ Kashwara operators, then we have $D'(H) = \Abs{\mu} - \#\{ j \mid g_j=e_0 \}$.
\end{remark}

\begin{figure}
    \centering
    \includegraphics[width=.7\linewidth]{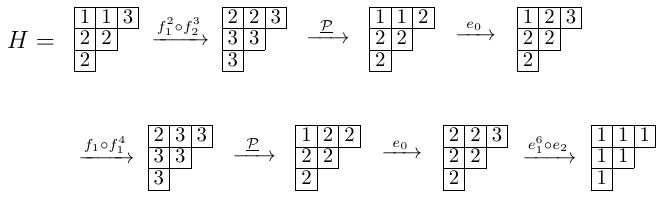}
    \caption{An example of a leading path.}
    \label{fig:example_leading_path}
\end{figure}

\subsection{$(\asl_n \oplus \asl_n)$-crystal on pairs of HWTs} \label{subs:crystal pairs HWT}

We now equip the set of pairs $(H,H')$ of horizontally weak tableaux of the same shape with two commuting $\asl_n$-crystal structures acting independently on each factor.
(Recall that we have identified $\HWT(\mu, n)$ with $\KR{\mu}$.) This makes $\bigsqcup_{\mu} \HWT(\mu,n)^2$ into an $(\asl_n \oplus \asl_n)$-crystal.

\begin{definition}
    For $k \in \{1,2\}$ and $i \in \{0,\dotsc,n-1\}$, define operators
    \[
        E^{(k)}_i,\, F^{(k)}_i \colon
        \bigsqcup_{\mu} \HWT(\mu,n)^2
        \;\longrightarrow\;
        \bigsqcup_{\mu} \HWT(\mu,n)^2\;\sqcup\;\{\zero\},
    \]
    which act componentwise for all $i$
    \begin{align}
        \label{eq:Kashiwara_skew}
        E^{(1)}_i = e_i \times \mathbf{1},
        \qquad\qquad
        E^{(2)}_i = \mathbf{1} \times e_i,
        \qquad\qquad
        F^{(1)}_i = f_i \times \mathbf{1},
        \qquad\qquad
        F^{(2)}_i = \mathbf{1} \times f_i.
    \end{align}
\end{definition}

\begin{definition}[Leading path for pairs of HWTs] \label{def:leading path pair HWT}
    Given a partition $\mu$ and a pair of elements $H,H' \in \HWT(\mu,n)$, a \defn{leading path for the pair $(H,H')$} is a composition of Kashiwara operators
    \[
        \mathcal{L}_{H,H'} = G^{(2)}_{k'} \cdots G^{(2)}_1 G^{(1)}_k \cdots G^{(1)}_1,
    \]
    where $G^{(1)}_j = g_j \times \mathbf{1}$, $G^{(2)}_j = \mathbf{1} \times g_j'$ with $g_j,g_{j}' \in \{e_i,f_i, 0\le i \le n-1\}$ are such that $\mathcal{L}_H = g_k \cdots g_1$ and $\mathcal{L}_{H'} = g_{k'}' \cdots g_1'$ are leading paths for $H$ and $H'$ respectively.
\end{definition}

If the partition $\mu$ is such that $\ell(\mu)\le n$, then for any pair $H,H' \in \HWT(\mu,n)$ a leading path $\mathcal{L}_{H,H'}$ exists and we have
\[
    \mathcal{L}_{H,H'}(H,H') = (H_{\max}, H_{\max}).
\]

\section{Commuting affine crystal structures on pairs of Young tableaux} \label{sec:crystals_pairs}

\subsection{Commuting affine crystal structures}

In this subsection, we equip the set $\bigsqcup_{\lambda,\rho} \SST(\lambda/\rho,n)^2$ with two commuting affine crystal structures.
However, unlike for pairs of HWT, the crystal operators do not entirely act on each component independently.
Let us first extend the action of classical Kashiwara operators on semistandard tableaux; for more details see \cite{BumpSchilling_crystal_book,kashiwara_nakashima1994}.

\begin{definition}
    For any $P \in \SST(\lambda/\rho,n)$ and $i\in \{1,\dots, n-1\}$, we define  $e_i(P)$ and $f_i(P)$ through the signature rule of \Cref{def:kashiwara_HWT} (1)-(4e) and (4f). Recall that the row reading word of a semistandard tableaux was defined in \Cref{def:rowreadingword}.
\end{definition}

We now define two families of Kashiwara operators on pairs of tableaux. In \Cref{prop:bicrystal_sym}, we will prove that they commute with each other.

\begin{definition}\label{def:Kashiwara_skew}
For $k \in \{1,2\}$ and $i \in \{0,\dotsc,n-1\}$, define operators
\[
    E^{(k)}_i,\, F^{(k)}_i \colon
    \bigsqcup_{\lambda,\rho}\SST(\lambda/\rho,n)^2
    \;\longrightarrow\;
    \bigsqcup_{\lambda,\rho} \SST(\lambda/\rho,n)^2\;\sqcup\;\{\zero\}.
\]
For $i\in \{1,\dots,n-1\}$ they act componentwise
\begin{align}
    \label{eq:Kashiwara_skew_tab}
    E^{(1)}_i = e_i \times \mathbf{1},
    \qquad\qquad
    E^{(2)}_i = \mathbf{1} \times e_i,
    \qquad\qquad
    F^{(1)}_i = f_i \times \mathbf{1},
    \qquad\qquad
    F^{(2)}_i = \mathbf{1} \times f_i,
\end{align}
while the 0-operators are defined by conjugation as 
\begin{align}\label{eq:EF0}
E_0^{(k)}=\iota_k\circ E^{(k)}_1\circ \iota_k^{-1},
\qquad\qquad
F_0^{(k)}=\iota_k\circ F^{(k)}_1\circ \iota_k^{-1},
\end{align}
where $\iota_k$ ($k=1,2$) was given in \cref{{def:cRSK_i1_i2}}. 
\end{definition}

\begin{example} \label{ex:Kashiwara_operators_tableaux}
    Consider the pair $(P,Q)$ of \Cref{ex:projection}. Below we compute $F^{(1)}_0 \circ E^{(1)}_1(P,Q)$
    \begin{equation} \label{eq:Kashiwara_operators_tableaux}
        \begin{split}
            \ytableausetup{aligntableaux = center,smalltableaux}
            &
            (P,Q) = \left(            \begin{ytableau}
            \none[\scriptstyle \ao{\cdots}] & \none[\scriptstyle \ao{\overline{1}}] & \none[\scriptstyle \ao{0}] & \none[\scriptstyle \ao{1}] & \none[\scriptstyle \ao{2}]& \none[\scriptstyle \ao{3}]
            \\
            \none[\scriptstyle \cdots] & *(gray) & *(gray) & *(gray) & *(gray) & 2
            \\
            \none[ \scriptstyle \cdots] & *(gray) & *(gray) & *(gray) & 1 
                \\
                \none[\scriptstyle \cdots]  & *(gray) & 1 & 2
            \end{ytableau}
            \,
            ,
            \, \begin{ytableau}
            \none[\scriptstyle \ao{\cdots}] & \none[\scriptstyle \ao{\overline{1}}] & \none[\scriptstyle \ao{0}] & \none[\scriptstyle \ao{1}] & \none[\scriptstyle \ao{2}]& \none[\scriptstyle \ao{3}]
            \\
            \none[\scriptstyle \cdots] & *(gray) & *(gray) & *(gray) & *(gray) & 2
            \\
            \none[ \scriptstyle \cdots] & *(gray) & *(gray) & *(gray) & 2 
                \\
                \none[\scriptstyle \cdots]  & *(gray) & 1 & 1
            \end{ytableau}
            \right)
            \xrightarrow[]{\quad E^{(1)}_1 \quad}
            \left(            \begin{ytableau}
            \none[\scriptstyle \ao{\cdots}] & \none[\scriptstyle \ao{\overline{1}}] & \none[\scriptstyle \ao{0}] & \none[\scriptstyle \ao{1}] & \none[\scriptstyle \ao{2}]& \none[\scriptstyle \ao{3}]
            \\
            \none[\scriptstyle \cdots] & *(gray) & *(gray) & *(gray) & *(gray) & \textcolor{red}{1}
            \\
            \none[ \scriptstyle \cdots] & *(gray) & *(gray) & *(gray) & 1 
                \\
                \none[\scriptstyle \cdots]  & *(gray) & 1 & 2
            \end{ytableau}
            \,
            ,
            \, \begin{ytableau}
            \none[\scriptstyle \ao{\cdots}] & \none[\scriptstyle \ao{\overline{1}}] & \none[\scriptstyle \ao{0}] & \none[\scriptstyle \ao{1}] & \none[\scriptstyle \ao{2}]& \none[\scriptstyle \ao{3}]
            \\
            \none[\scriptstyle \cdots] & *(gray) & *(gray) & *(gray) & *(gray) & 2
            \\
            \none[ \scriptstyle \cdots] & *(gray) & *(gray) & *(gray) & 2 
                \\
                \none[\scriptstyle \cdots]  & *(gray) & 1 & 1
            \end{ytableau}
            \right)
            \\
            &
            \xrightarrow[]{\quad \iota_1^{-1} \quad}
            \left(            \begin{ytableau}
            \none[\scriptstyle \ao{\cdots}] & \none[\scriptstyle \ao{\overline{2}}] & \none[\scriptstyle \ao{\overline{1}}] & \none[\scriptstyle \ao{0}] & \none[\scriptstyle \ao{1}] & \none[\scriptstyle \ao{2}]& \none[\scriptstyle \ao{3}]
            \\
            \none[\scriptstyle \cdots] & *(gray) & *(gray) & *(gray) & *(gray) & *(gray) & 2
            \\
            \none[ \scriptstyle \cdots] & *(gray) & *(gray) & *(gray) & *(gray) & 2 
                \\
                \none[\scriptstyle \cdots] & *(gray) & 1 & 2
            \end{ytableau}
            \,
            ,
            \, \begin{ytableau}
            \none[\scriptstyle \ao{\cdots}] & \none[\scriptstyle \ao{\overline{2}}] & \none[\scriptstyle \ao{\overline{1}}] & \none[\scriptstyle \ao{0}] & \none[\scriptstyle \ao{1}] & \none[\scriptstyle \ao{2}]& \none[\scriptstyle \ao{3}]
            \\
            \none[\scriptstyle \cdots] & *(gray) & *(gray) & *(gray) & *(gray) & *(gray) & 2
            \\
            \none[ \scriptstyle \cdots] & *(gray) & *(gray) & *(gray) & *(gray) & 2 
                \\
                \none[\scriptstyle \cdots] & *(gray) & 1 & 1
            \end{ytableau}
            \right)
            \xrightarrow[]{\quad F_1^{(1)} \quad}
            \left(            \begin{ytableau}
            \none[\scriptstyle \ao{\cdots}] & \none[\scriptstyle \ao{\overline{2}}] & \none[\scriptstyle \ao{\overline{1}}] & \none[\scriptstyle \ao{0}] & \none[\scriptstyle \ao{1}] & \none[\scriptstyle \ao{2}]& \none[\scriptstyle \ao{3}]
            \\
            \none[\scriptstyle \cdots] & *(gray) & *(gray) & *(gray) & *(gray) & *(gray) & 2
            \\
            \none[ \scriptstyle \cdots] & *(gray) & *(gray) & *(gray) & *(gray) & 2 
                \\
                \none[\scriptstyle \cdots] & *(gray) & 2 & 2
            \end{ytableau}
            \,
            ,
            \, \begin{ytableau}
            \none[\scriptstyle \ao{\cdots}] & \none[\scriptstyle \ao{\overline{2}}] & \none[\scriptstyle \ao{\overline{1}}] & \none[\scriptstyle \ao{0}] & \none[\scriptstyle \ao{1}] & \none[\scriptstyle \ao{2}]& \none[\scriptstyle \ao{3}]
            \\
            \none[\scriptstyle \cdots] & *(gray) & *(gray) & *(gray) & *(gray) & *(gray) & 2
            \\
            \none[ \scriptstyle \cdots] & *(gray) & *(gray) & *(gray) & *(gray) & 2 
                \\
                \none[\scriptstyle \cdots] & *(gray) & 1 & 1
            \end{ytableau}
            \right)
            \\
            &
            \xrightarrow[]{\quad \iota_1 \quad}
            \left(            \begin{ytableau}
            \none[\scriptstyle \ao{\cdots}] & \none[\scriptstyle \ao{\overline{2}}] & \none[\scriptstyle \ao{\overline{1}}] & \none[\scriptstyle \ao{0}] & \none[\scriptstyle \ao{1}] & \none[\scriptstyle \ao{2}]& \none[\scriptstyle \ao{3}]
            \\
            \none[\scriptstyle \cdots] & *(gray) & *(gray) & *(gray) & *(gray) & *(gray) & 1
            \\
            \none[ \scriptstyle \cdots] & *(gray) & *(gray) & *(gray) & *(gray) & 1 
                \\
                \none[\scriptstyle \cdots] & *(gray) & 1 & 1
            \end{ytableau}
            \,
            ,
            \, \begin{ytableau}
            \none[\scriptstyle \ao{\cdots}] & \none[\scriptstyle \ao{\overline{2}}] & \none[\scriptstyle \ao{\overline{1}}] & \none[\scriptstyle \ao{0}] & \none[\scriptstyle \ao{1}] & \none[\scriptstyle \ao{2}]& \none[\scriptstyle \ao{3}]
            \\
            \none[\scriptstyle \cdots] & *(gray) & *(gray) & *(gray) & *(gray) & *(gray) & 2
            \\
            \none[ \scriptstyle \cdots] & *(gray) & *(gray) & *(gray) & *(gray) & 2 
                \\
                \none[\scriptstyle \cdots] & *(gray) & 1 & 1
            \end{ytableau}
            \right) = F^{(1)}_0 \circ E^{(1)}_1(P,Q).
        \end{split}
    \end{equation}
    In the right hand side of the first line we highlighted in red the entry of $P$ that was modified by the action of $E^{(1)}_1$ Recall that the operator $\iota_1$ can also be evaluated through column internal insertion as described in \Cref{prop:iota2 by column insertion}.
\end{example}

\Cref{ex:Kashiwara_operators_tableaux} shows that, while the action of operators $E_i^{(k)},F_i^{(k)}$ for $i\in \{1,\dots,n-1\}$ and $k\in\{1,2\}$ does not modify the shape of tableaux $(P,Q)$, the same is not true for 0 operators $E_0^{(k)},F_0^{(k)}$. For instance, the action of $F_0^{(1)}$ in \eqref{eq:Kashiwara_operators_tableaux} shifted by one box to the left the third row of both tableaux. 

In the next example we observe the action of a 0-operator on stable pairs of tableaux.

\begin{example}
    The pair $(P,Q)$ below is stable and we apply $F^{(1)}_0$ to it
    \begin{equation} \label{eq:F0_PQ}
        \begin{split}
            \ytableausetup{aligntableaux = center,smalltableaux}
            &
            (P,Q) = \left(    \,        \begin{ytableau}
            *(gray) & *(gray) & *(gray) & *(gray) & *(gray) & *(gray) & *(gray) & 1 & 3 & 3
            \\
            *(gray) & *(gray) & *(gray) & 2 & 2 
                \\
           2
            \end{ytableau}
            \,
            ,
            \, \begin{ytableau}
                *(gray) & *(gray) & *(gray) & *(gray) & *(gray) & *(gray) & *(gray) & 1 & 1 & 3
            \\
            *(gray) & *(gray) & *(gray) & 2 & 2 
                \\
                1
            \end{ytableau}
            \,
            \right)
            \\
            &
            \xrightarrow[]{\,\, \iota_1^{-1} \,\,}
            \left(   \,         \begin{ytableau}
            *(gray) & *(gray) & *(gray) & *(gray) & *(gray) & 1 & 1 & 2
            \\
            *(gray) & *(gray) & *(gray) & 3 & 3 
                \\
            3
            \end{ytableau}
            \,
            ,
            \, \begin{ytableau}
            *(gray) & *(gray) & *(gray) & *(gray) & *(gray) & 1 & 1 & 3
            \\
            *(gray) & *(gray) & *(gray) & 2 & 2 
                \\
                1
            \end{ytableau}
            \,
            \right)
            \xrightarrow[]{\,\, F_1^{(1)} \,\,}
            \left(      \,      \begin{ytableau}
            *(gray) & *(gray) & *(gray) & *(gray) & *(gray) & 1 & 2 & 2
            \\
            *(gray) & *(gray) & *(gray) & 3 & 3 
                \\
            3
            \end{ytableau}
            \,
            ,
            \, \begin{ytableau}
            *(gray) & *(gray) & *(gray) & *(gray) & *(gray) & 1 & 1 & 3
            \\
            *(gray) & *(gray) & *(gray) & 2 & 2 
                \\
            1
            \end{ytableau}
            \,
            \right)
            \\
            &
            \xrightarrow[]{\,\, \iota_1 \,\,}
            \left(      \,      \begin{ytableau}
            *(gray) & *(gray) & *(gray) & *(gray) & *(gray) & *(gray) & 1 & 1 & 3
            \\
            2 & 2 
                \\
                2
            \end{ytableau}
            \,
            ,
            \, \begin{ytableau}
            *(gray) & *(gray) & *(gray) & *(gray) & *(gray) & *(gray) & 1 & 1 & 3
            \\
            *(gray) & *(gray) & *(gray) & 2 & 2 
                \\
              1
            \end{ytableau}
            \,
            \right)
            = F^{(1)}_0(P,Q).
        \end{split}
    \end{equation}
    Notice that the asymptotic stable pair associated to $(P,Q)$ is
    \[
        \Phi(P,Q) = \left( \,  \begin{ytableau} 1 & 3 & 3 \\ 2 & 2 \\ 2 \end{ytableau}
            \,
            ,
            \, 
            \begin{ytableau} 1 & 1 & 3 \\ 2 & 2 \\ 1 \end{ytableau}
            \,
            \right)
            = (H_1,H_2).
    \]
    Let us now compute $f_0(H_1)$ using \eqref{eq:0 operators by promotion}
    \begin{equation} \label{eq:F0_H}
        H_1=
        \begin{ytableau} 1 & 3 & 3 \\ 2 & 2 \\ 2 \end{ytableau}
        \xrightarrow{\quad \mathrm{pr}\quad}
        \begin{ytableau} 1 & 1 & 2 \\ 3 & 3 \\ 3 \end{ytableau}
        \xrightarrow{\quad  f_1 \quad}
        \begin{ytableau} 1 & 2 & 2 \\ 3 & 3 \\ 3 \end{ytableau}
        \xrightarrow{\quad \mathrm{pr}^{-1} \quad }
        \begin{ytableau} 1 & 1 & 3 \\ 2 & 2 \\ 2 \end{ytableau}
        = f_0(H_1).
    \end{equation}
    Compare now the change of $H_1$ in \eqref{eq:F0_H} with that of the $P$-tableau in \eqref{eq:F0_PQ}. We see that they are completely equivalent as long as we focus only on the labeled cells in $P$. That is, $\iota_1^{-1}$ act on the positive entries in $P$ in the same way as $\text{pr}$ on $H_1$. Moreover we notice that the row that gets shifted to the left by the action of $F^{(1)}_0$ is precisely the row that gets modified by the action of $F_1^{(1)}$. 
\end{example}

The above example suggests a precise correspondence between the operator $\iota_1$ and the inverse promotion, as well as a commutativity between the projection $\Phi$ and the action of Kashiwara operators $E^{(k)}_{i},F^{(k)}_{i}$ for $k\in\{1,2\}$ and $i\in \{0,\dots,n-1\}$. 
    
\begin{proposition} \label{prop:EF_stable}
    Fix $k \in \{1,2\}$ and let $(P_1,P_2)\in \SST(\lambda/\rho,n)^2$ be such that $\iota_k^{-1}(P_1,P_2)$ is stable in the sense of \cref{prop:stable}. Let $(H_1,H_2)=\Phi(P_1,P_2)$ and let $\mu$ be the shape of $H_1$ and $H_2$. 
    \begin{enumerate}[label = {\rm (\Roman*)}]
    \item \label{item:iota_k} Call $(P_1^*,P_2^*) = \iota_k^{-1}(P_1,P_2)$ and $H_k^*=\pr(H_k)$. Then, entries of $P_k^*$ are row-by-row equal to those of $H_k^*$. Moreover, calling $\lambda^*/\rho^*$ the skew shape of $(P_1^*,P_2^*)$, we have
    \begin{equation} \label{eq:shape_iota_k}
        \lambda_i^* = \lambda_i - \#\{ n\text{-entries at row $i$ of }P_k \}, \qquad \rho_i^* = \rho_i - \#\{ n\text{-entries at row $i$ of }P_k \}.
    \end{equation}
    \item \label{item:action_e_0} Suppose that $e_0(H_k) \neq \zero$ and that the action of $e_0$ changes the $r$-th row of $H_k$. Then, $(\widetilde{P}_1, \widetilde{P}_2)=E_0^{(k)}(P_1,P_2) \neq \zero$ and the line up of the positive entries in $\widetilde{P}_k$ coincides with that of $e_0(H_k)$. Moreover, calling $\widetilde{\lambda}/\widetilde{\rho}$ the skew shape of $(\widetilde{P}_1, \widetilde{P}_2)$, we have 
    \[
        \tilde{\lambda}_i= \lambda_i+\mathbf{1}_{i=r}, \qquad\qquad \tilde{\rho}_i= \rho_i+\mathbf{1}_{i=r}.
    \]
    \item \label{item:action_f_0} Similarly, if the action of $f_0(H_k) \neq \zero$ and that the action of $f_0$ changes the $r$-th row of $H_k$, we have $(\widehat{P}_1, \widehat{P}_2)=F_0^{(k)}(P_1,P_2) \neq \zero$. Moreover entries of $\widehat{P}_k$ are row-by-row equal to those of $f_0(H_k)$ and
   \[
        \widehat{\lambda}_i=\lambda_i-\mathbf{1}_{i=r}, \qquad\qquad \widehat{\rho}_i=\rho_i-\mathbf{1}_{i=r},
   \]
   where $\widehat{\lambda}/\widehat{\rho}$ is the skew shape of $(\widehat{P}_1, \widehat{P}_2)$.
    \end{enumerate}
\end{proposition}

\begin{proof}
    Call $k' = 3-k$, so that $k'=2$ if $k=1$ and $k'=1$ if $k=2$.
    
    Let us start by proving \ref{item:iota_k}. We use \Cref{prop:iota2 by column insertion} to evaluate the action of $\iota_k$ on $(P_1^*,P_2^*)$. Since $(P_1^*,P_2^*)$ is stable, the row content of $P^*_{k'}$ is equal to that of $P_{k'}$ and the two tableaux differ by a shift in the external and internal shapes. On the other hand, the $i$-th row of $P_k$ is obtained from the $i$-th row of $P_k^*$ as follows: first vacate all its cells with entry $1$, then subtract 1 from all remaining entries and finally append as many $n$-cells as $1$-cells were vacated. This proves the relation \eqref{eq:shape_iota_k}. Moreover, comparing the procedure transforming $P_k^*$ into $P_k$ with the promotion $\pr$, we see that $H_k = \pr^{-1}(H_k^*)$, since the stability of $(P_1^*,P_2^*)$ implies that the rows of $P_k$ and $P_k^*$ correspond to those of $H_k$ and $H_k^*$.

    The statements \ref{item:action_e_0}, \ref{item:action_f_0} are basic consequences of \ref{item:iota_k} and of the expressions $e_0=\pr \circ e_1 \circ \pr^{-1}$ and $f_0=\pr \circ f_1 \circ \pr^{-1}$, given in \Cref{rem:promotion}.
\end{proof}

In the following lemma we state that the column internal insertion in a tableau $P$ commutes with the action of classical Kashiwara operators.

\begin{lemma}
    \label{lem:com_EF_C}
    Let $P\in\bigsqcup_{\lambda,\rho}\SST(\lambda/\rho,n)$  and $\mathcal{C}_{[c]}$ be the column internal insertion introduced in \cref{def:c_insertion}.
    For any corner cell $c$ of $P$ and any choice of $g \in \{e_1,f_1,\dots, e_{n-1}, f_{n-1} \}$ we have
    \begin{align}
        \label{eq:com_EF_C}
       g\circ\mathcal{C}_{[c]}(P)=\mathcal{C}_{[c]}\circ g(P).
    \end{align}
\end{lemma}
\begin{proof}
    It is well known that the action of the classical Kashiwara operators commutes with another combinatorial operation called \defn{jeu de taquin slides (jdt)}; see e.g. ~\cite[Thm.~8.7]{BumpSchilling_crystal_book}. We will not introduce the jdt here, and we refer the interested reader to \cite{lothaire_2002,sagan2001symmetric} and references therein. The internal column insertion $\mathcal{C}_{[c]}$ on the tableau $P$ can be realized by a finite sequence of jdt slides; see \cite[App.~A1.2]{Stanley1999} and \cite[Prop.~A.4, Thm.~A.2, Rem~A.3]{imamura2021skewRSK}. This shows that the commutation relation \eqref{eq:com_EF_C} holds.
\end{proof}

\begin{lemma} \label{lem:com_EF_iota}
    Let $(P,Q)\in\bigsqcup_{\lambda,\rho}\SST(\lambda/\rho,n)^2$and let $(k,k')\in \{(1,2),(2,1)\}$. Then, for any $G^{(k)} \in \{E_i^{(k')},F_i^{(k')} \mid 0\le i \le n-1 \}$, we have
    \begin{align}
        \label{eq:com_EF_iota}
       G^{(k)}\circ \iota_{k'} (P,Q)= \iota_{k'} \circ G^{(k)} (P,Q).
    \end{align}
\end{lemma}

\begin{proof}
    Let us prove only the case $k=1,k'=2$, as the complementary case $k=2,k'=1$ can be shown in similar fashion. Let us first consider the case where $G$ is not a 0-Kashiwara operator $G^{(1)} = g \times \mathbf{1}$ for $g \in \{e_i,f_i: 1 \le i \le n-1 \}$. Denote by $(\overline{P},\overline{Q})=\iota_2(P,Q)$. By \Cref{prop:iota2 by column insertion}, we have
    \[
         \overline{P} = \mathcal{C}_{[c_1]} \circ \cdots \circ \mathcal{C}_{[c_s]}(P),
    \]
    where the columns $c_1,\dots, c_s$ are determined by the cells with entry 1 in $Q$. Then, by \Cref{lem:com_EF_C} and by the description of $\overline{Q}$ provided by \Cref{prop:iota2 by column insertion}, we have
    \begin{equation} \label{eq:com_EF_iota_classical}
        G^{(1)} \circ \iota_2(P,Q) = (g \circ \mathcal{C}_{[r_1]} \circ \cdots \circ \mathcal{C}_{[r_s]}(P), \overline{Q} ) = ( \mathcal{C}_{[r_1]} \circ \cdots \circ \mathcal{C}_{[r_s]} \circ g(P), \overline{Q} ) = \iota_2 \circ G^{(1)} (P,Q).
    \end{equation}
    This proves the relation \eqref{eq:com_EF_iota} for $k=1,k'=2$, whenever $G^{(1)}$ is not a 0-Kashiwara operator. Now, let $G^{(1)} \in \{E_0^{(1)}, F_0^{(1)}\}$ and by the definition \eqref{eq:EF0} we can write $G^{(1)} = \iota_1 \circ \widetilde{G}^{(1)} \circ \iota_1^{-1}$ for $\widetilde{G}^{(1)} \in \{E_1^{(1)}, F_1^{(1)}\}$. Then, we have
    \begin{equation*}
        \begin{split}
            G^{(1)} \circ \iota_2 
            & = \iota_1 \circ \widetilde{G}^{(1)} \circ \iota_1^{-1} \circ \iota_2 
            = \iota_1 \circ \widetilde{G}^{(1)} \circ \iota_2 \circ \iota_1^{-1} 
            \\
            &= \iota_1 \circ \iota_2\circ \widetilde{G}^{(1)} \circ \iota_1^{-1} 
            = \iota_2 \circ \iota_1 \circ \widetilde{G}^{(1)} \circ \iota_1^{-1} 
            = \iota_2 \circ G^{(1)}.
        \end{split}
    \end{equation*}
    Above, in the second and in the fourth equality we used the commutation relation \Cref{prop:properties of iota and cRSK}, \ref{item:commutation iota}, while in the third equality we used \eqref{eq:com_EF_iota_classical}.
    This completes the proof.
\end{proof}

\begin{theorem}
    \label{prop:bicrystal_sym}
    The families of Kashiwara operators $\{ E_i^{(1)}, F_i^{(1)} \mid 0\le i \le n-1 \}$ and $\{ E_i^{(2)}, F_i^{(2)} \mid 0\le i \le n-1 \}$ commute with each other.
    That is, they equip $\bigsqcup_{\lambda,\rho} \SST(\lambda/\rho,n)^2$ with a $(\asl_n \oplus \asl_n)$-crystal structure. 
\end{theorem}

\begin{proof}
    It is straightforward from the definition \eqref{eq:Kashiwara_skew_tab} that for all
    \[
        G^{(k)} \in \{ E_i^{(k)}, F_i^{(k)} \mid 1\le i \le n-1 \}, 
        \qquad \text{where } k\in \{1,2\},
    \]
    and all pairs of tableaux $(P,Q) \in \bigsqcup_{\lambda,\rho} \SST(\lambda/\rho,n)^2$ we have
    \[
        G^{(1)} \circ G^{(2)} (P,Q) = G^{(2)} \circ G^{(1)} (P,Q).
    \]
    This is because non-$0$ operators act component wise.
    Therefore we only need to show commutation relations involving the Kashiwara $0$ operators.
    Let $G_0^{(1)} = \iota_1 \circ G_1^{(1)} \circ \iota_1^{-1}$ for $G_1^{(1)} \in \{E^{(1)}_1,F^{(1)}_1\}$.
    Consider the non-$0$ operator $G_i^{(2)} \in \{ E_i^{(1)}, F_i^{(1)} \mid 1\le i \le n-1 \}$.
    Then, using \Cref{lem:com_EF_iota}, we have
    \begin{equation*}
        \begin{split}
            G^{(1)}_0 \circ G_i^{(2)}
            & = \iota_1 \circ G_1^{(1)} \circ \iota_1^{-1} \circ G_i^{(2)} 
            = \iota_1 \circ G_1^{(1)} \circ G_i^{(2)} \circ \iota_1^{-1} 
            \\
            &
            = \iota_1 \circ G_i^{(2)} \circ G_1^{(1)} \circ \iota_1^{-1} 
            = G_i^{(2)} \circ \iota_1 \circ G_1^{(1)} \circ \iota_1^{-1} 
            = G_i^{(2)} \circ G_0^{(1)}.
        \end{split}
    \end{equation*}
    The remaining cases of when $G^{(2)}$ is the Kashiwara $0$ operator (with possibly $G^{(1)}$ being a Kashiwara $0$ operator, where we need to use the above case) is analogous.
\end{proof}

\begin{theorem}
    \label{th:crystal_sym}
     The skew column RSK map is an isomorphism of $(\asl_n \oplus \asl_n)$-crystals; that is, for all $G\in \{ E_i^{(k)}, F_i^{(k)} \mid 0\le i \le n-1, 1\le k \le 2 \}$
     we have 
     \begin{align}
         \label{eq:crystal_sym}
         \mathbf{cRSK}\circ G=G\circ\mathbf{cRSK}.
     \end{align}
\end{theorem}

We call the commutation relations \eqref{eq:crystal_sym} the \defn{crystal symmetry of the column skew RSK dynamics}.

\begin{proof}
    By \Cref{prop:properties of iota and cRSK}, \ref{item:cRSK iota} we have $\cRSK = \iota_1^n = \iota_2^n$. For $k\in \{1,2\}$, let $G\in \{ E_i^{(k)}, F_i^{(k)} \mid 0\le i \le n-1\}$. Then by \Cref{lem:com_EF_iota}, $G$ commutes with $\iota_{k'}$, for $k'=3-k$ and hence with $\cRSK$.
\end{proof}

\subsection{Leading paths for pairs of skew tableaux}
\label{sec:leading_skew}

The result of \Cref{th:crystal_sym} allows us to relate the $(\asl_n \oplus \asl_n)$-crystal structures on the sets of pairs of skew tableaux $\bigsqcup_{\lambda,\rho} \SST(\lambda/\rho,n)^2$ and on the set $\bigsqcup_\mu \HWT(\mu,n)^2$, the latter defined in \Cref{subs:crystal pairs HWT}. For the next lemma recall the projections $\Phi,\Phi^-$ introduced in \Cref{def:Phi} and \Cref{def:Phi_minus}.

\begin{lemma}
    \label{lem:kashiwara_phi} 
    For any $k\in \{1,2\}$ and $i\in \{0,\dots,n-1\}$, we have
    \begin{equation} \label{eq:kashiwara_phi}
        E^{(k)}_i \circ \Phi = \Phi \circ E^{(k)}_i
        \qquad
        \text{and}
        \qquad
        F^{(k)}_i \circ \Phi = \Phi \circ F^{(k)}_i.
    \end{equation}
    In other words $\Phi$ is a strict crystal morphism. Similarly, also $\Phi^-$ is a strict crystal morphism.
\end{lemma}
\begin{proof}
    For any pair of semistandard Young tableaux $(P,Q)$, call $(H_1,H_2) = \Phi(P,Q)$. By \Cref{th:crystal_sym}, we have $\cRSK^t \circ E^{(k)}_i(P,Q) = E^{(k)}_i \circ \cRSK^t (P,Q)$ for all $t\in \mathbb{Z}$. Taking $t>0$ large enough this implies that, when $i \neq 0$, we have $\Phi \circ E^{(k)}_i(P,Q) = E^{(k)}_i(H_1,H_2)$. On the other hand case $i=0$ was already covered in \Cref{prop:EF_stable}\ref{item:action_e_0}. This proves the first relation in \eqref{eq:kashiwara_phi}. The remaining relation in \eqref{eq:kashiwara_phi} as well as the intertwining relations involving $\Phi^-$ can be proven analogously.
\end{proof}

In view of \cref{lem:kashiwara_phi}, we define the leading path for a pair of skew tableaux as the pull-back along $\Phi$ of the leading path for pairs of horizontally weak tableaux.

\begin{definition}
    \label{def:leading_PQ}
    Let $(P,Q)\in\bigsqcup_{\lambda,\rho} \SST(\lambda/\rho,n)$ and  $(H_1,H_2)=\Phi(P,Q)$.
    The \defn{leading path $\mathcal{L}_{P,Q}$ for the pair $(P,Q)$} is
    \[
        \mathcal{L}_{P,Q} = \mathcal{L}_{H_1,H_2},
    \]
    where $\mathcal{L}_{H_1,H_2}$ is the leading path for the pair $(H_1,H_2)$ introduced in \Cref{def:leading path pair HWT}.
\end{definition}

\begin{lemma}
    \label{lem:bottom_row}
    
    Let $\lambda, \rho \in \signatures$ be signatures, and let $(P_1,P_2)\in \SST(\lambda/\rho,n)$. Assume that $(H_1,H_2) = \Phi(P_1,P_2) \in \HWT(\mu,n)$ for some partition $\mu$ with $\ell(\mu) \le n$. Assume that, for some $k\in \{1,2\}$ and for some $H_\lw \in \KR{\mu}_\lw$ such that $D'(H_\lw) < D'(H_{\max})$, we have
    \[
       H_k = \underline{H_{\lw}},
    \]
    where the element $\underline{H_{\lw}}$ was introduced in \Cref{lem:explicit leading path}. Let $(\overline{P}_1,\overline{P}_2) = E^{(k)}_0(P_1,P_2)$ and let $\overline{\lambda} / \overline{\rho}$ we the shape of $\overline{P}_1,\overline{P}_2$. Then, for $\ell = \ell(\lambda)$, we have
    \begin{subequations}
    \begin{gather}
        \label{eq:bottom_row_shape}
        \overline{\lambda}_\ell = \lambda_\ell, \qquad\qquad \overline{\rho}_\ell = \rho_\ell,
        \\
        \label{eq:rho_+1}
        \abs{\overline{\rho}} = \abs{\rho} +1.
        \end{gather}
    \end{subequations}
\end{lemma}

\begin{proof}
    Let $\ell' = \ell(\mu)$, so that $\ell' \le \ell$ and let $m = \min(H_\lw)$. Denote by $H_\lw = h_{\ell'}' \otimes \cdots \otimes h_1'$ where $\ell' \le \ell$ and $\underline{H_\lw} = \underline{h_{\ell'}'} \otimes \cdots \otimes \underline{h_1'}$. Since $H_\lw$ is a lowest weight element, $\min(h_{\ell'}')=m$ only if $m = n$, in which case $H_\lw$ is in the same classically connected component of $H_{\max}$. Then, the hypothesis that $D'(H_\lw) < D'(H_{\max})$ implies that $h_{\ell'}'$ contains no $m$ entry and $\underline{h_{\ell'}'}$ contains no $1$ entries.

    Let us prove the Lemma in the case $k=1$, since an analogous proof also works for the $k=2$ case. By hypothesis we have $\Phi(P_1,P_2) = (\underline{H_\lw},H)$ for some $H\in \KR{\mu}$. Since $\underline{H_{\ell'}'}$ does not possess any $1$ entry, then also the last row of $P_1$ does not possess any $1$ entry: this is a basic consequence of \cref{prop:cRSK_upright}, of the fact that the map $\cRSK = \iota_2^n$ can be realized through column internal insertion (\Cref{prop:iota2 by column insertion}) and the definition of the projection $\Phi$. 

    By direct analysis of the action of the Kashiwara operator $E^{(1)}_0$, we see that the only way any of the equalities \eqref{eq:bottom_row_shape} fails is that the action of $E^{(1)}_1$ on $(\widetilde{P}_1,\widetilde{P}_2) = \iota_1^{-1}(P_1,P_2)$ turns a 2 entry in the last row of $\widetilde{P}_1$ into a 1 entry. Since the number of 2 entries in the last row of $\widetilde{P}_1$ is equal to the number of 1 entries in the last row of $P_1$ (which is zero), equalities \eqref{eq:bottom_row_shape} always hold if $H_\hw \neq H_{\max}$. Finally, \eqref{eq:rho_+1} follows from the definition of the operator $E^{(k)}_0$ and always holds as long as $E^{(1)}_0(P,Q) \neq \zero$.
\end{proof}

We can now examine the transformation of a pair of tableaux under leading paths.

\begin{theorem}
    \label{cor:bottom_row}
    Let $\lambda, \rho \in \signatures_\ell$ be signatures of length $\ell$ and let $(P,Q)\in \SST(\lambda/\rho,n)$. Assume that $(H_1,H_2) = \Phi(P_1,P_2) \in \HWT(\mu,n)^2$ for some partition $\mu$ with $\ell(\mu) \le n$. Then, for any leading paths $\mathcal{L}_{H_1},\mathcal{L}_{H_2}$ for $H_1, H_2$, letting $\mathcal{L}_{P,Q}$ be the corresponding leading path for $(P,Q)$ and letting $\overline{\lambda}/\overline{\rho}$ be the shape of $\mathcal{L}_{P,Q}(P,Q)$, we have
    \begin{equation} \label{eq:rho_difference}
        \abs{\overline{\rho}} = \abs{\rho} + 2 D'(H_{\max}) - D'(H_1) - D'(H_2)
    \end{equation}
    and
    \begin{equation} \label{eq:bottom_row_shape_lm}
        \bar{\lambda}_\ell=\lambda_\ell,
        \qquad\qquad
        \bar{\rho}_{\ell}=\rho_{\ell}.
    \end{equation}
\end{theorem}
\begin{proof}
    Let $\mathcal{L}_{H_1}$, $\mathcal{L}_{H_2}$ be the leading paths corresponding to $H_1,H_2$. They consist of classical Kashiwara operators and of 0-arrows joining $\underline{H_\lw} \to e_0(\underline{H_\lw})$ for some lowest weight element $H_\lw$. Then, each path $\mathcal{L}_{H_i}$ contains exactly $D'(H_{\max}) - D'(H_i)$ $e_0$ arrows and, by repeated applications of \Cref{lem:bottom_row}, the action of the path $\mathcal{L}_{P,Q}$ on the pair $(P,Q)$ transforms the internal shape according to \eqref{eq:rho_difference}. Similarly, by \Cref{lem:bottom_row} the length last row of the tableaux $(P,Q)$ never gets modified and \eqref{eq:bottom_row_shape_lm} also holds.
\end{proof}

\subsection{Scattering rules}

Using the crystal symmetry of \Cref{th:crystal_sym} we deduce asymptotic properties of the cRSK dynamics. For this recall the notion of stability of \Cref{def:stability}.

\begin{proposition}
    \label{prop:Kashiwara_ker}
        Let $(P,Q)\in\bigsqcup_{\lambda,\rho} \SST (\lambda/\rho,n)$ and for 
       \[
            G \in \{ \iota_k,E^{(k)}_i,F^{(k)}_i \mid 0\le i \le n-1, 1\le k \le 2 \},
        \]
        let $(\widetilde{P},\widetilde{Q}):=G(P,Q)$. Then we have
        \begin{enumerate} [label = {\rm (\Roman*)}]
            \item \label{item:same asympotic increment} $(P,Q)$ and $(\widetilde{P},\widetilde{Q})$ have a common positive and negative asymptotic shape;
            \item \label{item:same kernel}$\ker(P,Q)=\ker(\widetilde{P},\widetilde{Q})$.
       \end{enumerate}
\end{proposition}

\begin{proof}
    Consider first the case $G\in \{\iota_1, \iota_2\}$. It is evident from \Cref{prop:stable field}, \Cref{prop:properties of iota and cRSK}, \ref{item:iota2 field}, \ref{item:iota1 field} and \Cref{prop:properties_kernel}, \ref{item:ker_asymptotic_pos}, that the action of $G$ cannot modify neither the asymptotic shape nor the kernel of $(P,Q)$. 
    
    Consider then the case $G \in \{ E^{(k)}_i,F^{(k)}_i \mid 1\le i \le n-1, 1\le k \le 2 \}$. By \Cref{th:crystal_sym} and \Cref{prop:properties of iota and cRSK}, \ref{item:cRSK iota}, we have $\cRSK^t (\widetilde{P},\widetilde{Q}) = \cRSK^t\circ G(P,Q) = G \circ \cRSK^t (P,Q)$ for all $t\in\Z$.
    In particular, if $t$ is large enough so that both $\cRSK^t(P,Q)$ and $\cRSK^t (\widetilde{P},\widetilde{Q})$ are stable, it follows that the asymptotic increments and the kernel of $(P,Q)$ and $(\widetilde{P},\widetilde{Q})$ are the same since $G$ does not modify the shape of the tableaux it acts upon. 
    
    Finally, the statements \ref{item:same asympotic increment}, \ref{item:same kernel} for the case $G \in \{ E_0^{(k)}, F_0^{(k)}, 1\le k \le 2\}$, follow combining the above two cases, since 0-operators are conjugations of 1-operators and $\iota$-maps as in \eqref{eq:EF0}.
\end{proof}

\begin{theorem} \label{thm:scattering_rules}
    Fix $(P,Q) \in \bigsqcup_{\rho,\lambda} \SST (\lambda/\rho,n)^2$ and let $(H_1,H_2) = \Phi(P,Q)$ and $(H_1^-,H_2^-) = \Phi^-(P,Q)$. Assume that $H_1,H_2 \in \HWT(\mu,n)$ and $H_1^-,H_2^- \in \HWT(\widetilde{\mu},n)$ with $\ell(\mu) \le n$. Then, we have
    \begin{enumerate} [label = {\rm (\Roman*)}]
        \item \label{item:negative asymptotic shape} $\widetilde{\mu} = \mu^-$, where $\mu^- = \sigma^-(\mu)$ and $\sigma^-$ is the total reverse permutation; 
        \item \label{item:negative asymptotic tableaux} $H_{k}^- = \sigma^-(H_k)$ for $k\in \{1,2\}$, where $\sigma^-$ is the isomorphism of crystals defined in \Cref{rem:reverse_iso}.
    \end{enumerate}
\end{theorem}

\begin{proof}
    Consider a leading path $\mathcal{L}_{P,Q}$ for the pair $(P,Q)$ and let $(T,T)= \mathcal{L}_{P,Q}(P,Q)$ be its image. Then, $T = \overline{\rho} \preceq \overline{\lambda}$ is a tableaux with 1-entries only. By \Cref{prop:Kashiwara_ker}, the positive and negative asymptotic shapes $\mu, \widetilde{\mu}$ of the pair $(P,Q)$ are the same as those of the pair $(T,T)$. By \Cref{prop:cRSK_BBS} the cRSK dynamics with initial data $(T,T)$ is equivalent to the BBS dynamics with initial data $B=\Pr(T)$, where $\Pr$ is the projection of \Cref{def:map_to_BBS}. Hence, $\mu$ is also the asymptotic shape of the configuration $B$, which can be computed through the KKR map $\KKR(B)=(\mu,J)$ for some rigging $J$. By \Cref{thm:KKR linearization} this is equal, upon reversal, to the negative asymptotic shape $\widetilde{\mu} = \mu^-$ and this proves \Cref{item:negative asymptotic shape}.

    To prove \Cref{item:negative asymptotic tableaux} recall that, by \Cref{th:crystal_sym}, the map $\cRSK^{2t}:(P_{-t},Q_{-t}) \to (P_{t},Q_{t})$ commutes with the Kashiwara operators $E^{(k)}_i,F^{(k)}_i$ for $k\in\{1,2\}$ and $i\in \{0,\dots,n-1\}$.
    Let $\mathcal{G}_{P,Q}$ be a finite collection of compositions of Kashiwara operators so that the projection of the orbit of the pair $(P,Q)$ $\{\Phi(g(P,Q)) \mid g \in \mathcal{G}_{P,Q} \} = \HWT(\mu,n)$ is the full set of HWT of shape $\mu$. Take $t>0$ large enough so that all pairs in the sets $\{g(P_{-t},Q_{-t}) \mid g \in \mathcal{G}_{P,Q}\}$ and $\{(P_{t},Q_{t}) \mid g \in \mathcal{G}_{P,Q}\}$ are respectively negatively and positively stable. Then the map 
    \[
        \cRSK^{2t} \colon \{g(P_{-t},Q_{-t}) \mid g \in \mathcal{G}_{P,Q}\} \to \{(P_{t},Q_{t}) \mid g \in \mathcal{G}_{P,Q}\}
    \]
    induces a surjective morphism of crystals $\HWT(\mu^-,n)^2 \to \HWT(\mu,n)^2$, which maps the element $(H_{\max}^-,H_{\max}^-)$ to $(H_{\max},H_{\max})$, where $H_{\max}^- \in \HWT(\mu^-,n)$, $H_{\max} \in \HWT(\mu,n)$ are the unique elements with only 1 entries. By symmetry also, the map $\cRSK^{-2t}$ induces a surjective morphism $\HWT(\mu,n)^2 \to \HWT(\mu^-,n)^2,$ and we must have $(H_1,H_2) \to (\sigma^-(H_1),\sigma^-(H_2))$, since $\sigma^-$ is the unique crystal isomorphism between $\KR{\mu}$ and $\KR{\mu^-}$. This completes the proof.
\end{proof}

\section{The linearizing map: proof and examples}

\subsection{Proof of \Cref{th:main_bij}}

In this subsection we construct explicitly the bijection $\Upsilon_{\mathrm{col}}$ transforming a pair of semistandard tableaux $(P,Q)$ into a quadruple $(H_1,H_2,\kappa,\nu)$ where $H_1,H_2$ are horizontally weak tableaux of same partition shape, $\kappa$ is a list of integers and $\nu$ is a signature. Most of the necessary ingredients to our construction have already been defined and studied in the above section.

\begin{definition}[The map $\Upsilon_{\mathrm{col}}$] \label{def:upsilon}
    Let $n\in \N$ and let $(P,Q) \in \bigsqcup_{\rho,\lambda} \SST(\lambda/\rho,n)^2$. We define the map
    \[
        \Upsilon_{\mathrm{col}}(P,Q) = (H_1,H_2,\kappa,\nu)
    \]
    as follows. First, we set $ (H_1,H_2) = \Phi(P,Q)$, where $\Phi$ is the soliton data map of \Cref{def:Phi} and denote by $\mu$ be the common shape of $H_1,H_2$. Set
    \[
        n'=\max \{ \ell(\mu),n \}.
    \]
    We regard $H_1,H_2$ as elements of $\HWT(\mu,n')$ and $P,Q$ as elements in $\bigsqcup_{\rho,\lambda} \SST(\lambda/\rho,n')^2$.
    
    Let $\mathcal{L}_{H_1}, \mathcal{L}_{H_2}$ be the leading paths in $\HWT(\mu,n')$ constructed in \Cref{def:leading_path}. Using the corresponding operators on pairs of tableaux define $\mathcal{L}_{P,Q} = \mathcal{L}_{H_1,H_2}$ as in \Cref{def:leading_PQ}. Then
    \[
        (T,T) = \mathcal{L}_{P,Q} (P,Q),
    \]
    for some tableau $T \in \SST(1)$. Applying to $T$ the map $\Xi$ of \eqref{eq:Xi} we obtain a pair $(B,\nu) = \Xi(T)$. Finally let $(\mu,J) = \KKR(B)$ be the rigged configuration associated with the BBS state $B$ and set $\kappa = J+\mu$.
\end{definition}

\begin{remark}
    In the definition of $\Upsilon_{\mathrm{col}}$, the passage from $n$ to $n'=\max\{\ell(\mu),n\}$ should be understood only as an embedding of sets
    \[
        \HWT(\mu,n) \subseteq \HWT(\mu,n').
    \]
    It is not an embedding of $U_q(\asl_n)$-crystals.
    This enlargement of the alphabet is needed to guarantee the existence of the leading paths $\mathcal{L}_{H_1}$ and $\mathcal{L}_{H_2}$ used in the construction; see \Cref{def:leading_path}. Since $n'$ is fixed by definition, the construction is unambiguous.
\end{remark}

\begin{definition}[The inverse map $\Upsilon_{\mathrm{col}}^{-1}$] \label{def:upsilon_inv}
    Let $n \in \N$ and fix a quadruple 
    \[
        (H_1,H_2,\kappa,\nu) \in \bigsqcup_{\mu \in \mathbb{Y}}\HWT(\mu,n)^2 \times \mathcal{K}(\mu)\times \mathbb{S}.
    \]
    Let $\mu$ be the common shape of $H_1,H_2$ and let $J=\kappa-\mu$. Let $B =\KKR^{-1}(\mu,J)$ be the corresponding BBS state and define $T=\Xi^{-1}(B,\nu)$.
    
    Set $n'=\max(\ell(\mu),n)$. Viewing $H_1,H_2$ as elements of $\HWT(\mu,n')$, let $\mathcal{L}_{H_1}, \mathcal{L}_{H_2}$ be the leading paths for $H_1,H_2$ constructed in \Cref{def:leading_path}. Define $\mathcal{L}^{-1} = \mathcal{L}_{H_1,H_2}^{-1}$. Finally, set $\mathcal{L}^{-1}(T,T) = (P,Q)$ and define
    \[
        \Upsilon_{\mathrm{col}}^{-1} (H_1,H_2,\kappa,\nu) = (P,Q).
    \]
\end{definition}

We are now ready to prove our main theorem.

\begin{proof}[Proof of \Cref{th:main_bij}]
    The map $\Upsilon_{\mathrm{col}}$ constructed in \Cref{def:upsilon} is well posed and so is its inverse $\Upsilon_{\mathrm{col}}^{-1}$ given in \Cref{def:upsilon_inv}. The property \eqref{eq:PQHH} is part of the definition of $\Upsilon_{\mathrm{col}}$. We then turn to the volume identity \eqref{eq:volume_pre}. Let $\mathcal{L}_{P,Q}$ be the explicit leading map associated to the pair $(P,Q)$ and let $(T,T) = \mathcal{L}_{P,Q}(P,Q)$ with $T = \overline{\lambda} / \overline{\rho}$. Then, if $\lambda/\rho$ is the shape of $P,Q$ we have, combining \Cref{cor:bottom_row} and \Cref{prop:volume_BBS_tableau}
    \begin{equation*}
        \begin{split}
            |\overline{\rho}| &= |\rho| + 2 D'(H_{\max}) - D'(H_1) - D'(D_2) 
            \\
            &= |\nu| + |J| + |\mu| + 2 \Abs{\mu} 
        \end{split}
    \end{equation*}
    which implies \eqref{eq:volume_pre}, since, $\kappa = J+\mu$ and $D'(H_{\max}) = \Abs{\mu}$ by \Cref{lem:unique_highest}.
    
    Let us now establish relation \eqref{eq:length_pre}. Let $\boldsymbol{\lambda} = \boldsymbol{\lambda}[P,Q]$ be the field of interlacing signatures associated to $P,Q$. Then, by definition, all signatures $\boldsymbol{\lambda}(p)$ for $p\in \mathscr{C}_n$ have the same length $\ell = \ell(\lambda)$. Recall the definition \eqref{eq:nu} of the signature $\nu$. Then for $i \in \{1,\dots, \ell(\nu)\}$ the limit $\lim_{t \to - \infty} \boldsymbol{\lambda}_i(t,0)$ exists finite, while for $i \in \{ \ell(\nu)+1 ,\dots , \ell \}$ we have, by \Cref{prop:stable field} and \Cref{thm:scattering_rules}, 
    \[
        \lim_{t \to - \infty} \frac{\boldsymbol{\lambda}_i(t,0)}{t} = \mu_{i-\ell(\nu)}^- = \mu_{\ell(\mu) - i + \ell(\nu) + 1}. 
    \]
    This shows that $\ell = \ell(\nu) + \ell(\mu)$.

    Finally, assume that the shape $\lambda / \rho$ of the pair $P,Q$ is such that $\lambda,\rho$ are partitions (i.e. all their entries are non-negative). By \Cref{cor:bottom_row}, the action of the leading path $\mathcal{L}_{P,Q}$ on the pair $(P,Q)$ produces a pair $(T,T) = \mathcal{L}_{P,Q}(P,Q)$ where the tableau $T$ has shape $\overline{\lambda}/\overline{\rho}$ such that $\overline{\lambda} = \lambda_\ell \ge 0$ and $\overline{\rho} = \rho_\ell \ge 0$. Then, $\Xi(T) = (B,\nu)$ is such that $B \in \BBS_{>0}$ and $\nu$ is a signature with non negative entries. Since by \Cref{thm:KKR linearization}, the image of the KKR map $\KKR(B) = (\mu,J)$ is such that $J_i \ge - \mu_i$ we have $\kappa \in \mathcal{K}_+(\mu)$. This completes the proof.
\end{proof}

\begin{corollary}
    Let $(P,Q) \in \bigsqcup_{\rho,\lambda} \SST(\lambda/\rho,n)$ and let $(P',Q')=\cRSK(P,Q)$. Then, 
    \begin{equation} \label{eq:linearization cRSK}
        \text{if}
        \quad 
        \Upsilon_{\mathrm{col}}(P,Q)=(H_1,H_2;\kappa,\nu), 
        \quad 
        \text{we have} 
        \quad 
        \Upsilon_{\mathrm{col}}(P',Q')=(H_1,H_2;\kappa+ \mu,\nu).
    \end{equation}
\end{corollary}
\begin{proof}
    It is clear, by \Cref{def:Phi} of the map $\Phi$ and \eqref{eq:nu} that the elements $H_1,H_2,\nu$ are common in the image under $\Upsilon_{\mathrm{col}}$ of $(P,Q)$ and $(P',Q')$.
    Let $\mathcal{L}_{P,Q}$ be the leading map associated to the pair $(P,Q)$. By the crystal symmetry of the cRSK dynamics (\Cref{th:crystal_sym}), $\mathcal{L}_{P,Q}$ is also a leading map for the pair $(P',Q')$. Let $(T,T)=\mathcal{L}_{P,Q}(P,Q)$ and $(T',T')=\mathcal{L}_{P,Q}(P',Q')$. Then, always by \Cref{th:crystal_sym}, we have $(T',T') = \cRSK(T,T)$. By \Cref{prop:cRSK_BBS}, denoting $B=\Pr(T)$ and $B' = \Pr(T')$, we have $B' = \sfT (B)$ and their image under the KKR map is, by \Cref{thm:KKR linearization}
    \[
        \KKR (B) = (\mu, J),
        \qquad  
        \KKR (B') = (\mu, J+\mu).
    \]
    Then, by \Cref{def:upsilon}, \eqref{eq:linearization cRSK} follows.
\end{proof}

\subsection{An example}

Let us review the construction of the map $\Upsilon_{\mathrm{col}}$ through an example. Consider the pair of tableaux $(P,Q)$ given below
\begin{equation} \label{eq:pair_example_2}
    (P,Q)
    =
    \left(
    \,
    \begin{ytableau}
    	*(gray) & *(gray) & *(gray) & *(gray) & 1 & 3
    	\\
    	*(gray) & *(gray) & *(gray) & 2 &  3
    	\\
    	*(gray) & 2
        \\
        2
	\end{ytableau}
    ,
    \begin{ytableau}
    	*(gray) & *(gray) & *(gray) & *(gray) & 1 & 3
    	\\
    	*(gray) & *(gray) & *(gray) & 1 &  2
    	\\
    	*(gray) & 2
        \\
        2
	\end{ytableau}
    \,
    \right).
\end{equation}
Running the cRSK dynamics with initial data $(P_0,Q_0) = (P,Q)$, as in \Cref{fig:skew_RSK_dynamics_example_2}, we get
\[
    \Phi(P,Q) = (H_1 , H_2), 
    \qquad
    \text{with}
    \qquad
    H_1 = 
    \begin{ytableau}
    	1 & 3 & 3
    	\\
        2 & 2
    	\\
    	2
	\end{ytableau}
    ,
    \qquad
    H_2 = 
    \begin{ytableau}
    	1 & 1 & 3
    	\\
        2 & 2
    	\\
    	2
	\end{ytableau}
    .
\]
\begin{figure}
    \centering
    \includegraphics[width=.9\linewidth]{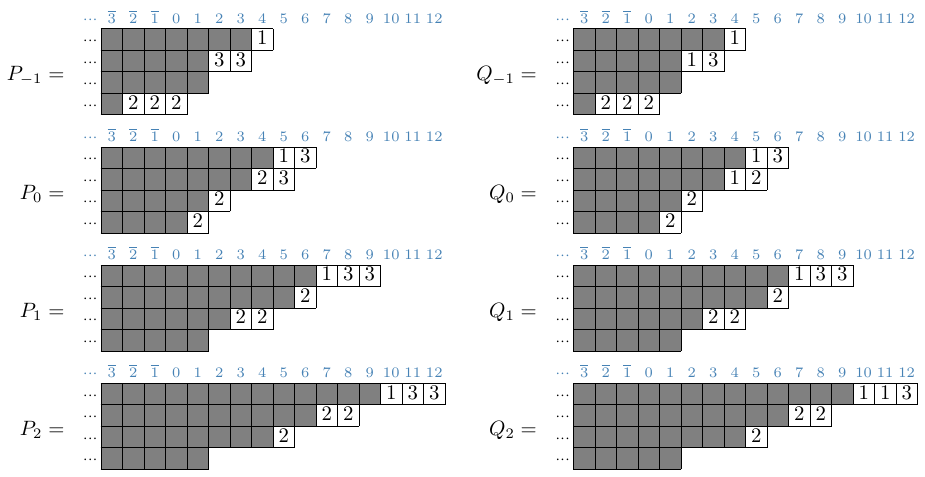}
    \caption{An example of a skew column RSK dynamics.}
    \label{fig:skew_RSK_dynamics_example_2}
\end{figure}
The explicit leading path for the horizontally weak tableau $H_2$ was presented in \Cref{fig:example_leading_path}, whereas the one for $H_1$ is given below
\[
    \includegraphics[width=.7\linewidth]{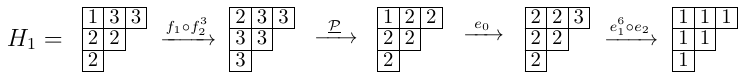}.
\]
Then, the leading path for the pair $(P,Q)$ is
\begin{equation} \label{eq:explicit_leading_path_example_2}
    \begin{split}
        \mathcal{L}_{P,Q} =&
        F^{(1)}_1 \circ \bigl( F^{(1)}_2 \bigr)^3 \circ \underline{\mathcal{P}}^{(1)} \circ E_0^{(1)} \circ \bigl( E^{(1)}_1 \bigr)^6 \circ E_2^{(1)}
        \\
        & \quad \circ \bigl( F_1^{(2)} \bigr)^2 \circ \bigl( F_2^{(2)} \bigr)^3 \circ \underline{\mathcal{P}}^{(2)} \circ E_0^{(2)} \circ  F_1^{(2)} \circ \bigl(  F_1^{(2)} \bigr)^4  \circ \underline{\mathcal{P}}^{(2)} \circ E_0^{(2)} \circ \left( E^{(2)}_1 \right)^6 \circ E_2^{(2)},
    \end{split}
\end{equation}
where we denoted by $\underline{\mathcal{P}}^{(1)} = \underline{\mathcal{P}} \times 1$ and $\underline{\mathcal{P}}^{(2)} = 1 \times \underline{\mathcal{P}}$ the componentwise action of the map $\underline{\mathcal{P}}$ of \eqref{eq:underline_P}.
We can then compute its action as in \Cref{fig:leading path tableaux example 2}, from which we find
\[
    \mathcal{L}_{P,Q}(P,Q) = (T,T) \qquad \text{with} \qquad 
    T = 
    \begin{ytableau}
    	*(gray) & *(gray) & *(gray) & *(gray) & *(gray) & *(gray) & 1 & 1 &1
    	\\
    	*(gray) & *(gray) & *(gray) & *(gray) & 1
    	\\
    	*(gray) & 1
		\\
		1
	\end{ytableau}
    .
\]
The tableau $T$ can be parameterized by the pair $(B,\nu) \in \BBS \times \mathbb{Y}$ as
\[
    \Xi(T) = (B,\nu)
    \qquad
    \text{with}
    \qquad
    B = \,\, 
    \begin{matrix}
        (\, 1 & 1 & \grc{0} & \grc{0} & 1 & \grc{0} & 1 & 1 & 1 & \grc{0} &\grc{\cdots} \, )
    \end{matrix}
    \,\,
    ,
    \qquad
    \nu = 
    \ydiagram{1}
    \,
    .
\]
We can finally compute the action of the KKR bijection on the BBS configuration $B$ as
\[
    \KKR (B) = (\mu,J) 
    \qquad
    \text{with}
    \qquad
    \mu = \ydiagram{3,2,1}
    ,
    \qquad
    J = (-3,-2,1).
\]
Then, setting $\kappa = J + \mu = (0,0,2)$, we conclude that, for the pair $(P,Q)$ of \eqref{eq:pair_example_2}, we have
\[
    \Upsilon_{\mathrm{col}} (P,Q) = \left(
    H_1=
    \begin{ytableau}
    	1 & 3 & 3
    	\\
        2 & 2
    	\\
    	2
	\end{ytableau}
    ,
    H_2=
    \begin{ytableau}
    	1 & 1 & 3
    	\\
        2 & 2
    	\\
    	2
	\end{ytableau} ; \kappa = (0,0,2), \nu = \ydiagram{1} \,
    \right).
\]
Notice that in this case we have $\kappa \in \mathcal{K}_+(\mu)$ and $\nu \in \mathbb{Y}$ as a result of the fact that the tableaux $P,Q$ have partition shape.

\begin{figure}
    \centering
    \includegraphics[width=\linewidth]{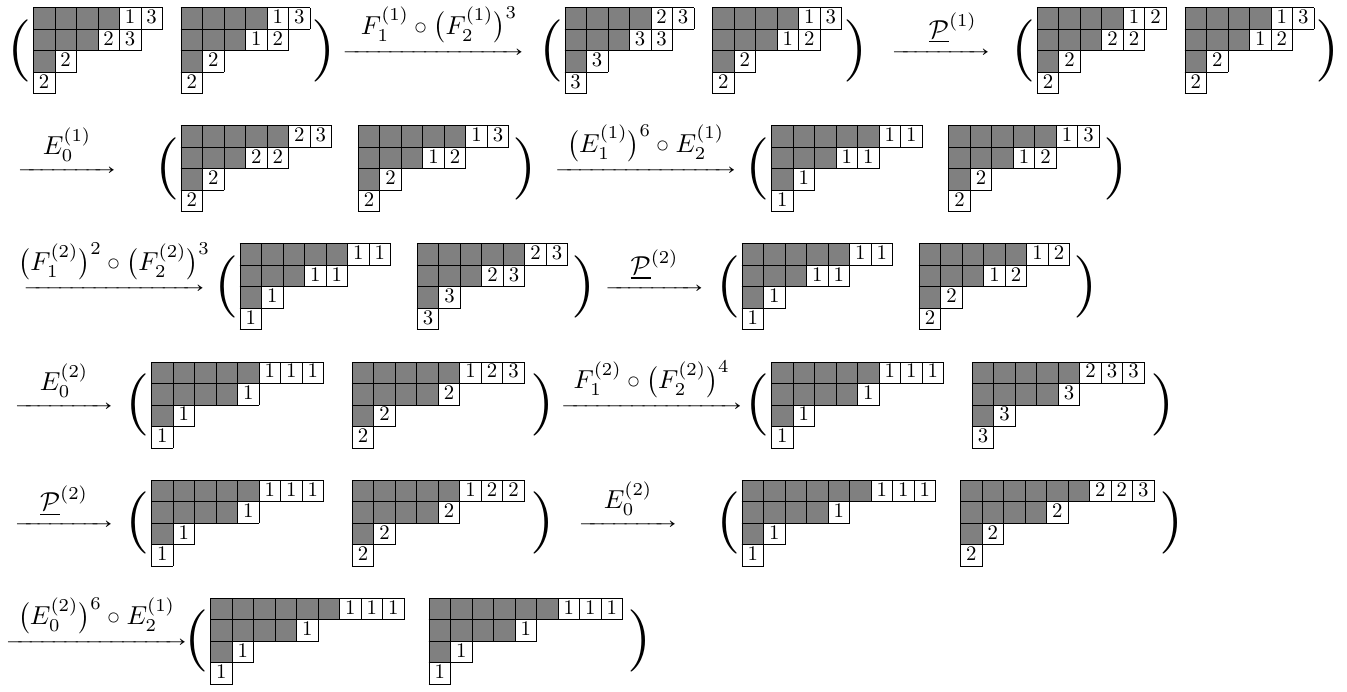}
    \caption{The evaluation of the leading path $\mathcal{L}_{P,Q}$ \eqref{eq:explicit_leading_path_example_2} on the pair $(P,Q)$ of \eqref{eq:pair_example_2}.}
    \label{fig:leading path tableaux example 2}
\end{figure}

\section{Generalized Greene's invariants} \label{sec:Greene}

In this section we prove \Cref{thm:LPP} which gives a last passage percolation characterization of the asymptotic shape $\mu$ of the cRSK dynamics.

\subsection{The skew row RSK dynamics and the skew column RSK dynamics} \label{subs:row_RSK_dynamics}

Let us describe a basic relation between the cRSK dynamics and the \emph{skew (row) RSK dynamics} introduced in \cite{imamura2021skewRSK}. We will only discuss the case of standard Young tableaux, as this is all we will need.

Given a signature $\lambda \in \mathbb{S}_\ell$ of length $\ell$ recall that its Young diagram is the set of cells $(i,j)\in \Z \times \Z_+$ such that $i \le \lambda_j$ and $j\in\{1,\dotsc,\ell\}$.
We define its \defn{transposed diagram} $\lambda^{\mathsf{t}}$ as the set of cells $(i,j)\in \Z_+ \times \Z$ such that $(j,i)$ belongs to the Young diagram of $\lambda$, i.e.\ $j \le \lambda_i$ and $i\in\{1,\dotsc,\ell\}$.
In other words, the transposed diagram $\lambda^\mathsf{t}$ is obtained from $\lambda$ by a reflection across the main diagonal of $\Z^2$.
When $\lambda$ is a partition, the notion of transposed $\lambda^\mathsf{t}$ coincides with the conjugate diagram $\lambda'$.
 
We have defined a standard Young tableau $P$ of shape $\lambda/\rho$ as a sequence of signatures $P=(\rho = \lambda^{(0)} \preceq \cdots \preceq \lambda^{(n)} = \lambda)$ such that $|\lambda^{(k)}/\lambda^{(k-1)}|\in \{0,1\}$ for all $k=1,\dots,n$. We define the \defn{transposed standard tableau} $P^\mathsf{t}$ as the sequence of transposed diagrams 
\[
    P^\mathsf{t}= \left( \rho^\mathsf{t} = \left(\lambda^{(0)} \right)^\mathsf{t} \dot{\preceq} \cdots \dot{\preceq} \left(\lambda^{(N)} \right)^\mathsf{t} = \lambda^\mathsf{t}  \right).
\]
(Here the notation $\eta^\mathsf{t} \dot{\preceq} \, \nu^\mathsf{t}$ means that $\eta \preceq \nu$.)
We can also represent the transposed standard tableau $P^\mathsf{t}$  by assigning label $k$ to the cell $(i,j) \in \Z_+ \times \Z$ if $\lambda^{(k)} / \lambda^{(k-1)} = \{(j,i)\} \in \Z \times \Z_+$. In other words, the transposed tableau $P^\mathsf{t}$ is obtained from $P$ through a reflection across the main diagonal of $\Z^2$.

Given a \defn{fully refined} Fomin field $\boldsymbol{\lambda}$, let $(P_t,Q_t)_{t \in \Z}$ be the associated cRSK dynamics with initial data $(P_0,Q_0)$ as in \Cref{def:cRSK}.
The \defn{skew row RSK dynamics} associated to the field $\boldsymbol{\lambda}$ is the sequence of pairs of transposed standard tableaux $(P_t^\mathsf{t}, Q_t^\mathsf{t})_{t \in \Z}$.
In \Cref{subs:alternative_description_cRSK} we presented an alternative construction of the cRSK dynamics through a sequence of column Schensted insertions.
An analogous description of the skew row RSK dynamics can be given replacing column insertions by row insertion. 

\begin{example} \label{ex:row RSK dynamics}
    The fully refined Fomin field of signatures $\boldsymbol{\lambda}$ of \Cref{fig:refinement_field} is associated to the skew row RSK dynamics $(P_t^\mathsf{t}, Q_t^\mathsf{t})$ which at times $t=-2,0,2$ take values
    \[
        \includegraphics[width=.9\linewidth]{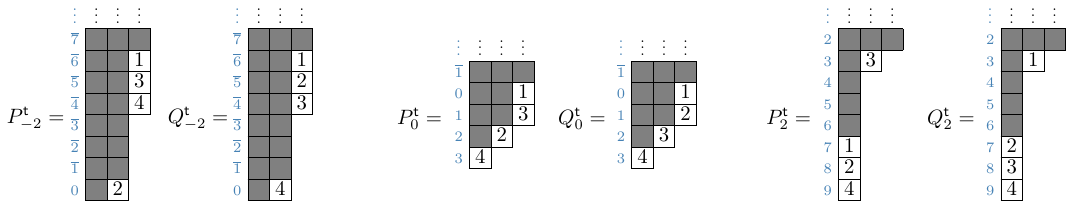}.
    \]
    Notice that tableaux $(P_t^\mathsf{t}, Q_t^\mathsf{t})$ are simply the transposition of tableaux $(P_t, Q_t)$ given in \Cref{fig:RSK dynamics example}, after standardization.
\end{example}

Just as in the case of the skew (column) RSK dynamics, we can define soliton data of skew row RSK dynamics $(P_t^\mathsf{t}, Q_t^\mathsf{t})$.
We define the \defn{asymptotic pair of vertically strict (standard) taubleax} $(V_1,V_2)$ associated to the skew row RSK dynamics $(P_t^\mathsf{t}, Q_t^\mathsf{t})$, or equivalently to the Fomin field $\boldsymbol{\lambda}$, as
\[
    V_1 = H_1^\mathsf{t}, \qquad\qquad V_2 = H_2^\mathsf{t},
\]
where $(H_1,H_2)$ is the asymptotic pair of horizontally strict tableaux associated to the Fomin field $\boldsymbol{\lambda}$. (Here, as above, the transposed of a horizontally weak tableau $H$ is the tableau $V=H^\mathsf{t}$ which has at cell $(i,j)$ the value $H$ has at cell $(j,i)$.) For example, the asymptotic pair of vertically strict tableaux associated to the dynamics of \Cref{ex:row RSK dynamics} is  
\[
    V_1 =
    \begin{ytableau}
        1 & 3
        \\
        2
        \\
        4
    \end{ytableau}
    \,
    ,
    \qquad\qquad
    \, 
    V_2=
    \begin{ytableau}
        2 & 1
        \\
        3
        \\
        4
    \end{ytableau}
    .
\]

Let us now cite a result from \cite{imamura2021skewRSK}, which will be useful to us.

\begin{proposition}[Special case of \cite{imamura2021skewRSK}, Theorem 6.6] \label{prop:Generalized_Greene_for_skew_row_RSK}
    Let $\boldsymbol{\lambda}$ be a fully refined Fomin field on $\mathscr{C}_n$ and let $\mathbf{c} \in \mathfrak{F}_{\mathrm{fin}}(\mathscr{C}_n',\{0,1\})$ be its environment; recall the notation \eqref{eq:set_environments}. Let $(P_t^\mathsf{t},Q_t^\mathsf{t})_{t \in \Z}$ be the skew row RSK dynamics associated to $\boldsymbol{\lambda}$ and let $(V_1,V_2)$ be the asymptotic pair of vertically strict tableau of shape $\mu$. Then, we have
    \[
        \mu_1' +\cdots +\mu_k' = \mathrm{LPP}^{\circlearrowleft}_k(\mathbf{c})
        \qquad
        \text{and}
        \qquad
        \mu_1+\dots+\mu_k = \mathrm{LPP}^{\nearrow}_k(\mathbf{c}).
    \]
\end{proposition}

\begin{remark}
    The result of \Cref{prop:Generalized_Greene_for_skew_row_RSK} cannot be naively generalized to Fomin fields which are not fully refined. The generalization to such cases involves changing the definition of the statistics $\mathrm{LPP}^{\nearrow}_k$ and $\mathrm{LPP}^{\circlearrowleft}_k$; see \cite{imamura2021skewRSK}, Theorem 6.6 for more details.
\end{remark}

\subsection{Generalized Greene's invariants for the skew column RSK dynamics} \label{subs:LPP proof}

In this subsection we prove \Cref{thm:LPP}. Let us first present the proof of items \ref{item:weight_environment}, \ref{item:volume_environment}, \ref{item:bij environment HWT restriction}.

\begin{proof}[Proof of \Cref{thm:LPP}  \ref{item:weight_environment}, \ref{item:volume_environment}, \ref{item:bij environment HWT restriction}]
    By the periodic column RSK correspondence of \Cref{prop:column RSK}, any environment $\mathbf{c} \in 
    \mathfrak{F}_{\mathrm{fin}}(\mathscr{C}_n',\N)$ is in bijection with a pair $(P,Q)$ with $\ker(P,Q) = \varnothing$. Then, we define $\widetilde{\Upsilon}_{\mathrm{col}}(\mathbf{c}) = (H_1,H_2,\kappa)$, where  $(H_1,H_2,\kappa,\nu = \varnothing) = \Upsilon_{\mathrm{col}}(P,Q)$. Let $\boldsymbol{\lambda} = \boldsymbol{\lambda}[P,Q]$ be the Fomin field of interlacing signatures on $\mathscr{C}_n$ associated to the pair $(P,Q)$ and denote by $\mu$ the shape of $H_1,H_2$.


    The restriction property \ref{item:bij environment HWT restriction} is immediate from the definition of $\widetilde{\Upsilon}_{\mathrm{col}}$. In fact, by \Cref{th:main_bij}, $\kappa \in \mathcal{K}_+(\mu)$ if and only if the tableaux $P,Q$ have partition shape. This happens if and only if for each $i,j \in \{0,\dots, n-1\}$ and $k\ge 0$ the signature $\boldsymbol{\lambda}(i,j+n k)$ is a partition. Then, by \Cref{prop:conservation_Fomin} and \eqref{eq:value_c_0} we have
    \begin{equation}
        \begin{split}
        &\mathbf{c}(i+\frac{1}{2},j+\frac{1}{2}+k n) 
        \\
        &
        = \sum_{s \ge 1} \boldsymbol{\lambda}_s'(i+1,j+1+k n) - \boldsymbol{\lambda}_s'(i,j+1+k n) - \boldsymbol{\lambda}_s'(i+1,j+k n) + \boldsymbol{\lambda}_s'(i,j+kn) 
        \\
        &
        = |\boldsymbol{\lambda}(i+1,j+1+k n) / \boldsymbol{\lambda}(i,j+1+k n)| - |\boldsymbol{\lambda}(i+1,j+k n) / \boldsymbol{\lambda}(i,j+k n)|
        = 0,
    \end{split}
    \end{equation}
    for each $i,j \in \{0,\dots, n-1\}$ and $k\ge 0$. Then, $\kappa \in \mathcal{K}_+(\mu)$ implies that $\mathbf{c}$ is supported in $\mathscr{C}_n^-$. The same argument also shows that if $\mathbf{c}$ is supported in $\mathscr{C}_n^-$ we also necessarily have $\kappa \in \mathcal{K}_+(\mu)$.

    
    The property \eqref{eq:weight_environment} relating the weights of tableaux $H_1,H_2$ with the environment $\mathbf{c}$ follows from relation \eqref{eq:sum environment} and from the definition of $H_1,H_2$ as the limits \eqref{eq:H_1 limit}, \eqref{eq:H_1 limit}. 

    To prove \eqref{eq:volume_environment}, let $K_*>0$ be large enough so that $\mathbf{c}(i',j' + Kn) = \mathbf{c}(i',j' - Kn) = 0$ for all $i',j' \in \{ \frac{1}{2},\dots,n-\frac{1}{2} \}$ and $K\ge K_*$. This implies that $\boldsymbol{\lambda}(i,j + Kn)$ (resp. $\boldsymbol{\lambda}(i,j - Kn)$) is a signature with only negative elements for all $K\ge K_*$. Then, we have
    \begin{equation}
        \begin{split}
            |\boldsymbol{\lambda}(0,K_* n)|
            &
            =
            \sum_{s \ge 1} \boldsymbol{\lambda}_s'(0,K_* n)
            \\
            &
            = \sum_{s \ge 1} \left[ \boldsymbol{\lambda}_s'(0,K_* n) - 2 \boldsymbol{\lambda}_s'(0,(K_* -1)-n) + \boldsymbol{\lambda}_s'(-n,(K_* -1)n) \right]
            \\
            &
            \hspace{20pt}
            + \sum_{s \ge 1} \left[ 2 \boldsymbol{\lambda}_s'(0,(K_* -1)n) - \boldsymbol{\lambda}_s'(-n,(K_* -1)n) \right]
            \\
            &
            = \sum_{i',j'\in\{ \frac{1}{2} ,\dots,n-\frac{1}{2} \} } \mathbf{c}(i',j'+ (K_* -1) n) 
            \\
            &
            \hspace{20pt} + \sum_{s \ge 1} \left[ 2 \boldsymbol{\lambda}_s'(0,(K_* -1)n) - \boldsymbol{\lambda}_s'(-n,(K_* -1)n) \right],
        \end{split}
    \end{equation}
    where in the last equality we used \eqref{eq:value_c_0}. Iterating the above relation we find
    \begin{equation} \label{eq:volume_environment_iter}
        \begin{split}
            |\boldsymbol{\lambda}(0,K_* n)|
            &
            = \sum_{k=1}^{2K_*} \sum_{i',j'\in\{ \frac{1}{2} ,\dots,n-\frac{1}{2} \} } k \,\mathbf{c}(i',j'+ (K_* -k) n) 
            \\
            &
            \hspace{20pt} + \sum_{s \ge 1} \left[ 2 \boldsymbol{\lambda}_s'(0,-K_* n) - \boldsymbol{\lambda}_s'(-n,-K_* n) \right],
            \\
            &
            = \sum_{k\in \mathbb{Z}} \sum_{i',j'\in\{ \frac{1}{2} ,\dots,n-\frac{1}{2} \} } k \,\mathbf{c}(i',j'+ (K_* -k) n)
            \\
            &
            = \sum_{k\in \mathbb{Z}} \sum_{i',j'\in\{ \frac{1}{2} ,\dots,n-\frac{1}{2} \} } k' \,\mathbf{c}(i',j' -k' n) - K_* \sum_{k\in \mathbb{Z}} \sum_{i',j'\in\{ \frac{1}{2} ,\dotsc,n-\frac{1}{2} \} } \,\mathbf{c}(i',j' -k' n),
        \end{split}
    \end{equation}
    where in the second equality, we extended the sum over $k \in \mathbb{Z}$, using the fact that the environment $\mathbf{c}$ is zero for these terms, and we suppressed the sum $ \sum_{s \ge 1} \left[ 2 \boldsymbol{\lambda}_s'(0,-K_* n) - \boldsymbol{\lambda}_s'(-n,-K_* n) \right]$, since every of its summands is zero. In the last equality we simply operated the change of variables $k' = k-K_*$. From \eqref{eq:weight_environment} and the fact that $|\boldsymbol{\lambda}(0,K n) | - |\boldsymbol{\lambda}(0,(K-1) n)|$ for any $K \in \mathbb{Z}$, we have 
    \[
        |\mu| = \sum_{k\in \mathbb{Z}} \sum_{i',j'\in\{ \frac{1}{2} ,\dots,n-\frac{1}{2} \} } \,\mathbf{c}(i',j' -k' n) = \abs{\boldsymbol{\lambda}(0,K n)} - \abs{\boldsymbol{\lambda}(0,(K-1) n)},
    \]
    and plugging the above relation in \eqref{eq:volume_environment_iter} we obtain
    \[
        \sum_{k\in \mathbb{Z}} \sum_{i',j'\in\{ \frac{1}{2} ,\dots,n-\frac{1}{2} \} } k' \,\mathbf{c}(i',j' -k' n) = \boldsymbol{\lambda}(0,0) .
    \]
    Then, since $\boldsymbol{\lambda}(0,0) = \rho$ is the empty shape of the tableaux $P,Q$, this proves \eqref{eq:volume_environment} using \eqref{eq:volume_pre}.

    \end{proof}

    From \Cref{prop:Generalized_Greene_for_skew_row_RSK} we can easily prove \Cref{thm:LPP}.

    \begin{proof}[Proof of \Cref{thm:LPP} \ref{item:LPP}]

    Adopt the notation used above in the subsection. Define also $\widetilde{\boldsymbol{\lambda}}$ the refinement of the field $\boldsymbol{\lambda}$, let $(\widetilde{P}_t,\widetilde{Q}_t)_{t \in \Z}$ be the associated cRSK dynamics and let $\widetilde{\mathbf{c}} \in \mathfrak{F}_{\mathrm{fin}}(\mathscr{C}_N',\{0,1\})$ be the corresponding environment, where $N=|\boldsymbol{\lambda}(n,0)/\boldsymbol{\lambda}(0,0)|$.
    
    Consider $(\widetilde{P}^\mathsf{t}_t, \widetilde{Q}^\mathsf{t}_t)_{t \in \Z}$ the skew row RSK dynamics associated to the field $\widetilde{\boldsymbol{\lambda}}$ and its soliton data consisting of a pair of vertically strict tableaux $(V_1,V_2)$ of shape $\mu'$. Then \Cref{prop:Generalized_Greene_for_skew_row_RSK} implies that
    \[
        \mu_1 +\cdots +\mu_k = \mathrm{LPP}^{\circlearrowleft}_k(\widetilde{\mathbf{c}})
        \qquad
        \text{and}
        \qquad
        \mu_1'+\dots+\mu_k' = \mathrm{LPP}^{\nearrow}_k(\widetilde{\mathbf{c}}),
    \]
    for all $k \ge 1$.
    To complete the proof we need to observe that, for all $k\ge 1$, we have 
    \[
        \mathrm{LPP}^{\circlearrowleft}_k(\widetilde{\mathbf{c}})= 
        \mathrm{LPP}^{\circlearrowleft}_k(\mathbf{c})
        \qquad
        \text{and}
        \qquad
        \mathrm{LPP}^{\nearrow}_k(\widetilde{\mathbf{c}}) = \mathrm{LPP}^{\nearrow}_k(\mathbf{c}).
    \]
    This is relatively straightforward from the construction of $\widetilde{\mathbf{c}}$ from $\mathbf{c}$. Below we only show the identity $\mathrm{LPP}^{\circlearrowleft}_k(\widetilde{\mathbf{c}})= \mathrm{LPP}^{\circlearrowleft}_k(\mathbf{c})$ as the remaining one can be proved analogously. 

    Recall, from \Cref{prop:standardization_field}, the embedding $\phi: \mathscr{C}_n \to \mathscr{C}_N$ and for each $p'\in \mathscr{C}_n'$ define the rectangle $\mathscr{R}(p') \subset \mathscr{C}_N$ of all faces in the rectangular sublattice of $\mathscr{C}_N$ with vertices $\phi( p' \pm \frac{1}{2} \mathbf{e}_1 \pm \frac{1}{2} \mathbf{e}_2 )$. Denote also by $\mathscr{D}(p')$ the set of all faces $\widetilde{p}\,' \subset \mathscr{R}(p')$ such that $\widetilde{\mathbf{c}} (\widetilde{p} \, ') = 1$. Then, $\mathscr{D}(p')$ consists of $\mathbf{c}(p')$ faces arranged along an anti-diagonal of $\mathscr{R}(p')$ as described in \Cref{prop:refinement_environment} and \Cref{fig:refinement_Fomin_cell}.
    
    Let $\mathbf{p}^{(1)},\dots,\mathbf{p}^{(k)}$ be $k$ non-intersecting simple loops in $\mathscr{C}_n'$. Then, we can construct $k$ non-intersecting simple loops $\widetilde{\mathbf{p}}^{(1)},\dots,\widetilde{\mathbf{p}}^{(k)}$ in $\mathscr{C}_N'$ such that 
    \[
        \sum_{i=1}^k \sum_{j = 1}^{2n+1} \mathbf{c}(\mathbf{p}_j^{(i)}) = \sum_{i=1}^k \sum_{j = 1}^{2n+1} \widetilde{\mathbf{c}}(\widetilde{\mathbf{p}}_j^{(i)}),
    \]
    by first replacing any node $\mathbf{p}_j^{(i)} \in \mathscr{C}_n'$ by a sequence of nodes crossing all faces of $\mathscr{D}(\mathbf{p}_j^{(i)})$ and later connecting these nodes with down-right paths. An example of such construction is given below
    \[
        \includegraphics[width=.4\linewidth]{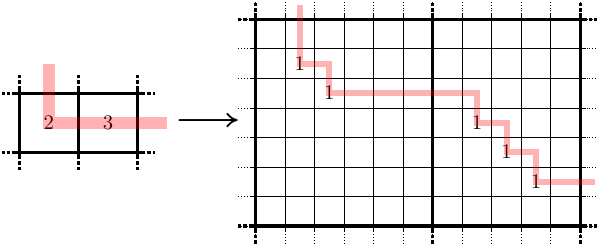}.
    \]
    Such procedure is possible since the set of faces $\mathscr{D}(\mathbf{p}_{j+1}^{(i)})$ lies strictly down and to the right of $\mathscr{D}(\mathbf{p}_{j}^{(i)})$ for all $j$. This implies that 
    \[
        \mathrm{LPP}^{\circlearrowleft}_k(\widetilde{\mathbf{c}}) \ge  
        \mathrm{LPP}^{\circlearrowleft}_k(\mathbf{c}).
    \]
    
    For the opposite inequality consider a $k$-tuple of non-intersecting down-right loops $\widetilde{\boldsymbol{q}}^{(1)}, \dots, \widetilde{\boldsymbol{q}}^{(k)}$ of $\mathscr{C}_N'$. We start by re-routing $\widetilde{\mathbf{q}}^{(1)}$ as follows: if $\widetilde{\mathbf{q}}^{(1)}$ intersects $\mathscr{R}(p')$ for some $p' \in \mathscr{C}_n'$, then we deform it so that contains all cells of $\mathscr{D}(p')$: such deformation can be obtained through a finite sequence of "flips" $(\widetilde{p}' , \widetilde{p}' + \mathbf{e}_1 , \widetilde{p}' + \mathbf{e}_1 -  \mathbf{e}_2) \leftrightarrow (\widetilde{p}' , \widetilde{p}' - \mathbf{e}_2 , \widetilde{p}' + \mathbf{e}_1 -  \mathbf{e}_2)$ of right-down turn into down-right turn (or viceversa) of $\widetilde{\mathbf{q}}^{(1)}$. If during such deformation the loop $\widetilde{\mathbf{q}}^{(1)}$ crosses other loops $\widetilde{\mathbf{q}}^{(i)}$, we push the latter in the direction suggested by the flips, so to avoid intersections. After completing such procedure we obtain a new $k$-tuple of non-intersecting down-right loops, which we call $\widetilde{\mathbf{q}}^{(1,1)},\widetilde{\mathbf{q}}^{(2,1)},\dots,\widetilde{\mathbf{q}}^{(k,1)}$ on $\mathscr{C}_N'$, with the property that
    \[
        \sum_{i=1}^k \sum_{j=1}^{2N+1} \widetilde{\mathbf{c}}(\widetilde{\mathbf{q}}^{(i,1)}_j) \ge \sum_{i=1}^k \sum_{j=1}^{2N+1} \widetilde{\mathbf{c}}(\widetilde{\mathbf{q}}^{(i)}_j).
    \]
    We then sequentially proceed to generate $k$-tuple of loops $\widetilde{\mathbf{q}}^{(1,\ell)},\widetilde{\mathbf{q}}^{(2,\ell)},\dots,\widetilde{\mathbf{q}}^{(k,\ell)}$ for $\ell=1,\dots,k$, where at each $\ell$
    we re-route the paths $\widetilde{\mathbf{q}}^{(j,\ell)}$ for $j>\ell$ with the same strategy described in the case $\ell = 1$ and the additional condition that if $\widetilde{\mathbf{q}}^{(j,\ell)}$ crosses some rectangle $\mathscr{R}(p')$ which is already crossed by $\widetilde{\mathbf{q}}^{(i,\ell)}$ with $i \le \ell$ (which implies that the $\widetilde{\mathbf{q}}^{(i,\ell)}$ contains all faces of $\mathscr{D}(p')$), then we deform $\widetilde{\mathbf{q}}^{(j,\ell)}$ so that it does not intersect the rectangle $\mathscr{R}(p')$. As in the $\ell = 1$ case we still have the property
    \[
        \sum_{i=1}^k \sum_{j=1}^{2N+1} \widetilde{\mathbf{c}}(\widetilde{\mathbf{q}}^{(i,\ell)}_j) \ge \sum_{i=1}^k \sum_{j=1}^{2N+1} \widetilde{\mathbf{c}}(\widetilde{\mathbf{q}}^{(i,\ell-1)}_j).
    \]

    By the end of this procedure we will be left with a $k$-tuple $\widetilde{\mathbf{p}}^{(i)} = \widetilde{\mathbf{q}}^{(i,k)}$ for $i=1,\dotsc,k$.
    By construction, none of the rectangles $\mathscr{R}(p')$ is crossed by more that one loop $\widetilde{\mathbf{p}}^{(i)}$ and whenever a loop $\widetilde{\mathbf{p}}^{(i)}$ crossed $\mathscr{R}(p')$ it also crosses the faces $\mathscr{D}(p')$.
    As a result we can associate to the $k$-tuple of loops $\widetilde{\mathbf{p}}^{(1)}, \dotsc, \widetilde{\mathbf{p}}^{(k)}$ on $\mathscr{C}_N'$ a $k$-tuple of non-intersecting loops $\mathbf{p}^{(1)},\dots,\mathbf{p}^{(k)}$ on $\mathscr{C}_n'$ with the property that
    \[
        \sum_{i=1}^k \sum_{j=1}^{2N+1} \widetilde{\mathbf{c}}(\widetilde{\mathbf{p}}^{(i)}_j) = \sum_{i=1}^k \sum_{j=1}^{2N+1} \mathbf{c}(\mathbf{p}^{(i)}_j).
    \]
    As a result we have
    \[
        \mathrm{LPP}^{\circlearrowleft}_k(\widetilde{\mathbf{c}}) \le \mathrm{LPP}^{\circlearrowleft}_k(\mathbf{c}),
    \]
    which completes the proof.
\end{proof}

\section{Symmetric environments and symmetric fields}

In this section we analyze special cases of the bijections of \Cref{th:main_bij} and \Cref{thm:LPP} under symmetric inputs. For this we define the transposition of an environment $\mathbf{c}$ as $\mathbf{c}^\mathsf{t}(i,j) = \mathbf{c}(j,i)$. Similar to \eqref{eq:set_environments}, for any set $B \subseteq \mathbb{Z}$ such that $0 \in B$, we define the set of symmetric environments 
\[
    \mathfrak{F}_{\mathrm{fin}}^{\mathrm{sym}}(\mathscr{C}_n',B) = \{\mathbf{c} \in \mathfrak{F}_{\mathrm{fin}}(\mathscr{C}_n',B) \mid \mathbf{c}^\mathsf{t} = \mathbf{c} \}.
\]

\begin{corollary} \label{cor:symmetric_column_RSK}
    Restricting the bijection $\Psi$ of \Cref{prop:column RSK} to symmetric inputs $(P,P)$ yields a bijection
    \begin{equation}
        \Psi\colon P \in \bigsqcup_{\lambda,\rho \in \mathbb{S}} \mathrm{SST} ( \lambda/\rho , n) \longrightarrow (\mathbf{c},\nu) \in \mathfrak{F}_{\mathrm{fin}}^{\mathrm{sym}}(\mathscr{C}_n',\N) \times \mathbb{S}.
    \end{equation}
\end{corollary}

The result of \Cref{cor:symmetric_column_RSK} is a variant of \cite[Corollary 4.4]{sagan1990robinson}. The fact that for any tableau $P$ the associated environment $\mathbf{c}$ such that $\Psi(P) = (\mathbf{c},\nu)$ is symmetric is consequence of the symmetry of the Fomin local rule.

For the next theorem recall that for any array $\alpha \in \mathbb{Z}^\ell$, $\mathrm{odd}(\alpha)$ is the number of its odd entries.

\begin{theorem} \label{cor:symmetric_Upsilon}
    Restricting the bijection $\Upsilon_{\mathrm{col}}$ of \Cref{th:main_bij} to symmetric inputs $(P,P)$ yields a bijection
    \begin{equation}
        \Upsilon_{\mathrm{col}} \colon P \in \bigsqcup_{\lambda,\rho \in \mathbb{S}} \mathrm{SST} ( \lambda/\rho , n)
        \longrightarrow (H,\kappa,\nu) \in
        \bigsqcup_{\mu\in\partitions} \HWT(\mu,n) \times \mathcal{K}(\mu) \times \mathbb{S}.
    \end{equation}
    The shapes $\rho, \lambda$ are integer partitions if and only if $\nu$ is also a partition and $\kappa \in \mathcal{K}_+(\mu)$.
    Moreover, we have
    \begin{equation} \label{eq:odd}
        \mathrm{odd}(\lambda) = \mathrm{odd}(\nu) + \mathrm{odd}(\kappa + \mu)
        \qquad
        \text{and}
        \qquad
        \mathrm{odd}(\rho) = \mathrm{odd}(\nu) + \mathrm{odd}(\kappa).
    \end{equation}
\end{theorem}

\begin{proof}
    The cRSK dynamics with initial data $(P,P)$ is a sequence of pairs of identical tableaux $(P_t,P_t)_{t \in \Z}$ as a consequence of the symmetry of the Fomin local rules. As a result the associated soliton data also consists of a pair of identical horizontally weak tableaux $\Phi(P,P) = (H,H)$. This proves the first part of the theorem. 

    Relations \eqref{eq:odd} are less trivial. Notice that we can assume that $\lambda,\rho$ are partitions, otherwise we would consider the tableau $\widetilde{P}$ obtained shifting each row of $P$ by $2m$ cells to the right with $2m \ge \rho_\ell$. The image under $\Upsilon_{\mathrm{col}}$ of $\widetilde{P}$ can be checked to be the triple $(H, \widetilde{\kappa} , \widetilde{\nu})$ with $\widetilde{\kappa}_i = \kappa + 2m$ and $\widetilde{\nu}_i = \nu_i + 2m$ and such even shifts leave relations \eqref{eq:odd} unchanged.
    Moreover, since relations \eqref{eq:odd} only concern the shape of the tableau $P$ and $H$ and the value of $\kappa$, we can further assume that $P$ is a standard tableau. The transposed tableau $P^\mathsf{t}$ can be put in bijection with the triple $(H^\mathsf{t}, \kappa, \nu')$, where $(H^\mathsf{t},H^\mathsf{t})$ is the asymptotic pair of vertically strict tableaux associated to the skew row RSK dynamics $(P_t^\mathsf{t},P_t^\mathsf{t})_{t \in \Z}$ with initial condition $(P^\mathsf{t}, P^\mathsf{t})$. Such bijection is a particular case of the one discovered in \cite{imamura2021skewRSK} and for that we can invoke \cite[Lemmas 3.2, 3.6]{He2023boundaryASEP}, which claims that
    \[
        \mathrm{odd}(\lambda') = \mathrm{odd}(\nu') + \mathrm{odd}(\mu' + \kappa), 
        \qquad
        \mathrm{odd}(\rho') = \mathrm{odd}(\nu') + \mathrm{odd}(\kappa),
    \]
    where $\mu'$ is the shape of $H^\mathsf{t}$. Transposing partitions $\lambda,\rho,\mu,\nu$ in the above relation proves \eqref{eq:odd}.
\end{proof}

For any environmnent $\mathbf{c} \in \mathfrak{F}_{\mathrm{fin}}(\mathscr{C}_n',\N)$ define
\[
    \mathrm{oddTr}(\mathbf{c}) = \sum_{k \ge 0} \sum_{i=1}^n \mathbf{1}_{\mathbf{c}(i+kn,i) \text{ is odd}}
    \qquad
    \text{and}
    \qquad
    \mathrm{oddTr}^-(\mathbf{c}) = \sum_{k \ge 0} (-1)^k \sum_{i=1}^n \mathbf{1}_{\mathbf{c}(i+kn,i)\text{ is odd}}.
\]

\begin{theorem} \label{cor:symmetric_env}
    Restricting the bijection $\widetilde{\Upsilon}_{\mathrm{col}}$ of \Cref{thm:LPP} to symmetric environments yields a bijection
    \begin{equation}
        \widetilde{\Upsilon}_{\mathrm{col}} \colon \mathbf{c} \in \mathfrak{F}_{\mathrm{fin}}^{\mathrm{sym}}(\mathscr{C}_n',\N) \longrightarrow ( H,\kappa) \in \bigsqcup_{\mu\in\partitions} \HWT(\mu,n) \times \mathcal{K}(\mu).
    \end{equation}
    The environment $\mathbf{c}$ is supported in $\mathscr{C}_n'^-$ if and only if $\kappa \in \mathcal{K}_+(\mu)$.
    Moreover, we have
    \begin{equation} \label{eq:odd_Trace}
        \mathrm{oddTr}(\mathbf{c}) = \mathrm{odd}(\kappa + \mu) +  \mathrm{odd}(\kappa)
        \qquad
        \text{and}
        \qquad
        \mathrm{oddTr}^-(\mathbf{c}) = \mathrm{odd}(\kappa + \mu) -  \mathrm{odd}(\kappa).
    \end{equation}
\end{theorem}

\begin{proof}
    The only non immediate statements are those in \eqref{eq:odd_Trace} and as in the proof of \Cref{cor:symmetric_Upsilon}, we will borrow results from \cite{imamura2021skewRSK,He2023boundaryASEP}.
    Let $\boldsymbol{\lambda} \colon \mathscr{C}_n \to \mathbb{S}$ be a symmetric Fomin field of signatures associated to the environment $\mathbf{c}$ and let $\widetilde{\boldsymbol{\lambda}} \colon \mathscr{C}_N \to \mathbb{S}$ be its refinement for $N \ge n$ chosen appropriately. Let $\widetilde{\mathbf{c}} \in \mathfrak{F}_{\mathrm{fin}}^\mathrm{sym}(\mathscr{C}_N',\{0,1\})$ be the environment associated to the field $\widetilde{\boldsymbol{\lambda}}$. The fields $\mathbf{c}$ and $\widetilde{\mathbf{c}}$ are related by \Cref{prop:refinement_environment} and \Cref{prop:standardization_field}. Then, defining
    \[
        \mathrm{Tr}(\mathbf{c}) = \sum_{k \ge 0} \sum_{i=1}^n \mathbf{c}(i+kn,i)
        \qquad
        \text{and}
        \qquad
        \mathrm{Tr}^-(\mathbf{c}) = \sum_{k \ge 0} (-1)^k \sum_{i=1}^n \mathbf{c}(i+kn,i),
    \]
    we notice that
    \begin{equation} \label{eq:equality_oddTr}
        \mathrm{oddTr}(\mathbf{c}) = \mathrm{oddTr}(\widetilde{\mathbf{c}}) = \mathrm{Tr}(\widetilde{\mathbf{c}})
        \qquad
        \text{and}
        \qquad
        \mathrm{oddTr}^-(\mathbf{c}) = \mathrm{oddTr}^-(\widetilde{\mathbf{c}}) = \mathrm{Tr}^-(\widetilde{\mathbf{c}}).
    \end{equation}
    The equalities $\mathrm{oddTr}(\widetilde{\mathbf{c}}) = \mathrm{Tr}(\widetilde{\mathbf{c}})$ and $\mathrm{oddTr}^-(\widetilde{\mathbf{c}}) = \mathrm{Tr}^-(\widetilde{\mathbf{c}})$ are straightforward, since $\widetilde{\mathbf{c}}$ only takes values in $\{0,1\}$. The equality $\mathrm{oddTr}(\mathbf{c}) = \mathrm{oddTr}(\widetilde{\mathbf{c}})$ follows from \Cref{prop:refinement_environment}. The refinement of the generic diagonal cell $(i,i)$, where the environment takes value $c=\mathbf{c}(i,i)$ produces a $N_i \times N_i$ square sublattice $\Lambda_{N_i,N_i} \subset \mathscr{C}_N$ for some $N_i$ where $\widetilde{\mathbf{c}}(p') = 1$ only if $p' = (A_i+k,A_i+c+1-k)$ for some $k \in \{1,\dots,c\}$ and where the number $A_i$ is determined through \Cref{prop:standardization_field} and \Cref{prop:refinement_environment}. It is clear that the set $\{ (A+k,A+c+1-k) \mid  1 \le k \le c\} \}$ admits a point on the diagonal only when $c$ is odd. Similarly one shows that $\mathrm{oddTr}^-(\mathbf{c}) = \mathrm{oddTr}^-(\widetilde{\mathbf{c}})$.
    
    Now consider the images $\widetilde{\Upsilon}_{\mathrm{col}}(\mathbf{c}) = (H, \kappa)$ and $\widetilde{\Upsilon}_{\mathrm{col}}(\widetilde{\mathbf{c}}) = (\widetilde{H}, \kappa)$. By construction $\widetilde{H}$ is the standardization of $H$ and both tableaux have the same shape $\mu$. Moreover the map $\widetilde{\Upsilon}_{\mathrm{row}}(\widetilde{\mathbf{c}}) = (\widetilde{H}^\mathsf{t}, \kappa)$, is a bijection between the set of fully refined environments $\widetilde{\mathbf{c}}$ and pairs $(\widetilde{H}^\mathsf{t},\kappa)$ where $\widetilde{H}^\mathsf{t}$ is a vertically strict standard tableau of shape $\mu'$ and $\kappa \in \mathcal{K}(\mu)$ which is a particular case of \cite[Corollary 8.2]{imamura2021skewRSK}. For the bijection $\widetilde{\Upsilon}_{\mathrm{row}}$ it was found in \cite[Lemma 3.6]{He2023boundaryASEP} that
    \[
        \mathrm{Tr}(\widetilde{\mathbf{c}}) = \mathrm{odd}(\kappa + (\mu')') +  \mathrm{odd}(\kappa)
        \qquad
        \text{and}
        \qquad
        \mathrm{Tr}^-(\widetilde{\mathbf{c}}) = \mathrm{odd}(\kappa + (\mu')') -  \mathrm{odd}(\kappa).
    \]
    In the above relations we kept the dependence on $\mu'$ as it is the shape of $\widetilde{H}^\mathsf{t}$.
    Since $(\mu')'= \mu$, combining these identities with relations \eqref{eq:equality_oddTr}, completes the proof of \eqref{eq:odd_Trace}.
\end{proof}

\section{Summation identities for transformed Hall--Littlewood polynomials} \label{sec:summation_id}

\subsection{Transformed Hall--Littlewood polynomials}

In this section we present bijective proofs of Cauchy, Littlewood identities and their refinements for transformed Hall--Littlewood polynomials.

\begin{definition} \label{def:transformed_Hall_Littlewood}
Fix parameters $q, x_1,\dots,x_n \in \mathbb{C}$ and a partition $\mu$.
The \defn{transformed Hall--Littlewood polynomial} is
\begin{equation}
    \THL_\mu (x;q) = \sum_{H \in \mathrm{HWT}(\mu,n)} q^{D'(H)} x^H.
\end{equation}
\end{definition}

There is another combinatorial description of the transformed Hall--Littlewood polynomials as the graded character of $\KR{\mu}$ (equivalently $\HWT(\mu, n)$).
Recall that $\wt_j(b)$ being the number of entries `$j$'s in $b$ and $D$ is the energy function defined in~\cref{def:energy_function}.

\begin{remark}
    The transformed Hall--Littlewood polynomials are symmetric with respect to the permutation of the $x$ variables. This is consequence of the fact that the action of classical Kashiwara operators induces an action of the symmetric group $S_n$ on $\HWT(\mu,n)$: see \cite[Thm.~7.2.2]{Kashiwara94}. Then, the relation $D'(H)=D(\eta(H))$ and an inspection of the involution $\eta$ of \eqref{eq:lusztig_involution}, implies that the polynomials $\THL_\mu (x;q)$ could be equivalently defined, replacing the dual energy $D'$ with $D$, as 
    \begin{equation}
        \THL_\mu (x;q) = \sum_{H \in \mathrm{HWT}(\mu,n)} q^{D(H)} x^H.
    \end{equation} 
    Such definition is more common in literature; see \cite[Thm.~2.10]{Kirillov-Schilling-Shimozono2002_RCbijection} and \cite{KKR86,KR86,Nakayashiki_Yamada}.
\end{remark}

\begin{remark}
    \label{rem:HL_conventions}
    Our convention for the transformed Hall--Littlewood polynomials $\THL_{\lambda}(x; q)$ match what were called the modified Hall--Littlewood polynomials in~\cite{Garbali_Wheeler_2020}.
    Indeed, these differ from what are usually called the modified Hall--Littlewood polynomials in the combinatorics literature (\textit{e.g.},~\cite{Haiman_Hilber_Schemes,haglund_haiman_loehr,Haglund08}), which are a specialization of the modified Macdonald polynomials $\widetilde{\mathscr{H}}_{\mu}(x; q, t)$ by $\THL_{\mu} = q^{\Abs{\mu}} \widetilde{\mathscr{H}}_{\mu}(x; 0, q^{-1})$, where $\Abs{\mu}$ is given in \eqref{eq:b_1}.

    The difference reflects our energy convention: the maximal weight tableau, whose entries are all $1$, has maximal rather than minimal energy as in e.g. ~\cite{HKOTT02,HKOTY02,HKOTY99,Inoue-Kuniba-Takagi2012BBS,Kirillov-Schilling-Shimozono2002_RCbijection,Kuniba-Sakamoto-Yamada2007_tau}. Thus, we have the specialization $\THL_{\mu}(1, 0, \ldots; q) = q^{\Abs{\mu}}$.
    In terms of Macdonald polynomials $P_\lambda(x;q,t),Q_\lambda(x;q,t)$, borrowing the notation of \cite{Macdonald1995}, we have
    \[
        \THL_{\mu}(x; q) = Q_{\mu}(x/(1-q); 0, q) = \omega P_{\mu'}(x; q,0),
    \]
    where the substitution $x/(1-q)$ is to be interpreted as a plethysm and $\omega$ is the Hall involution. The above equality relates the transformed Hall-Littlewood polynomials with the Hall-Littlewood polynomials $Q_\mu(x;0,q)$ and with the $q$-Whittaker polynomials $P_\mu(x;q,0)$.

\end{remark}

In the following subsections we will use the following coefficients
\begin{equation} \label{eq:b_mu}
    \mathpzc{b}_\mu(q) = \prod_{i \ge 1} \frac{1}{(q;q)_{m_i(\mu)}}
    \quad
    \text{and}
    \quad
    \mathpzc{b}_\mu(\alpha,\beta;q) = \prod_{i = 1,3,5,\dots } \frac{[\alpha+q\beta]_{q^2}^{m_i(\mu)}}{(q^2;q^2)_{m_i(\mu)}} \prod_{i =2,4,6,\dots} \frac{[1+q \alpha \beta]_{q^2}^{m_i(\mu)}}{(q^2;q^2)_{m_i(\mu)}} ,
\end{equation}
where $m_i(\mu) = \# \{ j: \mu_j=i \}$ is the multiplicity of $i$ in $\mu$ and the symbol $[a+b]_p^n$ denotes
\begin{equation} \label{eq:rogers_szego}
    [a+b]_p^n=\sum_{k=0}^n \binom{n}{k}_p a^{n-k} b^k,
    \qquad \text{with} \qquad 
    \binom{n}{k}_q = \frac{(q;q)_n}{(q;q)_k(q;q)_{n-k}}, \quad \text{for } 0\le k\le n.
\end{equation}

\subsection{Cauchy identities and refinements}

Here we prove identities \eqref{eq:Cauchy_id} from \cref{thm:Cauchy and Littlewood id} and \eqref{eq:refined_Cauchy_id} from \Cref{thm:refined Cauchy and Littlewood id}.

\begin{proof}[Proof of \cref{thm:Cauchy and Littlewood id}, eq. \eqref{eq:Cauchy_id}]

Using the geometric series to expand every factor $1/(1-q^k x_i y_j)$ in the left hand side of \eqref{eq:Cauchy_id}, we have
\begin{equation*}
    \begin{split}
        \prod_{i,j=1}^n \frac{1}{(x_i y_j;q)_\infty} &= \prod_{i,j=1}^n \prod_{k \ge 0} \left( \sum_{ c_{i,j,k} \ge 0 } \left( q^k x_i y_j \right)^{c_{i,j,k}} \right) 
        \\
        & = \sum_{ \mathbf{c} \in \mathfrak{F}_{\mathrm{fin}}(\mathscr{C}_n'^-,\N) } \left(q^{\sum_{i,j=1}^n\sum_{k\ge 0} k \, \mathbf{c}(i,j+k)} \right) \left(\prod_{i=1}^n x_i^{ \sum_{\ell \ge 0} \mathbf{c}(i,\ell) } \right) \left(\prod_{j=1}^n y_j^{ \sum_{\ell \ge 0} \mathbf{c}(\ell,j) } \right) 
        \\
        & = \sum_{ (H_1,H_2;\kappa) \in \bigsqcup_{\mu \in \mathbb{Y}}  \HWT(\mu,n)^2 \times \mathcal{K}_+(\mu) } q^{D'(H_1) + D'(H_2) + |\kappa|} x^{H_1} y^{H_2} 
        \\
        & = \sum_{\mu \in \mathbb{Y}} \left( \sum_{ \kappa \in \mathcal{K}_+(\mu) } q^{|\kappa|} \right) \left( \sum_{H_1 \in \HWT(\mu,n)} q^{D'(H_1)} x^{H_1} \right) \left( \sum_{H_1 \in \HWT(\mu,n)} q^{D'(H_2)} y^{H_2} \right)
        \\
        & = \sum_{\mu \in \mathbb{Y}}  \mathpzc{b}_\mu(q) \THL_{\mu}(x;q) \THL_\mu(y;q).
    \end{split}
\end{equation*}
Above, in the third equality we used \Cref{thm:LPP}, \ref{item:bij environment HWT restriction} and relations \eqref{eq:weight_environment} \eqref{eq:volume_environment}. In the last equality we used elementary geometric summations to write $\mathpzc{b}_\mu(q) = \sum_{ \kappa \in \mathcal{K}_+(\mu) } q^{|\kappa|}$.
\end{proof}

\begin{proof}[Proof of \cref{thm:refined Cauchy and Littlewood id}, eq. \eqref{eq:refined_Cauchy_id}]

Expanding Schur polynomials as generating functions of semistandard Young tableaux of skew shape we can manipulate the left hand side of \eqref{eq:refined_Cauchy_id} as
\begin{equation*}
    \begin{split}
        \sum_{ \substack{ \rho,\lambda \in \mathbb{Y} \\ \, \lambda_1'=N }} q^{|\rho|} s_{\lambda/\rho}(x) s_{\lambda/\rho}(y) &= \sum_{ \substack{ \rho,\lambda \in \mathbb{Y} \\ \, \lambda_1'=N }} \sum_{P,Q \in \SST(\lambda/\rho,n)} q^{|\rho|} x^P y^Q
        \\
        &= \sum_{ \substack{ \nu, \mu \in \mathbb{Y} \\ \nu_1'+\mu_1'=N } } \sum_{(H_1,H_2;\kappa) \in  \HWT(\mu,n)^2 \times \mathcal{K}_+(\mu) } q^{|\nu|+|\kappa|+D(H_1)+D(H_2)} x^{H_1} y^{H_2}
        \\
        &= \sum_{\ell = 0}^N  \bigg( \sum_{ \substack{ \nu \in \mathbb{Y} \\ \nu_1'=\ell } } q^{|\nu|} \bigg) \sum_{\substack{ \mu \in \mathbb{Y} \\ \mu_1'=N-\ell } }  \mathpzc{b}_\mu(q) \mathcal{Q}_\mu'(x;q) \mathcal{Q}_\mu'(y;q)
        \\
        &= \sum_{\ell = 0}^N  \frac{q^\ell}{(q;q)_\ell} \sum_{\substack{ \mu \in \mathbb{Y} \\ \mu_1'=N-\ell } }  \mathpzc{b}_\mu(q) \mathcal{Q}_\mu'(x;q) \mathcal{Q}_\mu'(y;q).
    \end{split}
\end{equation*}
Above, in the second equality we used \Cref{th:main_bij} and the fact that \eqref{eq:PQHH}, \eqref{eq:volume_pre} imply that, if $\Upsilon_{\mathrm{col}}(P,Q) = (H_1,H_2,\kappa,\nu)$, then $x^P = x^{H_1}$, $y^Q=y^{H_2}$ and $|\rho| = D(H_1)+D(H_2)+|\kappa|+|\nu|$. We also used \eqref{eq:length_pre} to relate $\lambda_1'$ to $\mu_1'+\nu_1'$. In the third and in the last last equality we used elementary geometric summations express $\mathpzc{b}_\mu(q)$ and to write $\frac{q^\ell}{(q;q)_\ell} = \sum_{  \nu \in \mathbb{Y}, \,\nu_1'=\ell  } q^{|\nu|}$.
\end{proof}

\subsection{Littlewood identities and refinements}

Let us recall a useful identity proven in \cite{imamura2021skewRSK}.

\begin{lemma}[\cite{imamura2021skewRSK}, Lemma 10.6] \label{lem:sum_from_IMS}
    Fix $z,q \in \mathbb{C}$ with $|q|<1$. Then, we have
    \[
        \sum_{\varkappa \in \mathbb{Y}~:~ \ell(\varkappa) \le m} z^{\mathrm{odd}(\varkappa)} q^{|\varkappa|} = \frac{[1+qz]_{q^2}^m}{(q^2;q^2)_m},
        \qquad
         \sum_{\varkappa \in \mathbb{Y}~:~ \ell(\varkappa) =m} z^{\mathrm{odd}(\nu)} q^{|\nu|} = \frac{[q^2+q z ]_{q^2}^m}{(q^2;q^2)_m}.
    \]
\end{lemma}

\begin{lemma} \label{lem:b_mu}
    Fix $\alpha,\beta,q \in \mathbb{C}$ with $|q|<1$. Then, we have
    \[
        \mathpzc{b}_{\mu}(\alpha,\beta,q) = \sum_{\kappa \in \mathcal{K} (\mu) } \alpha^{ \mathrm{odd}(\kappa + \mu)} \beta^{\mathrm{odd}(\kappa)} q^{|\kappa|}.
    \]
\end{lemma}
\begin{proof}
    We can decompose any $\kappa \in \mathcal{K}(\mu)$ as $\kappa = (\varkappa^{(1)},\varkappa^{(2)},\dots)$ with $\varkappa^{(i)} \in \mathbb{Y}$, and $\ell(\varkappa^{(i)}) \le m_i(\mu)$. Then, we have
    \begin{equation*}
        \sum_{\kappa \in \mathcal{K} (\mu) } \alpha^{ \mathrm{odd}(\kappa + \mu)} \beta^{\mathrm{odd}(\kappa)} q^{|\kappa|} 
        = \prod_{i\ge 1} \left( \sum_{\varkappa \in \mathbb{Y} ~:~ \ell(\varkappa) \le m_i(\mu)} \alpha^{\mathrm{odd}(i+\varkappa)} \beta^{\mathrm{odd}(\varkappa)} q^{|\varkappa|} \right).
    \end{equation*}
    Then, since for every $\varkappa \in \mathbb{Y}$ such that $ \ell(\varkappa) \le m_i(\mu)$ we have $\mathrm{odd}(i + \varkappa) = \mathrm{odd}(\varkappa)$ if $i$ is even and $\mathrm{odd}(i + \varkappa) = m_i(\mu) - \mathrm{odd}(\varkappa)$ if $i$ is odd, the right hand side in the above expression reduced, using \Cref{lem:sum_from_IMS}, to the factorized expression in \eqref{eq:b_mu}.
\end{proof}

\begin{proof}[Proof of \Cref{thm:Cauchy and Littlewood id}, eq. \eqref{eq:littlewood_id}]
    Expanding the right hand side of \eqref{eq:littlewood_id} we find
    \begin{equation*}
        \begin{aligned}
            \prod_{i = 1}^n & \frac{(-\alpha x_i;q^2)_\infty (-q \beta x_i;q^2)_\infty}{(x_i^2;q^2)_\infty} \prod_{1\le i < j \le n} \frac{1}{(x_i x_j;q^2)_\infty}
            \\
            & =\sum_{\mathbf{c} \in \mathfrak{F}_{\mathrm{fin}}^{\mathrm{sym}}(\mathscr{C}_n'^-,\N) } \left( \alpha^{\sum_{k \ge 0} \sum_{i=1}^n \mathbf{1}_{\mathbf{c}(i+2kn,i) \text{ is odd}} } \right)  \left( \beta^{\sum_{k \ge 0} \sum_{i=1}^n \mathbf{1}_{ \mathbf{c}(i+(2k+1)n,i) \text{ is odd}}  } \right) 
            \\
            & \qquad \qquad \qquad \qquad \qquad \times \left(q^{ \sum_{k\ge 0} \sum_{i,j=1}^n k \, \mathbf{c}(i+kn,j)} \right)  \left( \prod_{i=1}^n x_i^{\sum_{k \ge 0} \sum_{j=i}^n \mathbf{c}(i+kn,j)  }  \right)
            \\
            & = \sum_{\mathbf{c} \in \mathfrak{F}_{\mathrm{fin}}^{\mathrm{sym}}(\mathscr{C}_n'^-,\N) } (\alpha \beta)^{\frac{1}{2} \mathrm{oddTr}(\mathbf{c})} \left( \frac{\alpha}{\beta} \right)^{\frac{1}{2} \mathrm{oddTr}^-(\mathbf{c})}
            \left(q^{ \sum_{k\ge 0} \sum_{i,j=1}^n k \, \mathbf{c}(i+kn,j)} \right)  \left( \prod_{i=1}^n x_i^{\sum_{k \ge 0} \sum_{j=i}^n \mathbf{c}(i+kn,j)  }  \right)
            \\
            &
            = \sum_{(H;\kappa) \in \bigsqcup_{\mu \in \mathbb{Y} } \HWT(\mu,n) \times \mathcal{K}(\mu) }  \alpha^{ \mathrm{odd}(\kappa + \mu)} \beta^{\mathrm{odd}(\kappa)} q^{|\kappa|} q^{2 D(H)} x^{H}
            \\
            & =
            \sum_{\mu} c_\mu(\alpha,\beta;q) \THL_\mu(x;q^2).
        \end{aligned}
    \end{equation*}
    Above in the first equality we expanded through the geometric series the factors in the $q$-Pochhammer symbols in the denominator and combined them with those in the numerator in the right hand side. In the second equality we simply recognized the quantities $\mathrm{oddTr}(\mathbf{c}),\mathrm{oddTr}^-(\mathbf{c})$ in the exponents of coefficients $\alpha,\beta$ and in the third equality we used \Cref{cor:symmetric_env}. In the last equality we used \Cref{lem:b_mu}.
\end{proof}

\begin{proof}[Proof of \Cref{thm:refined Cauchy and Littlewood id}, eq. \eqref{eq:littlewood_id_refined}]

    We expand the skew Schur polynomials as generating functions of semistandard tableaux in the left hand side of  \eqref{eq:littlewood_id_refined} and we find
    \begin{equation}
        \begin{aligned}
            \sum_{\rho \subset \lambda : \lambda_1' = N} & \alpha^{\mathrm{odd}(\lambda)} \beta^{\mathrm{odd}(\rho)} q^{|
            \rho|} s_{\lambda/\rho}(x) 
            \\
            & = \sum_{\rho \subset \lambda : \lambda_1' = N} \sum_{P \in \SST(\lambda/\rho,n)} \alpha^{\mathrm{odd}(\lambda)} \beta^{\mathrm{odd}(\rho)} q^{|\rho|} x^P
            \\
            & = \sum_{ \substack{\nu, \mu \in \mathbb{Y} \\ \nu_1'+\mu_1' = N } } \sum_{(H_1,\kappa) \in \HWT(\mu,n) \times \mathcal{K}_+(\mu) } (\alpha\beta)^{\mathrm{odd}(\nu)} \alpha^{\mathrm{odd}(\kappa+\mu)} \beta^{\mathrm{odd}(\kappa)} q^{2D'(H)+|\kappa| + |\nu|} x^H
            \\
            &
            = \sum_{k=0}^N \left( \sum_{\nu \in \mathbb{Y}: \nu_1'=k} (\alpha \beta)^{\mathrm{odd}(\nu)} q^{|\nu|} \right) \left( \sum_{\mu \in \mathbb{Y}: \mu_1'=N-k} \mathpzc{b}_\mu (\alpha, \beta ,q) \mathscr{P}_\mu(x;q^2) \right).
        \end{aligned}
    \end{equation}
    Above, the second equality uses \Cref{cor:symmetric_Upsilon}, while the third equality uses \Cref{lem:b_mu}. Then, using \Cref{lem:sum_from_IMS} the right hand side reduces to the right hand side of \eqref{eq:littlewood_id_refined}, completing the proof.
\end{proof}

\appendix

\section{Alternative proof of \texorpdfstring{\cref{th:leading_f0}}{}}\label{sec:proof+1}

Here we give a direct proof of \cref{th:leading_f0}.
For this recall the notion of energy and dual energy function of \Cref{def:energy_function}. Below we state another property of the energy function that is useful to us.
To do so, we note that we can organize the energy function in the form of a corner transfer matrix by the scattering diagram in Figure~\ref{fig:scattering diagrams}.
Recall that $\sigma^- \in S_n$ is the longest element as in \Cref{rem:reverse_iso}.

\begin{lemma}[{\cite[Lemma 4.4]{Kuniba-Sakamoto-Yamada2007_tau}}]
\label{lem:D_winding}
    For $H \in \KR{\mu}$, the energy function $D(H)$ is equal to the total sum of the winding numbers in all $\ell(\ell-1)/2$ crossings~\eqref{eq:scattering_diagram} from the action of $\sigma^- \in S_{\ell}$.
\end{lemma}

\begin{example}\label{ex:bdb}
When $H$ is the second tableau in \cref{ex:HWTs}, i.e., $H= 12\otimes 223\otimes 1122$, we can calculate the local energies $D_1(H), D_2(H)$ as follows:
\begin{align*}
D_1(H) & = W(1122 \otimes 223) = 0,
\\ 
D_2(H) & = W_1\sigma_2\sigma_1(H) + W_2(H) = W(2223,12) + W(112,23) = 2+0 = 2.
\end{align*}
The corresponding diagram is illustrated in \Cref{fig:scattering diagrams}, right panel.
\end{example}

Next, let us describe the idea of our direct proof of~\Cref{th:leading_f0} using the example with $\mu=(4,3,2,2,1)$ and $n=5$.
Consider the tensor products
\begin{subequations}
\begin{align}\label{eq:tildeb_ex}
\overline{H}&=5\otimes 44\otimes 45\otimes 455\otimes 4444,
& D(\overline{H}) & = 9,
\\
\label{eq:f0_tildeb_ex}
f_0(\overline{H})&=5\otimes 44\otimes 45\otimes 145\otimes 4444,
& D(f_0(\overline{H})) & = 10,
\end{align}
\end{subequations}
and note that they only differ in the fourth elements ($455$ and $145$).
To see their energies, let us compare the scattering diagrams given in~\cref{fig:scattering_diagram}.

\begin{figure}
    \iftikz
    \[
    \begin{tikzpicture}[scale=.85]
    \draw (3, 1) node {$\overline{H}$};
    \node[anchor=south] (s1) at (0, -0.2) {$4444$};
    \node[anchor=south] (s2) at (1.5, -0.2) {$45{\color{red}5}$};
    \node[anchor=south] (s3) at (3, -0.2) {$45$};
    \node[anchor=south] (s4) at (4.5, -0.2) {$44$};
    \node[anchor=south] (s5) at (6, -0.2) {$5$};
    \node[anchor=west] (t1) at (6.8, -1) {$4445$};
    \node[anchor=west] (t2) at (6.8, -2) {$455$};
    \node[anchor=west] (t3) at (6.8, -3) {$45$};
    \node[anchor=west] (t4) at (6.8, -4) {$44$};
    \node[anchor=west] (t5) at (6.8, -5) {$4$};
    \draw (2.25,-1) node[anchor=north] {\tiny $445{\color{red}5}$};
    \draw (3.75,-1) node[anchor=north] {\tiny $4455$};
    \draw (5.25,-1) node[anchor=north] {\tiny $4444$};
    \draw (2.25,-2) node[anchor=north] {\tiny $444$};
    \draw (3.75,-2) node[anchor=north] {\tiny $44{\color{red}5}$};
    \draw (5.25,-2) node[anchor=north] {\tiny $55{\color{red}5}$};
    \draw (3.75,-3) node[anchor=north] {\tiny $44$};
    \draw (5.25,-3) node[anchor=north] {\tiny $44$};
    \draw (5.25,-4) node[anchor=north] {\tiny $44$};
    \draw (3,-1.5) node[anchor=west] {\tiny $4{\color{red}5}$};
    \draw (4.5,-1.5) node[anchor=west] {\tiny $55$};
    \draw (6,-1.5) node[anchor=west] {\tiny $4$};
    \draw (4.5,-2.5) node[anchor=west] {\tiny $44$};
    \draw (6,-2.5) node[anchor=west] {\tiny ${\color{red}5}$};
    \draw (6,-3.5) node[anchor=west] {\tiny $4$};
    \foreach \i in {1,2,3,4,5} {
      \draw[-,rounded corners] (s\i) -- (1.5*\i-1.5, -\i) -- (t\i);
    }
    \draw (1.5,-1) node[anchor=south east, color=purple] {\tiny $1$};
    \draw (3,-1) node[anchor=south east, color=purple] {\tiny $1$};
    \draw (4.5,-1) node[anchor=south east, color=purple] {\tiny $2$};
    \draw (6,-1) node[anchor=south east, color=purple] {\tiny $0$};
    \draw (3,-2) node[anchor=south east, color=purple] {\tiny $1$};
    \draw (4.5,-2) node[anchor=south east, color=purple] {\tiny $0$};
    \draw (6,-2) node[anchor=south east, color=purple] {\tiny $1$};
    \draw (4.5,-3) node[anchor=south east, color=purple] {\tiny $2$};
    \draw (6,-3) node[anchor=south east, color=purple] {\tiny $0$};
    \draw (6,-4) node[anchor=south east, color=purple] {\tiny $1$};
    \draw[color=red,line width=1pt] (s2) -- (1.5,-1) -- (3,-1) -- (3,-2) -- (6,-2) -- (6,-3) -- (t3);
    \end{tikzpicture}
    \quad
    \begin{tikzpicture}[scale=.85]
    \draw (3, 1) node {$f_0(\overline{H})$};
    \node[anchor=south] (s1) at (0, -0.2) {$4444$};
    \node[anchor=south] (s2) at (1.5, -0.2) {${\color{blue}1}45$};
    \node[anchor=south] (s3) at (3, -0.2) {$45$};
    \node[anchor=south] (s4) at (4.5, -0.2) {$44$};
    \node[anchor=south] (s5) at (6, -0.2) {$5$};
    \node[anchor=west] (t1) at (6.8, -1) {$4445$};
    \node[anchor=west] (t2) at (6.8, -2) {$455$};
    \node[anchor=west] (t3) at (6.8, -3) {$15$};
    \node[anchor=west] (t4) at (6.8, -4) {$44$};
    \node[anchor=west] (t5) at (6.8, -5) {$4$};
    \draw (2.25,-1) node[anchor=north] {\tiny ${\color{blue}1}445$};
    \draw (3.75,-1) node[anchor=north] {\tiny $4455$};
    \draw (5.25,-1) node[anchor=north] {\tiny $4444$};
    \draw (2.25,-2) node[anchor=north] {\tiny $444$};
    \draw (3.75,-2) node[anchor=north] {\tiny ${\color{blue}1}44$};
    \draw (5.25,-2) node[anchor=north] {\tiny ${\color{blue}1}55$};
    \draw (3.75,-3) node[anchor=north] {\tiny $44$};
    \draw (5.25,-3) node[anchor=north] {\tiny $44$};
    \draw (5.25,-4) node[anchor=north] {\tiny $44$};
    \draw (3,-1.5) node[anchor=west] {\tiny ${\color{blue}1}4$};
    \draw (4.5,-1.5) node[anchor=west] {\tiny $55$};
    \draw (6,-1.5) node[anchor=west] {\tiny $4$};
    \draw (4.5,-2.5) node[anchor=west] {\tiny $44$};
    \draw (6,-2.5) node[anchor=west] {\tiny ${\color{blue}1}$};
    \draw (6,-3.5) node[anchor=west] {\tiny $4$};
    \foreach \i in {1,2,3,4,5} {
      \draw[-,rounded corners] (s\i) -- (1.5*\i-1.5, -\i) -- (t\i);
    }
    \draw (1.5,-1) node[anchor=south east, color=purple] {\tiny $2$};
    \draw (3,-1) node[anchor=south east, color=purple] {\tiny $0$};
    \draw (4.5,-1) node[anchor=south east, color=purple] {\tiny $2$};
    \draw (6,-1) node[anchor=south east, color=purple] {\tiny $0$};
    \draw (3,-2) node[anchor=south east, color=purple] {\tiny $2$};
    \draw (4.5,-2) node[anchor=south east, color=purple] {\tiny $0$};
    \draw (6,-2) node[anchor=south east, color=purple] {\tiny $0$};
    \draw (4.5,-3) node[anchor=south east, color=purple] {\tiny $2$};
    \draw (6,-3) node[anchor=south east, color=purple] {\tiny $1$};
    \draw (6,-4) node[anchor=south east, color=purple] {\tiny $1$};
    \draw[color=blue,line width=1pt] (s2) -- (1.5,-1) -- (3,-1) -- (3,-2) -- (6,-2) -- (6,-3) -- (t3);
    \end{tikzpicture}
    \]
    \fi
    \caption{}
   \label{fig:scattering_diagram}
\end{figure}

In this diagrams we can observe:
\begin{enumerate}
\item The edges in the diagram can be classified into two cases.
    \begin{enumerate}
        \item The edges where the attached scattering data are the same between $\overline{H}$ and $f_0(\overline{H})$.
        We color these edges in black.
        \item Those edges where they are different (more precisely one element `5' in $\overline{H}$ is changed as `1' in $f_0(\overline{H})$).
        We colored these red for $\overline{H}$ and blue for $f_0(\overline{H})$.
    \end{enumerate}
    We can observe that the colored edges form paths that enter from one of the top vertical edges and  exit from one of the rightmost horizontal edges.
\item None of the edges from left or bottom of the colored paths have elements `$n$'.
\item We have the following explicit rules about the local energies:
\begin{align}
\label{eq:localrules}
W\left(\!\begin{tikzpicture}[scale=.3,baseline=-3]
\draw[-] (-1,0) -- (0,0) -- (0,-1);
\draw[-,color=red,line width=1pt] (0,1) -- (0,0) -- (1,0);
\end{tikzpicture}\!\right)
=
W\left(\!\begin{tikzpicture}[scale=.3,baseline=-3]
\draw[-] (-1,0) -- (0,0) -- (0,-1);
\draw[-,color=blue,line width=1pt] (0,1) -- (0,0) -- (1,0);
\end{tikzpicture}\!\right)
- 1,
\quad
W\left(\!\begin{tikzpicture}[scale=.3,baseline=-3]
\draw[-] (0,1) -- (0,0) -- (1,0);
\draw[-,color=red,line width=1pt] (-1,0) -- (0,0) -- (0,-1);
\end{tikzpicture}\!\right)
=
W\left(\!\begin{tikzpicture}[scale=.3,baseline=-3]
\draw[-] (0,1) -- (0,0) -- (1,0);
\draw[-,color=blue,line width=1pt] (-1,0) -- (0,0) -- (0,-1);
\end{tikzpicture}\!\right)
+ 1,
\quad
W\left(\!\begin{tikzpicture}[scale=.3,baseline=-3]
\draw[-] (0,1) -- (0,-1);
\draw[-,color=red,line width=1pt] (-1,0) -- (1,0);
\end{tikzpicture}\!\right)
=
W\left(\!\begin{tikzpicture}[scale=.3,baseline=-3]
\draw[-] (0,1) -- (0,-1);
\draw[-,color=blue,line width=1pt] (-1,0) -- (1,0);
\end{tikzpicture}\!\right)
= 0.
\end{align}
Here $W(\cdot)$ is the local energy for the scattering diagram ``$\cdot$'' defined in \cref{def:energy_function}. 
\end{enumerate}

In fact we can prove that the above properties (1)-(3) hold for any $\overline{H}$ with $H\in\KR{\mu}_{\mathrm{hw}}\setminus\{H_{\max}\}$.
We can show Equation~\eqref{eq:leading_f0} using them.
By (3), one sees that the sum of the local energies along the red path is one less than that along the blue path.
The local energies at the vertex surrounding to black edges are exactly the same between the two scattering diagrams.

For proving the above (1)-(3), we prepare two lemmas performing a comparison between two combinatorial $R$-matrices, here denoted $\mathcal{R}(a \otimes b) = d \otimes c$ and $\mathcal{R}(a' \otimes b') = d' \otimes c'$.
Hereafter we will always assume that the length of $a$ (resp.\ $a'$) is shorter than or equal to $b$ (resp.\ $b'$).
The associated local energies are denoted by $W(a\otimes b)$ and $W(a'\otimes b')$.

\begin{lemma}\label{lem:a'ad'd}
Suppose that
\begin{itemize}
    \item[A.1:] $a$ includes at least one $n$ and $a'$ is obtained from $a$ with one $n$ replaced by $1$;
    \item[A.2:] $b=b'$ and they do not include $n$.
\end{itemize}
Then we have the following output data
$c,d$ and $c',d'$.
\begin{itemize}
\item[C.1:] $c=c'$ and they do not have $n$.
\item[C.2:] $d$ and $d'$ differ in only one element. In $d'$, one $n$ element in $d$ is exchanged to $1$.
\item[C.3:] $W(a'\otimes b')=W(a\otimes b)+1$.
\end{itemize}
\end{lemma}

A graphical expression of \cref{lem:a'ad'd} is shown in \cref{fig:graph_a'ab'b}.
Recall the scattering diagram depiction of a combinatorial $R$-matrix~\eqref{eq:scattering_diagram}.

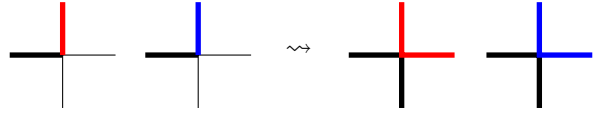
\begin{figure}[ht]
    \[
    \begin{tikzpicture}[scale=.7,baseline=0]
    \draw[-] (0,-1) -- (0,0) -- (1,0);
    \draw[-,line width=2pt] (-1,0) -- (0,0);
    \draw[-,color=red,line width=2pt] (0,1) -- (0,0);
    \end{tikzpicture}
    \quad
    \begin{tikzpicture}[scale=.7,baseline=0]
    \draw[-] (0,-1) -- (0,0) -- (1,0);
    \draw[-,line width=2pt] (-1,0) -- (0,0);
    \draw[-,color=blue,line width=2pt] (0,1) -- (0,0);
    \end{tikzpicture}
    \quad
    \rightsquigarrow
    \quad
    \begin{tikzpicture}[scale=.7,baseline=0]
    \draw[-,line width=2pt] (-1,0) -- (0,0) -- (0,-1);
    \draw[-,color=red,line width=2pt] (0,1) -- (0,0) -- (1,0);
    \end{tikzpicture}
    \quad
    \begin{tikzpicture}[scale=.7,baseline=0]
    \draw[-,line width=2pt] (-1,0) -- (0,0) -- (0,-1);
    \draw[-,color=blue,line width=2pt] (0,1) -- (0,0) -- (1,0);
    \end{tikzpicture}
    \]
    \caption{A graphical expression of \cref{lem:a'ad'd} as scattering diagrams. Let $h\in\{a,b,c,d\}$.
    When $h=h'$ and $h$ does not contain an $n$, we mark the edges for $h, h'$ in black.
    When an $n$ in $h$ is replaced with a $1$ in $h'$, we mark the edge for $h$ (resp.\ $h'$) in red (resp.\ blue).}
    \label{fig:graph_a'ab'b}
\end{figure}

\begin{example}
\label{ex:scattering+1}
\cref{fig:scattering+1} illustrates
the case for $n=3$, $a=1122233$, $a' = 1112223$ and $b=b'=11111222$, in which~\cref{lem:a'ad'd} is applicable.
Note that in this example, $b=b'$ and $c=c'$ while $a$ (resp.~$d$) differs from $a'$ (resp.~$d'$) in only one element.
Indeed, one ``$3$'' in $a$ (resp.~$d$) is replaced by a ``$1$'' in $a'$ (resp.~$d'$).
As in \cref{fig:graph_a'ab'b}, we mark the top vertical and right horizontal edges in red in the scattering diagram for $a,b,c,d$ and in blue for $a',b'c'd'$.

\begin{figure}[ht]
    \iftikz
    \[
    \begin{tikzpicture}[baseline=1cm]
        \draw[-,line width=2pt] (-1,0) node[anchor=south,scale=.7] {$b=11111222$} -- (0,0) -- (0,-1) node[anchor=north,scale=.7] {$c=1111222$};
        \draw[-,red,line width=2pt] (0,1) node[black,anchor=south,scale=.7] {$a=1122233$} -- (0,0) -- (1,0) node[black,anchor=south,scale=.7] {$d=11122233$};
    \end{tikzpicture}
    \quad
    \begin{tikzpicture}[scale=.8,baseline=0,>=to]
      \draw (.5,0) node[anchor=south] {$b$};
      \draw (0,0) grid (1,-3);
      \draw (2,0) grid (3,-3);
      \draw[-,rounded corners,dgreencolor] (.25,-.25) -- (.25,-.55) -- (1.1,-.55) -- (1.6,-1.5) -- (2.25,-1.5);
      \draw[-,rounded corners,dgreencolor] (.5,-.25) -- (.5,-.4) -- (1.25,-.4) -- (1.75,-1.35) -- (2.5,-1.35) -- (2.5,-1.5);
      \draw[-,rounded corners,dgreencolor] (.75,-.25) -- (1.4,-.25) -- (1.9,-1.2) -- (2.75,-1.2) -- (2.75,-1.5);
      \draw[-<,rounded corners,UQpurple] (.6,-.75) -- (1.1,-.75) -- (1.6,-1.7) -- (1.6,-3.2);
      \draw[<-,rounded corners,UQpurple] (1.6,.3) -- (1.6,-.25) -- (2.5,-.25) -- (2.6,-.5);
      \draw[-,rounded corners,dgreencolor] (.25,-1.5) -- (.25,-1.8) -- (1.1,-1.8) -- (1.5,-2.5) -- (2.3,-2.5);
      \draw[-,rounded corners,dgreencolor] (.5,-1.5) -- (.5,-1.65) -- (1.2,-1.65) -- (1.6,-2.35) -- (2.6,-2.35) -- (2.6,-2.5);
      \draw[-<<,rounded corners,UQpurple] (.75,-1.5) -- (1.3,-1.5) -- (1.3,-3.2);
      \draw[<<-,rounded corners,UQpurple] (1.3,.3) -- (1.3,-.1) -- (1.8,-.5) -- (2.3,-.5);
      \foreach \x/\y in {.25/-.25,.5/-.25,.75/-.25,.3/-.75,.6/-.75, .25/-1.5,.5/-1.5,.75/-1.5}
        \fill[black] (\x,\y) circle (0.07);
      \draw (2.5,0) node[anchor=south] {$a$};
      \foreach \x/\y in {2.3/-.5,2.6/-.5, 2.25/-1.5,2.5/-1.5,2.75/-1.5, 2.3/-2.5,2.6/-2.5}
        \fill[black] (\x,\y) circle (0.07);
      \draw (1.5,.45) node[scale=.8] {$\overbrace{\phantom{XX}}^{W(a,b) = 2}$};
    \end{tikzpicture}
    \qquad\quad
    \begin{tikzpicture}[baseline=1cm]
        \draw[-,line width=2pt] (-1,0) node[anchor=south,scale=.7] {$b'=11111222$} -- (0,0) -- (0,-1) node[anchor=north,scale=.7] {$c'=1111222$};
        \draw[-,blue,line width=2pt] (0,1) node[black,anchor=south,scale=.7] {$a'=1112223$} -- (0,0) -- (1,0) node[black,anchor=south,scale=.7] {$d'=11112223$};
    \end{tikzpicture}
    \quad
    \begin{tikzpicture}[scale=.8,baseline=0,>=to]
      \draw (.5,0) node[anchor=south] {$b'$};
      \draw (0,0) grid (1,-3);
      \draw (2,0) grid (3,-3);
      \draw[-,rounded corners,dgreencolor] (.25,-.25) -- (.25,-.55) -- (1.1,-.55) -- (1.6,-1.5) -- (2.25,-1.5);
      \draw[-,rounded corners,dgreencolor] (.5,-.25) -- (.5,-.4) -- (1.25,-.4) -- (1.75,-1.35) -- (2.5,-1.35) -- (2.5,-1.5);
      \draw[-,rounded corners,dgreencolor] (.75,-.25) -- (1.4,-.25) -- (1.9,-1.2) -- (2.75,-1.2) -- (2.75,-1.5);
      \draw[-<,rounded corners,UQpurple] (.6,-.75) -- (1.1,-.75) -- (1.75,-1.8) -- (1.75,-3.2);
      \draw[<-,rounded corners,UQpurple] (1.75,.3) -- (1.75,-.2) -- (2.75,-.2) -- (2.75,-.5);
      \draw[-,rounded corners,dgreencolor] (.5,-1.5) -- (.5,-1.65) -- (1.2,-1.65) -- (1.8,-2.5) -- (2.5,-2.5);
      \draw[-<<,rounded corners,UQpurple] (.75,-1.5) -- (1.5,-1.5) -- (1.5,-3.2);
      \draw[<<-,rounded corners,UQpurple] (1.5,.3) -- (1.5,-.1) -- (1.8,-.3) -- (2.5,-.3) -- (2.5,-.5);
      \draw[-<<<,rounded corners,UQpurple] (.25,-1.5) -- (.25,-1.8) -- (1.25,-1.8) -- (1.25,-3.2);
      \draw[<<<-,rounded corners,UQpurple] (1.25,.3) -- (1.25,-.15) -- (1.8,-.5) -- (2.25,-.5);
      \foreach \x/\y in {.25/-.25,.5/-.25,.75/-.25,.3/-.75,.6/-.75, .25/-1.5,.5/-1.5,.75/-1.5}
        \fill[black] (\x,\y) circle (0.07);
      \draw (2.5,0) node[anchor=south] {$a'$};
      \foreach \x/\y in {2.25/-.5,2.5/-.5,2.75/-.5, 2.25/-1.5,2.5/-1.5,2.75/-1.5, 2.5/-2.5}
        \fill[black] (\x,\y) circle (0.07);
      \draw (1.5,.45) node[scale=.8] {$\overbrace{\phantom{XX}}^{W(a',b') = 3}$};
    \end{tikzpicture}
    \]
    \fi
    \caption{}
    \label{fig:scattering+1}
\end{figure}
\end{example}

\begin{proof}[Proof of~\Cref{lem:a'ad'd}]
We prove the lemma based on the Nakayashiki--Yamada rule in \cref{def:Nakayashiki_yamada}.
Instead of {\it C.1} and {\it C.2}, we can say in the language of the algorithm that among the dots in the left boxes associated to $b$, the dots connected by the lines are exactly the same as those for $b'$.

Based on the algorithm, we connect each dot in the right boxes and each one in the left boxes with lines from the bottom dot to the top one in the right side.
Note that we can choose the order of connection arbitrarily because the positions of the unconnected dots and also the number of winding pairs do not depend on the order.
On the other hand, for the dots associated to $a'$ and $b'$, we first connect the dot corresponding to the element $1$ in $a'$ appeared by the replacement of an element $n$ in $a$ with a dot on the left.
The order of the second and the following connections is assumed to be the same as that for the case of ``$a$''.
In \cref{fig:boxdot1}, we show the examples of $n=4$.
The left figure represents $a=12344,~b=122233$ while the right one does $a'=11234,~b'=122233$, which satisfy the assumptions {\it A.1} and {\it A.2}.
The red superscripts mean the order of the connections.

\begin{figure}[t]
    \iftikz
    \[
    \begin{tikzpicture}[scale=.7,>=to]
      \draw (.5,0) node[anchor=south] {$b$};
      \draw (0,0) grid (1,-4);
      \draw (2,0) grid (3,-4);
      \draw[-,rounded corners,dgreencolor] (.5,-.5) -- (1.2,-.5) -- (1.8,-1.5) -- (2.5,-1.5);
      \draw[-,rounded corners,dgreencolor] (.75,-1.5) -- (1.2,-1.5) -- (1.8,-2.5) -- (2.5,-2.5);
      \draw[-<,rounded corners,UQpurple] (.5,-1.5) -- (.5,-1.8) -- (1.2,-1.8) -- (1.2,-4.2);
      \draw[<-,rounded corners,UQpurple] (1.2,.3) -- (1.2,-.5) -- (2.5,-.5);
      \draw[-,rounded corners,blue] (.3,-2.5) -- (.3,-2.8) -- (1.2,-2.8) -- (1.8,-3.5) -- (2.3,-3.5);
      \draw[-,rounded corners,dgreencolor] (.6,-2.5) -- (1.2,-2.5) -- (1.8,-3.2) -- (2.6,-3.2) -- (2.6,-3.5);
      \foreach \x/\y in {.5/-.5, .25/-1.5,.5/-1.5,.75/-1.5, .3/-2.5,.6/-2.5}
        \fill[black] (\x,\y) circle (0.07);
      \fill[blue,opacity=.3] (.3,-2.5) circle (.15);
      \draw (2.5,0) node[anchor=south] {$a$};
      \foreach \x/\y [count=\i] in {2.3/-3.5, 2.6/-3.5, 2.5/-2.5, 2.5/-1.5, 2.5/-.5}
        \fill[black] (\x,\y) circle (0.07) node[anchor=north west, color=red, scale=.4, node distance=8pt] {$\i$};
      \fill[blue,opacity=.3] (2.3,-3.5) circle (.15);
    \end{tikzpicture}
    \qquad\qquad
    \begin{tikzpicture}[scale=.7,>=to]
      \draw (.5,0) node[anchor=south] {$b'$};
      \draw (0,0) grid (1,-4);
      \draw (2,0) grid (3,-4);
      \draw[-,rounded corners,dgreencolor] (.5,-.5) -- (1.2,-.5) -- (1.8,-1.5) -- (2.5,-1.5);
      \draw[-,rounded corners,dgreencolor] (.75,-1.5) -- (1.2,-1.5) -- (1.8,-2.5) -- (2.5,-2.5);
      \draw[-<,rounded corners,UQpurple] (.5,-1.5) -- (.5,-1.8) -- (1.2,-1.8) -- (1.2,-4.2);
      \draw[<-,rounded corners,UQpurple] (1.2,.3) -- (1.2,-.5) -- (2.3,-.5);
      \draw[-<<,rounded corners,blue] (.3,-2.5) -- (.3,-2.8) -- (1.5,-2.8) -- (1.5,-4.2);
      \draw[<<-,rounded corners,blue] (1.5,.3) -- (1.5,-.2) -- (2.6,-0.2) -- (2.6,-.5);
      \draw[-,rounded corners,dgreencolor] (.6,-2.5) -- (1.5,-2.5) -- (1.8,-3.5) -- (2.5,-3.5);
      \foreach \x/\y in {.5/-.5, .25/-1.5,.5/-1.5,.75/-1.5, .3/-2.5,.6/-2.5}
        \fill[black] (\x,\y) circle (0.07);
      \fill[blue,opacity=.3] (.3,-2.5) circle (.15);
      \draw (2.5,0) node[anchor=south] {$a'$};
      \foreach \x/\y [count=\i] in {2.6/-.5, 2.5/-3.5, 2.5/-2.5, 2.5/-1.5, 2.3/-.5}
        \fill[black] (\x,\y) circle (0.07) node[anchor=north west, color=red, scale=.4, node distance=2pt] {$\i$};
      \fill[blue,opacity=.3] (2.6,-.5) circle (.15);
    \end{tikzpicture}
    \]
    \fi
    \caption{An example with $n = 4$ of how the Nakayashiki--Yamada algorithm changes under a $f_0$ that changes the winding number (i.e., local energy).}
    \label{fig:boxdot1}
\end{figure}
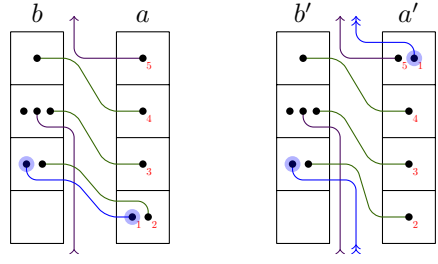

Comparing the boxes and dots for $a$ and $a'$, we find that the only difference is the position of the dots which we connect at first.
In $a$, the dot is in the bottommost (the $n$th) box while in $a'$ it is in the first box.
From the algorithm, the dot in the $n$th box in $a$ is connected to the one of bottommost dots among the ones which are in the $1$st through $(n-1)$th boxes in $b$.
On the other hand, the dot in the first box in $a'$ is connected to the one of the bottommost boxes among the whole ones.
Combining this with {\it A.2}, we see that they are connected to the same dot in the side of $b(=b')$.
Furthermore we find that after we connect them, the configurations of the boxes and dots have no differences and the order of connections are also exactly the same between the dots associated to $a$ and $a'$.
Thus we find that in the dots on the left sides associated to $b$ and $b'$, exactly the same dots are connected. This leads to {\it C.1} and {\it C.2}.

{\it C.3} also follows from the above considerations. It is clear that for the first connections, the connecting pair of dots is unwinding for $a$ and $b$ while for $a'$ and $b'$, it is winding.  In addition there are no differences for the other connecting pairs. This leads to {\it C.3}.
\end{proof}

\cref{lem:a'ad'd} stated above is about a comparison between two scattering diagrams as depicted 
in \cref{ex:scattering+1}, where $a'$ and $a$
differ in only one element while $b=b'$. 
The next lemma deals with the opposite case, i.e.~$a=a'$, while $b$ and $b'$ are different with only one element.

\begin{lemma}\label{lem:b'b}
Assume
\begin{itemize}
    \item[A.1:] $a=a'$ and they do not have $1$s.
    \item[A.2:] $b$ has at least one $n$ and $b'$ is obtained from $b$ by replacing one $n$ in $b$
    with $1$.
\end{itemize}

\noindent
(1) If additionally $W(a\otimes b) = 0$, then
\begin{itemize}
    \item[C.1:] $c=c'$ and they do not have $n$.
    \item[C.2:] $d$ has at least one $n$ and $d'$ is obtained from $d$ by replacing one $n$ in $d$ with $1$.
    \item[C.3:] $W(a'\otimes b') = 0$.
\end{itemize}

\noindent
(2) If additionally $W(a\otimes b) > 0$, then
\begin{itemize}
    \item[C.4:] $c$ has at least one $n$ and $c'$ is obtained from $c$ by replacing one $n$ in $c$ with $1$.
    \item[C.5:] $d=d'$ and they do not have $1$.
    \item[C.6:] $W(a'\otimes b') = W(a\otimes b) - 1$.
\end{itemize}
\end{lemma}

Similar to \cref{fig:graph_a'ab'b}, a graphical expression of~\cref{lem:b'b} is shown in~\cref{fig:graph_b'b} with the black, red and blue edges are drawn in the same way.
When $h=h'$ with $h\in\{a,b,c,d\}$ and $h$ does not contain a $1$s, we mark the edges in gray.
\begin{figure}[ht]
    \iftikz
    \[
    \begin{tikzpicture}[scale=.6,baseline=0]
    \draw[-] (0,-1) -- (0,0) -- (1,0);
    \draw[-,color=gray,line width=2pt] (0,1) -- (0,0);
    \draw[-,color=red,line width=2pt] (-1,0) -- (0,0);
    \draw (0,0) node[anchor=south east] {\tiny $H = 0$};
    \draw (-1,1) node[anchor=east] {\tiny(1)};
    \end{tikzpicture}
    \ 
    \begin{tikzpicture}[scale=.6,baseline=0]
    \draw[-] (0,-1) -- (0,0) -- (1,0);
    \draw[-,color=gray,line width=2pt] (0,1) -- (0,0);
    \draw[-,color=blue,line width=2pt] (-1,0) -- (0,0);
    \end{tikzpicture}
    \ 
    \rightsquigarrow
    \ 
    \begin{tikzpicture}[scale=.6,baseline=0]
    \draw[-,color=gray,line width=2pt] (0,1) -- (0,0);
    \draw[-,line width=2pt] (0,0) -- (0,-1);
    \draw[-,color=red,line width=2pt] (-1,0) -- (1,0);
    \end{tikzpicture}
    \ 
    \begin{tikzpicture}[scale=.6,baseline=0]
    \draw[-,color=gray,line width=2pt] (0,1) -- (0,0);
    \draw[-,line width=2pt] (0,0) -- (0,-1);
    \draw[-,color=blue,line width=2pt] (-1,0) -- (1,0);
    \end{tikzpicture}
    \hspace{30pt}
    \begin{tikzpicture}[scale=.6,baseline=0]
    \draw[-] (0,-1) -- (0,0) -- (1,0);
    \draw[-,color=gray,line width=2pt] (0,1) -- (0,0);
    \draw[-,color=red,line width=2pt] (-1,0) -- (0,0);
    \draw (0,0) node[anchor=south east] {\tiny $H > 0$};
    \draw (-1,1) node[anchor=east] {\tiny (2)};
    \end{tikzpicture}
    \quad
    \begin{tikzpicture}[scale=.6,baseline=0]
    \draw[-] (0,-1) -- (0,0) -- (1,0);
    \draw[-,color=gray,line width=2pt] (0,1) -- (0,0);
    \draw[-,color=blue,line width=2pt] (-1,0) -- (0,0);
    \end{tikzpicture}
    \ 
    \rightsquigarrow
    \ 
    \begin{tikzpicture}[scale=.6,baseline=0]
    \draw[-,color=gray,line width=2pt] (0,1) -- (0,0) -- (1,0);
    \draw[-,color=red,line width=2pt] (-1,0) -- (0,0) -- (0,-1);
    \end{tikzpicture}
    \ 
    \begin{tikzpicture}[scale=.6,baseline=0]
    \draw[-,color=gray,line width=2pt] (0,1) -- (0,0) -- (1,0);
    \draw[-,color=blue,line width=2pt] (-1,0) -- (0,0) -- (0,-1);
    \end{tikzpicture}
    \]
    \fi
    \caption{Graphical expression of \cref{lem:b'b}.}
    \label{fig:graph_b'b}
\end{figure}
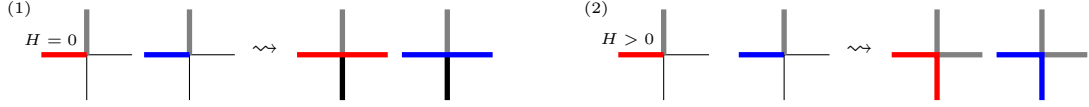

\begin{example}    \label{ex:scattering0-}
In~\cref{fig:scattering0-}, we show examples of the scattering diagrams for~\cref{lem:b'b} cases (1) and (2) with $n=3$.
In case (1), we see that $a=a'$ and $c=c'$, while $b$ and $b'$ (resp.~$d$ and $d'$) differ in only one element: one $n$ in $b$ (resp.~$d$) is replaced by $1$ in $b'$ (resp.~$d'$).
Similar things also happen in case~(2). In this case $a=a'$ and $d=d'$, while $b$ and $b'$ (resp.~$c$ and $c'$)
differ in one element in the same way as case~(1).
As in the above figures, to help compare the left and right scattering diagrams, we mark the edges where the data changes in red (resp.\ blue) on the left (resp.\ right) and those that it does not in black.
\end{example}

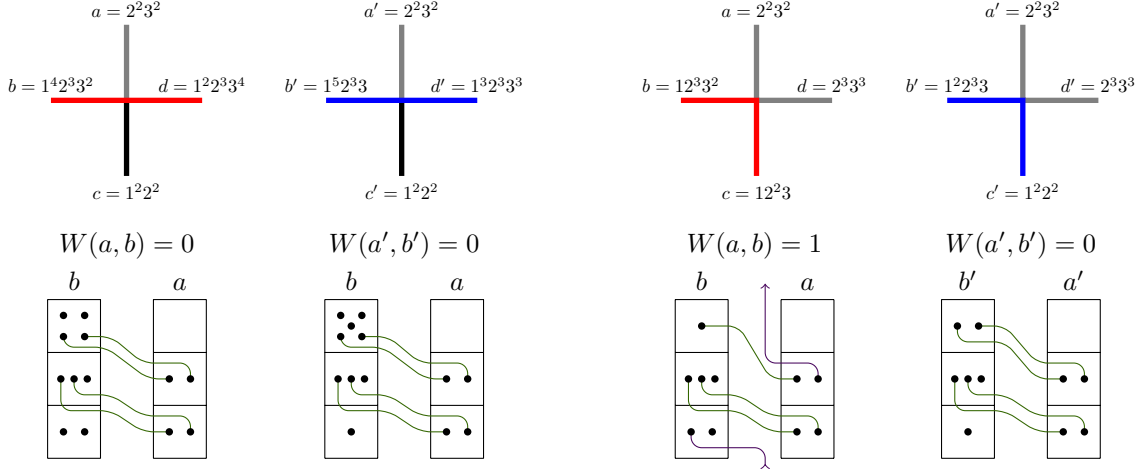
\begin{figure}
    \[
    \begin{array}{cc@{\hspace{40pt}}cc}
    \begin{tikzpicture}[scale=1,baseline=0]
    \draw[-,line width=2pt,color=gray] (0,1) node[anchor=south,color=black,scale=.7] {$a=2^2 3^2$} -- (0,0);
    \draw[-,line width=2pt] (0,0) -- (0,-1) node[anchor=north,color=black,scale=.7] {$c=1^2 2^2$};
    \draw[-,color=red,line width=2pt] (-1,0) node[anchor=south,color=black,scale=.7] {$b=1^4 2^3 3^2$} -- (1,0) node[anchor=south,color=black,scale=.7] {$d=1^2 2^3 3^4$};
    \end{tikzpicture}
    &
    \begin{tikzpicture}[scale=1,baseline=0]
    \draw[-,line width=2pt,color=gray] (0,1) node[anchor=south,color=black,scale=.7] {$a'=2^2 3^2$} -- (0,0);
    \draw[-,line width=2pt] (0,0)-- (0,-1) node[anchor=north,color=black,scale=.7] {$c'=1^2 2^2$};
    \draw[-,color=blue,line width=2pt] (-1,0) node[anchor=south,color=black,scale=.7] {$b'=1^5 2^3 3$} -- (1,0) node[anchor=south,color=black,scale=.7] {$d'=1^3 2^3 3^3$};
    \end{tikzpicture}
    & 
    \begin{tikzpicture}[scale=1,baseline=0]
    \draw[-,line width=2pt,color=gray] (0,1) node[anchor=south,color=black,scale=.7] {$a=2^2 3^2$} -- (0,0) -- (1,0) node[anchor=south,color=black,scale=.7] {$d=2^3 3^3$};
    \draw[-,color=red,line width=2pt] (-1,0) node[anchor=south,color=black,scale=.7] {$b=1 2^3 3^2$} -- (0,0) -- (0,-1) node[anchor=north,color=black,scale=.7] {$c=1 2^2 3$};
    \end{tikzpicture}
    &
    \begin{tikzpicture}[scale=1,baseline=0]
    \draw[-,line width=2pt,color=gray] (0,1) node[anchor=south,color=black,scale=.7] {$a'=2^2 3^2$} -- (0,0) -- (1,0) node[anchor=south,color=black,scale=.7] {$d'=2^3 3^3$};
    \draw[-,color=blue,line width=2pt] (-1,0) node[anchor=south,color=black,scale=.7] {$b'=1^2 2^3 3$} -- (0,0) -- (0,-1) node[anchor=north,color=black,scale=.7] {$c'=1^2 2^2$};
    \end{tikzpicture}
    \\[45pt]
    W(a,b) = 0
    &
    W(a',b') = 0
    &
    W(a,b) = 1
    &
    W(a',b') = 0
    \\
    \begin{tikzpicture}[scale=.7,>=to,baseline=0]
      \draw (.5,0) node[anchor=south] {$b$};
      \draw (0,0) grid (1,-3);
      \draw (2,0) grid (3,-3);
      \draw[-,rounded corners,dgreencolor] (.3,-.7) -- (.3,-.9) -- (1.1,-.9) -- (1.9,-1.5) -- (2.3,-1.5);
      \draw[-,rounded corners,dgreencolor] (.7,-.7) -- (1.2,-.7) -- (1.8,-1.2) -- (2.7,-1.2) -- (2.7,-1.5);
      \draw[-,rounded corners,dgreencolor] (.5,-1.5) -- (.5,-1.8) -- (1.2,-1.8) -- (1.8,-2.2) -- (2.7,-2.2) -- (2.7,-2.5);
      \draw[-,rounded corners,dgreencolor] (.25,-1.5) -- (.25,-2.1) -- (1.2,-2.1) -- (1.8,-2.5) -- (2.3,-2.5);
      \foreach \x/\y in {.3/-2.5,.7/-2.5, .25/-1.5,.5/-1.5,.75/-1.5, .3/-.3,.7/-.3,.3/-.7,.7/-.7}
        \fill[black] (\x,\y) circle (0.07);
      \draw (2.5,0) node[anchor=south] {$a$};
      \foreach \x/\y in {2.3/-2.5, 2.7/-2.5, 2.3/-1.5, 2.7/-1.5}
        \fill[black] (\x,\y) circle (0.07);
    \end{tikzpicture}
    &
    \begin{tikzpicture}[scale=.7,>=to,baseline=0]
      \draw (.5,0) node[anchor=south] {$b$};
      \draw (0,0) grid (1,-3);
      \draw (2,0) grid (3,-3);
      \draw[-,rounded corners,dgreencolor] (.3,-.7) -- (.3,-.9) -- (1.1,-.9) -- (1.9,-1.5) -- (2.3,-1.5);
      \draw[-,rounded corners,dgreencolor] (.7,-.7) -- (1.2,-.7) -- (1.8,-1.2) -- (2.7,-1.2) -- (2.7,-1.5);
      \draw[-,rounded corners,dgreencolor] (.5,-1.5) -- (.5,-1.8) -- (1.2,-1.8) -- (1.8,-2.2) -- (2.7,-2.2) -- (2.7,-2.5);
      \draw[-,rounded corners,dgreencolor] (.25,-1.5) -- (.25,-2.1) -- (1.2,-2.1) -- (1.8,-2.5) -- (2.3,-2.5);
      \foreach \x/\y in {.5/-2.5, .25/-1.5,.5/-1.5,.75/-1.5, .3/-.3,.7/-.3,.3/-.7,.7/-.7,.5/-.5}
        \fill[black] (\x,\y) circle (0.07);
      \draw (2.5,0) node[anchor=south] {$a$};
      \foreach \x/\y in {2.3/-2.5, 2.7/-2.5, 2.3/-1.5, 2.7/-1.5}
        \fill[black] (\x,\y) circle (0.07);
    \end{tikzpicture}
    &
    \begin{tikzpicture}[scale=.7,>=to,baseline=0]
      \draw (.5,0) node[anchor=south] {$b$};
      \draw (0,0) grid (1,-3);
      \draw (2,0) grid (3,-3);
      \draw[-,rounded corners,dgreencolor] (.5,-.5) -- (1.2,-.5) -- (1.8,-1.5) -- (2.3,-1.5);
      \draw[-,rounded corners,dgreencolor] (.5,-1.5) -- (.5,-1.8) -- (1.2,-1.8) -- (1.8,-2.2) -- (2.7,-2.2) -- (2.7,-2.5);
      \draw[-,rounded corners,dgreencolor] (.25,-1.5) -- (.25,-2.1) -- (1.2,-2.1) -- (1.8,-2.5) -- (2.3,-2.5);
      \draw[-<,rounded corners,UQpurple] (.3,-2.5) -- (.3,-2.8) -- (1.7,-2.8) -- (1.7,-3.2);
      \draw[<-,rounded corners,UQpurple] (1.7,.3) -- (1.7,-1.2) -- (2.7,-1.2) -- (2.7,-1.5);
      \foreach \x/\y in {.5/-.5, .25/-1.5,.5/-1.5,.75/-1.5, .3/-2.5,.7/-2.5}
        \fill[black] (\x,\y) circle (0.07);
      \draw (2.5,0) node[anchor=south] {$a$};
      \foreach \x/\y in {2.3/-2.5, 2.7/-2.5, 2.3/-1.5, 2.7/-1.5}
        \fill[black] (\x,\y) circle (0.07);
    \end{tikzpicture}
    &
    \begin{tikzpicture}[scale=.7,>=to,baseline=0]
      \draw (.5,0) node[anchor=south] {$b'$};
      \draw (0,0) grid (1,-3);
      \draw (2,0) grid (3,-3);
      \draw[-,rounded corners,dgreencolor] (.3,-.5) -- (.3,-.8) -- (1.2,-.8) -- (1.8,-1.5) -- (2.3,-1.5);
      \draw[-,rounded corners,dgreencolor] (.7,-.5) -- (1.2,-.5) -- (1.8,-1.2) -- (2.7,-1.2) -- (2.7,-1.5);
      \draw[-,rounded corners,dgreencolor] (.5,-1.5) -- (.5,-1.8) -- (1.2,-1.8) -- (1.8,-2.2) -- (2.7,-2.2) -- (2.7,-2.5);
      \draw[-,rounded corners,dgreencolor] (.25,-1.5) -- (.25,-2.1) -- (1.2,-2.1) -- (1.8,-2.5) -- (2.3,-2.5);
      \foreach \x/\y in {.5/-2.5, .25/-1.5,.5/-1.5,.75/-1.5, .3/-.5,.7/-.5}
        \fill[black] (\x,\y) circle (0.07);
      \draw (2.5,0) node[anchor=south] {$a'$};
      \foreach \x/\y in {2.3/-2.5, 2.7/-2.5, 2.3/-1.5, 2.7/-1.5}
        \fill[black] (\x,\y) circle (0.07);
    \end{tikzpicture}
    \end{array}
    \]
    \caption{Examples of scattering diagrams for~\cref{lem:b'b} with $n = 3$.}
    \label{fig:scattering0-}
\end{figure}

\begin{proof}[Proof of~\Cref{lem:b'b}]
    As in the proof of \cref{lem:a'ad'd}, we give the proof by considering the algorithm in \cref{def:Nakayashiki_yamada}.
    We compare two box and dots diagrams describing $a$, $b$ and $a'$, $b'$.
    Note that the dots in the left sides associated with $b$ and $b'$ can be classified into two types, the ones connected with the dots in the right sides associated with $a$ and $a'$, and the disconnected ones. 
    
    First we consider the case (1).
    To show {\it C.1}-{\it C.3}, it is sufficient to prove the following claim about the two types of dots:
    \begin{equation}
        \label{eq:b'b_claim1}
        \begin{minipage}{.8\textwidth}
        All connected dots for $b$ are exactly the same as those for $b'$, and they do not includes the dots on the bottom box representing $n$s.
        \end{minipage}
        \tag{$\ast$}
    \end{equation}
    The claim~\eqref{eq:b'b_claim1} is clear from the algorithm.
    From  {\it A.3}, we immediately see that in the box and dots diagram for $a$ and $b$, all dots in the left bottom box are not connected since otherwise we have $W(a\otimes b)>0$.
    Thus we find that even if we remove one dot in the left bottom box and put it in the left top box, the configuration of the connected dots  does not change.
    This leads to the proof of~\eqref{eq:b'b_claim1} and thus the case (1).

    Next we move to the case (2).
    To show {\it C.4} and {\it C.5}, it is sufficient to show the following claim about the dots associated with $b$ and $b'$: 
    \begin{equation}
        \label{eq:b'b_claim2}
        \begin{minipage}{.8\textwidth}
        All unconnected dots about $b$ is exactly the same as those about $b'$ and they do not include the dots in the top box representing $1$s.
        \end{minipage}
        \tag{$\ast\ast$}
    \end{equation}
     
     It is not difficult to check this property.
     As in the proof of \cref{lem:a'ad'd}, we allocate the numbers $1,2,\ldots$ to the dots on the right side associated with $a$ and $a'$ and connect them to the dots on the left side associated with $b$ and $b'$ in the increasing order of the labels.
     In \cref{fig:boxdot-1}, we show the example for the case $a=a'=1223333$, $b=112233344$, and $b'=111223334$.
     There we mark the labels in red.

    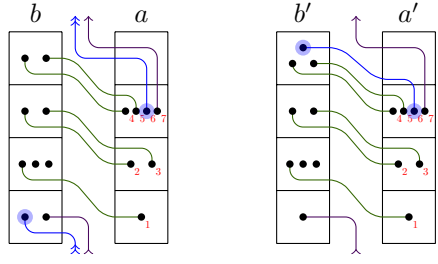
\begin{figure}[t]
    \iftikz
    \[
    \begin{tikzpicture}[scale=.7,>=to]
      \draw (.5,0) node[anchor=south] {$b$};
      \draw (0,0) grid (1,-4);
      \draw (2,0) grid (3,-4);
      \draw[-,rounded corners,dgreencolor] (.3,-0.5) -- (.3,-0.8) -- (1.2,-0.8) -- (1.8,-1.5) -- (2.2,-1.5);
      \draw[-,rounded corners,dgreencolor] (.7,-0.5) -- (1.2,-0.5) -- (1.8,-1.2) -- (2.4,-1.2) -- (2.4,-1.5);
      \draw[-<<,rounded corners,blue] (.3,-3.5) -- (.3,-3.8) -- (1.25,-3.8) -- (1.25,-4.2);
      \draw[<<-,rounded corners,blue] (1.25,.3) -- (1.25,-.5) -- (2.6,-.5) -- (2.6,-1.5);
      \draw[-<,rounded corners,UQpurple] (.7,-3.5) -- (1.5,-3.5) -- (1.5,-4.2);
      \draw[<-,rounded corners,UQpurple] (1.5,.3) -- (1.5,-.3) -- (2.8,-.3) -- (2.8,-1.5);
      \draw[-,rounded corners,dgreencolor] (.3,-1.5) -- (.3,-1.8) -- (1.2,-1.8) -- (1.8,-2.5) -- (2.3,-2.5);
      \draw[-,rounded corners,dgreencolor] (.7,-1.5) -- (1.2,-1.5) -- (1.8,-2.2) -- (2.7,-2.2) -- (2.7,-2.5);
      \draw[-,rounded corners,dgreencolor] (.25,-2.5) -- (.25,-2.8) -- (1.2,-2.8) -- (1.8,-3.5) -- (2.5,-3.5);
      \foreach \x/\y in {.3/-.5,.7/-.5, .3/-1.5,.7/-1.5, .25/-2.5,.5/-2.5,.75/-2.5, .3/-3.5,.7/-3.5}
        \fill[black] (\x,\y) circle (0.07);
      \fill[blue,opacity=.3] (.3,-3.5) circle (.15);
      \draw (2.5,0) node[anchor=south] {$a$};
      \foreach \x/\y [count=\i] in {2.5/-3.5, 2.3/-2.5, 2.7/-2.5, 2.2/-1.5, 2.4/-1.5, 2.6/-1.5, 2.8/-1.5}
        \fill[black] (\x,\y) circle (0.07) node[anchor=north west, color=red, scale=.4, node distance=8pt] {$\i$};
      \fill[blue,opacity=.3] (2.6,-1.5) circle (.15);
    \end{tikzpicture}
    \qquad\qquad
    \begin{tikzpicture}[scale=.7,>=to]
      \draw (.5,0) node[anchor=south] {$b'$};
      \draw (0,0) grid (1,-4);
      \draw (2,0) grid (3,-4);
      \draw[-,rounded corners,dgreencolor] (.3,-0.6) -- (.3,-0.8) -- (1.2,-0.8) -- (1.8,-1.5) -- (2.2,-1.5);
      \draw[-,rounded corners,dgreencolor] (.7,-0.6) -- (1.2,-0.6) -- (1.8,-1.2) -- (2.4,-1.2) -- (2.4,-1.5);
      \draw[-,rounded corners,blue] (.5,-.3) -- (1.2,-.3) -- (1.9,-.9) -- (2.6,-.9) -- (2.6,-1.5);
      \draw[-<,rounded corners,UQpurple] (.5,-3.5) -- (1.5,-3.5) -- (1.5,-4.2);
      \draw[<-,rounded corners,UQpurple] (1.5,.3) -- (1.5,-.3) -- (2.8,-.3) -- (2.8,-1.5);
      \draw[-,rounded corners,dgreencolor] (.3,-1.5) -- (.3,-1.8) -- (1.2,-1.8) -- (1.8,-2.5) -- (2.3,-2.5);
      \draw[-,rounded corners,dgreencolor] (.7,-1.5) -- (1.2,-1.5) -- (1.8,-2.2) -- (2.7,-2.2) -- (2.7,-2.5);
      \draw[-,rounded corners,dgreencolor] (.25,-2.5) -- (.25,-2.8) -- (1.2,-2.8) -- (1.8,-3.5) -- (2.5,-3.5);
      \foreach \x/\y in {.3/-.6,.7/-.6,.5/-.3, .3/-1.5,.7/-1.5, .25/-2.5,.5/-2.5,.75/-2.5, .5/-3.5}
        \fill[black] (\x,\y) circle (0.07);
      \fill[blue,opacity=.3] (.5,-.3) circle (.15);
      \draw (2.5,0) node[anchor=south] {$a'$};
      \foreach \x/\y [count=\i] in {2.5/-3.5, 2.3/-2.5, 2.7/-2.5, 2.2/-1.5, 2.4/-1.5, 2.6/-1.5, 2.8/-1.5}
        \fill[black] (\x,\y) circle (0.07) node[anchor=north west, color=red, scale=.4, node distance=8pt] {$\i$};
      \fill[blue,opacity=.3] (2.6,-1.5) circle (.15);
    \end{tikzpicture}
    \]
    \fi
    \caption{Example of how the winding number changes after applying $f_0$.}
    \label{fig:boxdot-1}
    \end{figure}

    Suppose that for the boxes and dots diagram about $a$ and $b$, the winding appears for the first time in the $k$th connection.
    We see from the algorithm that the dot labeled $k$ is connected with the dot in the bottom box on the left side while in the diagram associated with $a'$ and $b'$, the dot with the same label $k$ is connected with the dot in the top box corresponding to the element 1 in $a'$ which appears as the replacement of the element $n$ in $a$.
    In \cref{fig:boxdot-1}, $k=6$ and we mark the associated dots in blue.
    In addition, we see from the algorithm that the remaining dots have exactly the same pairings by the lines between two diagrams about $a$, $b$ and $a'$, $b'$.
    Thus the first statement in~\eqref{eq:b'b_claim2} follows.
    The second statement in~\eqref{eq:b'b_claim2} is clear from the algorithm and the assumption in {\it A.1} that $a(=a')$ does not have $1$s since in the diagram about $a$ and $b$, it is impossible to have unwinding pairs leaving unconnected dots in the top left box.
    Combining these facts we show the claim~\eqref{eq:b'b_claim2} and thus {\it C.4} and~{\it C.5}. 

    {\it C.6} also immediately follows from the derivation above.
    In the boxes and dots diagram about $a$ and $b$, the pairing between the dot labeled $k$ and the one in the bottom box is a winding pair while this winding changes to unwinding in the diagram about $a'$ and $b'$ since the dot labeled $k$ is connected with the dot in the top box.
    Combining this with the fact that other pairings are exactly the same between the diagrams about $a$, $b$ and $a'$, $b'$, we have shown {\it C.6}. 
\end{proof}

Using \cref{lem:a'ad'd} and \cref{lem:b'b}, we obtain the proof of  \cref{th:leading_f0}.

\begin{proof}[Alternative proof of \cref{th:leading_f0}]
From \cref{lem:f0barb}, we see that $\overline{H}$ and $f_0(\overline{H})$ differ in only one factor.
Suppose that the $k$th factors from the right are the different ones.
As stated in the proof of \cref{lem:f0barb}, the $1$st through $(k-1)$th factors do not have the element $n$s while $(k+1)$th through $\ell$ factors do not have the element $1$s.
Thus the input data of the scattering diagrams of $\overline{H}$ and $f_0(\overline{H})$ can be expressed graphically as
\begin{equation}
\label{fig:inputdata}
    \begin{tikzpicture}[xscale=.7,yscale=.5,every node/.style={scale=0.7},baseline=-2cm]
      \foreach \i [count=\y] in {1,2,4,5,6,7,9} {
        \draw[-,rounded corners] (\i,0) -- (\i,-\y) -- (10,-\y);
      }
    \draw (1,0) node[anchor=south] {$b_1$};
    \draw[-,line width=2pt] (1,0) -- (1,-1);
    \draw (2,0) node[anchor=south] {$b_2$};
    \draw[-,line width=2pt] (2,0) -- (2,-1);
    \draw (4,0) node[anchor=south] {$b_{k-1}$};
    \draw[-,line width=2pt] (4,0) -- (4,-1);
    \draw (5,0) node[anchor=south] {$b_k$};
    \draw[-,red,line width=2pt] (5,0) -- (5,-1);
    \draw (6,0) node[anchor=south] {$b_{k+1}$};
    \draw[-,gray,line width=2pt] (6,0) -- (6,-1);
    \draw (7,0) node[anchor=south] {$b_{k+2}$};\draw[-,gray,line width=2pt] (7,0) -- (7,-1);
    \draw (9,0) node[anchor=south] {$b_{\ell}$};\draw[-,gray,line width=2pt] (9,0) -- (9,-1);
    \draw (3,0) node[anchor=south] {$\cdots$};
    \draw (3,-1) node[anchor=south] {$\cdots$};
    \draw (3,-2) node[anchor=south] {$\cdots$};
    \foreach \i in {0,1,2,4,5}
      \draw (8,-\i) node[anchor=south] {$\cdots$};
    \foreach \i in {4,5,6,9}
      \draw (\i+.3,-2.4) node[anchor=west] {$\vdots$};
    \draw (8,-5.4) node {$\ddots$};
    \draw (9.3,-5.4) node[anchor=west] {$\vdots$};
    \draw (8,-2.4) node {$\ddots$};
    \end{tikzpicture}
    \qquad\qquad
    \begin{tikzpicture}[xscale=.7,yscale=.5,every node/.style={scale=0.7},baseline=-2cm]
      \foreach \i [count=\y] in {1,2,4,5,6,7,9} {
        \draw[-,rounded corners] (\i,0) -- (\i,-\y) -- (10,-\y);
      }
    \draw (1,0) node[anchor=south] {$b_1$};
    \draw[-,line width=2pt] (1,0) -- (1,-1);
    \draw (2,0) node[anchor=south] {$b_2$};
    \draw[-,line width=2pt] (2,0) -- (2,-1);
    \draw (4,0) node[anchor=south] {$b_{k-1}$};
    \draw[-,line width=2pt] (4,0) -- (4,-1);
    \draw (5,0) node[anchor=south] {$b_k$};
    \draw[-,blue,line width=2pt] (5,0) -- (5,-1);
    \draw (6,0) node[anchor=south] {$b_{k+1}$};
    \draw[-,gray,line width=2pt] (6,0) -- (6,-1);
    \draw (7,0) node[anchor=south] {$b_{k+2}$};\draw[-,gray,line width=2pt] (7,0) -- (7,-1);
    \draw (9,0) node[anchor=south] {$b_{\ell}$};\draw[-,gray,line width=2pt] (9,0) -- (9,-1);
    \draw (3,0) node[anchor=south] {$\cdots$};
    \draw (3,-1) node[anchor=south] {$\cdots$};
    \draw (3,-2) node[anchor=south] {$\cdots$};
    \foreach \i in {0,1,2,4,5}
      \draw (8,-\i) node[anchor=south] {$\cdots$};
    \foreach \i in {4,5,6,9}
      \draw (\i+.3,-2.4) node[anchor=west] {$\vdots$};
    \draw (8,-5.4) node {$\ddots$};
    \draw (9.3,-5.4) node[anchor=west] {$\vdots$};
    \draw (8,-2.4) node {$\ddots$};
    \end{tikzpicture}
\end{equation}
Using \cref{lem:a'ad'd,lem:b'b}, we will show that under the scatterings with the input data as in \eqref{fig:inputdata} the red and blue edges form a common path starting from the $k$th top vertical edge from the left ending at some right horizontal edge.
  
First we apply \cref{lem:a'ad'd} to the vertices adjacent to the red and blue edges.
Note that since all factors from the first through $(k-1)$th do not contain an $n$, so all edges from the left of the edges colored in red or blue are colored in black.
Thus applying \cref{lem:a'ad'd} to the scattering diagram depiction, we get the diagram
\[
    \begin{tikzpicture}[xscale=.55,yscale=.4,every node/.style={scale=0.6},baseline=-10]
    \draw[-] (0.5,0) node[anchor=east] {$\cdots$} -- (4.5,0) node[anchor=west] {$\cdots$};
    \draw[-,line width=2pt] (0.5,0) -- (2,0);
    \draw[-] (0.5,-1) node[anchor=east] {$\cdots$} -- (4.5,-1) node[anchor=west] {$\cdots$};
    \draw[-,line width=2pt] (0.5,-1) -- (2,-1);
    \draw[-,line width=2pt] (1,1) node[anchor=south] {$b_{k-1}$} -- (1,-1.5) node[anchor=north] {$\vdots$};
    \draw[-] (2,1) node[anchor=south] {$b_k$} -- (2,-1.5) node[anchor=north] {$\vdots$};
    \draw[-,line width=2pt,red] (2,1) -- (2,0);
    \draw[-] (3,1) node[anchor=south] {$b_{k+1}$} -- (3,-1.5) node[anchor=north] {$\vdots$};
    \draw[-,line width=2pt,gray] (3,1) -- (3,0);
    \draw[-] (4,1) node[anchor=south] {$b_{k+2}$} -- (4,-1.5) node[anchor=north] {$\vdots$};
    \draw[-,line width=2pt,gray] (4,1) -- (4,0);
    \end{tikzpicture}
    \quad
    \begin{tikzpicture}[xscale=.55,yscale=.4,every node/.style={scale=0.6},baseline=-10]
    \draw[-] (0.5,0) node[anchor=east] {$\cdots$} -- (4.5,0) node[anchor=west] {$\cdots$};
    \draw[-,line width=2pt] (0.5,0) -- (2,0);
    \draw[-] (0.5,-1) node[anchor=east] {$\cdots$} -- (4.5,-1) node[anchor=west] {$\cdots$};
    \draw[-,line width=2pt] (0.5,-1) -- (2,-1);
    \draw[-,line width=2pt] (1,1) node[anchor=south] {$b_{k-1}$} -- (1,-1.5) node[anchor=north] {$\vdots$};
    \draw[-] (2,1) node[anchor=south] {$b_k$} -- (2,-1.5) node[anchor=north] {$\vdots$};
    \draw[-,line width=2pt,blue] (2,1) -- (2,0);
    \draw[-] (3,1) node[anchor=south] {$b_{k+1}$} -- (3,-1.5) node[anchor=north] {$\vdots$};
    \draw[-,line width=2pt,gray] (3,1) -- (3,0);
    \draw[-] (4,1) node[anchor=south] {$b_{k+2}$} -- (4,-1.5) node[anchor=north] {$\vdots$};
    \draw[-,line width=2pt,gray] (4,1) -- (4,0);
    \end{tikzpicture}
    \;
    \rightsquigarrow
    \;
    \begin{tikzpicture}[xscale=.55,yscale=.4,every node/.style={scale=0.6},baseline=-10]
    \draw[-] (0.5,0) node[anchor=east] {$\cdots$} -- (4.5,0) node[anchor=west] {$\cdots$};
    \draw[-,line width=2pt] (0.5,0) -- (2,0) -- (2,-1.5) node[anchor=north] {$\vdots$};
    \draw[-] (0.5,-1) node[anchor=east] {$\cdots$} -- (4.5,-1) node[anchor=west] {$\cdots$};
    \draw[-,line width=2pt] (0.5,-1) -- (3,-1);
    \draw[-,line width=2pt] (1,1) node[anchor=south] {$b_{k-1}$} -- (1,-1.5) node[anchor=north] {$\vdots$};
    \draw[-,line width=2pt,red] (2,1) node[anchor=south,color=black] {$b_k$} -- (2,0) -- (3,0);
    \draw[-] (3,1) node[anchor=south] {$b_{k+1}$} -- (3,-1.5) node[anchor=north] {$\vdots$};
    \draw[-,line width=2pt,gray] (3,1) -- (3,0);
    \draw[-] (4,1) node[anchor=south] {$b_{k+2}$} -- (4,-1.5) node[anchor=north] {$\vdots$};
    \draw[-,line width=2pt,gray] (4,1) -- (4,0);
    \end{tikzpicture}
    \quad
    \begin{tikzpicture}[xscale=.55,yscale=.4,every node/.style={scale=0.6},baseline=-10]
    \draw[-] (0.5,0) node[anchor=east] {$\cdots$} -- (4.5,0) node[anchor=west] {$\cdots$};
    \draw[-,line width=2pt] (0.5,0) -- (2,0) -- (2,-1.5) node[anchor=north] {$\vdots$};
    \draw[-] (0.5,-1) node[anchor=east] {$\cdots$} -- (4.5,-1) node[anchor=west] {$\cdots$};
    \draw[-,line width=2pt] (0.5,-1) -- (3,-1);
    \draw[-,line width=2pt] (1,1) node[anchor=south] {$b_{k-1}$} -- (1,-1.5) node[anchor=north] {$\vdots$};
    \draw[-,line width=2pt,blue] (2,1) node[anchor=south,color=black] {$b_k$} -- (2,0) -- (3,0);
    \draw[-] (3,1) node[anchor=south] {$b_{k+1}$} -- (3,-1.5) node[anchor=north] {$\vdots$};
    \draw[-,line width=2pt,gray] (3,1) -- (3,0);
    \draw[-] (4,1) node[anchor=south] {$b_{k+2}$} -- (4,-1.5) node[anchor=north] {$\vdots$};
    \draw[-,line width=2pt,gray] (4,1) -- (4,0);
    \end{tikzpicture}
\]
  
In the next step, if the local energy in the vertex labeled $(\dagger)$ in
\begin{subequations}
\begin{align}
\label{fig:sc0}
    \begin{tikzpicture}[xscale=.55,yscale=.4,every node/.style={scale=0.6},baseline=-10]
    \draw[-] (0.5,0) node[anchor=east] {$\cdots$} -- (4.5,0) node[anchor=west] {$\cdots$};
    \draw[-,line width=2pt] (0.5,0) -- (2,0) -- (2,-1.5) node[anchor=north] {$\vdots$};
    \draw[-] (0.5,-1) node[anchor=east] {$\cdots$} -- (4.5,-1) node[anchor=west] {$\cdots$};
    \draw[-,line width=2pt] (0.5,-1) -- (3,-1);
    \draw[-,line width=2pt] (1,1) node[anchor=south] {$b_{k-1}$} -- (1,-1.5) node[anchor=north] {$\vdots$};
    \draw[-] (3,1) node[anchor=south] {$b_{k+1}$} -- (3,-1.5) node[anchor=north] {$\vdots$};
    \draw[-,line width=2pt,gray] (3,1) -- (3,0) node[anchor=south west,black,scale=.7] {$(\ast)$};
    \draw[-] (4,1) node[anchor=south] {$b_{k+2}$} -- (4,-1.5) node[anchor=north] {$\vdots$};
    \draw[-,line width=2pt,gray] (4,1) -- (4,0);
    \draw[-,line width=2pt,red] (2,1) node[anchor=south,color=black] {$b_k$} -- (2,0) -- (3,0);
    \end{tikzpicture}
    \quad
    \begin{tikzpicture}[xscale=.55,yscale=.4,every node/.style={scale=0.6},baseline=-10]
    \draw[-] (0.5,0) node[anchor=east] {$\cdots$} -- (4.5,0) node[anchor=west] {$\cdots$};
    \draw[-,line width=2pt] (0.5,0) -- (2,0) -- (2,-1.5) node[anchor=north] {$\vdots$};
    \draw[-] (0.5,-1) node[anchor=east] {$\cdots$} -- (4.5,-1) node[anchor=west] {$\cdots$};
    \draw[-,line width=2pt] (0.5,-1) -- (3,-1);
    \draw[-,line width=2pt] (1,1) node[anchor=south] {$b_{k-1}$} -- (1,-1.5) node[anchor=north] {$\vdots$};
    \draw[-] (3,1) node[anchor=south] {$b_{k+1}$} -- (3,-1.5) node[anchor=north] {$\vdots$};
    \draw[-,line width=2pt,gray] (3,1) -- (3,0) node[anchor=south west,black,scale=.7] {$(\ast)$};
    \draw[-] (4,1) node[anchor=south] {$b_{k+2}$} -- (4,-1.5) node[anchor=north] {$\vdots$};
    \draw[-,line width=2pt,gray] (4,1) -- (4,0);
    \draw[-,line width=2pt,blue] (2,1) node[anchor=south,color=black] {$b_k$} -- (2,0) -- (3,0);
    \end{tikzpicture}
    \;
    & \rightsquigarrow
    \;
    \begin{tikzpicture}[xscale=.55,yscale=.4,every node/.style={scale=0.6},baseline=-10]
    \draw[-] (0.5,0) node[anchor=east] {$\cdots$} -- (4.5,0) node[anchor=west] {$\cdots$};
    \draw[-,line width=2pt] (0.5,0) -- (2,0) -- (2,-1.5) node[anchor=north] {$\vdots$};
    \draw[-] (0.5,-1) node[anchor=east] {$\cdots$} -- (4.5,-1) node[anchor=west] {$\cdots$};
    \draw[-,line width=2pt] (3,0) -- (3,-1.5) node[anchor=north] {$\vdots$};
    \draw[-,line width=2pt] (0.5,-1) -- (4,-1);
    \draw[-,line width=2pt] (1,1) node[anchor=south] {$b_{k-1}$} -- (1,-1.5) node[anchor=north] {$\vdots$};
    \draw[-,line width=2pt,gray] (3,1) node[anchor=south,black] {$b_{k+1}$} -- (3,0) node[anchor=south west,black,scale=.7] {$(\dagger)$};
    \draw[-] (4,1) node[anchor=south] {$b_{k+2}$} -- (4,-1.5) node[anchor=north] {$\vdots$};
    \draw[-,line width=2pt,gray] (4,1) -- (4,0);
    \draw[-,line width=2pt,red] (2,1) node[anchor=south,color=black] {$b_k$} -- (2,0) -- (4,0);
    \end{tikzpicture}
    \quad
    \begin{tikzpicture}[xscale=.55,yscale=.4,every node/.style={scale=0.6},baseline=-10]
    \draw[-] (0.5,0) node[anchor=east] {$\cdots$} -- (4.5,0) node[anchor=west] {$\cdots$};
    \draw[-,line width=2pt] (0.5,0) -- (2,0) -- (2,-1.5) node[anchor=north] {$\vdots$};
    \draw[-] (0.5,-1) node[anchor=east] {$\cdots$} -- (4.5,-1) node[anchor=west] {$\cdots$};
    \draw[-,line width=2pt] (3,0) -- (3,-1.5) node[anchor=north] {$\vdots$};
    \draw[-,line width=2pt] (0.5,-1) -- (4,-1);
    \draw[-,line width=2pt] (1,1) node[anchor=south] {$b_{k-1}$} -- (1,-1.5) node[anchor=north] {$\vdots$};
    \draw[-,line width=2pt,gray] (3,1) node[anchor=south,black] {$b_{k+1}$} -- (3,0) node[anchor=south west,black,scale=.7] {$(\dagger)$};
    \draw[-] (4,1) node[anchor=south] {$b_{k+2}$} -- (4,-1.5) node[anchor=north] {$\vdots$};
    \draw[-,line width=2pt,gray] (4,1) -- (4,0);
    \draw[-,line width=2pt,blue] (2,1) node[anchor=south,color=black] {$b_k$} -- (2,0) -- (4,0);
    \end{tikzpicture}
\\
\label{fig:sc-1}
    \begin{tikzpicture}[xscale=.55,yscale=.4,every node/.style={scale=0.6},baseline=-10]
    \draw[-] (0.5,0) node[anchor=east] {$\cdots$} -- (4.5,0) node[anchor=west] {$\cdots$};
    \draw[-,line width=2pt] (0.5,0) -- (2,0) -- (2,-1.5) node[anchor=north] {$\vdots$};
    \draw[-] (0.5,-1) node[anchor=east] {$\cdots$} -- (4.5,-1) node[anchor=west] {$\cdots$};
    \draw[-,line width=2pt] (0.5,-1) -- (3,-1) node[anchor=south west,scale=.7] {$(\ddag)$};
    \draw[-,line width=2pt] (1,1) node[anchor=south] {$b_{k-1}$} -- (1,-1.5) node[anchor=north] {$\vdots$};
    \draw[-] (3,1) node[anchor=south] {$b_{k+1}$} -- (3,-1.5) node[anchor=north] {$\vdots$};
    \draw[-,line width=2pt,gray] (3,1) -- (3,0) node[anchor=south west,black,scale=.7] {$(\dagger)$};
    \draw[-] (4,1) node[anchor=south] {$b_{k+2}$} -- (4,-1.5) node[anchor=north] {$\vdots$};
    \draw[-,line width=2pt,gray] (4,1) -- (4,0);
    \draw[-,line width=2pt,red] (2,1) node[anchor=south,color=black] {$b_k$} -- (2,0) -- (3,0);
    \end{tikzpicture}
    \quad
    \begin{tikzpicture}[xscale=.55,yscale=.4,every node/.style={scale=0.6},baseline=-10]
    \draw[-] (0.5,0) node[anchor=east] {$\cdots$} -- (4.5,0) node[anchor=west] {$\cdots$};
    \draw[-,line width=2pt] (0.5,0) -- (2,0) -- (2,-1.5) node[anchor=north] {$\vdots$};
    \draw[-] (0.5,-1) node[anchor=east] {$\cdots$} -- (4.5,-1) node[anchor=west] {$\cdots$};
    \draw[-,line width=2pt] (0.5,-1) -- (3,-1) node[anchor=south west,scale=.7] {$(\ddag)$};
    \draw[-,line width=2pt] (1,1) node[anchor=south] {$b_{k-1}$} -- (1,-1.5) node[anchor=north] {$\vdots$};
    \draw[-] (3,1) node[anchor=south] {$b_{k+1}$} -- (3,-1.5) node[anchor=north] {$\vdots$};
    \draw[-,line width=2pt,gray] (3,1) -- (3,0) node[anchor=south west,black,scale=.7] {$(\dagger)$};
    \draw[-] (4,1) node[anchor=south] {$b_{k+2}$} -- (4,-1.5) node[anchor=north] {$\vdots$};
    \draw[-,line width=2pt,gray] (4,1) -- (4,0);
    \draw[-,line width=2pt,blue] (2,1) node[anchor=south,color=black] {$b_k$} -- (2,0) -- (3,0);
    \end{tikzpicture}
    \;
    &\rightsquigarrow
    \;
    \begin{tikzpicture}[xscale=.55,yscale=.4,every node/.style={scale=0.6},baseline=-10]
    \draw[-,line width=2pt] (0.5,0) node[anchor=east] {$\cdots$} -- (2,0) -- (2,-1.5) node[anchor=north] {$\vdots$};
    \draw[-] (0.5,-1) node[anchor=east] {$\cdots$} -- (4.5,-1) node[anchor=west] {$\cdots$};
    \draw[-,line width=2pt] (3,0) -- (3,-1.5) node[anchor=north] {$\vdots$};
    \draw[-,line width=2pt] (0.5,-1) -- (3,-1) node[anchor=south west,scale=.7] {$(\ddag)$};
    \draw[-,line width=2pt] (1,1) node[anchor=south] {$b_{k-1}$} -- (1,-1.5) node[anchor=north] {$\vdots$};
    \draw[-,line width=2pt,gray] (3,1) node[anchor=south,black] {$b_{k+1}$} -- (3,0) -- (4.5,0) node[anchor=west,black] {$\cdots$};
    \draw[-,line width=2pt,gray] (4,1) node[anchor=south,black] {$b_{k+2}$} -- (4,-1.5) node[anchor=north,black] {$\vdots$};
    \draw[-,line width=2pt,red] (2,1) node[anchor=south,color=black] {$b_k$} -- (2,0) -- (3,0) node[anchor=south west,black,scale=.7] {$(\dagger)$} -- (3,-1) -- (4,-1);
    \end{tikzpicture}
    \quad
    \begin{tikzpicture}[xscale=.55,yscale=.4,every node/.style={scale=0.6},baseline=-10]
    \draw[-,line width=2pt] (0.5,0) node[anchor=east] {$\cdots$} -- (2,0) -- (2,-1.5) node[anchor=north] {$\vdots$};
    \draw[-] (0.5,-1) node[anchor=east] {$\cdots$} -- (4.5,-1) node[anchor=west] {$\cdots$};
    \draw[-,line width=2pt] (3,0) -- (3,-1.5) node[anchor=north] {$\vdots$};
    \draw[-,line width=2pt] (0.5,-1) -- (3,-1) node[anchor=south west,scale=.7] {$(\ddag)$};
    \draw[-,line width=2pt] (1,1) node[anchor=south] {$b_{k-1}$} -- (1,-1.5) node[anchor=north] {$\vdots$};
    \draw[-,line width=2pt,gray] (3,1) node[anchor=south,black] {$b_{k+1}$} -- (3,0) -- (4.5,0) node[anchor=west,black] {$\cdots$};
    \draw[-,line width=2pt,gray] (4,1) node[anchor=south,black] {$b_{k+2}$} -- (4,-1.5) node[anchor=north,black] {$\vdots$};
    \draw[-,line width=2pt,blue] (2,1) node[anchor=south,color=black] {$b_k$} -- (2,0) -- (3,0) node[anchor=south west,black,scale=.7] {$(\dagger)$} -- (3,-1) -- (4,-1);
    \end{tikzpicture}
\end{align}
\end{subequations}
is 0, we apply \cref{lem:b'b}(1) to this vertex while if the local energy is positive, we apply \cref{lem:b'b}(2) to the vertex labeled $(\dag)$ then \cref{lem:a'ad'd} to the one labeled $(\ddag)$ in order.
The pictoral descriptions~\eqref{fig:sc0} and~\eqref{fig:sc-1} illustrates the former and latter cases respectively.
Repeating this step, we complete the proof of the formation of the red and blue paths.

The relation \eqref{eq:leading_f0} readily follows from this property. Note that in any red and blue paths,  
\begin{align}\label{eq:num_difference}
\#\left(
\begin{tikzpicture}[scale=.4,baseline=-3]
\draw[-] (-1,0) -- (0,0) -- (0,-1);
\draw[-,color=red,line width=2pt] (0,1) -- (0,0) -- (1,0);
\end{tikzpicture}
\right)
-
\#\left(
\begin{tikzpicture}[scale=.4,baseline=-3]
\draw[-] (0,1) -- (0,0) -- (1,0);
\draw[-,color=red,line width=2pt] (-1,0) -- (0,0) -- (0,-1);
\end{tikzpicture}
\right)
=
\#\left(
\begin{tikzpicture}[scale=.4,baseline=-3]
\draw[-] (-1,0) -- (0,0) -- (0,-1);
\draw[-,color=blue,line width=2pt] (0,1) -- (0,0) -- (1,0);
\end{tikzpicture}
\right)
-
\#\left(
\begin{tikzpicture}[scale=.4,baseline=-3]
\draw[-] (0,1) -- (0,0) -- (1,0);
\draw[-,color=blue,line width=2pt] (-1,0) -- (0,0) -- (0,-1);
\end{tikzpicture}
\right)
= 1
\end{align}
always holds.
Here $\#(\cdot)$ represents the number of the diagram ``$\cdot$'' in the colored paths.
Combining this with the relations about the local energy, {\it C.3} in \cref{lem:a'ad'd} and {\it C.3} and {\it C.6} in \cref{lem:b'b}, we get \eqref{eq:leading_f0}.
\end{proof}

\bibliographystyle{abbrv} 
\bibliography{bib}

\end{document}